\documentclass{article}
\usepackage{tikz}
\usetikzlibrary{shapes,backgrounds}
\usetikzlibrary{positioning} 
\usetikzlibrary{topaths} 
\usepackage{graphicx} 
\usepackage{color}
\usepackage{comment}
\usepackage{booktabs} 
\usepackage{array} 
\usepackage{paralist} 
\usepackage{verbatim} 
\usepackage{subcaption} 
\usepackage{mathdots} 
\usepackage[psdextra]{hyperref}
\usepackage{amsmath}
\usepackage{amsfonts}
\usepackage{amssymb}
\usepackage{amsthm}
\usepackage{tikz-cd}
\usepackage{epigraph}
\usepackage{mathtools}
\newtheorem{theorem}{Theorem}[section]
\newtheorem{lemma}[theorem]{Lemma}
\newtheorem{proposition}[theorem]{Proposition}
\newtheorem{corollary}[theorem]{Corollary}
\newtheorem{definition}[theorem]{Definition}
\newtheorem{example}[theorem]{Examples}

\newtheorem{remark}[theorem]{Remark}

\newtheorem{assumptions}[theorem]{Assumptions}
\usepackage{setspace}
\usepackage[normalem]{ulem}

\usepackage[left=3cm, right=3cm, top=2cm]{geometry}

\DeclareMathOperator*{\argmin}{arg\,min}

\newcommand{\x}[1]{{\color{red} #1}}

\DeclareRobustCommand\longtwoheadrightarrow
     {\relbar\joinrel\twoheadrightarrow}

\usepackage[most]{tcolorbox}
\usetikzlibrary{patterns}
\pgfdeclarepatternformonly{mystrikeout}{\pgfqpoint{-1pt}{-1pt}}{\pgfqpoint{11pt}{11pt}}{\pgfqpoint{10pt}{10pt}}%
{
  \pgfsetlinewidth{0.4pt}
  \pgfpathmoveto{\pgfqpoint{0pt}{0pt}}
  \pgfpathlineto{\pgfqpoint{10.1pt}{10.1pt}}
  \pgfusepath{stroke}
}
\newtcolorbox{soutbox}{breakable,
 enhanced jigsaw,
 opacityback=0,
 parbox=false,
 boxrule=0mm,
 top=0mm,bottom=0pt,left=0pt,right=0pt,
 boxsep=0pt,
 frame hidden,
 finish={\fill[pattern=mystrikeout, pattern color=blue] (frame.north west) rectangle (frame.south east);}
}

\title{An extension to Banach stackings of the Brezis--Pazy semigroup-convergence theorem, with applications to $\lambda$-convex gradient flows}

\author{Samuel Mercer \and Yves van Gennip}

\date{}

\begin{document}

\maketitle

\begin{abstract}
A 1972 theorem by Brezis and Pazy establishes the uniform convergence of nonlinear semigroups generated by $\omega$-accretive operators on a Banach space. Our goal is to expand the setting of this theorem to include nonlinear semigroups that are acting on different Banach spaces. This is useful, for example, to prove discrete-to-continuum convergence for graph-based gradient flows. We name the general setting in which our theorem holds a Banach stacking.

We give three main applications of the extended theorem that are of independent interest. The first establishes uniform convergence of semigroups in a Banach stacking if the generators of the semigroups converge pointwise. The second is a proof of uniform convergence for gradient flows of $\Gamma$-converging $\lambda$-convex functions on a Banach stacking of Hilbert spaces; the third a proof of uniform convergence for gradient flows of $\Gamma$-converging functions that satisfy a convexity condition that was formulated by B\'enilan and Crandall in a 1991 publication (and which we term `$P_0$-convexity') on a Banach stacking of $L^p$ spaces, corresponding to the $TL^p$ space introduced by Garc\'ia Trillos and Slep\v{c}ev in a 2016 paper.
\end{abstract}

\section{Introduction}\label{sec:introduction}

Let $X_\infty$ be a Banach space and, for all $n\in \mathbb{N}$, let $\{S_{A_n}(t)\}_{t\geq 0}$ be a semigroup acting on $X_\infty$ generated by an $\omega$-accretive (Definition~\ref{def:omegamaccretive}) operator $A_n$ on the space $X_\infty$. In \cite[Theorem 3.1]{brezis1972convergence} Brezis and Pazy prove that if the resolvents of the operators $A_n$ converge pointwise, then the corresponding semigroups converge uniformly over any finite interval $[0,T]$. 

We are interested in extending this theorem to a setting in which each operator $A_n$, and thus each generated semigroup $\{S_{A_n}(t)\}_{t\geq 0}$, is defined on a (potentially) different Banach space $X_n$. In order to still have a meaningful concept of convergence, we require each of the spaces $X_n$ to be embedded into a unifying metric space $\mathcal{M}$; naturally we will need a few assumptions on the properties of these embeddings.

We present a new axiomatic structure in Definition~\ref{BanStack} imposed on the collection $(X_1,X_2,...,X_\infty)$, which we call a Banach stacking. One can think of this structure as being similar to a vector bundle; in fact, the second example of a Banach stacking in Examples~\ref{BSexmpl} is built off a vector bundle with a Riemannian metric structure on the base space. Figure~\ref{fig:banach} shows a schematic representation of a Banach stacking, with the functions $\xi_n$ being embeddings of $X_n$ into $\mathcal{M}$, satisfying the requirements in Definition~\ref{BanStack}. In particular we emphasise that the spaces $X_n$, with $n\in \mathbb{N}$, do not need to embed into the space $X_\infty$.

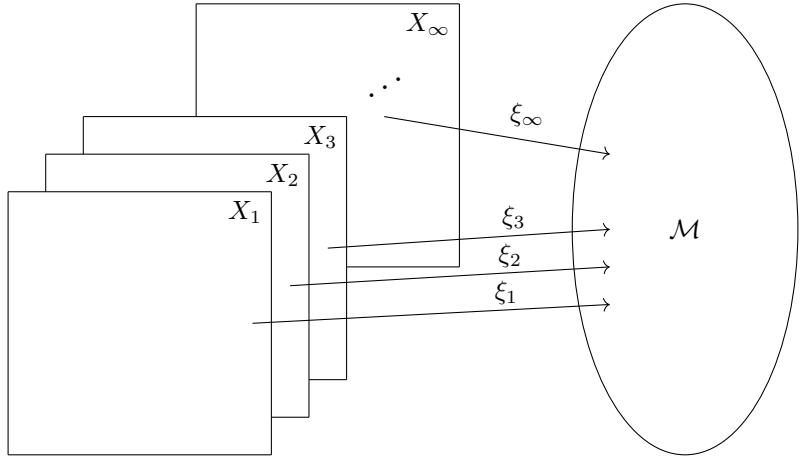
\begin{figure}[ht]
    \begin{center}
    \begin{tikzpicture}[scale=0.5]
\draw (0,0) -- (7,0) ;
\draw (0,0) -- (0,7) ;
\draw (0,7) -- (7,7) ;
\draw (7,7) -- (7,0) ;

\draw (1,7) -- (1,8) ;
\draw (1,8) -- (8,8) ;
\draw (8,8) -- (8,1) ;
\draw (8,1) -- (7,1) ;

\draw (2,8) -- (2,9) ;
\draw (2,9) -- (9,9) ;
\draw (9,9) -- (9,2) ;
\draw (9,2) -- (8,2) ;

\draw (5,9) -- (5,12) ;
\draw (5,12) -- (12,12) ;
\draw (12,12) -- (12,5) ;
\draw (12,5) -- (9,5) ;

\node at (10,10) {$\Large{\iddots}$} ;

\node at (7,7) [anchor=north east] {$X_1$} ;
\node at (8,8) [anchor=north east] {$X_2$} ;
\node at (9,9) [anchor=north east] {$X_3$} ;
\node at (12,12) [anchor=north east] {$X_\infty$} ;

\draw (18,6) ellipse (3 and 6) ;
\node at (18,6) {$\mathcal{M}$} ;

\draw (6.5,3.5) edge[->] node[above,xshift=1cm]{$\xi_1$} (16,4)  ;
\draw (7.5,4.5)  edge[->] node[above,xshift=0.8cm]{$\xi_2$} (16,5) ;
\draw (8.5,5.5)  edge[->] node[above,xshift=0.6cm]{$\xi_3$} (16,6) ;
\draw (10,9) edge[->] node[above,xshift=0.4cm]{$\xi_\infty$} (16,8) ;
\end{tikzpicture}
\end{center}
    \caption{Schematic illustration of a Banach stacking}
    \label{fig:banach}
\end{figure}

With this structure in hand we prove Theorem~\ref{theorem1}, which extends the Brezis--Pazy theorem to Banach stackings.

Our definition for a Banach stacking is primarily motivated by discrete-to-continuum limits. In \cite{garcia2016continuum} the metric space $TL^p(\Omega)$ is introduced to allow for $\Gamma$-convergence results of discrete graph-based functionals to continuum functionals (we defer giving details of $\Omega\subset\mathbb{R}^d$). In Proposition~\ref{tlpBanach} we show that $TL^p(\Omega)$ fits into our Banach stacking structure, which opens the door for the application of our extended Brezis--Pazy theorem to gradient flows in a $TL^p(\Omega)$ setting.

In Proposition~\ref{OpConvergence} we prove uniform convergence of semigroups on intervals $[0,T]$ in a Banach stacking under the assumption that their generating operators converge pointwise. Our other two main applications for this extended Brezis--Pazy theorem avoid this assumption on the operators in favour of other assumptions. They are the convergence results Theorems~\ref{Hilthm} and~\ref{thm:P0theorem1} for gradient flows of functions over a Banach stacking that satisfy certain $\Gamma$-convergence, equicoercivity, and convexity conditions. These results also build upon classical results regarding convergence of gradient flows such as those in \cite{EvoEq} and \cite{SerfatyGam}. 

One notable difference of our results compared to those in \cite{EvoEq} and \cite{SerfatyGam} is that we require no well-preparedness \footnote{Well-preparedness of the initial conditions tends to refer to any additional assumptions on the sequence $(x_n)_{n\in\mathbb{N}}$ of initial conditions beyond its convergence. There is no single specific definition that is used in the literature. For example, if we have a sequence of energies $(\Phi_n)_{n\in\mathbb{N}}$ that $\Gamma$-converges to an energy $\Phi_\infty$, well-preparedness may mean $\Phi_n(x_n)\rightarrow\Phi_\infty(x_\infty)$ as $n\rightarrow \infty$.} of the initial conditions, instead restricting ourselves to $\lambda$-convex functions (see Definition~\ref{def:lambdaconvex}) in Theorem~\ref{Hilthm}, including the non-convex $\lambda<0$ case, and $P_0$-convex functions (see Definition~\ref{def:Pconvex}) in Theorem~\ref{thm:P0theorem1}, and exploiting the energy bound for gradient flows proved in Theorem~\ref{GradFlowTheorem}. In Section~\ref{sec:literature} we will see that both \cite{EvoEq} and \cite{SerfatyGam} require conditions on the functions that are strictly weaker than $\lambda$-convexity. Surprisingly, without well-preparedness assumptions, we are still able to deduce convergence of the `energy' of the gradient flows for every positive time $t>0$.

In Theorem~\ref{Hilthm} the spaces $X_n$ and $X_\infty$ are required to be Hilbert spaces and the functions $\lambda$-convex. We obtain uniform convergence of the gradient flows on closed and bounded time intervals and pointwise convergence of the values of the functions along those flows.

In Theorem~\ref{thm:P0theorem1} we treat the $TL^p(\Omega)$ setting, in which $X_n:=L^p(\Omega;\mu_n)$ and $X_\infty:=L^p(\Omega;\mu_\infty)$ for probability measures $\mu_n$ and $\mu_\infty$ on an open subset $\Omega\subset \mathbb{R}^d$. We require a stronger convexity assumption than in Theorem~\ref{Hilthm} ($P_0$-convexity, which implies convexity for lower semicontinuous functionals, but is itself not implied by $\lambda$-convexity for any $\lambda\geq 0$; see Propositions~\ref{Pprop} and~\ref{prop:counterexample}), but are able to weaken the coercivity conditions. We compare Theorems~\ref{Hilthm} and \ref{thm:P0theorem1} with other results in the literature in further detail in Section~\ref{sec:literature}. 
Additionally, Proposition~\ref{prop:assumpreduce} provides an alternative, stronger set of assumptions to that in Theorem~\ref{thm:P0theorem1} which we expect to be easier to verify in practice.

\subsection{Setting and main results}

\subsubsection{Convex analysis}

In various places in this paper we consider functions on a Hilbert space that are $\lambda$-convex. In other sources $\lambda$-convexity is also called $\omega$-convexity; it is a particular case of $\phi$-convexity defined in \cite{EvoEq}. For some intuition we mention that a function $\Phi$ on a Hilbert space with norm $\|\cdot\|$ is $\lambda$-convex if and only if $\Phi(\cdot)-\frac{\lambda}{2}\Vert \cdot \Vert^2$ is convex. Moreover, if $\lambda \geq 0$ ($\lambda>0$), a $\lambda$-convex function is (strictly) convex. For proofs of these statements we refer to Lemmas~\ref{lem:lambdaconvexequivalent} and~\ref{lem:lambdaconvexanstrictlyconvex}. 

In order to study convergence of gradient flows, a convexity assumption of this type is necessary because of the stability it introduces with respect to initial conditions. This is especially important for discrete-to-continuum limits where the initial condition must be approximated. We illustrate this in Figure~\ref{fig:lambdaconvexity} with the graphs of two functions, one being $\lambda$-convex and the other not. 

\begin{figure}[ht]
\centering
\begin{subcaptionblock}{.5\textwidth}
\centering
\begin{tikzpicture}[scale=0.5]

\draw (0,0)--(5,5) ;
\draw (5,5)--(10,0) ;

\draw[dashed,->] (4,5)--(0,1) ;
\node at (2,3) [anchor=south east] {\scriptsize fast moving} ;
\draw[dashed,->] (6,5)--(10,1) ;
\node at (8,3) [anchor=south west] {\scriptsize fast moving} ;

\end{tikzpicture}
\caption{Non-$\lambda$-convex}
\end{subcaptionblock}%
\begin{subcaptionblock}{.5\textwidth}
\centering
\begin{tikzpicture}[scale=0.5]

\draw (5,5) parabola (0,0) ;
\draw (5,5) parabola (10,0) ;

\draw[dashed,->] (4,5) to[out=190, in =55] (0.5,1.8) ;
\node at (2,5) [anchor=south] { \scriptsize slow moving} ;
\draw[dashed,->] (6,5) to[out=-10, in =125] (9.5,1.8) ;
\node at (8,5) [anchor=south] { \scriptsize slow moving} ;

\end{tikzpicture}
\caption{$\lambda$-convex}
\end{subcaptionblock}%
\caption{Graphs of (a) a function that is not $\lambda$-convex and (b) a function that is $\lambda$-convex for some $\lambda<0$. The dashed lines with arrows show the trajectories of gradient flows of the functions.}\label{fig:lambdaconvexity}
\end{figure} 

We observe it is not possible to deduce uniform convergence of gradient flows without some assumption such as $\lambda$-convexity. For the non-$\lambda$-convex example in Figure~\ref{fig:lambdaconvexity}, two gradient flows that start from arbitrarily close initial conditions located on opposite sides of the maximimizer of the function will, after a time $T$, have moved away from each other in opposite directions over a distance proportional to $T$. In the $\lambda$-convex case the gradient flows still move in opposite directions, but the gradient at the peak is zero. Hence the closer the two initial conditions are together the slower the respective gradient flows move apart from one another. The intuition we gain from this is that $\lambda$-convexity introduces some stability to the gradient flow.

In Theorem~\ref{GradFlowTheorem} we prove that the value $\Phi(u(t))$ along the gradient flow $u$ of a $\lambda$-convex function $\Phi$ (with initial condition $x_0$) is bounded above by a Moreau envelope: $\Phi(u(t))\leq [\Phi]^{\kappa(t,\lambda)}(x_0)$. As is not uncommon, we will call $\Phi(u(t))$ the energy of the gradient flow, even if there is no direct physical justification to view $\Phi(u(t))$ as an energy. For the definition of Moreau envelope we refer to Definition~\ref{MoreauProx} and $\kappa(t,\lambda)$ is defined in Theorem~\ref{GradFlowTheorem}.The importance of this bound is that the Moreau envelope satisfies a number of nice analytical properties, which we detail in Section~\ref{sec:convex}. These will be particularly useful for us when proving our convergence results in Theorems~\ref{Hilthm} and~\ref{thm:P0theorem1}.

In \cite[Theorem 4.3.2]{greenbook}, the authors prove a very similar result to Theorem~\ref{GradFlowTheorem}, the only difference being the parameter in the Moreau envelope. In fact, for the applications of Theorem~\ref{GradFlowTheorem} in the current paper, \cite[Theorem 4.3.2]{greenbook} would suffice, yet we include Theorem~\ref{GradFlowTheorem} as its bound is strictly stronger in the case when $\lambda>0$ and the proof strategy we use is novel and could be useful for future research.

We illustrate the bound from Theorem~\ref{GradFlowTheorem} in Figure~\ref{fig:enter-label} for the convex function ${\Phi: \mathbb{R} \to \mathbb{R}, x \mapsto |x|}$ and a time $T>0$. The figure shows the graphs of $x\mapsto \Phi(x)$, $x\mapsto [\Phi]^T(x)$, and $x\mapsto \Phi(u_x(T))$, where $u_x$ denotes the gradient flow of $\Phi$ starting from $x$.

\begin{figure}[ht]
    \centering
    \begin{tikzpicture}[scale=1]

    \draw (3,0) -- (7,0) ;
    \draw [stealth-stealth] (0,-0.7) -- (10,-0.7) ;
    \node at (5,-0.7) [anchor=north]{$x\in \mathbb{R}$} ;
    \draw (3,0) -- (0,3) ;
    \draw (7,0) -- (10,3) ;

    \draw (0,5) -- (5,0) ;
    \draw (5,0) -- (10,5) ;

    \draw[dashed] (5,0) parabola (2,1.5) ; 
    \draw[dashed] (2,1.5) -- (0,3.5) ;
    \draw[dashed] (5,0) parabola (8,1.5) ; 
    \draw[dashed] (8,1.5) -- (10,3.5) ;

    \node at (2.5,2.5) [anchor=south west]{$\Phi(x)$} ;
    \node at (1.5,1.7) [anchor=south west]{$[\Phi]^T(x)$} ;
    \node at (2,1) [anchor=north east]{$\Phi(u_x(T))$} ;
    \node at (3,0) [anchor=north]{$-T$} ;
    \node at (7,0) [anchor=north]{$T$} ;
    \end{tikzpicture}
    \caption{Moreau-envelope bound for $\Phi(x):=|x|$}
    \label{fig:enter-label}
\end{figure}
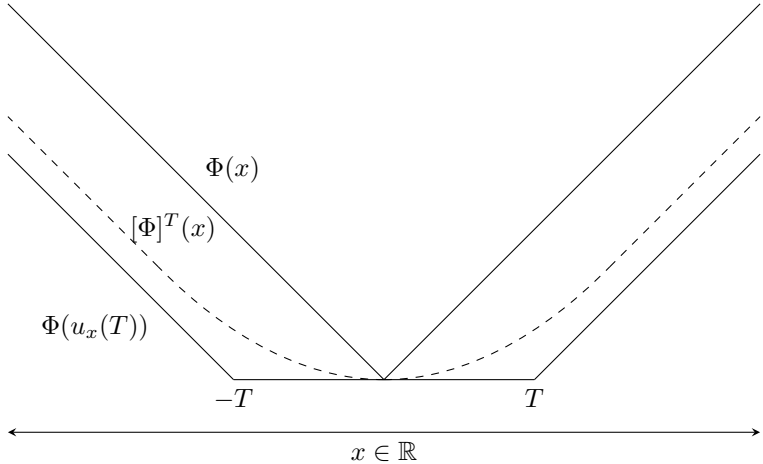

\subsubsection{Convergence of semigroups and gradient flows}\label{sec:convgradflowintro}

In order to use our generalised Brezis--Pazy theorem to prove convergence results for gradient flows, we first have to express the gradient flows as semigroups generated by suitably accretive operators, which we do in Section~\ref{sec:convex}, in particular Lemma~\ref{lem:gradflowsemigroup}. Different notions of accretivity are explained in Section~\ref{sec:accretivity}. The operators will be the subdifferentials (adjusted for $\lambda$-convexity) of the functions $\Phi$ of which the gradient flows are taken and the required accretivity will be guaranteed by convexity properties of the functions. For applications such as total-variation flow, which we will address shortly, we would like to formulate our flow as a semigroup generated by an $L^2$ subdifferential, yet find that the domain of that operator is not fitting for our purposes. In Section~\ref{sec:subdifferentials} we define an operator $\partial_{L^P}\Phi$ whose domain is a suitable $L^p$ space for $p\neq 2$. To ensure the required accretivity properties of this new operator, we consider $P_0$-convexity of $\Phi$; we return to this topic briefly near the end of Section~\ref{sec:convgradflowintro}.

Second, we need to establish convergence of the resolvents. One way to do this is via Proposition~\ref{OpConvergence}, which deduces pointwise convergence of the resolvents from pointwise convergence of the underlying operators. The latter, however, is not always easy to prove and thus for our two main applications of the generalised Brezis--Pazy theorem in Theorems~\ref{Hilthm} and~\ref{thm:P0theorem1}, we do not use Proposition~\ref{OpConvergence}, but we make use of $\Gamma$-convergence of the functions of which the gradient flows are taken to establish the required convergence of the resolvents.

Our result in Theorem~\ref{Hilthm} holds for a Banach stacking of Hilbert spaces. We prove uniform convergence on intervals $[0,T]$ of the gradient flows of a sequence of $\lambda$-convex, lower-semicontinuous functions on Hilbert spaces under the assumption that the functions $\Gamma$-converge and satisfy an equicoercivity property in the Banach stacking. We also establish pointwise convergence of the values of the functions along the flow. 

Finally, Theorem~\ref{thm:P0theorem1}, is specifically about gradient flows in the Banach stacking $TL^p(\Omega)$. As in the previous result, we are still interested in a gradient flow with respect to the inner product of a Hilbert space, namely an $L^2$ inner product, as such gradient flows are very common in the literature (e.g., the total-variation flow below), yet we require the $\Gamma$-convergence and compactness assumptions to hold over the Banach stacking $TL^p(\Omega)$. If $p<2$, that means the compactness assumption (which is formulated in terms of the equicoercivity property of Definitions~\ref{def:equicoercivity} and~\ref{BSdefs}) is strictly weaker than if we were to assume a similar compactness property in the $L^2$ topology.

To motivate why these alternative assumptions in Theorem~\ref{thm:P0theorem1} are necessary, we give the example of the total-variation functional (see \cite[Definition 3.4]{Ambrosio2000})
\[
TV(u) := \sup\left\{\int_\Omega u \operatorname{div} \varphi \,|\, \varphi\in C^1_c(\Omega;\mathbb{R}^d) \text{ with } \|\varphi\|_{L^\infty(\Omega;\mathbb{R}^d)}\leq1 \right\},
\]
defined on $L^1(\Omega)$, where $\Omega\subset \mathbb{R}^d$ is the open unit ball. (In our results later in this paper we consider function spaces on open subsets $\Omega$ equipped with probability measures, so this example fits into our framework up to a normalization factor.) The total-variation flow is the gradient flow of $TV$ with respect to the $L^2(\Omega)$ inner product. However, the standard compactness assumption in $L^2(\Omega)$ that one would typically require for studying the convergence of gradient flows---namely that every bounded sequence $(\psi_n)_{n\in\mathbb{N}}$ in $L^2(\Omega)$ for which also $(TV(\psi_n))_{n\in\mathbb{N}}$ is bounded, has a converging subsequence in $L^2(\Omega)$---fails, as the following counterexample shows.

If $u\in W^{1,1}(\Omega)$, then $TV(u)=\int_\Omega |\nabla u(x)| \, dx$. Let $\Omega$ be the unit ball in $\mathbb{R}^d$ and let $\psi\neq 0$ be a smooth function defined on $\mathbb{R}^d$ with compact support in $\Omega$. After rescaling if necessary, we may assume that $\Vert \psi\Vert^2_{L^2(\Omega)}=1$. Define the sequence $(\psi_n)_{n\in \mathbb{N}}$ by, for all $n\in\mathbb{N}$, $\psi_n(x):=n^{\frac{d}{2}}\psi(nx)$. Then $\Vert \psi_n\Vert^2_{L^2(\Omega)}=\Vert \psi\Vert^2_{L^2(\Omega)}=1$, so the sequence $(\psi_n)_{n\in\mathbb{N}}$ is bounded in $L^2(\Omega)$. Moreover, $TV(\psi_n)=n^{1-\frac{d}{2}}\int|\nabla\psi(x)| \, dx$, thus, if $d\geq 2$, the sequence $(TV(\psi_n))_{n\in \mathbb{N}}$ is also bounded. However, the sequence $(\psi_n)_{n\in\mathbb{N}}$ has no subsequence that converges in $L^2(\Omega)$; indeed, if such a subsequence were to exist, then the limit function would have unit $L^2(\Omega)$ norm which contradicts that, for all $x\in \mathbb{R}^d\setminus\{0\}$, $(\psi_n(x))_{n\in\mathbb{N}}$ converges to $0$. Hence the standard compactness assumption, as explained above, is not satisfied.

On the other hand, by a known result (see \cite[Theorem 3.47]{Ambrosio2000}), if $A\subset L^1(\Omega)$ is a subset such that $\{\Vert u \Vert_{L^1(\Omega)}\}_{u\in A}$ and $\{TV(u)\}_{u\in A}$ are bounded, then $A$ is relatively compact in $L^p(\Omega)$ for $p<\frac{d}{d-1}$. Thus we can recover some equicoercivity, just for $p<2$, i.e. an exponent smaller than that of the Lebesgue-space topology with respect to which we take the gradient flow. In order to overcome this difficulty we will assume a stronger notion of convexity which we call $P_0$-convexity (see Section~\ref{sec:subdifferentials}). This notion is taken from the theory of completely accretive operators \cite[Lemma 7.1]{benilan1991comp}. This assumption introduces a number of other useful properties for the semigroup that describes the gradient flow. In particular, Theorem~\ref{BenCran1991} connects the $L^2$ subdifferential, which is the operator that generates the gradient flow with respect to the $L^2$ topology (see Lemma~\ref{lem:gradflowsemigroup}), to an m-accretive subdifferential operator on $L^p$. This allows us to consider $L^2$ gradient flows, even of functionals whose equicoercivity and $\Gamma$-convergence properties hold with respect to the topology of a different $L^p$ space. Moreover, Lemma~\ref{lemAcrete} shows that the subdifferential of a $P_0$-convex functional contracts the $L^q$ norm for any $q\geq 1$, which will be a useful property for our proofs.

This leads to Theorem~\ref{thm:P0theorem1}, which establishes uniform convergence on intervals $[0,T]$ of gradient flows of $P_0$-convex, lower-semicontinuous, non-negative functionals on $TL^{\max(2,p)}(\Omega)$ (for some $p\geq 1$) and pointwise convergence of the values of the functionals along the flow, under conditions that (when paired with Proposition~\ref{prop:assumpreduce}) take the form of 
\begin{itemize}
    \item a limit-inferior inequality for the functionals over $TL^p$;
    \item a limit-superior inequality for the functionals over $TL^{\max(2,p)}$;
    \item an equicoercivity property giving compactness in $TL^p$; and
    \item a uniform bound controlling $L^p$ norms in terms of $L^2$ norms plus the functionals, possibly via some interpolation result.
\end{itemize}

\subsection{Relations to the literature}\label{sec:literature}

An early result regarding convergence of gradient flows was introduced by Degiovanni \textit{et al.} \cite{EvoEq}. In this paper the authors prove that for a sequence of functions, which is either $\phi$-convex (\cite[Definition 1.16]{EvoEq}) or has a $\phi$-monotone subdifferential (\cite[Definition 1.4]{EvoEq}), the gradient flows of the functions converge, assuming the functions themselves $\Gamma$-converge. In particular \cite[Theorem 4.11]{EvoEq} overlaps with the problems we investigate, since $\lambda$-convexity implies $\phi$-convexity (for details we refer to the discussion following Definition 1.16 in \cite{EvoEq}). 

The convergence results in \cite{EvoEq}, however, are stated in a single Hilbert space $H$, which is one of the key differences with our results that concern convergence in Banach stackings. Moreover, the convergence of gradient flows in \cite{EvoEq} is given in $W^{1,p}(0,T;H)$, whereas we prove uniform convergence in $C([0,T];\mathcal{M})$ for a particular metric space $\mathcal{M}$. 
Another key difference between our work and \cite[Theorem 4.11]{EvoEq} is that we are able to do away with all well-preparedness assumptions on the initial conditions, whereas \cite[Theorem 4.11]{EvoEq} requires that the function values and the gradients of the initial conditions are uniformly bounded. Similarly to \cite[Theorem 4.11]{EvoEq}, our result also gives the convergence of the function values along the gradient flows (of those functions).

Other well-known contributions to the study of convergence of gradient flows are Serfaty's work in \cite{SerfatyGam}, originally applied to the study of gradient flows for the Ginzburg--Landau functional by Sandier and Serfaty in \cite{SerfatyGinz}. 

In \cite[Theorem 2]{SerfatyGam} the author proves that, if the gradient flows of a $\Gamma$-converging sequence of functions on a metric space converge pointwise, then their limit is the gradient flow of the $\Gamma$-limit. Some key assumptions required for this result are a lower bound on the slopes, a lower bound on the metric derivatives and well-prepared initial conditions. Per Proposition~\ref{prop:derivativelowbound}, in the setting of a Banach stacking the lower bound on the metric derivatives holds; the lower bounds on the slopes is also known to hold for $\lambda$-convex functions \cite[Proposition 8]{Ortner}. In order to apply \cite[Theorem 2]{SerfatyGam}, pointwise convergence of the gradient flows has to be established first. To do so, typically one has to assume some kind of compactness, analogous to the equicoercivity assumptions we include in our results.

As mentioned previously, we do not require a well-preparedness assumption for the initial conditions in Theorems~\ref{Hilthm} and \ref{thm:P0theorem1}. Uniform-in-time convergence of the gradient flows is not included in \cite[Theorem 2]{SerfatyGam}; we are able to establish it through our generalised Brezis--Pazy result of Theorem~\ref{theorem1}. Moreover, Banach stackings do not not fit into the framework of \cite[Theorem 2]{SerfatyGam}. In particular, the limiting gradient flow in \cite[Theorem 2]{SerfatyGam} is taken with respect to the limiting metric space ($\mathcal{S}$ in the notation of \cite{SerfatyGam}), whereas the limiting gradient flow in our results is taken with respect to the norm in the Banach space $X_\infty$, rather than with respect to the metric on the metric space $\mathcal{M}$ into which the Banach spaces $X_n$ embed.

There are further developments that build on the theory introduced in \cite{SerfatyGam}. The key observation for these is that the maximal-slope formulation of a gradient flow is equivalent to an energy-dissipation inequality (sometimes referred to as an energy-dissipation estimate, or energy-dissipation balance when equality holds). The works \cite{Liero2017,EDPtilting,VarConvergence} take this perspective and provide an analysis similar to that in \cite{SerfatyGam}. In general, they conclude that if gradient systems\footnote{Some authors refer to gradient flows as gradient systems or classical gradient systems and to non-gradient-flow gradient systems as generalised gradient systems or non-classical gradient systems. We will not do so here.} (which are generalisations of a gradient flow; see \cite[Section 1]{Liero2017}) converge to a limit (similar to a $\Gamma$-limit) and the solutions of the induced evolution equations converge to a limit pointwise, then the limiting evolution equation is induced by the limiting gradient system. The type of convergence of the gradient system typically required in the literature involves (a kind of) $\Gamma$-convergence of the dissipation potential, which appears in the energy-dissipation inequality. The theorems in these works require assumptions similar to those in \cite[Theorem 2]{SerfatyGam}.

It can be interesting future research to study if the ideas we introduce here can be applied to generalised gradient systems, especially since it is well known that gradient flows can converge to a non-gradient-flow gradient system \cite[Section 3.3.3]{Liero2017}; this falls outside the scope of the current paper.

We also like to point the reader towards Example 3.4 in \cite{Liero2017}, which illustrates that the particulars of the well-preparedness condition can influence the resulting limiting gradient system. Since in this example the energy function in the gradient system is not $\lambda$-convex, this suggests that our restriction to gradient flows of $\lambda$-convex functions is especially important for being able to drop well-preparedness assumptions. 

The work \cite{florentine} takes an approach like the one in the current paper as it derives a theorem similar to \cite[Theorem 2]{SerfatyGam}, except that it is based on the minimizing movement scheme, which allows for some error in each step of the scheme. The connection with our work follows from \eqref{eq:minimizinglimit} and Lemma~\ref{lem:gradflowsemigroup}, which together equate a semigroup generated by a suitable operator, which is our main analytical tool, with minimizing movements.    

There are a number of papers which which focus on particular discrete-to-continuum gradient-flow convergence results. In \cite{NonlocalPLaplace,Fadili2020,Fadili2023,Fadili2023Sparse,Fadili2024} discrete-to-continuum limits for solutions of $p$-Laplacian evolution equations are established. Notably, \cite{NonlocalPLaplace} imports the language of graphons as a limiting object for a sequence of graphs to discuss the discrete-to-continuum limit for the flows. We also refer to \cite{vanGennipGigaOkamoto26} which proves discrete-to-continuum convergence for total-variation flow and the solutions of one-dimensional Allen--Cahn equations. In work currently in preparation \cite{mefuture} we will apply the results set out in this paper to strengthen the known convergence results in \cite{vanGennipGigaOkamoto26} as well as broaden their setting.

A key tool of analysis that we use in our work is semigroup theory, a broad overview of which can be found in \cite{ito2002} with a particular focus on approximation results. As mentioned before, our primary strategy builds on and extends a result by Brezis and Pazy from \cite{brezis1972convergence}. We note that the linear analogue of this result is given in \cite{Ito1998} for discrete to continuum limits under the name `Trotter--Kato theorem'. Furthermore, we use the theory of completely accretive operators that was introduced in \cite{benilan1991comp}. This theory has been used in the analysis of total-variation flow in \cite{Mazon2001,MinTVFlow,Mazon2005}. Its use in those papers has guided us to its usefulness in our analysis; in particular Lemma~\ref{lemAcrete} and Theorem~\ref{BenCran1991} are key ingredients in our proofs . 

Our setting of a Banach stacking is inspired by the metric space $TL^p(\Omega)$ introduced in \cite{garcia2016continuum}. This setting has been used in a number of discrete-to-continuum results about convergence of minimizers, for example in \cite{von2016,Garcia2017,Osting2017,Thorpe2017,Alonso2018,Davis2018,Garcia2018,Dejan2019,Theil2019,Garcia2020,slepcev2020,Thorpe2020,Dunlop2020,Kaplan2020,ThorpeVanGennip23}.

\subsection{Section breakdown}

The paper is organised as follows.
\begin{itemize}
    \item In Section~\ref{sec:nonlinsemi} we set up the notation and basic concepts regarding operators and nonlinear semigroup that we will use.
    \item Section~\ref{sec:convex} gives definitions and results regarding $\lambda$-convex functions and gradient flows. Of particular interest is Theorem~\ref{GradFlowTheorem}, which provides a bound on the energy along a gradient flow in terms of a Moreau envelope. This is an improvement to the bound given in \cite[Theorem 4.3.2]{greenbook} for the case $\lambda>0$.
    \item Section~\ref{sec:subdifferentials} is dedicated to a stronger notion of convexity that was introduced (without name) in \cite[Lemma 7.1]{benilan1991comp}, which we give the name $P_0$-convexity. We also introduce nonstandard subdifferentials which are effectively extensions of the standard $L^2$ subdifferential to $L^p$. Together, these notions will be necessary for extending an $L^2$ gradient flow to the space $L^p$ whilst guaranteeing the semigroup contracts the $L^p$ norm.
    \item Section~\ref{sec:Banachstackings} introduces our definition of a Banach stacking; the setting we require for our gradient flow convergence results. Additionally, in Proposition~\ref{tlpBanach} we prove that $TL^p(\Omega)$ satisfies the axioms of a Banach stacking.
    \item In Section~\ref{sec:BrazisPazyextension} we extend the result by Brezis and Pazy from \cite[Theorem 3.1]{brezis1972convergence} to the setting of a Banach stacking. This result allows us to deduce uniform convergence of semigroups from pointwise convergence of resolvents.
    \item Section~\ref{sec:convresolvents} explores under what conditions we can deduce pointwise convergence of resolvents in a Banach stacking and thus, by our result in Section~\ref{sec:BrazisPazyextension}, also uniform convergence of semigroups. In particular, we prove the three main applications of Theorem~\ref{theorem1}: Proposition~\ref{OpConvergence} and Theorems~\ref{Hilthm} and~\ref{thm:P0theorem1}. Additionally, we prove Proposition~\ref{prop:assumpreduce}, which investigates in further detail the assumptions required for Theorem~\ref{thm:P0theorem1}.
    \item In the appendices we prove some basic properties of operators that we need (Appendix~\ref{app:basicproperties});  provide an overview of the required properties of $\lambda$-convex functions and subdifferentials (Appendix~\ref{sec:C}); give a proof of Proposition~\ref{prop:derivativelowbound} which establishes a lower bound on the metric derivatives in the setting of a Banach stacking (Appendix~\ref{sec:B}); give examples of Banach stackings (Appendix~\ref{sec:BanStackEx}); and give the main definitions and results regarding classical $\Gamma$-convergence and its formulation on Banach stackings (Appendix~\ref{app:Gammaconvergence}).
\end{itemize}

\section{Semigroups and notation}\label{sec:nonlinsemi}

In this paper we do not include $0$ in $\mathbb{N}$ and we use $\mathbb{N}_{\infty}:=\mathbb{N}\cup \{\infty\}$. Moreover the subset notation $\subset$ allows for equality. Given a function $f:X \rightarrow (-\infty,+\infty] := \mathbb{R} \cup \{+\infty\}$ defined on a set $X$, we define $\operatorname{dom}(f) \subset X$ to be the effective domain of $f$, i.e., 
$\operatorname{dom}(f):=\{x\in X \,|\, f(x)\neq +\infty \}$. We say $f$ is proper if $\operatorname{dom}(f)\neq \emptyset$. 

Here we shall introduce the notation/conventions used throughout the paper for operators and semigroups. The definitions we use in Section~\ref{sec:nonlinsemi} are taken from \cite[Chapter 2 and 3]{barbu2010nonlinear}. 

\subsection{Operators}\label{sec:operatorsbasic}

\begin{definition}\label{def:operator}
An operator on a set $X$ is a subset $A\subset X \times X$. We write ${A(x):= \{ y\in X \,| \, (x,y)\in A \}}$ and define the domain of $A$ to be the set ${D(A) := \{x\in X \,|\, A(x) \neq \emptyset \} \subset X}$ and the range of $A$ as ${\operatorname{Range}(A) := \cup_{x\in X}A(x) \subset X}$.
\end{definition}

To lighten the use of parentheses in the operator notation, when there is no risk of ambiguity we will also write $Ax$ instead of $A(x)$.

\begin{definition}\label{def:operatorsubset}
Given two operators $A$ and $B$ on a set $X$, $A\subset B$ is interpreted in the standard way for sets. Thus, $A\subset B$ if and only if for all $x\in X$, $A(x) \subset B(x)$.
\end{definition}

\begin{definition}\label{def:restriction}
Let $A$ be an operator on a set $X$ and let $Y$ be a subset of $X$. We define the restriction of $A$ to $Y$ by $A|_{Y}:= A \cap (Y \times Y)$. 
\end{definition}

\begin{definition}\label{def:inverse}
The inverse of an operator $A$ on a set $X$ is another operator $A^{-1}$ on $X$ defined by
\begin{equation*}
    A^{-1}:=\{ (x,y)\in X\times X \,|\, (y,x) \in A \} .
\end{equation*}
Moreover, for all $y\in X$ we define
\begin{equation*}
    A^{-1}(y):=\{ x\in X \,|\, (x,y) \in A\} .
\end{equation*}
\end{definition}

\begin{definition}\label{def:identity}
We shall denote the identity operator $I_X$ on a set $X$ by $I_X:=\{ (x,x) \in X \times X \,|\, x\in X \}$.
\end{definition}

Next we will assume that $X$ is a topological vector space and introduce some further concepts.

\begin{definition}\label{def:pointwiseclosure}
The pointwise closure (in the topology of $X$) of an operator $A$ on $X$, denoted by $\overline{A}$, is the closure of $A$ in the product topology of $X\times X$. We say an operator $A$ is pointwise closed if $\overline{A}=A$.
\end{definition}

\begin{definition}\label{def:operatorsoperations}
For operators $A,B$ on $X$ and $\lambda \in \mathbb{R}$ we define:
\begin{itemize}
    \item the sum of $A$ and $B$ as $A+B:= \{ (x,y+z) \in X\times X\, | \, (x,y)\in A \; \text{and} \; (x,z) \in B \}$;
    \item the scalar multiple of $A$ as $\lambda A:= \{ (x,\lambda y) \in X\times X\, | \, (x,y)\in A \}$;
    \item and the resolvent operator of $A$ at $\lambda$ as $R_{\lambda}(A) := (I_X+\lambda A)^{-1}$.
\end{itemize}
All these are again operators on $X$.
\end{definition}

\begin{definition}\label{def:rangecondition}
We say an operator $A$ on $X$ satisfies the range condition if there exists a $\delta>0$ such that, for all $\lambda\in (0,\delta)$,
\begin{equation}
\label{rangeCon}
\operatorname{Range}(I_X+\lambda A) \supset \overline{D(A)}.    
\end{equation}
\end{definition}

In Appendix~\ref{app:basicproperties} we prove some basic properties involving domains, ranges, restrictions inverses, pointwise closures, sums, and scalar multiples of operators.

\subsection{Notions of accretivity}\label{sec:accretivity}
In the remainder of Section~\ref{sec:nonlinsemi},, unless explicitly stated differently, $X$ is a Banach space with norm $\Vert \cdot \Vert$.

\begin{definition}\label{def:accretive}
An operator $A$ is accretive (with respect to the norm $\|\cdot\|$) if, for all $(x,y),(\hat{x},\hat{y})\in A$ and all $\lambda>0$,
\begin{equation*}
\Vert x - \hat{x} + \lambda(y - \hat{y}) \Vert \geq \Vert x - \hat{x} \Vert. 
\end{equation*}
\end{definition}
It is worth noting that if $A=\emptyset \subset X\times X$, then it is accretive.

\begin{remark}
    When we say that an operator is accretive `on $X$', we mean with respect to the norm that has been defined on, or is considered standard for, $X$. In those cases where accretivity is intended with respect to a different norm on the same set, we will be careful to indicate this using the language of Definition~\ref{def:accretiveN}. Analogous disclaimers hold for the concepts of maximal, $\omega$-, m-, and $\omega$-m-accretivity that will be introduced later.
\end{remark}

\begin{definition}\label{def:accretiveN}
Let $A$ be an operator on $X$ and $N:X \rightarrow [0,+\infty]:=[0,+\infty)\cup\{+\infty\}$ a non-negative function. We say that $A$ is accretive with respect to $N$ if, for all $(u,v),(u',v')\in A$ and for all $\lambda>0$,
$$ N(u-u'+\lambda(v-v')) \geq N(u-u').$$
\end{definition}

\begin{remark}\label{rem:intersectionclosure}
It follows directly from Definition~\ref{def:accretive} that if an operator $A$ on $X$ is accretive and $Y\subset X$, then $A|_Y$ is also an accretive operator with respect to the norm on $X$ (and thus also with respect to the subspace norm on $Y$ induced by the norm on $X$). Moreover, since the norm $\|\cdot\|$ on $X$ is a continuous function on $X$, also $\overline{A}$ is an accretive operator on $X$.
\end{remark}

If $X$ is a Hilbert space, then accretivity has a simple characterisation which in some sources (e.g. \cite{barbu2010nonlinear}) is referred to as monotonicity and which is given in Lemma~\ref{Hilbertaccretive}.

\begin{lemma}\label{Hilbertaccretive}
Suppose that $X$ is a Hilbert space with an inner product $\langle \cdot,\cdot \rangle$. Then $A$ is accretive on $X$ if and only if, for all $(x,y),(\hat{x},\hat{y})\in A$, 
$$ \langle y- \hat{y}, x-\hat{x} \rangle \geq 0.$$
\end{lemma}
\begin{proof}
Let $(x,y),(\hat{x},\hat{y})\in A$. The inequality in Definition~\ref{def:accretive} is equivalent to the inequality 
$\Vert x - \hat{x} + \lambda(y - \hat{y}) \Vert^2 \geq \Vert x - \hat{x} \Vert^2.$ 
Expanding this into inner products and rearranging terms, this is equivalent to
\begin{equation}\label{eq:accretiveequivalent}
\lambda^2 \langle y-\hat y, y-\hat y\rangle + 2 \lambda \langle y-\hat y, x-\hat x\rangle \geq 0.
\end{equation}
If $A$ is accretive, the inequality in \eqref{eq:accretiveequivalent} holds for all $\lambda>0$. Dividing both sides by $\lambda$ and taking the limit $\lambda \downarrow 0$, we obtain $\langle y- \hat{y},x-\hat{x} \rangle \geq 0$.

On the other hand, if $\langle y- \hat{y},x-\hat{x} \rangle \geq 0$, then \eqref{eq:accretiveequivalent} holds for all $\lambda>0$ and thus $A$ is accretive.
\end{proof}

In the remainder of this section, $X$ again denotes a Banach space.

\begin{definition}\label{def:m-accretive}
We say an operator $A$ on $X$ is m-accretive if it is accretive and, for all $\lambda>0$, $\operatorname{Range}(I_X+\lambda A) = X$.
\end{definition}

Next we give a simpler and equivalent characterisation of m-accretivity.

\begin{proposition}\label{prop:maccretive}
An operator $A$ on $X$ is m-accretive if and only if it is accretive and $\operatorname{Range}(I_X+ A) = X$.
\end{proposition}
\begin{proof}
See \cite[Definition 3.1 and Proposition 3.3]{barbu2010nonlinear}.
\end{proof}

\begin{proposition}
\label{accprop}
Let $A$ be an m-accretive operator on $X$, then, for all $\lambda>0$, $I_X+\lambda A$ is m-accretive on $X$.
\end{proposition}
\begin{proof}
Let $\lambda>0$ and $\gamma>0$. We note that $$I_X+\gamma(I_X+\lambda A) = (1+\gamma)\left(I_X+ \frac{\lambda \gamma}{1+\gamma}A \right)$$
and thus, by m-accretivity of $A$, $\operatorname{Range}(I_X+\gamma(I_X + \lambda A) ) = X$.

Furthermore if $(x,y),(\hat{x},\hat{y}) \in A$, then $(x,x+ \lambda y),(\hat{x},\hat{x}+\lambda \hat{y} A) \in I_X + \lambda A $. We see that,
\begin{align*}
\Vert x-\hat{x} + \gamma( x+ \lambda y - \hat{x} -\lambda \hat{y} ) \Vert &= (1+\gamma)\left\Vert x-\hat{x} + \frac{\gamma \lambda}{1+\gamma}(  y - \hat{y} ) \right\Vert\\
&\geq (1+\gamma) \Vert x -\hat{x} \Vert 
\geq \Vert x- \hat{x} \Vert,
\end{align*}
where the first inequality follows from accretivity of $A$.
\end{proof}

\begin{definition}\label{def:maximalmonotone}
    An accretive operator $A$ on $X$ is called maximal accretive if, for all accretive operators $B$ on $X$, $A\subset B$ implies $A=B$.
\end{definition}

We state an important result regarding m-accretive operators. 

\begin{proposition}\label{propMaccret}
Every m-accretive operator on a Banach space $X$ is maximal accretive. If $X$ is a Hilbert space, an accretive operator on $X$ is maximal accretive if and only if it is m-accretive.
\end{proposition}
\begin{proof}
    See \cite[Remark 3.1]{barbu2010nonlinear}.
\end{proof}

\begin{definition}\label{def:contraction}
An operator $A$ is a contraction on $X$ if, for all $(x,y),(\hat{x},\hat{y})\in A$,
$$ \Vert y - \hat{y} \Vert \leq \Vert x - \hat{x} \Vert.$$
\end{definition}

We note that, contrary to other common definitions of contraction, in this context it is usual not to require the inequality in the definition to be strict.

\begin{remark}\label{rem:contraction}
If an operator $A$ on $X$ has the property that there exists an $L>0$ such that, for all $(x,y),(\hat{x},\hat{y})\in A$,
$\Vert y - \hat{y} \Vert \leq L \Vert x - \hat{x} \Vert,$ then, for all $x\in X$, if $y, \hat y \in A(x)$, then $y=\hat y$. Thus $A(x)$ is either empty or a singleton. So one can reasonably think of $A$ as an $L$-Lipschitz-continuous function (i.e., a Lipschitz-continuous function for which $L$ is a Lipschitz constant) from $D(A)$ to $X$.

In particular, if $A$ is a contraction on $X$, then per Definition~\ref{def:contraction} $A$ satisfies an inequality as above with $L=1$ and thus we can view $A$ as operator or as $1$-Lipschitz-continuous function from $D(A)$ to $X$. In such cases, in a slight abuse of notation, we may use $A$ to denote both the operator as well as the function that maps $x\in D(A)$ to the only element in $A(x)$.
\end{remark}

\begin{proposition}\label{prop:accretivecontraction}
An operator $A$ is accretive if and only if for all $\lambda>0$, $R_{\lambda}(A)$ is a contraction.
\end{proposition}
\begin{proof}
To prove the `only if' statement, suppose $A$, is accretive. Let $\lambda>0$ and suppose $(z,x),(\hat{z},\hat{x}) \in R_{\lambda}(A)$. Then, by definition of the resolvent, there exist $y\in A(x)$ and $\hat{y}\in A(\hat{x})$ such that $z=x+\lambda y$ and $\hat{z}=\hat{x}+\lambda \hat{y}$. Hence, by accretivity, we have
    $
        \Vert x- \hat{x} \Vert \leq \Vert x- \hat{x} + \lambda(y-\hat{y}) \Vert 
        = \Vert z -\hat{z} \Vert
    $ 
and thus $R_\lambda(A)$ is a contraction.

To prove the `if' statement, assume that, for all $\lambda>0$, $R_\lambda(A)$ is a contraction. Let $(x,y),(\hat{x},\hat{y}) \in A$. Then $z:=x+\lambda y$ and $\hat z:=\hat x+\lambda \hat y$ are both in $D(R_\lambda(A))$; moreover $x \in R_\lambda(A)(z)$ and $\hat x \in R_\lambda(A)(\hat z)$. Thus, by the contraction property of $R_\lambda(A)$, we have
$
\|x-\hat x\| \leq \|z-\hat z\| = \|x-\hat x + \lambda (y-\hat y)\|.
$ 
It follows that $A$ is accretive.    
\end{proof}

Lastly we will weaken the notion of accretivity which in some cases is called $\omega$-accretivity. This will be relevant to our discussion of $\lambda$-convex functions.

\begin{definition}\label{def:omegamaccretive}
Let $\omega \in \mathbb{R}$. We say that an operator $A$ is $\omega$-accretive if $A+\omega I_X$ is accretive. We say that $A$ is $\omega$-m-accretive if $A+\omega I_X$ is m-accretive.
\end{definition}

For all $\omega \in \mathbb{R}$ we define the interval $\mathfrak{I}_{\omega}$ by
\begin{equation}
\label{OmegaInterval}
 \mathfrak{I}_{\omega} :=
\begin{cases}
 (0,1/\omega) \; &\text{if} \; \omega>0,  \\
 (0,+\infty) \;  &\text{if} \; \omega\leq 0.
\end{cases}
\end{equation}

\begin{proposition}
\label{Omega1}
Let $\omega \in \mathbb{R}$ and let $A$ be an $\omega$-accretive operator on $X$. For all $\lambda \in \mathfrak{I}_{\omega}$ and all $(x,y),(\hat{x},\hat{y})\in A$ we have
$$ \Vert x-\hat{x} +\lambda(y-\hat{y}) \Vert \geq (1-\lambda \omega) \Vert x- \hat{x} \Vert.$$
Equivalently, for all $\lambda \in \mathfrak{I}_{\omega}$ and all $(x,y),(\hat{x},\hat{y}) \in R_{\lambda}(A)$ we have
\[
\Vert y- \hat{y} \Vert \leq \frac{1}{1-\lambda \omega}\Vert x- \hat{x} \Vert.
\]
\end{proposition}

\begin{proof}
Firstly we note that $(x,\omega x +y),(\hat{x},\omega \hat{x}+\hat{y})\in A+\omega I_X$ and,
\begin{align*}
x-\hat{x} +\lambda(y-\hat{y}) &= (1-\lambda\omega)(x-\hat{x}) +\lambda(\omega x+y-\omega \hat{x}-\hat{y}) \\
&= (1-\lambda \omega)\left[ x-\hat{x} + \frac{\lambda}{1-\lambda \omega} (\omega x+y-\omega \hat{x}-\hat{y}).     \right] 
\end{align*}
By our assumptions $\lambda \omega<1$ and $A+\omega I_X$ is accretive. So we conclude,
\begin{align*}
\Vert x-\hat{x} +\lambda(y-\hat{y}) \Vert &= (1-\lambda \omega) \left\Vert  x-\hat{x} + \frac{\lambda}{1-\lambda \omega} (\omega x+y-\omega \hat{x}-\hat{y}) \right\Vert \\
&\geq (1-\lambda\omega)\Vert x-\hat{x} \Vert .
\end{align*}
\end{proof}

\begin{proposition}
\label{Omega2}
Let $\omega \in \mathbb{R}$ and $A$ be an $\omega$-m-accretive operator on $X$. Then, for all $\lambda \in \mathfrak{I}_{\omega}$ (as defined in \eqref{OmegaInterval}), we have
$$ \operatorname{Range}(I_X + \lambda A) = X.$$
\end{proposition}

\begin{proof}
Firstly we note $A+\omega I_X$ is m-accretive. So we have,
\begin{align*}
 I_X + \lambda A &= (1-\lambda\omega)I_X + \lambda(\omega I_X + A) \\
 &= (1-\lambda\omega)\left[ I_X + \frac{\lambda}{1-\lambda \omega} (\omega I_X + A)     \right] .
\end{align*}
Since $\frac{\lambda}{1-\lambda \omega}>0$ we observe $\operatorname{Range}\left( I_X + \frac{\lambda}{1-\lambda \omega} (\omega I_X + A)     \right)=X.$ Hence for any $y\in X$ we can choose $x\in X$ such that,
$$ \frac{1}{1-\lambda\omega}y \in \left( I_X + \frac{\lambda}{1-\lambda \omega} (\omega I_X + A)     \right)(x).$$
From this we conclude $y\in  (I_X + \lambda A)(x)$.
\end{proof}

\begin{remark}\label{rem:rangecondition}
By Proposition~\ref{Omega2}, the range condition in Definition~\ref{def:rangecondition} is satisfied whenever $A$ is $\omega$-m-accretive. In this case we can take the interval $(0,\delta)$ to be $\mathfrak{I}_{\omega}$ as defined in \eqref{OmegaInterval}. In particular, if $A$ is m-accretive, the interval can be taken to be $(0,+\infty)$.
\end{remark}

\begin{remark}
\label{accremark}
    Let $\lambda>0$. By Remark~\ref{rem:contraction} and Proposition~\ref{prop:accretivecontraction}, if $A$ is an accretive operator on $X$, then we can view $R_\lambda(A)$ as a $1$-Lipschitz-continuous function from $D(R_\lambda(A))$ to $X$. If $A$ is an m-accretive operator, then by Definition~\ref{def:m-accretive} $\operatorname{Range}(I_X+\lambda A) = X$ and thus, for all $x\in X$, $R_\lambda(A)(x) \neq \emptyset$, i.e. $D(R_\lambda(A))=X$, and thus we can identify $R_{\lambda}(A)$ with a $1$-Lipschitz-continuous function from $X$ to $X$.

    Now let $\omega\in \mathbb{R}$ and $\lambda\in \mathfrak{I}_{\omega}$ instead. In a similar fashion as above, if $A$ is an $\omega$-accretive operator on $X$, then by Remark~\ref{rem:contraction} and Proposition~\ref{Omega1}, we can view $R_\lambda(A)$ as a $\frac{1}{1-\lambda \omega}$-Lipschitz-continuous function from $D(R_\lambda(A))$ to $X$. If $A$ is $\omega$-m-accretive, then, by Proposition~\ref{Omega2}, $D(R_\lambda(A))=X$ and thus we can identify $ R_{\lambda}(A)$ with a $\frac{1}{1-\lambda \omega}$-Lipschitz-continuous function from $X$ to $X$.

    For convenience, when $A$ is an accretive, m-accretive, $\omega$-accretive, or $\omega$-m-accretive operator on $X$ we shall use the notation $R_{\lambda}(A)$ to refer both to the resolvent operator and to the function described above interchangeably. To lighten the notation, we may write $R_\lambda(A)x$ instead of $R_\lambda(A)(x)$.
\end{remark}

In the following remark we explain why the range condition~\eqref{rangeCon} is important. It allows us to iterate the resolvent map when considered as a function, which in particular means that the backward-Euler approximation of the Cauchy problem that follows in \eqref{AbCauchy} is solvable.

\begin{remark}\label{rem:whyrangecon}
    By Remark~\ref{rem:domainrange} and the definition of the resolvent in Definition~\ref{def:operatorsoperations}, we know that $D(R_\lambda(A)) = \operatorname{Range}(I_X+\lambda A)$. Thus, if $A$ satisfies the range condition in \eqref{rangeCon} for a given $\delta>0$, then, for all $\lambda \in (0,\delta)$, $\overline{D(A)} \subset D(R_\lambda(A))$.
    Additionally, $\operatorname{Range}(R_{\lambda}(A)) = D(I_X+\lambda A) \subset D(A)$, where the inclusion holds since, if $x\in D(I_X+\lambda A)$, then there exist $y \in X$ and $z\in D(A)$ such that $y=x+\lambda z$. Thus, for all $\lambda \in (0,\delta)$,
    $$\operatorname{Range}(R_{\lambda}(A)) \subset \overline{D(A)} \subset D(R_\lambda(A)).$$
    Hence, for all $x\in \overline{D(A)}$, all $\lambda \in (0,\delta)$, and all $k\in\mathbb{N}$, $R^k_\lambda(A)(x) \neq \emptyset$, where $R_\lambda^k(A):= (R_\lambda(A))^k$. In this paper, we use $R_\lambda^k(A)$ exclusively in the setting in which we can view the resolvent $R_\lambda(A)$ as a Lipschitz-continuous function, as per Remark~\ref{accremark}, and interpret $R_\lambda^k(A)$ as the usual $k$th power of that function.
        
    To avoid confusion stemming from our choice of notation, we emphasise that in general $R^k_\lambda(A)$ does not equal $R_\lambda(R^{k-1}_\lambda(A))$, but rather $R^k_\lambda(A)(x) = R_\lambda(A)(R_\lambda^{k-1}(x))$. 
\end{remark}

\subsection{Semigroups}\label{sec:semigroups}

We recall that, also in Section~\ref{sec:semigroups}, $X$ denotes a Banach space with norm $\|\cdot\|$.

We want to study the following Cauchy problem for an operator $A$ and initial condition $x\in X$:
\begin{equation} \label{AbCauchy}
\begin{cases}
u'+Au \ni 0, \\
u(0)=x. \\
\end{cases}   
\end{equation}

\begin{definition}
Given $T\in [0,+\infty)$, we define a strong solution to \eqref{AbCauchy} on $[0,T]$ to be a function $u\in C([0,T];X)\cap W^{1,1}((0,T);X)$ for which the inclusion in \eqref{AbCauchy} holds for a.e. $t\in (0,T)$ and $u(0)=x$.
\end{definition}

We also require a concept of mild solution for \eqref{AbCauchy}.

\begin{definition}\label{def:epsdiscr}
Let $\varepsilon>0$. We define an $\varepsilon$-discretization $\mathcal{P}_{\varepsilon}$ of the interval $[0,T]$ of size $N\in \mathbb{N}$ to be a tuple of times $(t_0, \ldots, t_N)$ with $0 = t_0<t_1<...<t_N\leq T$ and, for all $i\in \{1, \ldots, N\}$, $t_i-t_{i-1}\leq \varepsilon$ and $T-\varepsilon < t_N \leq T$ .
\end{definition}

\begin{definition}
\label{def:discsolution}

Given $T>0$ and $\varepsilon>0$, consider a triple $(\mathcal{P}_{\varepsilon},v,\theta)$ such that $\mathcal{P}_{\varepsilon} = (t_0, \ldots, t_N)$ is an $\varepsilon$-discretization of the interval $[0,T]$ of size $N$ and $v,\theta$ are functions $v,\theta: \mathcal{P}_{\varepsilon} \rightarrow X $. Conventionally we shall assume $\theta(0)=0$. For each $0\leq i \leq N$ we denote $v_i = v(t_i)$ and $\theta_i =\theta(t_i)$.

We say that the triple $(\mathcal{P}_{\varepsilon},v,\theta)$ is an $\varepsilon$-approximate solution to (\ref{AbCauchy}) if the following are satisfied:
\begin{itemize}
    \item $\Vert v_0 - x\Vert \leq \varepsilon$;
    \item for all $1\leq i \leq N$,
    $$ \frac{v_i-v_{i-1}}{t_i-t_{i-1}} +Av_i \ni \theta_i \; ;$$
    \item $\sum_{i=1}^N (t_i-t_{i-1})\Vert \theta_i \Vert \leq \varepsilon.$
\end{itemize}

We shall denote the piecewise constant interpolation of $v$ by a map $\bar{v}:[0,T]\rightarrow X$, such that for $1\leq i\leq N$,
$ \bar{v}(t) = v(t_i) $ for $t\in (t_{i-1},t_i]$, $\bar{v}(0)=v(0)$ and $\bar{v}(t)=v(t_N)$ for $t\in(t_N,T]$. 

\end{definition}

\begin{definition}\label{def:mildsolution}
Given $T>0$, we say $u$ is a mild solution to \eqref{AbCauchy} on $[0,T]$ if $u\in C([0,T];X)$ and for any $\varepsilon>0$ there exists an $\varepsilon$-approximate solution $(\mathcal{P}_{\varepsilon},v,\theta)$  to \eqref{AbCauchy} such that for any $t\in [0,T]$, $\Vert u(t) -\bar{v}(t) \Vert \leq \varepsilon$ and $u(0)=x$.

Moreover, we say that $u\in C([0,+\infty);X)$ is a mild solution to \eqref{AbCauchy} on $[0,+\infty)$ if for any $T<+\infty$ we have that $u|_{[0,T]}$ is a mild solution to \eqref{AbCauchy} on $[0,T]$.  
\end{definition}

\begin{proposition}
\label{strongimpliesmild}
Let $u$ be a strong solution to \eqref{AbCauchy} on $[0,T]$, then $u$ is also a mild solution to \eqref{AbCauchy} on $[0,T]$.
\end{proposition}
\begin{proof}
See \cite[Definition 4.3]{barbu2010nonlinear}.
\end{proof}

\begin{proposition}
\label{mildsolutionsProp}
\begin{enumerate}
    \item If $u$ is a mild solution to \eqref{AbCauchy} on $[0,T]$, then for all $t\in [0,T]$, $u(t)\in \overline{D(A)}$. 
    \item If $u$ is a mild solution to \eqref{AbCauchy} on $[0,T]$ it is also a mild solution for the Cauchy problem in \eqref{AbCauchy} on $[0,T]$ with $A$ replaced by $\overline{A}$. 
    \item Let $0<T_1<T_2$, then if $u$ is a mild solution to \eqref{AbCauchy} on $[0,T_2]$, then $u|_{[0,T_1]}$ is a mild solution to \eqref{AbCauchy} on $[0,T_1]$.
\end{enumerate}
\end{proposition}
\begin{proof}
See \cite[ Theorem A.9]{andreu2010nonlocal}.
\end{proof}

\begin{definition}\label{def:semigroup}
Let $A$ be an operator on $X$. We define
\begin{equation*}
    D(S_A) :=  \{ x\in X \, | \, \eqref{AbCauchy} \text{ has a unique mild solution } u(t) \text{ on } [0,+\infty)\}.
\end{equation*}
We note that \eqref{AbCauchy} depends on $x$ via the initial condition $u(0)=x$. 
Furthermore, for each $t\in [0,+\infty)$ we define the map $S_A(t):D(S_A) \rightarrow X  $ by $S_A(t)x := u(t)$, where $u(t)$ is the unique mild solution to \eqref{AbCauchy} on $[0,+\infty)$ with $u(0)=x$. We refer to the collection of maps $\{S_A(t)\}_{t\in [0,+\infty)}$ as the semigroup of $A$ (or generated by $A$) and to $D(S_A)$ as the domain of the semigroup.
\end{definition}

\begin{remark}\label{rem:propertiessemigroups}
From Definition~\ref{def:semigroup} we obtain some properties of the semigroup of an operator $A$.
\begin{itemize}
    \item For any $t\in [0,+\infty)$ we have $\operatorname{Range}(S_A(t)) \subset D(S_A).$
    \item Let $t,s\in [0,+\infty)$, then we have $S_A(t+s)=S_A(t) \circ S_A(s)$. In particular we can identify the collection $\{S_A(t)\}_{t\in [0,+\infty)}$ as a left action of the semigroup $([0,+\infty),+)$ on the Banach space $X$. This is often referred to as the semigroup property. 
    \item For all $x\in D(S_A)$, $S_A(t)x$ is continuous in $t$.
    \item For all $x\in D(S_A)$, $S_A(0)x=x$.
\end{itemize}
\end{remark}

\begin{lemma}
\label{subsetOp}
 Given two operators $A,B$ such that $A \subset B$. Then $D(S_A) \subset D(S_B)$ and for any $x\in D(S_A)$ we have $S_A(t)x = S_B(t)x$.   
\end{lemma}

\begin{proof}
This follows from the definitions of semigroups and mild solutions, because if $(\mathcal{P}_{\varepsilon},v,\theta)$ is an $\varepsilon$-approximate solution to \eqref{AbCauchy} then it is also an $\varepsilon$-approximate solution to \eqref{AbCauchy} with the operator $A$ replaced with $B$. 
\end{proof}

Next we state the celebrated Crandall--Liggett theorem. We recall that the range condition is given in Definition~\ref{def:rangecondition}. For the meaning of the notation $R^n_{t/n}(A)$, we refer to Remark~\ref{rem:whyrangecon}.

\begin{theorem}[Crandall--Liggett]
\label{CranLig}
Let $A$ be an $\omega$-accretive operator on $X$ satisfying the range condition~\eqref{rangeCon} for a given $\delta>0$ and let $x\in \overline{D(A)}$. Then \eqref{AbCauchy} has a unique mild solution $u\in C([0,+\infty),X)$, which, for all $t>0$, is given by
$$ u(t) = \lim_{n\rightarrow \infty} R_{t/n}^n(A) x .$$
This limit is taken with respect to the topology on $X$.
\end{theorem}
\begin{proof}
See \cite[Theorem I]{CrandallLiggett71}.
\end{proof}

\begin{remark}
    In Theorem~\ref{CranLig}, if $t/n \geq \delta$, then $A$ is not guaranteed to satisfy the range condition and thus $R_{t/n}^n(A)$ is not guaranteed to be a non-empty operator (see Remark~\ref{rem:whyrangecon}). Since we are only interested in the limit $n\to\infty$, however, this does not impact the result of the theorem.
\end{remark}

\begin{corollary}
\label{SemigroupCon}
Suppose $A$ is $\omega$-accretive on $X$ and satisfies the range condition~\eqref{rangeCon}. Then $\overline{D(A)} \subset D(S_A)$ and, for all $x,y\in \overline{D(A)}$ and all $t\in [0,+\infty)$,
$$ \Vert S_A(t)x - S_A(t)y \Vert \leq e^{\omega t}\Vert x-y \Vert. $$
\end{corollary}
\begin{proof}
We can apply the Crandall--Liggett theorem (Theorem~\ref{CranLig}) together with the definition of $S_A$ in Definition~\ref{def:semigroup} to deduce that, for all $z\in \overline{D(A)}$,
$$ S_A(t)z = \lim_{n\rightarrow \infty} R_{t/n}^n(A) z .$$
By iteratively applying Remark~\ref{accremark}, we obtain that, for all $n\in \mathbb{N}$, $R_{t/n}^n(A)$ can be identified with a Lipschitz-continuous function with Lipschitz constant $(1-t\omega /n)^{-n}$ and domain $X$. Then we deduce, for all $t\in [0,\infty)$, 
\begin{align*}
    \Vert S_A(t) x - S_A(t) y \Vert &= \lim_{n\rightarrow \infty} \Vert R_{t/n}^n(A) x - R_{t/n}^n(A) y \Vert
    \leq \lim_{n\rightarrow \infty} (1-t\omega/n)^{-n}  \Vert x -y \Vert\\
    &= e^{\omega t} \Vert x -y \Vert.
\end{align*}
\end{proof}

\section{Properties of \texorpdfstring{$\lambda$}{lambda}-convex functions and gradient flows}\label{sec:convex}

Throughout Section~\ref{sec:convex} we assume $H$ to be a Hilbert space with inner product $\langle \cdot , \cdot \rangle$ and corresponding norm $\| \cdot \|$. 

\subsection{Definition and properties of \texorpdfstring{$\lambda$}{lambda}-convex functions}

\begin{definition}\label{def:lambdaconvex}
Let $\Phi:H\rightarrow (-\infty,+\infty]$ and $\lambda \in \mathbb{R}$. We say that $\Phi$ is $\lambda$-convex if, for all $x,y\in H$ and for all $t\in [0,1]$,
\[
\Phi( tx + (1-t)y) \leq t\Phi(x) +(1-t) \Phi(y) - \frac{1}{2} \lambda t(1-t) \Vert x-y \Vert^2.
\]
\end{definition}
It is clear from the definition above that $0$-convexity corresponds to the standard notion of convexity.

A generalisation of Definition~\ref{def:lambdaconvex} to functions $\Phi$ defined on metric spaces is possible; see for example \cite[Section 2.2]{vanGennipGigaOkamoto26}. In this paper we will only be interested in the case where $H$ is a Hilbert space.

\begin{definition}\label{def:subdifferential}
Let $\Phi: H \rightarrow (-\infty,+\infty]$. We define the subdifferential $\partial \Phi$ to be the operator on $H$ determined by $(x,y)\in \partial\Phi$ if and only if, for all $h\in H$,
$$ \Phi(x+h) \geq \Phi(x) + \langle y,h \rangle .$$
\end{definition}

Appendix~\ref{sec:C} contains a review of the basic properties of $\lambda$-convex functions and subdifferentials. If $\Phi$ is not convex, the subdifferential $\partial \Phi$ is not guaranteed to have the desirable properties of Proposition~\ref{subdiffProp}. In the case that $\Phi$ is $\lambda$-convex on $H$, however, the function $\Phi - \frac\lambda2 \|\cdot\|^2$ is convex on $H$, which motivates Definition~\ref{def:subdiflambda}.

\begin{definition}\label{def:subdiflambda}
Let $\lambda \in \mathbb{R}$ and $\Phi: H \rightarrow (-\infty,+\infty]$. We define the operator $\partial^{\lambda}\Phi$ on $H$ by
$$ \partial^{\lambda}\Phi := \partial\big( \Phi-\frac{\lambda}{2}\Vert \cdot \Vert^2 \big) + \lambda I_H .$$
We note that $\partial^0 \Phi = \partial \Phi$.
\end{definition}

Since we have introduced $\partial^\lambda \Phi$ to have a useful notion of subdifferential for a broader class of functions $\Phi$ than only the convex functions and since, by Lemma~\ref{lem:lambdaconvexanstrictlyconvex}, $\lambda$-convex functions are convex if $\lambda \geq 0$, it is natural to ask whether in this case $\partial^\lambda \Phi$ equals the standard subdifferential $\partial \Phi$. Part 2 of Lemma~\ref{lem:subdiffinclusion} in Appendix~\ref{sec:C} (with $\lambda_1=0 < \lambda_2$) gives an affirmative answer for the case when $\Phi$ is $\lambda$-convex.

\begin{proposition}\label{prop:partiallambdaaccretive}
Let $\lambda \in \mathbb{R}$ and $\Phi:H \rightarrow (-\infty,+\infty]$. Then $\partial^{\lambda} \Phi$ is $(-\lambda)$-accretive.
\end{proposition}
\begin{proof}
By Definition~\ref{def:omegamaccretive} and Lemma~\ref{Hilbertaccretive}, $\partial^\lambda \Phi$ is $(-\lambda)$-accretive if, for all $(x,y), (\hat x, \hat y) \in \partial^\lambda \Phi - \lambda I_H = \partial(\Phi-\frac\lambda2 \|\cdot\|^2)$, $\langle y-\hat y, x-\hat x\rangle \geq 0$. 
Therefore, assume that $(x,y),(\hat{x},\hat{y})\in \partial(\Phi-\frac\lambda2 \|\cdot\|^2)$. Then, if for $(x,y)$ we choose $h=\hat x -x$ in Definition~\ref{def:subdifferential}, we obtain
\[
\Phi(\hat x) - \frac\lambda2 \|\hat x\|^2 \geq \Phi(x) - \frac\lambda2 \|x\|^2 + \langle y, \hat x-x\rangle.
\]
Similarly, if for $(\hat x, \hat y)$ we choose $h=x-\hat x$ in Definition~\ref{def:subdifferential}, we get
\[
\Phi(x) - \frac\lambda2 \|x\|^2 \geq \Phi(\hat x) - \frac\lambda2 \|\hat x\|^2 + \langle \hat y, x-\hat x\rangle.
\]
Subtracting the second inequality from the first, we obtain the required inequality ${\langle y-\hat y, x-\hat x\rangle \geq 0}$.
\end{proof}

\begin{theorem}\label{thm:partiallambdamaccretive}
Let $\lambda \in \mathbb{R}$, $\Phi:H \rightarrow (-\infty,+\infty]$ be proper, $\lambda$-convex, and lower semicontinuous. Then $\partial^{\lambda} \Phi$ is $(-\lambda)$-m-accretive.
\end{theorem}
\begin{proof}
By Definition~\ref{def:omegamaccretive}, $(-\lambda)$-m-accretivity of $\partial^\lambda \Phi$ is equivalent to m-accretivity of $\partial^{\lambda}\Phi-\lambda I$. By the definition of $\partial^\lambda \Phi$ in Definition~\ref{def:subdiflambda} $\partial^{\lambda}\Phi-\lambda I = \partial\big( \Phi-\frac{\lambda}{2}\Vert \cdot \Vert^2 \big)$. Since $\Phi-\frac{\lambda}{2}\Vert \cdot \Vert^2$ is proper, convex (i.e., $0$-convex) and lower semicontinuous, it suffices to prove this theorem in the case that $\lambda=0$. A proof of this case can be found in \cite[Theorem 2.8]{barbu2010nonlinear}, where we note that on a Hilbert space maximal monotonicity of an operator (which is what the conclusion of \cite[Theorem 2.8]{barbu2010nonlinear} gives us) is equivalent to m-accretivity. Indeed, comparing \cite[Definition 2.1]{barbu2010nonlinear} with Lemma~\ref{Hilbertaccretive} and Definition~\ref{def:maximalmonotone}, we see that on Hilbert spaces the concept of (maximal) monotonicity in \cite{barbu2010nonlinear} is equivalent to (maximal) accretivity; Proposition~\ref{propMaccret} establishes equivalence between maximal accretivity and m-accretivity in the Hilbert space setting.
\end{proof}

\begin{definition}
\label{MoreauProx}
Let $\Phi:H \rightarrow (-\infty,+\infty]$ be proper. We define the Moreau envelope with parameter $\gamma>0$ to be the function $[\Phi]^{\gamma}:H \rightarrow [-\infty,+\infty)$ defined by
$$ [\Phi]^{\gamma}(x) := \underset{y\in H}{\inf} \left(\Phi(y) + \frac{1}{2\gamma}\Vert y-x \Vert^2\right).$$
Moreover, we define the operator $J_{\gamma}(\Phi)$ on $H$ to be such that $(x,y)\in J_{\gamma}(\Phi)$ if and only if
$$ [\Phi]^{\gamma}(x) = \Phi(y) + \frac{1}{2\gamma}\Vert y-x \Vert^2.$$
We call $J_{\gamma}(\Phi)$ the proximal operator.
\end{definition}

Next we establish some properties of Moreau envelopes and afterwards prove a lemma that connects the proximal operator to the operator $\partial^{\lambda}\Phi$.

\begin{remark}
\label{rem:envelopesupport} Let $\Phi:H \rightarrow (-\infty,+\infty]$ be proper. The graph of a Moreau envelope of a function lies below the graph of that function. Indeed, for all $\gamma>0$ and all $x\in H$,
$$ [\Phi]^{\gamma}(x)\leq \Phi(x) + \frac{1}{2\gamma}\Vert x-x \Vert^2 = \Phi(x).$$
Moreover the envelope coincides with the function at minimizers. Given a minimizer $x^*$ of $\Phi$ we note that, for all $\gamma>0$ and all $x\in H$,
$$ \Phi(x^*) \leq \Phi(x) + \frac{1}{2\gamma}\Vert x-x^* \Vert^2$$
and so $\Phi(x^*) \leq [\Phi]^{\gamma}(x^*)$; hence $\Phi(x^*) = [\Phi]^{\gamma}(x^*)$.
\end{remark}

We recall that the definition of the interval $\mathfrak{I}_{-\lambda}$ is given in \eqref{OmegaInterval}. In what follows, in a slight abuse of notation, if $J_\gamma(\Phi)(x)$ is a singleton, we use the same notation $J_{\gamma}(\Phi)(x)$ (or $J_{\gamma}(\Phi)x$) for the set and for the single element in the set.

\begin{proposition}\label{envelopeProp}
Let $\lambda \in \mathbb{R}$, $\Phi:H \rightarrow (-\infty,+\infty]$ be proper, $\lambda$-convex, and lower semicontinuous. Then the following hold.
\begin{enumerate}
    \item For all $\gamma \in \mathfrak{I}_{-\lambda}$ and for all $x\in H$, $J_{\gamma}(\Phi)(x)$ is a singleton.
    \item For all $x\in H$, $\gamma \mapsto [\Phi]^{\gamma}(x)$ is non-increasing. Moreover, if $\lambda \geq 0$, then $\gamma \mapsto \Phi(J_{\gamma}(\Phi)x)$ is non-increasing and $[\Phi]^{\gamma}(x) \rightarrow \inf_{y\in H} \Phi(y)$ as $\gamma \rightarrow \infty$.
    \item For all $x\in H$, $[\Phi]^{\gamma}(x) \rightarrow \Phi(x)$ as $\gamma\rightarrow 0$. 
    \item For all $\gamma>0$ and $\delta>0$ such that $\gamma +\delta \in \mathfrak{I}_{-\lambda}$ we have
    $$ [[\Phi]^{\gamma}]^{\delta} = [\Phi]^{\gamma+\delta}.$$
\end{enumerate}
\end{proposition}
\begin{proof}
\begin{enumerate}
    \item Let $\gamma\in \mathfrak{I}_{-\lambda}$ and $x\in H$. The function $y\mapsto \Phi(y)+\frac{1}{2\gamma}\Vert y-x \Vert^2$ is $\big(\frac{1}{\gamma} + \lambda \big)$-convex. By the definition of $\mathfrak{I}_{-\lambda}$ in \eqref{OmegaInterval}, $\frac{1}{\gamma} + \lambda>0$ and thus Proposition \ref{ConvexMin} tells us that the infimum in the definition of the Moreau envelope $[\Phi]^\gamma(x)$ is achieved at a unique $y\in H$. Hence, for every $x\in H$ there is exactly one $y\in H$ such that $(x,y)\in J_\gamma(\Phi)$. This proves the claim.
    
    \item Let $x\in H$. If $\gamma \leq \delta$, then, for all $y\in H$,
    $$\Phi(y)+\frac{1}{2\gamma} \Vert y-x \Vert^2 \geq \Phi(y)+\frac{1}{2\delta} \Vert y-x \Vert^2$$
    and so $[\phi]^{\gamma}(x) \geq [\phi]^{\delta}(x)$.

    Next assume that $\lambda\geq 0$. A proof that $\gamma \mapsto \Phi(J_{\gamma}(\Phi)x)$ is non-increasing can be found in \cite[Proposition 12.27]{Bauschke}.

    Given $\varepsilon>0$, choose $z\in H$ such that $\Phi(z) \leq \inf_{y\in H} \Phi(y) + \varepsilon$. For $\gamma \geq \frac{1}{2\varepsilon} \Vert z-x \Vert^2$ we have
    $$ \inf_{y\in H} \Phi(y) \leq [\Phi]^{\gamma}(x) \leq \Phi(z) + \frac{1}{2\gamma} \Vert x-z \Vert^2 \leq \inf_{y\in H} \Phi(y) +2\varepsilon.$$
    From this we conclude that $[\Phi]^{\gamma}(x) \rightarrow \inf_{y\in H} \Phi(y)$ as $\gamma \rightarrow +\infty$. 
    
    \item The case when $\lambda \geq 0$ is proved in \cite[Proposition 12.33 (ii)]{Bauschke}. Here we address the case when $\lambda<0$. Let $x\in H$ and define the function $\Psi:H \rightarrow (-\infty,+\infty]$ by
    \begin{equation}\label{eq:Psi}
    \Psi(y):= \Phi(y) - \frac{\lambda}{2} \Vert y-x \Vert^2.
    \end{equation}
    We note that $\Psi$ is convex and lower semicontinuous. Moreover,
    \begin{equation}\label{eq:Psiequality}
    \Phi(y) +\frac{1}{2\gamma}\Vert y - x \Vert^2 = \Psi(y) +\frac{1+\gamma \lambda}{2\gamma} \Vert y-x \Vert^2.
    \end{equation}
    Since, for all $\gamma \in \mathfrak{I}_{-\lambda}$, $1+\gamma \lambda >0$, we deduce that 
    $ [\Phi]^{\gamma}(x) = [\Psi]^{\frac{\gamma}{1+\gamma \lambda}}(x)$.
    Given that the result for $\lambda =0$ is already proven and $\frac{\gamma}{1+\gamma \lambda} \rightarrow 0$ as $\gamma \rightarrow 0$ we conclude that
    $$ [\Phi]^{\gamma}(x) = [\Psi]^{\frac{\gamma}{1+\gamma \lambda}}(x) \rightarrow \Psi(x) = \Phi(x) \; \; \text{as} \; \; \gamma \rightarrow 0.$$
    
    \item See \cite[Proposition 12.22 (ii)]{Bauschke}.
\end{enumerate}
\end{proof}

Next we extend part 3 of Proposition~\ref{envelopeProp} to sequences in $\mathfrak{I}_{-\lambda}$ with limits other than zero.

\begin{proposition}
\label{envelopeProp2}
Let $\lambda \in \mathbb{R}$, $\Phi:H \rightarrow (-\infty,+\infty]$ be proper, $\lambda$-convex, and lower semicontinuous. Assume the sequence $(t_n)_{n\in \mathbb{N}} \subset \mathfrak{I}_{-\lambda}$ and $t_* \in \mathfrak{I}_{-\lambda}$ are such that $t_n \rightarrow t_*$ as $n\rightarrow \infty$. Let $x\in H$. Then $[\Phi]^{t_n}(x) \rightarrow [\Phi]^{t_*}(x)$ as $n\rightarrow \infty$.
\end{proposition}

\begin{proof}
Let $(t_n)$, $t_*$, and $x$ be as in the statement of the proposition. First we prove the result for $\lambda\geq0$. Since the sequence $(t_n)$ converges to $t_* \in \mathfrak{I}_{-\lambda}$, and thus in particular has a positive limit, there exist $\alpha, \beta \in (0,+\infty)$ such that, for all $n\in \mathbb{N}$, $\alpha \leq t_n \leq \beta$. Then, by part 2 of Proposition~\ref{envelopeProp}, $\Phi(J_{t_n}(\Phi)x) + \frac{1}{2t_n} \Vert J_{t_n}(\Phi)x - x \Vert^2 = [\Phi]^{t_n}(x) \leq  [\Phi]^{\alpha}(x)$. Moreover, again by part 2 of the same proposition, $\Phi(J_{t_n}(\Phi)x) \geq \Phi(J_{\beta}(\Phi)x)$ and hence
$ \Vert J_{t_n}(\Phi)x - x \Vert^2 \leq 2\beta \big( [\Phi]^{\alpha}(x) - \Phi(J_{\beta}(\Phi)x) \big).$ From this we deduce that
$$ \Vert J_{t_n}(\Phi) x\Vert^2 \leq 2\Vert x \Vert^2+ 2 \Vert J_{t_n}(\Phi)x - x \Vert^2 \leq 2\Vert x \Vert^2+4\beta\big( [\Phi]^{\alpha}(x) - \Phi(J_{\beta}(\Phi)x) \big). $$
Importantly, the sequence $(J_{t_n}(\Phi) x)_{n\in \mathbb{N}}$ is bounded and so by the Banach--Alaoglu theorem we deduce that there exists a $z\in H$ such that $(J_{t_n}(\Phi) x)$ converges weakly to $z$. Next we compute that
\begin{align}\label{eq:liminfinequalityMoreau}
    [\Phi]^{t_*}(x) &\leq \Phi(z) + \frac{1}{2t_*} \Vert x-z \Vert^2
    \leq \underset{n\rightarrow \infty}{\liminf} \Phi(J_{t_n}(\Phi)x) + \frac{1}{2t_n} \Vert J_{t_n}(\Phi)x - x\Vert^2\notag\\
    &= \underset{n\rightarrow \infty}{\liminf} [\Phi]^{t_n}(x).
\end{align}
The first inequality follows from the definition of the Moreau envelope. The second inequality follows since both $\Phi$ and the norm $\Vert \cdot \Vert$, being convex, are also weakly lower semicontinuous (see \cite[Theorem 9.1]{Bauschke}). Furthermore, 
\begin{align*}
    [\Phi]^{t_n}(x) &\leq \Phi(J_{t_*}(\Phi)x)+\frac{1}{2t_n} \Vert J_{t_*}(\Phi) x -x \Vert^2
    \rightarrow  \Phi(J_{t_*}(\Phi)x)+\frac{1}{2t_*} \Vert J_{t_*}(\Phi) x -x \Vert^2\\
    &= [\Phi]^{t_*}(x),
\end{align*}
where the convergence is as $n\to\infty$. It follows that $\underset{n\rightarrow \infty}{\limsup} [\Phi]^{t_n}(x) \leq [\Phi]^{t_*}(x)$, which, together with the inequality in \eqref{eq:liminfinequalityMoreau}, establishes that $[\Phi]^{t_n}(x) \rightarrow [\Phi]^{t_*}(x)$ as $n\rightarrow \infty$. 

Next we prove the result for the case when $\lambda<0$, using a similar approach as in the proof of part 3 of Proposition~\ref{envelopeProp}. Define the function $\Psi:H \rightarrow (-\infty,+\infty]$ as in \eqref{eq:Psi}. Then $\Psi$ is convex and lower semicontinuous. Moreover, for all $y\in H$ and all $n\in \mathbb{N}$, the equality in \eqref{eq:Psiequality} holds with $\gamma=t_n$ and thus, since $1+t_n \lambda >0$, we obtain $[\Phi]^{t_n}(x) = [\Psi]^{\frac{t_n}{1+t_n \lambda}}(x)$.
    Because the result is already proven for the case when $\lambda =0$ and $\frac{t_n}{1+t_n \lambda} \rightarrow \frac{t_*}{1+t_* \lambda}$ as $n \rightarrow \infty$, we conclude that
    $$ [\Phi]^{t_n}(x) = [\Psi]^{\frac{t_n}{1+t_n \lambda}}(x) \rightarrow [\Psi]^{\frac{t_*}{1+t_* \lambda}}(x) = [\Phi]^{t_*}(x) \; \; \text{as} \; \; n \rightarrow \infty.$$
\end{proof}

\begin{lemma}
\label{ProxResolve}
Let $\lambda \in \mathbb{R}$ and let $\Phi:H \rightarrow (-\infty,+\infty]$ be proper, $\lambda$-convex, and lower semicontinuous. Then, for all $\gamma \in \mathfrak{I}_{-\lambda}$, 
$$J_{\gamma}(\Phi) = R_{\gamma}(\partial^{\lambda}\Phi).$$
In this case we can identify $J_{\gamma}(\Phi)$ with a Lipschitz-continuous function on $H$ with Lipschitz constant $\frac{1}{1+\gamma \lambda}$.
\end{lemma}

\begin{proof}
Let $\lambda\in \mathbb{R}$, $\gamma \in \mathfrak{I}_{-\lambda}$, $x\in H$. By Theorem~\ref{thm:partiallambdamaccretive} $\partial^\lambda \Phi$ is m-accretive and thus, since $\gamma>0$, by Remark~\ref{accremark}, $R_\gamma(\partial^\lambda \Phi)$ is a singleton. Define $y$ to be the unique element in $R_{\gamma}(\partial^{\lambda}\Phi)(x)$. By the definition of the resolvent in Definition~\ref{def:operatorsoperations} we have $\frac{x-y}{\gamma} \in \partial^{\lambda} \Phi(y)$. It follows that
$$ \left(y,\frac{x-y}{\gamma} - \lambda y \right) \in \partial\big( \Phi - \frac{\lambda}{2} \Vert \cdot\Vert^2 \big) .   $$
Then, for all $z \in H$,
\begin{align*}
\Phi(z) + \frac{1}{2\gamma} \Vert z -x \Vert^2 &= \Phi(z) -\frac{\lambda}{2} \Vert z \Vert^2 + \frac{\lambda}{2} \Vert z \Vert^2 +  \frac{1}{2\gamma} \Vert z -x \Vert^2\\
&\geq \Phi(y) -\frac{\lambda}{2} \Vert y \Vert^2 + \left\langle \frac{x-y}{\gamma} - \lambda y, z-y \right\rangle  + \frac{\lambda}{2} \Vert y +z-y \Vert^2  + \frac{1}{2\gamma} \Vert z -y + y - x \Vert^2\\
&= \Phi(y) -\frac{\lambda}{2} \Vert y \Vert^2 +\frac{1}{\gamma} \langle x-y, z-y \rangle -\lambda \langle y,z-y \rangle \\
&+\frac{\lambda}{2} \Vert y \Vert^2 +\lambda \langle y,z-y \rangle + \frac{\lambda}{2}\Vert z-y \Vert^2 + \frac{1}{2\gamma} \Vert z-y \Vert^2 + \frac{1}{\gamma} \langle y-x, z-y \rangle + \frac{1}{2\gamma} \Vert y -x \Vert^2\\
&= \Phi(y) + \frac{1}{2\gamma} \Vert y-x \Vert^2 
+ \left( \frac{\lambda}{2} + \frac{1}{2\gamma}\right)\Vert z-y \Vert^2\\
&\geq \Phi(y) + \frac{1}{2\gamma} \Vert y-x \Vert^2.
\end{align*}
The first inequality follows from the definition of subdifferential in Definition~\ref{def:subdifferential} and the second follows since $\frac{1}{\gamma} + \lambda >0$, because $\gamma\in \mathfrak{I}_{-\lambda}$. 

From the inequality above we conclude that $y$ minimizes $\Phi(\cdot) + \frac{1}{2\gamma}\Vert\cdot -x \Vert^2$ over $H$; moreover by part 1 of Proposition~\ref{envelopeProp} it is the unique minimizer. Thus our result holds.

The final claim of the lemma holds by Remark~\ref{accremark}, because, by Theorem~\ref{thm:partiallambdamaccretive}, $\partial^\lambda \Phi$ is $(-\lambda)$-m-accretive.
\end{proof}

Now that we have established the prerequisites for $\lambda$-convex functions, we will define gradient flows and establish some of their properties.

\subsection{Definition and properties of gradient flows}

\begin{definition}
\label{gradflowdef}
Let $\lambda \in \mathbb{R}$, $\Phi:H \rightarrow (-\infty,+\infty]$ be proper, $\lambda$-convex, and lower semicontinuous. We say a curve $u:(0,+\infty)\rightarrow H$ is the gradient flow of $\Phi$ (over $H$) starting at $x_0\in H$ if $u$ is locally absolutely continuous, $u(t)\rightarrow x_0$ as $t\downarrow 0$, and $u$ satisfies the following inequality called the evolution variational inequality: for a.e. $t\in [0,+\infty)$ and for all $v\in H$,
$$ \frac{1}{2}\frac{d}{dt}\Vert u(t) -v \Vert^2 + \frac{1}{2}\lambda \Vert u(t) -v \Vert^2 + \Phi(u(t)) \leq \Phi(v) .$$
\end{definition}

\begin{remark}\label{rem:gradientflow}
Let $\Phi$ be as in Definition~\ref{gradflowdef}. Then by \cite[Theorem 4.0.4, i]{greenbook}, for all $x_0 \in \overline{\operatorname{dom}(\Phi)}$ the gradient flow $u$ of $\Phi$ starting at $x_0$ exists and is unique. Moreover, by \cite[Theorem 4.0.4, ii]{greenbook}, for all $t>0$, $u(t) \in \operatorname{dom}(\Phi)$, and thus by \cite[Lemma 2.2]{vanGennipGigaOkamoto26} $t\mapsto \Phi(u(t))$ is non-increasing on $[0,\infty)$ where, as a matter of convention, we define $u(0):=x_0$. 

Moreover, gradient flows satisfy the semigroup property (see Remark~\ref{rem:propertiessemigroups} in the following sense. If we define a family $\{S_t\}_{t\geq 0}$ of maps $S_t:\overline{\operatorname{dom}(\Phi)}\rightarrow H$ by $S_t(x_0):=u(t)$, where $u$ is the gradient flow starting from $x_0$, then, for all $T_1,T_2>0$,
$$ S_{T_1+T_2}(x_0)=S_{T_2}(S_{T_1}(x_0)).$$
\end{remark}

There are other possible ways to define a gradient flow. We have chosen the formulation in Definition~\ref{gradflowdef} since it suits our current purposes, but note that this definition coincides with other common formulations. Notably, a gradient flow as in Definition~\ref{gradflowdef} is also a curve of maximal slope with respect to the upper gradient $|\partial \Phi|$ (as in \cite[Definition 1.2.4]{greenbook}). Moreover, it also coincides with a minimizing movement for $\Phi$ (as in \cite[Definition 2.0.6]{greenbook}). For full details on these equivalences we refer the reader to \cite[Theorem 4.0.4]{greenbook}. In the following remark, we collect the relevant results needed to connect Definition~\ref{gradflowdef} with the notion of curves of maximal slope. We will not go into full detail here and refer to the appropriate places in \cite{greenbook} for further information.

\begin{remark}
\label{rem:maxslope}
Let $\Phi: H \to (-\infty, +\infty]$ satisfy all the conditions in Definition~\ref{gradflowdef}, let $u:[0,+\infty)\rightarrow H$ be the gradient flow starting at $x_0 \in \overline{\operatorname{dom}(\Phi)}$ as in Definition~\ref{gradflowdef}, and let $\delta>0$. Then, by the semigroup property in \cite[Theorem 4.0.4, iv]{greenbook}, the map $t\mapsto u(t+\delta)$ is the gradient flow of $\Phi$ starting at $u(\delta)$. Moreover, as noted in \cite[Theorem 4.0.4, ii]{greenbook}, $u(\delta)\in \operatorname{dom}(\Phi)$. Now combining  \cite[Theorem 2.3.3]{greenbook} and \cite[Corollary 2.4.10]{greenbook} we note $t\mapsto u(t+\delta)$ is the unique curve of maximal slope for $\Phi$ with respect to the strong upper gradient $|\partial\Phi|$ starting from $u(\delta)$. The definition of maximal slopes with respect to upper gradients is given in \cite[Definition 1.3.2]{greenbook}. We emphasise here that \cite[Theorem 2.3.3]{greenbook} requires the initial condition of $u$ to be in $\operatorname{dom}(\Phi)$, which we cannot guarantee for $x_0$, hence the need to start from $u(\delta)$. 

Since $t\mapsto u(t+\delta)$ is a curve of maximal slope, \cite[Remark 1.3.3]{greenbook} establishes that the map $t\mapsto \Phi(u(t+\delta))$ is locally absolutely continuous on $(0,\infty)$. Given that $\delta>0$ was chosen arbitrarily, we deduce that $t\mapsto \Phi(u(t))$ is locally absolutely continuous on $(0,\infty)$.

We also note that \cite[Theorem 2.3.3]{greenbook} states that the following energy identity holds for all $T>\delta$:
$$ \frac{1}{2} \int^T_\delta |u'|^2(t)dt +  \frac{1}{2} \int^T_\delta |\partial\Phi|^2(u(t))dt + \Phi(u(T)) = \Phi(u(\delta)) . $$
From this we read that $t\mapsto \Phi(u(t))$ is non-increasing on $[\delta, +\infty)$. Since $\delta>0$ is arbitrary, the function is non-increasing on $(0,\infty)$. In the case that $x_0\notin \operatorname{dom}(\Phi)$, we have $\Phi(x_0)=+\infty$, and thus $t\mapsto \Phi(u(t))$ is non-increasing on $[0,+\infty)$. If instead $x_0\in \operatorname{dom}(\Phi)$, we could simply take $\delta=0$ from the start of the argument and again conclude that $t\mapsto \Phi(u(t))$ is non-increasing on $[0,+\infty)$. This gives an alternative justification of \cite[Lemma 2.2]{vanGennipGigaOkamoto26}. 
\end{remark}

\begin{remark}\label{rem:minimizingmovement}
    Again, let $\Phi: H \to (-\infty, +\infty]$ satisfy all the conditions in Definition~\ref{gradflowdef} and let $u:[0,+\infty)\rightarrow H$ be the gradient flow starting at $x_0 \in \overline{\operatorname{dom}(\Phi)}$ as in Definition~\ref{gradflowdef}. As a consequence of \cite[Theorem 4.0.4, i]{greenbook} $u$ is a minimizing movement for $\Phi$ starting at $x_0$. The definition of minimizing movements is given in \cite[Definition 2.0.6]{greenbook}. 

In particular, recalling the operator $J_\gamma(\Phi)$ from Definition~\ref{MoreauProx}, by \cite[Theorem 4.0.4]{greenbook} we have, for all $t\in (0,+\infty)$
\begin{equation}\label{eq:minimizinglimit}
u(t) = \lim_{n \rightarrow \infty}J_{t/n}^n(\Phi)(x_0),
\end{equation}
where $J^n_{t/n}(\Phi) := \left(J_{t/n}(\Phi)\right)^n$. Similarly as in Remark~\ref{rem:whyrangecon}, this should be understood as an $n$-fold composition of the function associated with $J_{t/n}(\Phi)$. Indeed, for $n$ large enough we have $t/n\in \mathfrak{I}_{-\lambda}$ and thus, by part~1 of Proposition~\ref{envelopeProp}, $J_{t/n}(\Phi)(x_0)$ is a singleton and $J_{t/n}(\Phi)$ can be understood as a function from $H$ to $H$.

In this paper, our use of the minimizing-movement formulation of gradient flows is restricted to the use of the formula in \eqref{eq:minimizinglimit}. 
\end{remark}

We will make use of the minimizing-movement formulation mentioned in Remark~\ref{rem:minimizingmovement}, and equation \eqref{eq:minimizinglimit} specifically, for two results. First, in Theorem~\ref{GradFlowTheorem} we improve the a-priori estimate for the regularizing effect given in \cite[Theorem 4.3.2]{greenbook} via a proof strategy different from the one in \cite{greenbook}. Our result will be limited to the setting of a Hilbert space, whereas \cite[Theorem 4.3.2]{greenbook} holds in a metric space. We expect that our result can be generalised to metric spaces as well (see also the short discussion at the end of this section), but since this is not required for our purposes here, this falls outside of the scope of this paper. Second, in Lemma~\ref{lem:gradflowsemigroup} we use the minimizing-movement formulation to characterise our gradient flows in terms of a semigroup generated by a suitable operator.

\begin{remark}\label{rem:kappa}
In preparation of Theorem~\ref{GradFlowTheorem} we note that the Moreau envelope $[\Phi]^{\kappa(t,\lambda)}(x_0)$ that appears in the statement of that theorem is well defined. To be precise: if $\lambda \in \mathbb{R}$, $t> 0$, $\mathfrak{I}_{-\lambda}$ is the interval from \eqref{OmegaInterval}, and $\kappa$ is as in Theorem~\ref{GradFlowTheorem}, then $\kappa(t,\lambda)\in \mathfrak{I}_{-\lambda}$. 

Indeed, if $\lambda\geq 0$, then $\mathfrak{I}_{-\lambda}=(0,+\infty)$ and thus $\kappa(t,\lambda)\in \mathfrak{I}_{-\lambda}$. If $\lambda<0$, then $\mathfrak{I}_{-\lambda}=(0,1/|\lambda|)$ and
\[
0<\kappa(t,\lambda) = \frac{e^{2\lambda t}-1}{2\lambda} = \frac{1-e^{-2|\lambda| t}}{2|\lambda|} < \frac{1}{|\lambda|}.
\]
Thus $\kappa(t,\lambda)\in \mathfrak{I}_{-\lambda}$ as claimed.
\end{remark}

\begin{theorem}
\label{GradFlowTheorem}
Let $\lambda \in \mathbb{R}$, $\Phi:H \rightarrow (-\infty,+\infty]$ be proper, $\lambda$-convex, and lower semicontinuous. Let $u$ be the unique gradient flow of $\Phi$ starting at $x_0 \in \overline{\operatorname{dom}(\Phi)}$. Define a function $\kappa: (0,+\infty) \times \mathbb{R} \to \mathbb{R}$ by $\kappa(t,s):= \frac{e^{2s t}-1}{2s}$ for $s\neq 0$ and $\kappa(t,0):=t$. Then, for all $t> 0$,
$$ \Phi(u(t))\leq [\Phi]^{\kappa(t,\lambda)}(x_0).$$
\end{theorem}

We will find the previous estimate useful as it gives a quantitative estimate on how the energy of the gradient flow stays below a Moreau envelope. That Moreau envelope will be easier to work with analytically because it is more regular, for example in the case that $\Phi$ is lower semicontinuous and convex it has been proved that $[\Phi]^{\gamma}$ is Fr\'echet differentiable with Lipschitz-continuous gradient (see \cite[Proposition 12.30]{Bauschke}.) For convenience in reading the calculations below, we write $[\Phi]^{\wedge}\{\gamma\}:=[\Phi]^\gamma$.

\begin{proof}[Proof of Theorem~\ref{GradFlowTheorem}]
By Remark~\ref{rem:minimizingmovement} we know that $u$ is a minimizing movement, in particular \eqref{eq:minimizinglimit} holds. Moreover, we emphasise that, if $t> 0$, then, for $n\in \mathbb{N}$ sufficiently large, $\frac{t}n \in \mathfrak{I}_{-\lambda}$ and thus, by part 1 of Proposition~\ref{envelopeProp}, the set $J_{t/n}^n(\Phi)x_0$ contains a single element which we again denote by $J_{t/n}^n(\Phi)x_0$. 

For later use in this proof we compute that, for all $\lambda\neq 0$ and all $t> 0$, 
\begin{equation}\label{eq:somesum1}
\sum_{k=0}^{n-1} \frac{t}{n} \cdot \left(1+ \frac{\lambda t}{n} \right)^{2k} = \frac{t}{n} \left[ \frac{(1+ \frac{\lambda t}{n} )^{2n} -1}{(1+ \frac{\lambda t}{n} )^2-1} \right] = \left[ \frac{(1+ \frac{\lambda t}{n} )^{2n} -1}{2\lambda + \frac{\lambda^2 t}{n}} \right]
\end{equation}
and, for $\lambda=0$ and all $t> 0$, 
\begin{equation}\label{eq:somesum2}
\sum_{k=0}^{n-1} \frac{t}{n} \cdot \left(1+\frac{\lambda t}{n}\right)^{2k} = t.
\end{equation}

Next we prove by induction on $n$ that, for all $\lambda\in \mathbb{R}$, all $n\in \mathbb{N}$, and all $t> 0$ with $t/n \in \mathfrak{I}_{-\lambda}$,
\begin{equation}
\label{induction1}
 \Phi(J_{t/n}^n(\Phi)x_0) \leq [\Phi]^{\wedge}\left\{ \sum_{k=0}^{n-1} \frac{t}{n} \cdot \left(1+\frac{\lambda t}{n}\right)^{2k} \right\} (x_0).
\end{equation}
Before we start the induction argument, we shall establish that both sides of the inequality in \eqref{induction1} are well defined. By our earlier observation, $J_{t/n}^n(\Phi)x_0$ is single valued and thus $\Phi(J_{t/n}^n(\Phi)x_0)$ is unambiguously defined. Moreover, if $\lambda\geq 0$, then $\mathfrak{I}_{-\lambda}=(0,\infty)$ and thus the quantity in curly brackets in the right-hand side of \eqref{induction1} is in $\mathfrak{I}_{-\lambda}$, which implies that the Moreau envelope in \eqref{induction1} is well defined. If, on the other hand, $\lambda<0$, then
$$ 0< \frac{(1+ \frac{\lambda t}{n} )^{2n} -1}{2\lambda + \frac{\lambda^2 t}{n}}  = \frac{1-(1- \frac{|\lambda| t}{n} )^{2n} }{2|\lambda|- \frac{\lambda^2 t}{n}}<\frac{1 }{2|\lambda|- \frac{\lambda^2 }{|\lambda|}}=\frac{1}{|\lambda|}.$$
The first inequality holds, because by assumption in \eqref{induction1}, $t/n\in \mathfrak{I}_{-\lambda} = (0, 1/|\lambda|)$ and thus $1-(1+ \frac{\lambda t}{n} )^{2n} > 0$ and $|\lambda| \left(2-\frac{|\lambda| t}{n}\right) > 0$. So, by \eqref{eq:somesum1}, again the quantity in curly brackets in the right-hand side of \eqref{induction1} is in $\mathfrak{I}_{-\lambda}$ and hence the Moreau envelope in \eqref{induction1} is well defined.

For $n=1$, \eqref{induction1} is equivalent to
\begin{equation}\label{eq:envelopeinequality}
\Phi(J_t(\Phi)x_0)\leq [\Phi]^t(x_0),
\end{equation}
which holds by the definition of the proximal operator in Definition~\ref{MoreauProx}. For the induction step, assume that \eqref{induction1} holds for a specific $n\in \mathbb{N}$. Then, for all $t> 0$ with $t/n\in \mathfrak{J}_{-\lambda}$,
\begin{align*}
\Phi(J_{t/(n+1)}^{n+1}(\Phi)x_0) &\leq [\Phi]^{t/(n+1)} \big( J_{t/(n+1)}^n(\Phi)x_0 \big)\\
&= \underset{y\in H}{\argmin} \; \Phi(y) + \frac{n+1}{2t} \Vert y - J^n_{t/(n+1)}(\Phi)x_0 \Vert^2\\
&\leq \underset{z\in H}{\argmin} \; \Phi(J^n_{t/(n+1)}(\Phi)z) + \frac{n+1}{2t} \Vert J^n_{t/(n+1)}(\Phi)z - J^n_{t/(n+1)}(\Phi)x_0 \Vert^2\\
&\leq \underset{z\in H}{\argmin} \; \Phi(J^n_{t/(n+1)}(\Phi)z) + \frac{n+1}{2t} \left(\frac{1}{1+t\lambda/(n+1)} \right)^{2n}\Vert z-x_0 \Vert^2\\
&\leq \underset{z\in H}{\argmin} \; [\Phi]^{\wedge}\left\{\sum_{k=0}^{n-1} \frac{t}{n+1} \cdot \left(1+ \frac{\lambda t}{n+1}\right)^{2k} \right\}(z) + \frac{n+1}{2t} \left(\frac{1}{1+t\lambda/(n+1)} \right)^{2n}\Vert z-x_0 \Vert^2\\
&= \left[[\Phi]^{\wedge}\left\{\sum_{k=0}^{n-1} \frac{t}{n+1} \cdot \left(1+ \frac{\lambda t}{n+1}\right)^{2k} \right\}\right]^{\wedge}\left\{\frac{t}{n+1} \left(1+ \frac{\lambda t}{n+1}\right)^{2n}\right\}(x_0)\\
&= [\Phi]^{\wedge}\left\{\sum_{k=0}^{n} \frac{t}{n+1} \cdot \left(1+ \frac{\lambda t}{n+1}\right)^{2k} \right\} (x_0).
\end{align*}
The first inequality holds by \eqref{eq:envelopeinequality} with $\frac{t}{(n+1)}$ instead of $t$ and $J_{t/(n+1)}^n(\Phi)x_0$ instead of $x_0$. The first and second equalities follow from the definition of the Moreau envelope. The second inequality holds since $\{y=J^n_{t/(n+1)}(\Phi)z \,|\, z\in H\} \subset H$. The third inequality follows from Lemma~\ref{ProxResolve}, specifically the statement in that lemma about the Lipschitz constant for the proximal operator. The fourth inequality is a consequence of the induction assumption. Part 4 of Proposition \ref{envelopeProp} establishes the final equality. This establishes that \eqref{induction1} holds for all $n\in \mathbb{N}$.

By taking the limit $n\to\infty$ in \eqref{eq:somesum1} and \eqref{eq:somesum2}, we find that, for all $\lambda \in \mathbb{R}$,
$$\sum_{k=0}^{n-1} \frac{t}{n} \cdot \left(1+ \frac{\lambda t}{n} \right)^{2k} \rightarrow 
\left.\begin{cases}
    \frac{e^{2\lambda t}-1}{2\lambda}, &\text{if } \lambda \neq 0,\\
    t, &\text{if } \lambda = 0
\end{cases}\right\}
= \kappa(t,\lambda).
$$
Given that $\Phi$ is lower semicontinuous and using \eqref{eq:minimizinglimit}, we conclude that
$$ \Phi(u(t)) \leq \underset{n\rightarrow \infty}{\liminf} \Phi(J_{t/n}^n(\Phi)x_0) \leq \underset{n\rightarrow \infty}{\liminf} [\Phi]^{\wedge}\left\{\sum_{k=0}^{n-1} \frac{t}{n} \cdot \left(1+\frac{\lambda t}{n}\right)^{2k} \right\} (x_0) = [\Phi]^{\kappa(t,\lambda)}(x_0) .$$
The second inequality holds by \eqref{induction1} and the third equality follows from Proposition~\ref{envelopeProp2}. 
\end{proof}

\begin{remark}
\label{remark:quadratic}
We illustrate the bound in Theorem~\ref{GradFlowTheorem} for a quadratic function $\Phi: \mathbb{R} \rightarrow (-\infty,+\infty]$ such that $\Phi(x) := \frac{\lambda}{2} x^2$ for some $\lambda>0$. For a starting point $x_0 \in \mathbb{R}$ the gradient flow of $\Phi$ is given by $u(t) = e^{-\lambda t}x_0$ and so $\Phi(u(t))= \frac{\lambda}{2}e^{-2\lambda t} x_0^2$.
For $\gamma>0$ we calculate
\begin{align*}
[\Phi]^{\gamma}(x_0) &= \inf_{y\in \mathbb{R}} \left\{ \frac{\lambda}{2} y^2 + \frac{1}{2\gamma} (y-x_0)^2 \right\}
= \inf_{y\in \mathbb{R}} \left\{ \left(\frac{\lambda}{2} + \frac{1}{2\gamma}\right)y^2 - \frac{1}{\gamma} x_0y + \frac{1}{2\gamma} x_0^2 \right\}\\
&= \frac{1+ \gamma \lambda}{2\gamma}\inf_{y\in \mathbb{R}} \left\{ y^2 - \frac{2}{1+ \lambda\gamma} x_0y + \frac{1}{1+\gamma \lambda} x_0^2 \right\} \\
&= \frac{1+ \gamma \lambda}{2\gamma}\inf_{y\in \mathbb{R}} \left\{ \left(y-\frac{1}{1+\gamma \lambda} x_0 \right)^2 + \frac{1}{1+\gamma \lambda} x_0^2 - \frac{1}{(1+\gamma \lambda)^2} x_0^2 \right\}\\
&= \left(\frac{1}{2\gamma} - \frac{1}{2\gamma(1+\gamma \lambda) } \right) x_0^2 
= \frac{\lambda}{2(1+ \gamma \lambda)} x_0^2 .
\end{align*}
We note that $\Phi$ is proper, $\lambda$-convex, and (lower semi)continuous and indeed
\begin{align*}
[\Phi]^{\kappa(t,\lambda)}(x_0) - \Phi(u(t)) &= \frac{\lambda}{2(1+\frac{e^{2\lambda t}-1}{2})} x_0^2 - \frac{\lambda}{2e^{2\lambda t}}x_0^2
= \frac{\lambda}{1+e^{2\lambda t}}x_0^2 - \frac{\lambda}{2e^{2\lambda t}}x_0^2\\
&= \frac\lambda{2e^{2\lambda t}} \frac{e^{2\lambda t}-1}{1+e^{2\lambda t}} 
= \frac\lambda{2e^{2\lambda t}} \frac{1-e^{-2\lambda t}}{1+e^{-2\lambda t}} x_0^2 >0.
\end{align*}
Moreover, this shows us that $[\Phi]^{\kappa(t,\lambda)}(x_0) - \Phi(u(t)) \downarrow 0$ as $\lambda\to\infty$ and thus the bound in Theorem~\ref{GradFlowTheorem} is tight, if, as formulated in the theorem, we wish it to hold for all $\lambda>0$. Whether the bound is sharp, in the sense that equality can be achieved, is unknown to the authors at the moment. 

Furthermore, since $\Phi$ achieves its minimum value $0$ at $x_*:=0$, we have that
\[
\Phi(u(t))-\Phi(x_*) = \frac\lambda2 e^{-2\lambda t} x_0^2 = \frac{1-e^{-2\lambda t}}2 \frac1{2\kappa(t,\lambda)} x_0^2.
\]
Since $0<\frac{1-e^{-2\lambda t}}2 <\frac12$, this also illustrates the result of Corollary~\ref{cor:decayrate} below, for $\lambda>0$ and $x=x_*$. 
\end{remark}

\begin{corollary}\label{cor:decayrate}
Let $\lambda \geq 0$ and let $\Phi:H \rightarrow (-\infty,+\infty]$ be proper, $\lambda$-convex, and lower semicontinuous. Assume that $u$ is the unique gradient flow of $\Phi$ starting at $x_0 \in \overline{\operatorname{dom}(\Phi)}$. Then we have, for all $x\in H$ and all $t>0$,
\begin{equation}\label{eq:Phiinequality}
\Phi(u(t)) -\Phi(x)\leq \frac{1}{2\kappa(t,\lambda)} \Vert x_0 - x \Vert^2,
\end{equation}
where $\kappa(t,\lambda) \to +\infty$ as $t\to+\infty$.

In particular, if $\Phi$ has a minimum at $x_* \in H$, then $\Phi(u(t)) \to \Phi(x_*)$ as $t\to \infty$.
\end{corollary}
\begin{proof}
Since $\lambda\geq 0$, we have $I_{-\lambda} = (0,+\infty)$ and thus, for all $t>0$, $\kappa(t,\lambda)$ is well defined and positive and the inequality in Theorem~\ref{GradFlowTheorem} holds. By the definition of the Moreau envelope in Definition~\ref{MoreauProx} we obtain that, for all $x\in H$ and all $t>0$,
$$ [\Phi]^{\kappa(t,\lambda)}(x_0) \leq \Phi(x)+\frac{1}{2\kappa(t,\lambda)} \Vert x_0 - x \Vert^2.$$
Thus, by Theorem~\ref{GradFlowTheorem}, for all $x\in H$ and all $t>0$,
\[
\Phi(u(t)) \leq [\Phi]^{\kappa(t,\lambda)}(x_0) \leq \Phi(x)+\frac{1}{2\kappa(t,\lambda)} \Vert x_0 - x \Vert^2,
\]
from which the inequality \eqref{eq:Phiinequality} follows. 
The limit for $\kappa(t,\lambda)$ is immediate from its definition in Theorem~\ref{GradFlowTheorem}.

Now assume that $\Phi$ has a minimum at $x_*\in H$, then, for all $t>0$, $\Phi(u(t))-\Phi(x_*) \geq 0$. Thus, with $x=x_*$, the inequality \eqref{eq:Phiinequality} gives
\[
|\Phi(u(t)) -\Phi(x_*)|\leq \frac{1}{2\kappa(t,\lambda)} \Vert x_0 - x_* \Vert^2 \to 0 \qquad \text{as } t\to\infty.
\]
\end{proof}

\begin{remark}
In the case that $\lambda>0$, Corollary~\ref{cor:decayrate} gives us the same rate of decay as the one in \cite[Theorem 2.4.14]{greenbook}.
\end{remark}

For the following lemma, we recall the semigroup notation from Definition~\ref{def:semigroup}.

\begin{lemma}\label{lem:gradflowsemigroup}
Let $\lambda \in \mathbb{R}$, $\Phi:H \rightarrow (-\infty,+\infty]$ be proper, $\lambda$-convex and lower semicontinuous. Let $u$ be the unique gradient flow of $\Phi$ starting at $x_0 \in \overline{\operatorname{dom}(\Phi)}$. Then, for all $t\in[0,+\infty)$,
$$ u(t) = S_{\partial^{\lambda}\Phi}(t)x_0 .$$
\end{lemma}

\begin{proof} For the notation $J^n_{t/n}(\Phi)$ we refer to Remark~\ref{rem:minimizingmovement} and for $R^n_{t/n}(\partial^\lambda \Phi)$ to Remark~\ref{rem:whyrangecon}.

If $t=0$, then, by the last property of semigroups listed in Remark~\ref{rem:propertiessemigroups}, ${S_{\partial^\lambda \Phi}(t)x_0=x_0=u(t)}$. 

Now assume that $t>0$. By Remark~\ref{rem:minimizingmovement} we know that \eqref{eq:minimizinglimit} holds. By \eqref{OmegaInterval}, if $\lambda \geq 0$, then, for all $n\in \mathbb{N}$, $\frac{t}n \in \mathfrak{I}_{-\lambda}$. If $\lambda < 0$, then $\frac{t}n \in \mathfrak{I}_{-\lambda}$ if $n > t|\lambda|$. Hence in both cases we can use \eqref{eq:minimizinglimit} and Lemma~\ref{ProxResolve} to obtain
\begin{equation}\label{eq:gradflowandresolvents}
u(t) = \lim_{n \rightarrow \infty} J_{t/n}^n(\Phi)x_0 = \lim_{n \rightarrow \infty} R_{t/n}^n(\partial^{\lambda}\Phi)x_0.
\end{equation}
Finally, by the Crandall--Liggett theorem in Theorem~\ref{CranLig} and the definition of the semigroup generated by $\partial^\lambda \Phi$ in Definition~\ref{def:semigroup}, we find that
\[
\lim_{n \rightarrow \infty} R_{t/n}^n(\partial^{\lambda}\Phi)x_0 = S_{\partial^{\lambda}\Phi}(t)x_0.
\]
The proof is completed by combining these three identities.
\end{proof}

\begin{remark}
In the setting of Lemma~\ref{lem:gradflowsemigroup}, if $\gamma \leq \lambda$, then, by Lemma~\ref{lem:lambdaconvexanstrictlyconvex}, $\Phi$ is also $\gamma$-convex and thus $u(t) = S_{\partial^{\gamma}\Phi}(t)x_0$. In fact, by part~2 of Lemma~\ref{lem:subdiffinclusion}, $\partial^{\gamma}\Phi=\partial^{\lambda}\Phi$.
\end{remark}

\begin{remark}
By Theorem~\ref{thm:partiallambdamaccretive} $\partial^{\lambda}\Phi$ is $(-\lambda)$-m-accretive if the assumptions on $\Phi$ in Lemma~\ref{lem:gradflowsemigroup} are satisfied and so $x_0 \in D(S_{\partial^{\lambda}\Phi})$ by Corollary~\ref{SemigroupCon}.
\end{remark}

Throughout Section~\ref{sec:convex} we have restricted ourselves to the setting of a Hilbert space $H$. The key concepts of $\lambda$-convexity, gradient flows, and the Moreau envelope can readily be generalized to a metric space, typically by replacing the norm by a metric. However, there are two places where the inner-product structure is important for the proofs of our main results of this section. The first is in the proof of Lemma~\ref{ProxResolve}, where we expand the norm in the definition of the proximal operator via the inner product. The second is the semigroup property of the Moreau envelope given by part~4 of Proposition~\ref{envelopeProp} which we use to prove Theorem~\ref{GradFlowTheorem}. In particular, the inner-product structure is used in the proof of \cite[Proposition 12.22 (ii)]{Bauschke} (which in turn is the source of our proof of part~4 of Proposition~\ref{envelopeProp}) to rewrite certain norms.

To what extent the results of this section can be generalised beyond the Hilbert-space setting is a worthwhile question for future research. For example, we know from \cite[Theorem 4.3.2]{greenbook} that the result from Theorem~\ref{GradFlowTheorem} holds more generally on metric spaces, just with a smaller parameter in the Moreau envelope in the case $\lambda>0$.

Another interesting question is whether the bound we give in Theorem~\ref{GradFlowTheorem} can be reproduced for other types of infimal convolution (see \cite[Definition 12.1]{Bauschke}) besides the Moreau envelope. The Moreau envelope is a natural choice in the current context because of its connection to gradient flows via the proximal operator. However, when studying $p$-curves of maximal slope (see \cite[Definition 1.3.2]{greenbook}) for $p\neq 2$ we conjecture the suitable replacement to be an infimal convolution with $\Vert \cdot \Vert^p$ (where $\|\cdot\|$ denotes the norm on the Banach space under consideration) instead of $\|\cdot\|^2$.

\section{Subdifferentials over \texorpdfstring{$L^p$}{Lp} and \texorpdfstring{$P_0$}{P0}-convexity}\label{sec:subdifferentials}

For this section we fix $\Omega \subset \mathbb{R}^d$ to be an open subset and $\mu$ to be a Radon measure defined on $\Omega$. By $\mathfrak{F}(\Omega;\mu)$ we denote the set of all equivalence classes of Borel-measurable functions on $\Omega$, where two functions are considered equivalent if and only if they differ on a null set with respect to $\mu$. As is the usual convention for such spaces, in a slight abuse (and simplification) of notation we will treat $\mathfrak{F}(\Omega;\mu)$ as a set of functions modulo equality up to $\mu$-null sets. We note that $\mathfrak{F}(\Omega;\mu)$ is real vector space and thus the notion of convexity is meaningful for subsets of $\mathfrak{F}(\Omega;\mu)$. 

Given a function $\Phi:\mathfrak{F}(\Omega;\mu) \rightarrow (-\infty,+\infty]$ we are interested in studying the gradient flow of $\Phi|_{L^2(\Omega;\mu)}$ over the Hilbert space $L^2(\Omega;\mu)$. In particular, we will consider the subdifferential $\partial \Phi|_{L^2(\Omega;\mu)}$ as an operator on $L^2(\Omega;\mu)$ and study the semigroup it generates in the sense of Definition~\ref{def:semigroup}. By Lemma~\ref{lem:gradflowsemigroup} this indeed gives us a representation of the desired gradient flow. 

For the reasons explained in the introduction (at the start of Section~\ref{sec:convgradflowintro}), for all $p\in [1, +\infty)$, we want to adjust $\partial \Phi|_{L^2(\Omega;\mu)}$ to an operator $\partial_{L^p} \Phi$ on $L^p(\Omega;\mu)$ that satisfies the equality in \eqref{eq:subgradidentity} and then consider the semigroup it generates as a flow in $L^p(\Omega;\mu)$. We mention that this adjusted operator will satisfy $\partial_{L^2}\Phi = \partial \Phi|_{L^2(\Omega;\mu)}$; see Remark~\ref{rem:adjustedsubdiff}.

In this section we give the definitions relevant to this adjustment and establish conditions on $\Phi$ such that \eqref{eq:subgradidentity} holds. Our notation and theory will predominantly be taken from the work on completely accretive operators by Benilan and Crandall \cite{benilan1991comp}.

\begin{definition}
\label{Subdiffdef}
Let $1\leq p < +\infty$ and $\Phi: \mathfrak{F}(\Omega;\mu) \rightarrow (-\infty,+\infty]$. The subdifferential of $\Phi$ over $L^p(\Omega;\mu)$, denoted by $\partial_{L^p} \Phi$, is an operator on $L^p(\Omega;\mu)$ defined as follows. For all $u, v\in L^p(\Omega;\mu)$, $(u,v)\in \partial_{L^p} \Phi$ if and only if, for all $h\in L^p(\Omega;\mu)$ with $hv \in L^1(\Omega;\mu)$, 
$$ \Phi(u+h) \geq \Phi(u) + \int_\Omega hv \, d\mu.$$
\end{definition}

\begin{remark}\label{rem:adjustedsubdiff}
In Definition~\ref{Subdiffdef} above the (adjusted) subdifferential is always defined with respect to the $L^2$ inner product (and thus in particular $\partial_{L^2}\Phi = \partial \Phi|_{L^2(\Omega;\mu)}$) but has domain and range in $L^p(\Omega;\mu)$. Our later results will justify this adjusted subdifferential as useful and help us to extend the semigroup it generates to act on an $L^p$ space in a natural way. 

If $\mu$ is finite and $p\geq 2$ we have $L^p(\Omega;\mu) \subset L^2(\Omega;\mu)$. Directly from Definition \ref{Subdiffdef} we then have 
$ \partial_{L^2} \Phi|_{L^p(\Omega;\mu)} \subset  \partial_{L^p} \Phi.$
\end{remark}

The next proposition gives some important properties of the subdifferential.

\begin{proposition} \label{SubdiffProp}
Let $1\leq p < +\infty$. Suppose $\Phi,\Psi:L^p(\Omega;\mu) \rightarrow (-\infty,+\infty]$ and $\lambda>0$, then
\begin{enumerate}
    \item $\partial_{L^p} \Psi + \partial_{L^p} \Phi \subset \partial_{L^p}(\Phi + \Psi) $ and
    \item $\lambda \partial_{L^p} \Phi = \partial_{L^p}(\lambda \Phi).$
\end{enumerate}
Additionally, if $L^p(\Omega;\mu) \subset L^2(\Omega,\mu)$, then
\begin{enumerate}
\setcounter{enumi}{2}
    \item for all $u,v\in L^p(\Omega;\mu)$, $v\in \partial_{L^p} \Phi(u)$ if and only if
$\displaystyle
u \in \argmin_{w\in L^p(\Omega;\mu)} \left( \Phi(w) - \int_\Omega vw \, d\mu \right)
$,
\item and
\begin{equation*}
\partial_{L^p} \left(w \mapsto \frac{1}{2}\int_\Omega |w|^2 \, d\mu \right) = I_{L^p(\Omega; \mu)}.    
\end{equation*}
\end{enumerate}
\end{proposition}
\begin{proof}
\begin{enumerate}
    \item Let $(u,v)\in \partial_{L^p} \Psi$ and $(u,w)\in \partial_{L^p} \Phi$. For all $h\in L^p(\Omega;\mu)$ we have
    \begin{align*}
    \Phi(u+h)+ \Psi(u+h) &\geq \Phi(u) + \int_\Omega hv d\mu + \Psi(u) + \int_\Omega hw d\mu\\
    &= \Phi(u) + \Psi(u) + \int_\Omega h(v+w) d\mu.
    \end{align*}
Hence $(u,v+w) \in \partial_{L^p}( \Phi + \Psi)$.
\item This follows from an expansion analogous to the one in the proof of part 1.
\item Let $u,v\in L^p(\Omega; \mu)$. By assumption $v \in L^2(\Omega;\mu)$ and, if $w\in L^p(\Omega;\mu)$, then also $w-u \in L^2(\Omega;\mu)$; thus $(w-u)v\in L^1(\Omega;\mu)$. Therefore $(u,v)\in \partial_{L^p} \Phi$ if and only if, for all $w\in L^p(\Omega;\mu)$,
$$ \Phi(w) \geq \Phi(u) + \int_\Omega (w-u)v d\mu;$$
equivalently,
$$ \Phi(w) - \int_\Omega vw d\mu \geq \Phi(u) - \int_\Omega vu d\mu.$$
Moreover, this is equivalent to $\displaystyle
u \in \argmin_{w\in L^p(\Omega;\mu)} \left( \Phi(w) - \int_\Omega vw \, d\mu \right)
$.
\item Let $u,h \in L^p(\Omega;\mu)$. By assumption $u,h \in L^2(\Omega;\mu)$, hence $uh\in L^1(\Omega;\mu)$. Then
\begin{align*}
 \frac{1}{2}\int_\Omega |u+h|^2  d\mu   &=  \frac{1}{2}\int_\Omega |u|^2 d\mu +  \int_\Omega hu d\mu 
 + \frac{1}{2}\int_\Omega |h|^2 d\mu\\
 &\geq  \frac{1}{2}\int_\Omega |u|^2 d\mu + \int_\Omega hu d\mu .
\end{align*}
Thus $(u,u) \in \partial_{L^p} \left(w \mapsto \frac{1}{2}\int_\Omega |w|^2 \, d\mu \right)$, which establishes that 
$I_{L^p(\Omega; \mu)} \subset \partial_{L^p} \left(w \mapsto \frac{1}{2}\int_\Omega |w|^2 \, d\mu \right)$. 
On the other hand, if $(u,v) \in \partial_{L^p} \left(w \mapsto \frac{1}{2}\int_\Omega |w|^2 \, d\mu \right)$, then, by the same expansion of $|u+h|^2$ as above, we have, for all $h\in L^p(\Omega;\mu)$,
\[
\frac{1}{2}\int_\Omega |u|^2 d\mu + \int_\Omega hv d\mu \leq \frac{1}{2}\int_\Omega |u|^2 d\mu + \int_\Omega hu d\mu 
 + \frac{1}{2}\int_\Omega |h|^2 d\mu.
\]
Choosing $h=v-u$, we obtain
\[
\int_\Omega |v-u|^2 \, d\mu \leq \frac12 \int_\Omega |v-u|^2 \, d\mu,
\]
hence $u=v$ in $L^p(\Omega;\mu)$. Thus also $\partial_{L^p} \left(w \mapsto \frac{1}{2}\int_\Omega |w|^2 \, d\mu \right) \subset I_{L^p(\Omega; \mu)}$.
\end{enumerate}
\end{proof}

\begin{definition}\label{def:Pconvex}
Define the class of functions
\[
P_0 := \{ g \in C^{\infty}(\mathbb{R}) \,|\, 0 \notin \operatorname{spt}(g), \text{ } \operatorname{spt}(g') \text{ is compact, and } 0 \leq g' \leq 1 \},
\]
where $\operatorname{spt}(g')$ denotes the support of $g'$. Let $X$ be a convex subset of $\mathfrak{F}(\Omega;\mu)$. We then say a functional $\Phi: \mathfrak{F}(\Omega;\mu) \rightarrow (-\infty, +\infty]$ is $P_0$-convex on $X$ if, for all $u,v \in X$ and all $g \in P_0$,
\begin{equation} \label{Pconvex}
\Phi(u+g \circ (v-u)) + \Phi(v-g\circ(v-u)) \leq \Phi(u) + \Phi(v) .
\end{equation}
\end{definition}

The concept of $P_0$-convexity in Definition~\ref{def:Pconvex} is taken from \cite[Lemma 7.1]{benilan1991comp}, where it is introduced without name as condition in the lemma. This notion is very useful for our paper and we have decided to name it. 

\begin{remark}\label{rem:P0subset}
    It follows directly from Definition~\ref{def:Pconvex} that any functional that is $P_0$-convex on a convex subset of $\mathfrak{F}(\Omega;\mu)$, is also $P_0$-convex on any convex subset of $X$.
\end{remark}

\begin{remark}\label{rem:gcomposedwithu}
Let $g\in P_0$, then $\operatorname{sign}(g(x))= \operatorname{sign}(x)$. Indeed, since $g$ is smooth and $g'\geq 0$, $g$ is monotonically increasing, and because $0\not\in \operatorname{spt}$, $g(0)=0$. Thus if $x<0$ then $g(x)<0$, and if $x>0$, then $g(x)>0$. Moreover, since $|g'|\leq 1$, we find, for all $x \in \mathbb{R}$,
$
|g(x)| \leq \int_{[0,x]} |g'(y)| \, dy \leq |x|.
$
From this inequality it also follows that, if $1\leq p <+\infty$, then, for all $u\in L^p(\Omega;\mu)$ and all $g\in P_0$, we have $g \circ u \in L^p(\Omega;\mu)$, since $\int_\Omega |g\circ u|^p \, d\mu \leq \int_\Omega |u|^p \, d\mu < +\infty$.
\end{remark}

\begin{proposition} \label{Pprop}
Let $1\leq p<+\infty$. If the functional $\Phi: \mathfrak{F}(\Omega;\mu) \rightarrow (-\infty, +\infty]$ is $P_0$-convex on $L^p(\Omega;\mu)$ and lower semicontinuous with respect to the topology of $L^p(\Omega;\mu)$, then its restriction to $L^p(\Omega;\mu)$ is convex. 
\end{proposition}
\begin{proof}
    See \cite[Proposition 7.2]{benilan1991comp}. This result can be used since  $L^p(\Omega;\mu)$ is normal (in the sense of \cite[Definition 2.8]{benilan1991comp}) and contains $L^{1\cap\infty}(\Omega)$ as a dense subset.
\end{proof}

Proposition~\ref{Pprop} shows that $P_0$-convexity is stronger than convexity for lower-semicontinuous functions. Proposition~\ref{prop:counterexample} below tells us that it even is strictly stronger and in fact cannot be implied by $\lambda$-convexity for any $\lambda\geq 0$. 

\begin{proposition}
\label{prop:counterexample}
Let $x,y\in \Omega$ be distinct and $\mu:= \frac{1}{2} \delta_{x} + \frac{1}{2} \delta_{y}$. Then, for all $\lambda\geq 0$, there exists a function $\Phi:\mathfrak{F}(\Omega;\mu) \rightarrow \mathbb{R}$ which is continuous with respect to the topology of $L^2(\Omega;\mu)$ and $\lambda$-convex, yet which is not $P_0$-convex on $L^2(\Omega;\mu)$.
\end{proposition}
\begin{proof}
First we observe that, for all $u\in \mathfrak{F}(\Omega;\mu)$,
\[
\|u\|_{L^2(\Omega;\mu)}^2 = \frac12 \left(u^2(x) + u^2(y)\right) < +\infty
\]
and thus if we equip $\mathfrak{F}(\Omega;\mu)$ with the $L^2(\Omega;\mu)$ norm, then $\mathfrak{F}(\Omega;\mu)=L^2(\Omega;\mu)$ as Hilbert spaces.

Let $\lambda\geq 0$ and define
$$ \Phi(u):= (\lambda+1)\big(u(x)+u(y)\big)^2+\frac{\lambda}{2}\left(u^2(x)+u^2(y)\right).$$
Then $\Phi$ is continuous with respect to the topology of $L^2(\Omega;\mu)$. Using Lemma~\ref{lem:lambdaconvexequivalent} it is also not difficult to see that $\Phi$ is $\lambda$-convex. If we take $u,v\in \mathfrak{F}(\Omega;\mu)$ such that $u(x)=1$, $u(y)=0$, $v(x)=0$, and $v(y)=1$, and choose $g\in P_0$ such that $g(-1)=0$ and $g(1)=\frac{1}{2}$, then $(g\circ(v-u))(x) = g(-1)=0$, $(g\circ(v-u))(y)=g(1)=\frac12$,
and
$\Phi(u)=\Phi(v)=(\lambda+1) + \frac\lambda2=1+\frac32\lambda$. Thus
\begin{align*}
    \Phi(u+g\circ(v-u)) + \Phi(v-g\circ(v-u)) &= (\lambda+1)\big(u(x)+(g\circ(v-u))(x)+u(y)+(g\circ(v-u))(y)\big)^2\\
    &+ (\lambda+1)\big(v(x)-(g\circ(v-u))(x)+v(y)-(g\circ(v-u))(y)\big)^2\\
&\hspace{0.3cm}+\frac{\lambda}{2}\left([u(x)+(g\circ(v-u))(x)]^2+[u(y)+(g\circ(v-u)(y))]^2\right)\\    &\hspace{0.3cm}+\frac{\lambda}{2}\left([v(x)-(g\circ(v-u))(x)]^2+[v(y)-(g\circ(v-u)(y))]^2\right)\\   
&= \frac52(\lambda+1) + \frac32 \frac\lambda2 = \frac52+\frac{13}4\lambda > 2 + 3\lambda = \Phi(u)+\Phi(v).
\end{align*}
\end{proof}

\begin{lemma}
\label{basicP0}
Let $1\leq p <+\infty$ and $\lambda>0$. Suppose $\Phi_1,\Phi_2: \mathfrak{F}(\Omega;\mu) \rightarrow (-\infty,+\infty]$ are $P_0$-convex. Then $\lambda \Phi_1$ and $\Phi_1 +\Phi_2$ are both $P_0$-convex.
\end{lemma}

\begin{proof}
    This follows immediately from the defining condition in ($\ref{Pconvex}$).
\end{proof}

Next we state a result from \cite{benilan1991comp} that illustrates why $P_0$-convexity is a useful notion.

\begin{lemma} \label{lemAcrete}
Let $1\leq p<+\infty$ and $\Phi:\mathfrak{F}(\Omega;\mu) \rightarrow (-\infty,+\infty]$ be $P_0$-convex on $L^p(\Omega;\mu)$ and lower semicontinuous on $L^p(\Omega;\mu)$. Then, for all $1\leq r<+\infty$, $\partial_{L^p} \Phi$ is accretive with respect to the $L^r(\Omega;\mu)$ norm.
\end{lemma}
\begin{proof}
    See \cite[Definition 2.1 and Lemma 7.1]{benilan1991comp}. This lemma can be used because $L^r(\Omega;\mu)$ is normal, in the sense of \cite[Definition 2.8]{benilan1991comp}.
\end{proof}

By Proposition~\ref{prop:accretivecontraction} accretivity of $\partial_{L^p} \Phi$ with respect to the topology of $L^r(\Omega;\mu)$, as in the conclusion of Lemma~\ref{lemAcrete}, means that the resolvent $R_\lambda(\partial_{L^p}\Phi)$ contracts the $L^r(\Omega;\mu)$ norm. We can use this to establish that the semigroup generated by $\partial_{L^p}\Phi$ also contracts the $L^r(\Omega;\mu)$ norm. We will do so in Corollary~\ref{cor:semigroupAccretive}, but first require the following lemma.

\begin{lemma}
\label{lem:lscab}
For all $p,q\geq 1$, the norm functional $\Vert \cdot \Vert_{L^q(\Omega;\mu)}$ is lower semicontinuous on the space $L^p(\Omega;\mu)$. 
\end{lemma}

\begin{proof}
Let $K\in\mathbb{N}$ and define $\Omega_K:=\Omega\cap(-K,K)^d$. Also define the thresholding operator $\tau_K:L^p(\Omega;\mu)\rightarrow L^p(\Omega;\mu)$ by
$$\tau_Ku(x):=\max\Big(-K,\min\big(K,u(x)\big)\Big).$$ 
For later use we note that, for all $u\in L^p(\Omega;\mu)$ and all $x\in \Omega$, $|\tau_Ku(x)|\leq |u(x)|$.
By $\chi_{\Omega_K}$ we denote the indicator function of the subset $\Omega_K$, i.e. the function whose value equals $1$ on $\Omega_K$ and $0$ on $\Omega\setminus\Omega_K$.

Let $u\in L^p(\Omega;\mu)$ and suppose that $(u_n)_{n\in\mathbb{N}}\subset L^p(\Omega;\mu)$ is a sequence such that $u_n\rightarrow u$ in $L^p(\Omega;\mu)$ as $n\to\infty$. By direct computation in each of the different cases of the values $u(x)$ and $u_n(x)$ being in $(-\infty, -K)$, $[-K,K]$, or $(K,+\infty)$, it follows that, for all $x\in \Omega$, $|\tau_K u(x) - \tau_K u_n(x)|\leq |u(x)-u_n(x)|$ and thus also $$|\chi_{\Omega_K}(x)\tau_K u(x) - \chi_{\Omega_K}(x)\tau_K u_n(x)|\leq |u(x)-u_n(x)|.$$
Hence $\chi_{\Omega_K}\tau_Ku_n\rightarrow \chi_{\Omega_K}\tau_Ku$ in $L^p(\Omega;\mu)$ as $n\rightarrow \infty$. Moreover, for all $n\in\mathbb{N}$, $\tau_Ku_n$ has bounded support in $\Omega_K$ and $|\tau_Ku_n|\leq K$. Thus, by a standard interpolation argument using H\"older's inequality (see \cite[Lemma 244E(b)]{Fremlin2}), $\chi_{\Omega_K}\tau_Ku_n\rightarrow \chi_{\Omega_K}\tau_Ku$ in $L^q(\Omega;\mu)$ as $n\rightarrow \infty$. So we deduce that
$$ \int_{\Omega_K} |\tau_K u(x)|^q \, d\mu(x)= \liminf_{n\rightarrow\infty} \int_{\Omega_K} |\tau_K u_n(x)|^q \, d\mu(x) \leq \liminf_{n\rightarrow\infty} \int_{\Omega} | u_n(x)|^q \, d\mu(x).$$
For all $x\in\Omega$, $|\chi_{\Omega_K}(x)\tau_Ku(x)|\rightarrow |u(x)|$ as $K\rightarrow \infty$ and this convergence is monotone. So, by the monotone convergence theorem for integrals (\cite[Theorem 2.14]{Folland99}),
\[
\int_\Omega |u(x)|^q \, d\mu(x)=\lim_{K\rightarrow\infty} \int_{\Omega_K} |\tau_K u(x)|^q \, d\mu(x).
\]
Combining this with the previous inequality, we get
$$ \int_{\Omega}|u(x)|^q \, d\mu(x) \leq \liminf_{n\rightarrow\infty} \int_{\Omega} |u_n(x)|^q \, d\mu(x).$$
\end{proof}

\begin{corollary}
\label{cor:semigroupAccretive}
Let $\Phi:\mathfrak{F}(\Omega;\mu) \to (-\infty,+\infty]$ be proper, lower semicontinuous, and $P_0$-convex on $L^2(\Omega;\mu)$ and let $S_{\partial_{L^2} \Phi}$ be the semigroup that gives the gradient flow of $\Phi|_{L^2(\Omega;\mu)}$ over $L^2(\Omega;\mu)$ in the sense of Lemma~\ref{lem:gradflowsemigroup}. Then, for all $r\in [1, +\infty)$ and all $t\geq 0$,
\[
\Vert S_{\partial_{L^2} \Phi}(t)x -  S_{\partial_{L^2} \Phi}(t)y\Vert_{L^r(\Omega;\mu)} \leq \Vert x-y \Vert_{L^r(\Omega;\mu)}.
\]
\end{corollary}
\begin{proof}
    First we observe that, by Proposition~\ref{Pprop} and Lemma~\ref{lem:lambdaconvexanstrictlyconvex}, we can indeed take $\lambda=0$ in the generator $\partial^\lambda \Phi|_{L^2(\Omega;\mu)}$ of the semigroup in Lemma~\ref{lem:gradflowsemigroup}. Moreover, by Remark~\ref{rem:adjustedsubdiff}, $\partial \Phi|_{L^2(\Omega;\mu)}=\partial_{L^2}\Phi$.

    If $t=0$, then the result is trivially true according to the fourth property of semigroups in Remark~\ref{rem:propertiessemigroups}.

    Now assume that $t>0$. By Lemma~\ref{lemAcrete}, $\partial \Phi_{L^2}$ is accretive with respect to $\Vert \cdot \Vert_{L^r(\Omega;\mu)}$ which itself is a lower-semicontinuous functional on $L^2(\Omega;\mu)$ according to Lemma~\ref{lem:lscab}. So we conclude that
\begin{align*}
    \Vert S_{\partial_{L^2} \Phi}(t)x -  S_{\partial_{L^2} \Phi}(t)y\Vert_{L^r(\Omega;\mu)}&= \left\Vert \lim_{n\rightarrow \infty} \left(  R_{t/n}^n(\partial_{L^2}\Phi)x -  R_{t/n}^n(\partial_{L^2}\Phi)y      \right) \right\Vert_{L^r(\Omega;\mu)} \\
    &\leq \lim_{n\rightarrow\infty}  \left\Vert \left(  R_{t/n}^n(\partial_{L^2}\Phi)x -  R_{t/n}^n(\partial_{L^2}\Phi)y      \right) \right\Vert_{L^r(\Omega;\mu)}\\ 
    &\leq \Vert x-y \Vert_{L^r(\Omega;\mu)},
\end{align*}   
where the equality holds by \eqref{eq:gradflowandresolvents} and the last inequality follows from Proposition~\ref{prop:accretivecontraction}.
\end{proof}

The following three lemmas give some examples of classes of $P_0$-convex functionals.

\begin{lemma}
\label{lem:P0Energy}
Let $\mu$ be absolutely continuous with respect to the Lebesgue measure on $\Omega$. Consider a measurable function $J:\Omega \times \mathbb{R}^d \rightarrow [0,+\infty]$ that is lower semicontinuous and convex in the $\mathbb{R}^d$ component. Define the functional $\Phi: \mathfrak{F}(\Omega;\mu) \rightarrow (-\infty,+\infty]$ by
$$\Phi(u) := 
\begin{cases}
 \int_{\Omega} J(x, \nabla u(x)) \, d\mu(x),& \text{if } u\in W^{1,1}_{\text{loc}}(\Omega), \\
 +\infty,& \text{otherwise}.
\end{cases}
$$
Then $\Phi$ is $P_0$-convex on $L^1(\Omega;\mu)$.
\end{lemma}
\begin{proof}
See \cite[Example 7.9 and Lemma 7.10]{benilan1991comp}. In particular, take $D=W^{1,1}_{loc}(\Omega)$ and $X=L^1(\Omega;\mu)$ in the lemma and recall that $L^1(\Omega;\mu)$ is normal by \cite[Definition 2.8]{benilan1991comp}.
\end{proof}

We prove a lemma similar to the one above that applies in the discrete setting.

\begin{lemma}
\label{discreteP0}
Let $n\in \mathbb{N}$ and let $\Omega_n := \{ x_1,x_2,...,x_n \}$ be a finite subset of $\Omega$. Let $\mu_n := \frac{1}{n}\sum_{i=1}^n \delta_{x_i}$, where $\delta_{x_i}$ is the Dirac measure at $x_i$. Given a convex function $L: \mathbb{R} \rightarrow [0,+\infty)$ and an $n\times n$ matrix $A$ with non-negative real entries $A_{ij}\geq 0$, define a functional $\Phi: \mathfrak{F}(\Omega;\mu_n) \to \mathbb{R}$ by
$$ \Phi(u) := \sum_{i,j=1}^n A_{ij} L\big(u(x_i) - u(x_j)\big).$$
Then $\Phi$ is $P_0$-convex on $\mathfrak{F}(\Omega;\mu_n)$.

\end{lemma}

\begin{proof}
For notational convenience we write $u_i:=u(x_i)$ for functions $u\in L^p(\Omega;\mu_n)$. Let $g \in P_0$, and $u,v \in L^p(\Omega;\mu_n)$. By Remark~\ref{rem:gcomposedwithu} $g\circ(u-v), g\circ(v-u) \in L^p(\Omega;\mu_n)$ and by the mean value theorem, for all $i,j \in \{1, \ldots, n\}$, there exists a $\xi_{ij} \in [v_i - u_i, v_j - u_j]\cup[v_j - u_j, v_i - u_i]$ such that
$$ g(v_i -u_i) - g(v_j -u_j) = g'(\xi_{i,j})\big((v_i-u_i) - (v_j -u_j)\big).$$

It follows that 
\begin{align*}&\Phi( u + g\circ(v-u) ) + \Phi(v -g\circ(v-u)) \\
&\hspace{0.05cm}= \sum_{i,j=1}^n A_{ij} \big[L\big(
u_i + g(v_i-u_i) - u_j - g(v_j-u_j)\big) 
+ L \big(
v_i - g(v_i-u_i) - v_j + g(v_j-u_j)\big)\big]\\
&\hspace{0.05cm}= \sum_{i,j=1}^n A_{ij} \big[  L\big( u_i -u_j + g'(\xi_{ij})\left((v_i-u_i) - (v_j -u_j)\right)\big) 
+ L\big( v_i -v_j - g'(\xi_{ij})\left((v_i-u_i) - (v_j -u_j)\right) \big)\big] \\
&\hspace{0.05cm}= \sum_{i,j=1}^n A_{ij}\big[ L\big( (1-g'(\xi_{ij}))(u_i -u_j) + g'(\xi_{ij})(v_i- v_j ) \big) 
+ L\big( (1-g'(\xi_{ij}))(v_i -v_j) + g'(\xi_{ij})(u_i- u_j ) \big) \big] \\
&\hspace{0.05cm}\leq \sum_{i,j=1}^n A_{ij}\big[ (1-g'(\xi_{ij}))L(u_i -u_j) + g'(\xi_{ij})L(v_i- v_j ) 
+  (1-g'(\xi_{ij}))L(v_i -v_j) + g'(\xi_{ij})L(u_i- u_j ) \big] \\
&\hspace{0.05cm}= \sum_{i,j=1}^n A_{ij}[ L(u_i -u_j) + L(v_i- v_j ) ] 
= \Phi(u) + \Phi(v) .
\end{align*}
For the inequality above, we used that $L$ is convex and $0\leq g' \leq 1$.
\end{proof}

\begin{lemma}
\label{quadP0}
Define the quadratic map $Q:\mathfrak{F}(\Omega;\mu) \rightarrow [0,+\infty]$ by
$$ Q(u) := \frac{1}{2}\int_{\Omega} |u|^2 \, d\mu. $$
Then $Q$ is $P_0$-convex on $\mathfrak{F}(\Omega;\mu)$.
\end{lemma}

\begin{proof} To prove the inequality in \eqref{Pconvex} we may assume without loss of generality that $Q(u),Q(v)<+\infty$, otherwise the inequality holds automatically. Therefore let $g\in P_0$ and $v,u \in L^2(\Omega,\mu)$. By definition of $P_0$ we have that $0 \notin \text{spt}(g)$ and hence there exists an $a>0$ such that $g|_{[-a,a]}=0$. Moreover since $\text{spt}(g')$ is compact we note there exists a $C>0$ such that $|g|\leq C$ on $\mathbb{R}$.

Since $v-u \in L^2(\Omega,\mu)$, the set $\{x \in \Omega \,|\, |v(x)-u(x)|\geq a\}$ has finite $\mu$-measure. We then deduce that $|g(u-v)|\leq C$ and
\begin{align*}
\int_{\Omega} |g(u-v)|d\mu &= \int_{\{x \in \Omega \,|\, |v(x)-u(x)|\geq a\}} |g(u-v)| d\mu \leq
C\mu \big( \{x \in \Omega \,|\, |v(x)-u(x)|\geq a\} \big)\\
 &< +\infty .
\end{align*}
Therefore $g\circ(v-u)\in L^1(\Omega;\mu)\cap L^{\infty}(\Omega;\mu)$. We then make the following expansion:
\begin{align*}
&\hspace{0.5cm} Q(u+g\circ(v-u)) + Q(v-g\circ(v-u))\\
&= Q(u) + Q(v) 
+ \int_{\Omega} \big( |g\circ(v-u)(x)|^2-(v-u)(x)g\circ(v-u)(x) \big) \, d\mu(x).
\end{align*}
In particular, by H\"older's inequality, this expansion is well defined because each term is integrable. We recall from Remark~\ref{rem:gcomposedwithu} that, by definition of $P_0$, for all $y\in \mathbb{R}$, $|g(y)|\leq |y|$ and $\operatorname{sign}(g(y))= \operatorname{sign}(y)$, and thus also $$|g(y)|^2 \leq |y| \, |g(y)| = y \operatorname{sign}(y) \, |g(y)| = y \operatorname{sign}(g(y)) \, |g(y)| = yg(y).$$ Using this in the integrand in the equality above with $y=(v-u)(x)$, we conclude that
$$Q(u+g\circ(v-u)) + Q(v-g\circ(v-u)) \leq Q(u) + Q(v).$$
\end{proof}

Next we state some key results from \cite{benilan1991comp} regarding m-accretive operators. First we require a definition.

\begin{definition}\label{def:sumtopology}
Let $1\leq r,s \leq +\infty$. Let $(u_n)_{n\in \mathbb{N}} \subset \mathfrak{F}(\Omega;\mu)$ and $u_\infty \in \mathfrak{F}(\Omega;\mu)$. We say that $u_n \rightarrow u_{\infty}$ as $n \rightarrow \infty$ in the topology of $L^r(\Omega;\mu)+L^s(\Omega;\mu)$ if there exist sequences $(f_n)_{n\in \mathbb{N}} \subset L^r(\Omega;\mu)$ and $(g_n)_{n\in \mathbb{N}} \subset L^s(\Omega;\mu)$ and functions $f_\infty \in L^r(\Omega;\mu)$ and $g_\infty \in L^s(\Omega;\mu)$ such that the following hold: (i) For all $n\in \mathbb{N}_{\infty}$, $u_n =f_n+g_n$, (ii) $f_n\rightarrow f_{\infty}$ in $L^r(\Omega;\mu)$ as $n \rightarrow \infty$, and (iii) $g_n \rightarrow g_{\infty}$ in $L^s(\Omega;\mu)$ as $n \rightarrow \infty$.
\end{definition}

\begin{theorem}[Ph. Benilan, M. G. Crandall, 1991] \label{BenCran1991}
Let $1\leq p <+\infty$ and $\Phi: \mathfrak{F}(\Omega;\mu) \rightarrow (-\infty,+\infty]$. Suppose $\Phi$ is lower semicontinuous with respect to the topology of $L^2(\Omega;\mu)+L^p(\Omega;\mu)$ and $P_0$-convex on $L^p(\Omega;\mu)$. Assume that $(0,0) \in \partial_{L^p} \Phi$. Then $\overline{\partial_{L^p} \Phi}$ is m-accretive on $L^p(\Omega;\mu)$ and
\begin{equation}\label{eq:subgradidentity}
    \overline{\partial_{L^2}\Phi \cap (L^p(\Omega;\mu)\times L^p(\Omega;\mu))} = \overline{\partial_{L^p} \Phi},
\end{equation}
where the pointwise closures are taken in the topology of $L^p(\Omega;\mu)$.
\end{theorem}

\begin{proof}
See \cite[Theorem 7.4]{benilan1991comp}. The second part is not explicitly stated within the work but is contained within the proof. In particular the authors prove,
$ \overline{\partial_{L^{1\cap\infty}}\Phi} $ is m-accretive on $L^p(\Omega)$ where the prior operator is defined in the same manner from Definition~\ref{Subdiffdef} and $L^{1\cap\infty}(\Omega):=L^1(\Omega)\cap L^\infty(\Omega)$. 
We then observe from Definition~\ref{Subdiffdef} that $\partial_{L^{1\cap\infty}}\Phi \subset \partial_{L^p}\Phi$ and additionally $\partial_{L^{1\cap\infty}}\Phi \subset \partial_{L^2}\Phi \cap (L^p(\Omega;\mu)\times L^p(\Omega;\mu))$. By Lemma~\ref{lemAcrete} we note that both operators on the RHS of the previous inclusions are accretive on $L^p(\Omega)$. By applying m-accretivity of $ \overline{\partial_{L^{1\cap\infty}}\Phi} $ we conclude,
$$ \overline{\partial_{L^2}\Phi \cap (L^p(\Omega;\mu)\times L^p(\Omega;\mu))}=\overline{\partial_{L^{1\cap\infty}}\Phi} = \overline{\partial_{L^p}\Phi}.$$
\end{proof}

\section{Banach stackings}\label{sec:Banachstackings}

In this section we define Banach stackings. We need this concept to compare elements from different Banach spaces that exist inside a larger metric space. Importantly, as we will see in Definition~\ref{BanStack}, if $(X_n)$ is a sequence of Banach spaces, the Banach-stacking construct will allow us to define convergence for sequences $(x_n)_{n\in \mathbb{N}}$ with $x_n\in X_n$, even if the Banach space $X_\infty$ which contains the limiting elements does not contain the spaces $X_n$. The case in which the spaces $X_n$ \textit{do} embed (in a $1$-Lipschitz-continuous manner) into $X_\infty$ provides a special case of a Banach stacking, as per part~1 of Examples~\ref{BSexmpl}.

One of our main results, stated in Theorem~\ref{theorem1}, is an extension of a theorem by Brezis and Pazy to Banach stackings. The original theorem in \cite[Theorem 3.1]{brezis1972convergence} provides a sufficient condition for the convergence of semigroups generated (in the sense of Definition~\ref{def:semigroup}) by a sequence of operators that are all defined on the same Banach space. In Theorem~\ref{theorem1} we extend this to the setting in which the operators are defined on (potentially) different Banach spaces in a Banach stacking. This setting will feature in our main results: Proposition~\ref{OpConvergence}, Theorem~\ref{Hilthm}, and Theorem~\ref{thm:P0theorem1}.

After defining Banach stackings and proving some basic properties, we give a number of concrete examples of Banach stackings in Examples~\ref{BSexmpl}. We will be interested in one of these in particular, namely the space $TL^p$ (Definition~\ref{def:TLp}), which was introduced in \cite{garcia2016continuum}. The second part of this section is devoted to establishing some results for $TL^p$ that will be relevant for Section~\ref{sec:TLp}.

From here on we use the term `sequence' also for infinite tuples indexed by $\mathbb{N}_\infty$ rather than by the usual $\mathbb{N}$. In cases where ambiguity may arise, we explicitly mention the index set. Terminology such as `bounded sequence' generalises straightforwardly to such sequences. 

We often need to use sequences $(x_n)$ (where $n$ varies over $\mathbb{N}$ or $\mathbb{N}_\infty$) that satisfy that, for all $n$, $x_n \in X_n$, where $(X_n)$ is a sequence of sets. We indicate this element-wise set membership by $(x_n) \subset_n X_n$. If all $X_n$ are the same set $X$, we use the standard notation $(x_n)\subset X$.

\begin{definition}\label{BanStack}
Suppose we have
\begin{itemize}
    \item a sequence of Banach spaces $\big((X_n,\Vert \cdot \Vert_n)\big)_{n\in \mathbb{N}_\infty}$;
    \item a metric space $(\mathcal{M},d)$; and
    \item a sequence $(\xi_n)_{n\in \mathbb{N}_\infty}$ of functions $\xi_n:X_n \rightarrow \mathcal{M}$.
\end{itemize}  

Given a sequence $(x_n)_{n\in \mathbb{N}_\infty} \subset_n X_n$, we say it is convergent if $\xi_n(x_n) \rightarrow \xi_{\infty}(x_{\infty})$ in $\mathcal{M}$ as $n\to\infty$. In that case we write
$$ x_n \underset{\mathcal{M}}\longtwoheadrightarrow x_{\infty} \qquad \text{as } n\to\infty.$$

We call the collection $\big((X_n, \xi_n)_{n\in \mathbb{N}_\infty}, \mathcal{M}\big)$ a Banach stacking\footnote{In the interest of concise notation, we do not explicitly include the Banach space norms $\|\cdot\|_n$ and metric $d$ in the notation of a Banach stacking. Where ambiguities may arise, we will specify these.} if all of the following hold.
\begin{enumerate}[(i)]
    \item For all $n\in \mathbb{N}_\infty$, $\xi_n$ is $1$-Lipschitz continuous.
    \item For all $x_{\infty} \in X_{\infty}$ there exists a sequence $(x_n)_{n\in \mathbb{N}} \subset_n X_n$ such that $x_n \underset{\mathcal{M}}\longtwoheadrightarrow x_{\infty}$, as $n\to\infty$. 
    \item Addition and scaling are continuous with respect to $\mathcal{M}$. In other words, for all sequences $(x_n)_{n\in \mathbb{N}_\infty} \subset_n X_n$ and $(y_n)_{n\in \mathbb{N}_\infty} \subset_n X_n$ for which $ x_n \underset{\mathcal{M}}\longtwoheadrightarrow x_{\infty}$ and $ y_n \underset{\mathcal{M}}\longtwoheadrightarrow y_{\infty}$, as $n\to\infty$, and for all $\lambda\in \mathbb{R}$, it holds that
    \[
        x_n+y_n \underset{\mathcal{M}}\longtwoheadrightarrow x_{\infty}+y_{\infty} ,
    \]
    and
    \[
        \lambda x_n \underset{\mathcal{M}}\longtwoheadrightarrow \lambda x_{\infty},
    \]
    as $n\to\infty$.
    \item The norms of $X_n$ are continuous over the metric space $\mathcal{M}$. In other words, for all sequences $(x_n)_{n\in \mathbb{N}_\infty} \subset_n X_n$ for which $x_n \underset{\mathcal{M}}\longtwoheadrightarrow x_{\infty}$ as $n\rightarrow \infty$, it holds that
    \[
    \Vert x_n \Vert_{n} \rightarrow \Vert x_{\infty} \Vert_{\infty} \qquad \text{as } n\rightarrow \infty.
    \]
\end{enumerate}
We refer to $X_{\infty}$ as the limit space and $\mathcal{M}$ as the unifying space. If we want to emphasise the Banach space whose zero element we are considering, then, for $n\in \mathbb{N}_\infty$, we write $0_n$ for the zero of $X_n$. For all $n,m\in \mathbb{N}_{\infty}$ and $x\in X_n,y\in X_m$ we define
\begin{equation}\label{eq:dnm}
d^{n,m}(x,y):= d(\xi_n(x),\xi_m(y)).
\end{equation}
\end{definition}

For an illustration of this concept we refer back to figure~\ref{fig:banach}.

\begin{remark}
Given a Banach stacking $\big((X_n, \xi_n)_{n\in \mathbb{N}_\infty}, \mathcal{M}\big)$ and an increasing sequence $(k_n)_{n\in \mathbb{N}}\subset \mathbb{N}$. Set $k_\infty:=\infty$. Then $\big((X_{k_n}, \xi_{k_n})_{n\in \mathbb{N}_\infty}, \mathcal{M}\big)$ also defines a Banach stacking. Moreover, for a sequence $(x_n)_{n\in\mathbb{N}_\infty}\subset_n X_n$ we write $x_{k_n}\underset{\mathcal{M}}\longtwoheadrightarrow x_{\infty}$ as $n\rightarrow \infty$ if the subsequence $(x_{k_n})_{n\in \mathbb{N}_\infty}$ converges in the Banach stacking $\big((X_{k_n}, \xi_{k_n})_{n\in \mathbb{N}_\infty}, \mathcal{M}\big)$.         
\end{remark}

\begin{definition}
Let $\big((X_n, \xi_n)_{n\in \mathbb{N}_\infty}, \mathcal{M}\big)$ be a Banach stacking and $\mathcal{A}$ be a set. Assume, for all $a\in \mathcal{A}$, there exists a sequence $\big(x_n(a)\big)_{n\in\mathbb{N}_\infty}\subset_n X_n$. We say that $x_n(a)$ converges to $x_\infty(a)$ uniformly in $a\in \mathcal{A}$ in the Banach stacking $\big((X_n, \xi_n)_{n\in \mathbb{N}_\infty}, \mathcal{M}\big)$ if
$$ \sup_{a\in\mathcal{A}} d\big( \xi_n(x_n(a)),\xi_\infty(x_\infty(a)) \big) \rightarrow 0 \qquad \text{as } n\rightarrow \infty.$$
\end{definition}

\begin{remark}\label{rem:triangle}
In the Banach stacking setting of Definition~\ref{BanStack}, given $l,m,n \in \mathbb{N}_{\infty}$, $x\in X_l$, $y \in X_m$, and $z\in X_n$, the triangle inequality for the metric $d$ implies that
$$ d^{l,n}(x,z) \leq d^{l,m}(x,y) + d^{m,n}(y,z).$$

Furthermore, the Lipschitz continuity of $\xi_n$ with $1$ as a Lipschitz constant in condition~(i) of Definition~\ref{BanStack} implies that, for all $n\in \mathbb{N}_\infty$ and all $x,y\in X_n$,
\[
d^{n,n}(x,y) \leq \|x-y\|_n.
\]
\end{remark}

\begin{proposition}
\label{zeroConvergence}
Let $\big((X_n, \xi_n)_{n\in \mathbb{N}_\infty}, \mathcal{M}\big)$ be a Banach stacking. Then $0_n \underset{\mathcal{M}}\longtwoheadrightarrow 0_{\infty}$ as $n \rightarrow \infty$. 
\end{proposition}
\begin{proof}
By property~(ii) of Definition~\ref{BanStack} there exists a sequence $(x_n)_{n\in \mathbb{N}} \subset_n X_n$ such that $x_n \underset{\mathcal{M}}\longtwoheadrightarrow 0_{\infty}$ as $n \rightarrow \infty$. By property (iv) of Definition \ref{BanStack} we have $\Vert x_n \Vert_n \rightarrow 0$ as $n\rightarrow \infty$. Then, by Remark~\ref{rem:triangle}, we have
\begin{align*}
d^{n,\infty}(0_n,0_{\infty}) &\leq d^{n,n}(0_n,x_n)+ d^{n,\infty}(x_n,0_{\infty})
\leq \Vert x_n \Vert_n + d^{n,\infty}(x_n,0_{\infty})
\rightarrow 0,
\end{align*}
as $n\to\infty$, where the second inequality follows from property (i) in Definition~\ref{BanStack}.
\end{proof}

As a matter of independent interest, we note that the lower bound on the metric derivatives of absolutely continuous curves, a property established in \cite[Theorem 3.9]{SerfatyGam}, also holds over a Banach stacking. A precise statement is given below. The definitions of absolutely continuous curves and metric derivatives are given in Definition~\ref{def:ACcurves} and Theorem~\ref{thm:metderivative}.

\begin{proposition}\label{prop:derivativelowbound}
Let $\big((X_n, \xi_n)_{n\in \mathbb{N}_\infty}, \mathcal{M}\big)$ be a Banach stacking, $(u_n)_{n\in \mathbb{N}} \subset_n AC(0,T;X_n)$, and $u_\infty \in AC(0,T;X_\infty)$ (see Definition~\ref{def:ACcurves}). Assume that, for all $t\in(0,T)$,
 ${u_n(t) \underset{\mathcal{M}}\longtwoheadrightarrow u_\infty(t)}$ as $n\rightarrow \infty$. 
Then, for all $t\in(0,T)$,
$$ \liminf_{n\rightarrow \infty} \int_0^t |u_n'|^2(s) \, ds \geq \int_0^t |u_\infty'|^2(s) \, ds .$$
Here each metric derivative $|u_n'|$ is defined with respect to the metric on $X_n$, as per the definition in Theorem~\ref{thm:metderivative}.
\end{proposition}

\begin{proof}
See Appendix~\ref{sec:B}.
\end{proof}

So far we have defined a Banach stacking for the index set $\mathbb{N}_{\infty}$. However, it will also be useful to be able to have an arbitrary topological space $I$ as the index set. We extend our previous definition to allow for such an index set.

\begin{definition}\label{def:BanStackI}
Let $I$ be a topological space. Suppose $(\mathcal{M},d)$ is a metric space, ${\big((X_i, \|\cdot\|_i)\big)_{i\in I}}$ a collection of Banach spaces, and $(\xi_i)_{i\in I}$ a collection of maps $\xi_i :X_i \rightarrow \mathcal{M}$. Then we say that $\big((X_i,\xi_i)_{i\in I},\mathcal{M} \big)$ is a Banach stacking if, for all sequences $(i_n)_{n\in \mathbb{N}_\infty} \subset I$ with $i_n \rightarrow i_{\infty}$ as $n\rightarrow \infty$, we have that $\big((X_{i_n},\xi_{i_n})_{n\in \mathbb{N}_{\infty}},\mathcal{M} \big)$ is a Banach stacking according to Definition~\ref{BanStack}.
\end{definition}

\begin{remark}
For notational simplicity, in Definitions~\ref{BanStack} and~\ref{def:BanStackI} we have not incorporated the norms of the Banach spaces in the notations $\big((X_n,\xi_n)_{n\in \mathbb{N}_\infty},\mathcal{M} \big)$ and $\big((X_i,\xi_i)_{i\in I},\mathcal{M} \big)$ for Banach stackings. In concrete examples, it will be clearly stated what the intended norms are. In the abstract setting we write $\|\cdot\|_n$ for the norm on $X_n$ and $\|\cdot\|_i$ for the norm on $X_i$.
\end{remark}

In the proof of Lemma~\ref{lem:Resolve} we will need to make use of a density argument, which we establish in the following lemma.

\begin{lemma}\label{thResLem}
Let $\big((X_n, \xi_n)_{n\in \mathbb{N}_\infty}, \mathcal{M}\big)$ be a Banach stacking. Assume that, for all $n\in \mathbb{N}_\infty$, $G_n: X_n \to X_n$ is a $1$-Lipschitz-continuous function and $\mathcal{A} \subset X_{\infty}$ is a dense subset such that, for all $y_\infty \in \mathcal{A}$, there exists a sequence $(y_n)_{n\in \mathbb{N}} \subset_n X_n$ such that $y_n \underset{\mathcal{M}}\longtwoheadrightarrow y_{\infty}$ and $G_n(y_n) \underset{\mathcal{M}}\longtwoheadrightarrow G_{\infty}(y_{\infty})$ as $n\rightarrow \infty$. If $(x_n)_{n\in \mathbb{N}_\infty} \subset_n X_n$ is a sequence such that $x_n \underset{\mathcal{M}}\longtwoheadrightarrow x_{\infty}$ as $n\rightarrow \infty$, then $G_n(x_n) \underset{\mathcal{M}}\longtwoheadrightarrow G_{\infty}(x_{\infty})$ as $n\rightarrow \infty$.
\end{lemma}
\begin{proof}
Let $(x_n)_{n\in \mathbb{N}_\infty} \subset_n X_n$ be a sequence such that $x_n \underset{\mathcal{M}}\longtwoheadrightarrow x_{\infty}$ as $n\rightarrow \infty$ and let $\varepsilon>0$. By density of $\mathcal{A}$, there exists a $y_{\infty}\in \mathcal{A}$ such that $\Vert x_{\infty} - y_{\infty} \Vert_{\infty} \leq \varepsilon$. By assumption there exists a sequence $(y_n)_{n\in \mathbb{N}} \subset_n X_n$ such that $y_n \underset{\mathcal{M}}\longtwoheadrightarrow y_{\infty}$ and $G_n(y_n) \underset{\mathcal{M}}\longtwoheadrightarrow G_{\infty}(y_{\infty}) $ as $n\rightarrow \infty$.

By properties (iii) and (iv) in Definition~\ref{BanStack}, we have $\Vert y_n - x_n \Vert_n \rightarrow \Vert y_{\infty} - x_{\infty} \Vert_{\infty}$ as $n\rightarrow \infty$. Moreover, 
\begin{align*}
    d^{n,\infty}\big(G_n(x_n),G_\infty(x_{\infty}) \big) &\leq d^{n,n}\big(G_n(x_n), G_n(y_n) \big) + d^{n,\infty}\big(G_n(y_n), G_\infty(y_{\infty}) \big)\\
    &\hspace{0.4cm} +  d^{\infty,\infty}\big(G_\infty(y_{\infty}), G_\infty(x_{\infty}) \big)\\
    &\leq d^{n,n}\big(x_n, y_n \big) + d^{n,\infty}\big(G_n(y_n), G_\infty(y_{\infty}) \big)  +  d^{\infty,\infty}\big(y_{\infty}, x_{\infty} \big)\\
    &\leq \Vert x_n - y_n \Vert_n + d^{n,\infty}\big(G_n(y_n), G_\infty(y_{\infty}) \big) + \Vert x_{\infty} - y_{\infty} \Vert_{\infty}.
\end{align*}
The first inequality holds by the triangle inequality, the second because, for all $n\in \mathbb{N}_\infty$, $G_n$ is a contraction, and the third by property (i) in Definition~\ref{BanStack}. Taking the limit superior on both sides, we obtain
\[
    \limsup_{n\rightarrow \infty}  d^{n,\infty}\big(G_n(x_n), G_\infty (x_{\infty}) \big) \leq 2\Vert x_{\infty} - y_{\infty} \Vert_{\infty}
    \leq 2\varepsilon.
\]
Since $\varepsilon$ was chosen arbitrarily we conclude the desired result.
\end{proof}

In Examples~\ref{BSexmpl} in Appendix~\ref{sec:BanStackEx} we give a few examples of Banach stackings. The last one of those is of greatest interest to us in this paper and in Proposition~\ref{tlpBanach} we prove it is indeed a Banach stacking. To be able to define this particular Banach stacking and work with it, we first give the definitions of certain concepts from optimal transport.

\begin{definition}
Let $(\Omega,\Sigma)$ be a $\sigma$-algebra and consider two probability measures $\mu$ and $\nu$ defined on $\Omega$. We define $\Gamma(\mu,\nu)$ to be the set of probability measures $\pi$ on $(\Omega \times\Omega, \Sigma \otimes \Sigma)$ that satisfy, for all $A,B \in \Sigma$, $\pi(A \times \Omega)=\mu(A)$ and
$\pi(\Omega \times B) = \nu(B)$.
We refer to $\Gamma(\mu,\nu)$ as the set of transport plans.
\end{definition}

\begin{definition}\label{def:Wasserstein}
Let $\Omega \subset \mathbb{R}^d$ be an open subset and let $1\leq p < +\infty$. Denote by $\mathcal{P}_p(\Omega)$ the set of Borel probability measures on $\Omega$ with finite $p^{\text{th}}$ moment about zero. On $\mathcal{P}_p(\Omega)$ we define a metric $d_p$ given by
$$ d_p(\mu,\nu) := \inf_{\pi \in \Gamma(\mu,\nu)} \left(\int_{\Omega \times \Omega} |x-y|^p \, d\pi(x,y) \right)^{1/p} .$$
This is referred to as the $p$-Wasserstein metric.

If $\Omega$ is bounded, then we also define the $\infty$-transport distance on $\mathcal{P}(\Omega)$, which is the set of all Borel probability measures on $\Omega$. This distance is given by the metric $d_{\infty}$ that is defined as
$$ d_{\infty}(\mu,\nu):= \inf_{\pi \in \Gamma(\mu,\nu)} \operatorname{esssup}_{\pi} \left(\Omega\times \Omega \ni (x,y) \mapsto |x-y| \in [0,\infty)\right).$$ 
\end{definition}

To help understand the definition of $d_\infty$ above, we recall that if $f$ is a $\pi$-measurable real-valued function, then the essential supremum of $f$ with respect to $\pi$, denoted by $\operatorname{essup}_\pi(f)$, is the infimum of all $s\in \mathbb{R}$ such that $\pi\left(f^{-1}([s,+\infty))\right)=0$.

\begin{definition}\label{def:TLp}
Let $\Omega \subset \mathbb{R}^d$ be an open subset and let $1\leq p < +\infty$. Denote $TL^p(\Omega)$ to be the set of all pairs $(u,\mu)$ with $\mu \in \mathcal{P}_p(\Omega)$ and $u\in L^p(\Omega;\mu)$. We define the metric $d_{TL^p}$ on $TL^p(\Omega)$ by
$$ d_{TL^p(\Omega)}\big( (u,\mu),(v,\nu) \big) := \inf_{\pi \in \Gamma(\mu,\nu)} \left(\int_{\Omega \times \Omega} |u(x)-v(y)|^p + |x-y|^p \, d\pi(x,y) \right)^{1/p}.$$
\end{definition}

We note that both $d_p$, $d_\infty$, and $d_{TL^p}$ do in fact define metrics and we refer the reader to \cite{garcia2016continuum} and references therein for a more detailed description.

We will need Propositions~\ref{opLemma} and~\ref{tlpBS} regarding $p$-stagnating sequences (Definition~\ref{def:pstagnating}) and optimal transport distances, based on very similar results from \cite{garcia2016continuum}, to prove that the example in part~6 of Examples~\ref{BSexmpl} indeed forms a Banach stacking. 

For the remainder of this section we assume that $p\in [1,+\infty)$ and $\Omega$ is an open subset of $\mathbb{R}^d$.

\begin{definition}\label{def:pstagnating} Let $\mu$ and, for all $n\in \mathbb{N}$, $\mu_n$ be Borel probability measures on $\Omega$. We say
a sequence of transport plans $(\pi_n)_{n\in \mathbb{N}} \subset_n \Gamma(\mu_n,\mu)$ is $p$-stagnating if
$$ \int_{\Omega\times \Omega} |x-y|^p \, d\pi_n(x,y) \rightarrow 0 \qquad \text{as } n\to\infty. $$ 
\end{definition}

\begin{remark}
\label{rem:p-stagnating}
Comparing Definitions~\ref{def:Wasserstein} and~\ref{def:pstagnating}, we observe that there exists a $p$-stagnating sequence $(\pi_n)_{n\in\mathbb{N}}\subset \Gamma(\mu_n,\mu_\infty)$ if and only if $(\mu_n)_{n\in\mathbb{N}}$ converges to $\mu_\infty$ in the metric space $(\mathcal{P}_p(\Omega),d_p)$. The `only if' statement follows directly from the definitions, while the `if' statement requires an approximating sequence of transport plans converging to the infimum in the $p$-Wasserstein metric.
\end{remark}

\begin{proposition} \label{opLemma}
Given a Borel probability measure $\mu$ on $\Omega$, a $p$-stagnating sequence $(\pi_n)_{n\in \mathbb{N}} \subset \Gamma(\mu,\mu)$, and $u\in L^p(\Omega,\mu)$, we have
$$ \int_{\Omega \times \Omega} |u(x)-u(y)| \, d\pi_n(x,y) \rightarrow 0 \qquad \text{as } n\to\infty.$$ 
\end{proposition}
\begin{proof}
   This is proven in \cite[Lemma 3.10]{garcia2016continuum}. In \cite{garcia2016continuum} the domain $\Omega$ ($D$ in the notation of \cite{garcia2016continuum}) is assumed to be bounded, but abandoning that assumption does not influence the proof of \cite[Lemma 3.10]{garcia2016continuum}. 
\end{proof}

\begin{proposition} \label{tlpBS}
Let $(u_n,\mu_n)_{n\in \mathbb{N}_\infty}$ be a sequence in $TL^p(\Omega)$. Then the following are equivalent.

\begin{enumerate}[(i)]
\item In the $TL^p(\Omega)$ metric, $(u_n,\mu_n)\rightarrow (u_\infty,\mu_\infty)$ as $n\rightarrow \infty$.
    \item There exists a $p$-stagnating sequence $(\pi_n)_{n\in\mathbb{N}} \subset_n \Gamma(\mu_n,\mu_{\infty})$ such that
    \begin{equation}
\label{BSeq2}
\int_{\Omega \times \Omega} |u_n(x)-u_{\infty}(y)|^p \, d\pi_n(x,y) \rightarrow 0 \qquad \text{as } n\to\infty.
\end{equation}
    \item There exists at least one $p$-stagnating sequence $(\hat \pi_n)_{n\in\mathbb{N}} \subset_n \Gamma(\mu_n,\mu_{\infty})$. Moreover for every $p$-stagnating sequence $(\pi_n)_{n\in\mathbb{N}} \subset_n \Gamma(\mu_n,\mu_{\infty})$ \eqref{BSeq2} holds.
\end{enumerate}
\end{proposition}
\begin{proof}
It is immediate that (iii) implies (ii). In fact, by \cite[Lemma 3.11]{garcia2016continuum} (ii) and (iii) are equivalent. As also noted in the proof of Proposition~\ref{opLemma}, in \cite{garcia2016continuum} the domain $\Omega$ is assumed to be bounded; moreover, \cite[Lemma 3.11]{garcia2016continuum} also requires $(\mu_n)_{n\in\mathbb{N}}$ to converge weakly as measures to $\mu_\infty$. The proof of \cite[Lemma 3.11]{garcia2016continuum} remains valid, however, even if both these assumptions are dropped. 

Let $(\varepsilon_n)_{n\in\mathbb{N}} \subset (0,\infty)$ be a sequence that converges to $0$. By Definition~\ref{def:TLp}, there exists a sequence $(\pi_n)\subset_n\Gamma(\mu_n,\mu_\infty)$ such that, for all $n\in \mathbb{N}$, 
\[
d_{TL^p(\Omega)}^p\big( (u_n,\mu_n),(u_\infty,\mu_\infty) \big) + \varepsilon_n = \int_{\Omega \times \Omega} |u(x)-v(y)|^p + |x-y|^p \, d\pi_n(x,y).
\]
Thus (i) implies immediately that $(\pi_n)$ is $p$-stagnating and satisfies \eqref{BSeq2}. Hence (ii) holds.

Conversely, if (ii) is satisfied, then
\[
d_{TL^p(\Omega)}^p\big( (u_n,\mu_n),(u_\infty,\mu_\infty) \big) \leq \int_{\Omega\times\Omega} |u_n(x)-u_\infty(y)|^p + |x-y|^p \, d\pi_n(x,y) \to 0 \qquad \text{as } n\to\infty,
\]
hence (i) holds.

For completeness we mention that the above proof of the equivalence of (i) and (ii) mirrors the proof of \cite[Proposition 3.12]{garcia2016continuum}, in particular the equivalence of claims 1 and 3 in that proposition. Again, we observe that \cite[Proposition 3.12]{garcia2016continuum} assumes a bounded domain $\Omega$. Without that assumption the proof remains valid, if the statement of weak convergence of $(\mu_n)_{n\in\mathbb{N}}$ is removed from claim 3 in \cite[Proposition 3.12]{garcia2016continuum}. 
\end{proof}

We remind the reader that if $\Omega_1$ and $\Omega_2$ are measurable spaces (with corresponding $\sigma$-algebras), $\mu$ is a measure on $\Omega_1$, and $f: \Omega_1\to \Omega_2$ a function, then the pushforward measure $f\#\mu$ is a measure on $\Omega_2$ defined by, for all measurable subsets $A\in \Omega_2$, $f\#\mu(A):=\mu(f^{-1}(A))$.

\begin{proposition}\label{tlpBanach}
For all $\mu \in \mathcal{P}_p(\Omega)$, define $\xi_{\mu}: L^p(\Omega;\mu) \rightarrow TL^p(\Omega)$ by $\xi_{\mu}(u) := (u,\mu)$. Then $\big((L^p(\Omega,\mu), \xi_{\mu})_{\mu \in \mathcal{P}_p(\Omega)}, TL^p(\Omega)\big)$ (see part~6 of Examples~\ref{BSexmpl}) is a Banach stacking. 
\end{proposition}
\begin{proof}
Let $(\mu_n)_{n\in\mathbb{N}_\infty} \subset \mathcal{P}_p(\Omega)$ be a sequence such that $\mu_n \rightarrow \mu_{\infty}$ as $n \rightarrow \infty$ with respect to the $p$-Wasserstein distance. By Definition~\ref{def:BanStackI} it is sufficient to prove that $\big( (L^p(\Omega,\mu_n)_{n\in\mathbb{N}_\infty},\xi_n),TL^p(\Omega) \big)$ forms a Banach stacking, where $\xi_n(u):=(u,\mu_n)$. We shall do so by proving conditions~(i)--(iv) from Definition~\ref{BanStack}. For convenience we let $d$ denote the $TL^p$ metric from Definition~\ref{def:TLp}.

\begin{enumerate}
    \item Let $n\in \mathbb{N}_{\infty}$ and $u,v\in L^p(\Omega;\mu_n)$. Define the Borel map $\mathcal{D}:\Omega \rightarrow \Omega\times \Omega$ by $\mathcal{D}(x):=(x,x)$ and define the transport plan $\pi:= \mathcal{D}\# \mu_n$. Because $\pi \in \Gamma(\mu_n, \mu_n)$, it follows from \eqref{eq:dnm} and Definition~\ref{def:TLp} that
    \begin{align*}
      d^{n,n}(u,v) &\leq \left( \int_{\Omega \times \Omega} |u(x)-v(y)|^p + |x-y|^p \, d\pi(x,y) \right)^{1/p}\\
      &=\left( \int_{\Omega}  |u(x)-v(x)|^p \, d\mu_n(x) \right)^{1/p}
      = \Vert u-v \Vert_{L^p},
    \end{align*}
 as required by condition~(i) in Definition~\ref{BanStack}.

    \item Let $u_{\infty} \in L^p(\Omega;\mu_{\infty})$ and let $(\pi_n)_{n\in \mathbb{N}} \subset_n \Gamma(\mu_n,\mu_{\infty})$ be a $p$-stagnating sequence of transport plans. Such a sequence exists by Remark~\ref{rem:p-stagnating}, since $\mu_n\to \mu_\infty$ with respect to the $p$-Wasserstein distance as $n\to\infty$. Let $n\in \mathbb{N}$. By the disintegration theorem, see \cite[Theorem 5.3.1]{greenbook} for $\mu_n$-a.e. $x\in \Omega$ there exists a Borel probability measure $\eta_x^{(n)}$ on $\Omega$ such that, for all $\pi_n$-integrable functions $f:\Omega\times \Omega \to \mathbb{R}$ the map
    $$x \mapsto \int_{\Omega} f(x,y) \, d\eta_x^{(n)}(y) $$
    is $\mu_n$-integrable and 
    $$ \int_{\Omega \times \Omega} f(x,y) \, d\pi_n(x,y) = \int_{\Omega} \int_{\Omega} f(x,y) d\eta_x^{(n)}(y) \, d\mu_n(x).$$
   Here we are disintegrating the measure $\pi_n$ on $\Omega\times\Omega$ with respect to the function $(x,y)\mapsto x$ defined on $\Omega\times\Omega$. We note that the map $(x,y)\mapsto u_{\infty}(y)$ is $\pi_n$-integrable, so we can define, for $\mu_n$-a.e. $x\in \Omega$,    
    $$ u_n(x) := \int_{\Omega} u_{\infty}(y) \, d\eta_x^{(n)}(y)$$
   and this function $u_n$ is $\mu_n$-integrable. We wish to show that $(u_n)_{n\in \mathbb{N}}$ converges to $u_{\infty}$ in the $TL^p$ topology as $n\to \infty$. For this purpose, for all $n\in \mathbb{N}$, we define the Borel probability measure $\tau_n$ on $\Omega \times \Omega$ such that for a Borel set $A\subset \Omega \times \Omega$ we have 
    $$ \tau_n(A) = \int_{\Omega}\int_{(y,z)\in A} \, d\eta_x^{(n)}(y) \, d\eta_x^{(n)}(z) \,  d\mu_n(x). $$
    We note also, given a Borel set $B\subset \Omega$, that
    \[
        \tau_n(B\times \Omega) = \int_{\Omega} \int_{B} \int_{\Omega} \, d\eta_x^{(n)}(z) \, d\eta_x^{(n)}(y) \, d\mu_n(x)
        = \int_{(x,y)\in \Omega \times B} \, d\pi_n(x,y) = \mu_{\infty}(B) .
    \]
    A similar argument shows that $\tau_n(\Omega \times B) = \mu_{\infty}(B)$ and thus $\tau_n \in \Gamma(\mu_{\infty},\mu_{\infty})$. Next we show that $\tau_n$ is $p$-stagnating. Indeed 
    \begin{align*}
    &\left(\int_{\Omega \times \Omega} |y-z|^p \, d\tau_n(y,z) \right)^{1/p} = \left(\int_{\Omega}\int_{\Omega} \int_{\Omega} |y-x+x-z|^p \, d\eta_x^{(n)}(y) \, d\eta_x^{(n)}(z) \, d\mu_n(x) \right)^{1/p} \\
    &\leq \left( \int_{\Omega}\int_{\Omega} \int_{\Omega} |y-x|^p \, d\eta_x^{(n)}(y) \, d\eta_x^{(n)}(z) \, d\mu_n(x) \right)^{1/p} + \left( \int_{\Omega}\int_{\Omega} \int_{\Omega} |x-z|^p  \, d\eta_x^{(n)}(y) \, d\eta_x^{(n)}(z) \, d\mu_n(x) \right)^{1/p} \\
    &= 2\left( \int_{\Omega\times \Omega} |x-y|^p \, d\pi_n(x,y) \right)^{1/p} \rightarrow 0 \qquad \text{as } n\to\infty.
    \end{align*}
    The inequality above follows from Minkowski's inequality and the second equality holds since $(\pi_n)_{n\in \mathbb{N}}$ is $p$-stagnating by assumption. Next we note that 
    \begin{align*}
    \int_{\Omega\times \Omega} |u_n(x)-u_{\infty}(y)|^p \, d\pi_n(x,y) 
    &= \int_{\Omega\times \Omega} \left|\int_{\Omega} \big(u_{\infty}(z)-u_{\infty}(y)\big) \, d\eta_x^{(n)}(z)\right|^p \, d\pi_n(x,y) \\
    &\leq \int_{\Omega}\int_{\Omega} \int_{\Omega} |u_{\infty}(z) - u_{\infty}(y)|^p \, d\eta_x^{(n)}(z) \, d\eta_x^{(n)}(y) \, d\mu_n(x) \\
    &= \int_{\Omega}\int_{\Omega} \int_{\Omega} |u_{\infty}(z) - u_{\infty}(y)|^p \, d\tau_n(z,y) \rightarrow 0,
    \end{align*}
   as $n\to\infty$. The inequality above holds by Jensen's inequality and the convergence at the end holds by Proposition~\ref{opLemma} since the sequence $(\tau_n)_{n\in \mathbb{N}}$ is stagnating. Thus the sequence $(u_n)_{n\in \mathbb{N}} \subset_n L^p(\Omega;\mu_n)$ converges to $u_{\infty}$ in $TL^p(\Omega)$ in the sense of Banach stackings (i.e., as in Definition~\ref{BanStack}), as required.

    \item Assume we have two sequences $(u_n)_{n\in \mathbb{N}_{\infty}}\subset_n L^p(\Omega;\mu_n)$ and $(v_n)_{n\in \mathbb{N}_{\infty}} \subset_n L^p(\Omega;\mu_n)$ such that $(u_n)_{n\in \mathbb{N}}$ converges to $u_{\infty}$ and $(v_n)_{n\in \mathbb{N}}$ converges to $v_{\infty}$ in $TL^p(\Omega)$ in the sense of Banach stackings. By definition of $\xi_n$, this means that part (i) of Proposition~\ref{tlpBS} is satisfied for both $(u_n, \mu_n)_{n\in\mathbb{N}_\infty}$ and $(v_n, \mu_n)_{n\in\mathbb{N}_\infty}$ and thus, by part (iii) there exists a $p$-stagnating sequence $(\pi_n)_{n\in \mathbb{N}} \subset_n \Gamma(\mu_n,\mu_{\infty})$ which satisfies \eqref{BSeq2} for both $(u_n)_{n\in\mathbb{N}_\infty}$ and $(v_n)_{n\in\mathbb{N}_\infty}$. Then, by Minkowski's inequality,
    \begin{align*}
        &\left( \int_{\Omega\times\Omega} |u_n(x)+v_n(x) - u_{\infty}(y) -v_{\infty}(y)|^p \, d\pi_n(x,y) \right)^{1/p} \\
        &\leq 
        \left( \int_{\Omega\times\Omega} |u_n(x) - u_{\infty}(y)|^p \, d\pi_n(x,y) \right)^{1/p}
        + \left( \int_{\Omega\times\Omega} |v_n(x) -v_{\infty}(y)|^p \, d\pi_n(x,y) \right)^{1/p} \\
        &\rightarrow 0 \qquad \text{as } n\to\infty.
    \end{align*}
    Again by Proposition~\ref{tlpBS} we now obtain that $u_n + v_n \underset{TL^p(\Omega)}\longtwoheadrightarrow u_{\infty} + v_{\infty}$.
    Via a similar argument we also find that, for all $\lambda \in \mathbb{R}$, $\lambda x_n \underset{TL^p(\Omega)}\longtwoheadrightarrow \lambda x_{\infty}.$

    \item Assume the sequence $(u_n)_{n\in \mathbb{N}_\infty} \subset_n L^p(\Omega;\mu_n)$ is such that $(u_n,\mu_n)$ converges to $(u_{\infty},\mu_{\infty})$. By comparing the definitions of $d_p$ and $d_{TL^p(\Omega)}$ in Definitions~\ref{def:Wasserstein} and~\ref{def:TLp}, respectively, we see that $(\mu_n)_{n\in\mathbb{N}}$ converges to $\mu_\infty$ in the metric space $(\mathcal{P}_p(\Omega),d_p)$. Hence, by Remark~\ref{rem:p-stagnating}, a $p$-stagnating sequence $(\pi_n)_{n\in \mathbb{N}} \subset_n  \Gamma(\mu_n,\mu_{\infty})$ exists. Then, Minkowski's inequality,
    \begin{align*}
    &\big|\Vert u_n \Vert_{L^p} - \Vert u_{\infty} \Vert_{L^p}\big| \\
    &= \left|\left(\int_{\Omega \times \Omega}  |u_n(x)|^p \, d\pi_n(x,y) \right)^{1/p} - \left(\int_{\Omega \times \Omega}|u_{\infty}(y)|^p \, d\pi_n(x,y) \right)^{1/p}\right| \\
    &\leq \left(\int_{\Omega \times \Omega} |u_n(x)-u_{\infty}(y)|^p \, d\pi_n(x,y) \right)^{1/p} \rightarrow 0 \qquad \text{as } n\to\infty,
    \end{align*}
    where convergence follows from \eqref{BSeq2}.
\end{enumerate}
\end{proof}

The next lemma gives sufficient conditions to obtain convergence in $TL^r$ from convergence in $TL^p$ with $r>p$.

\begin{lemma}
\label{tlplem}
    Let $\Omega \subset \mathbb{R}^N$ be open and let $\mu,\nu$ be Borel probability measures on $\Omega$. Let $1\leq p <+\infty$ and $p<q\leq+\infty$. Also let $r\in [p,q)$. Let $u\in L^p(\Omega;\mu)$ and $v\in L^p(\Omega;\nu)$.
    
    Assume there exists a $C<+\infty$ such that $\max\left(\Vert u \Vert_{L^q(\Omega;\mu)}, \Vert v \Vert_{L^q(\Omega;\nu)}\right) \leq C$. Then there exists a $\theta \in (0,1]$ (that is determined by $p$, $r$, and $q$) such that
$$ d_{TL^r} \big( (u,\mu) ,(v,\nu) \big) \leq d_r(\mu,\nu) + (2C)^{\frac{q(1-\theta)}{r}} d_{TL^p}\big( (u,\mu) , (v,\nu) \big)^{\frac{\theta p}{r}}.$$

\end{lemma}

\begin{proof} Let $\theta := \frac{q-r}{q-p} \in (0,1]$ such that $r= \theta p + (1-\theta) q$ and let $\pi \in \Gamma(\mu,\nu)$. We note that
\begin{align*}
&\int_{\Omega \times \Omega} |u(x)-v(y)|^r \, d\pi(x,y)
= \int_{\Omega \times \Omega} |u(x)-v(y)|^{p\theta}|u(x)-v(y)|^{q(1-\theta)} \, d\pi(x,y)\\
&\leq \left( \int_{\Omega \times \Omega} |u(x)-v(y)|^p \, d\pi(x,y) \right)^{\theta} \left( \int_{\Omega \times \Omega} |u(x)-v(y)|^q \, d\pi (x,y)\right)^{1-\theta}\\
&\leq \left( \int_{\Omega \times \Omega} |u(x)-v(y)|^p \, d\pi(x,y) \right)^{\theta} \left(\left( \int_{\Omega \times \Omega} |u(x)|^q \, d\pi(x,y) \right)^{\frac{1}{q}} +\left( \int_{\Omega \times \Omega} |v(y)|^q \, d\pi(x,y) \right)^{\frac{1}{q}} \right)^{q(1-\theta)}\\
&= \left( \int_{\Omega \times \Omega} |u(x)-v(y)|^p \, d\pi(x,y) \right)^{\theta} \left( \Vert u \Vert_{L^q} + \Vert v \Vert_{L^q} \right)^{q(1-\theta)}\\
&\leq  \left( 2C \right)^{q(1-\theta)} \left( \int_{\Omega \times \Omega} |u(x)-v(y)|^p \, d\pi (x,y) \right)^{\theta}.
\end{align*}
The first inequality holds by H\"older's inequality, the second by Minkowski's inequality. By again applying Minkowski's inequality, now in combination with Definitions~\ref{def:Wasserstein} and~\ref{def:TLp}, we obtain
\begin{align*}
d_{TL^r} \big( (u,\mu) ,(v,\nu) \big) &\leq d_r(\mu,\nu) +  \left( \int_{\Omega \times \Omega} |u(x)-v(y)|^r \, d\pi(x,y) \right)^{\frac{1}{r}} ,\\
&\leq d_r(\mu,\nu) +  \left( 2C \right)^{\frac{q(1-\theta)}{r}} \left( \int_{\Omega \times \Omega} |u(x)-v(y)|^p \, d\pi (x,y) \right)^{\frac{\theta}{r}}.
\end{align*}
Taking the infimum of the above over all $\pi\in \Gamma(\mu,\nu)$ yields our result.

\end{proof}

\begin{lemma}
\label{tlplem2}
Let $\Omega \subset \mathbb{R}^N$ be open and let $(\mu_n)_{n\in \mathbb{N}_\infty}$ be a sequence of Borel probability measures on $\Omega$. Let $1\leq p<q <+\infty$. Suppose we have a sequence $(u_n)_{n\in \mathbb{N}_\infty} \subset_n L^q(\Omega;\mu_n)$ such that $(u_n,\mu_n)$ converges to $(u_{\infty},\mu_{\infty})$ in $TL^q(\Omega)$ as $n\rightarrow \infty$. Then $(u_n,\mu_n)$ converges to $(u_{\infty},\mu_{\infty})$ in $TL^p(\Omega)$ as $n\rightarrow \infty$. 
\end{lemma}

\begin{proof}
First we note that, since $\mu_n$ is a finite measure, if $u_n\in L^q(\Omega;\mu_n)$, then, by H\"older's inequality, also $u_n \in L^p(\Omega;\mu_n)$. By Proposition~\ref{tlpBS} there exists a $q$-stagnating sequence $(\pi_n)_{n\in \mathbb{N}}$ of transport plans $\pi_n \in \Gamma(\mu_n,\mu_{\infty})$ and
$$\int_{\Omega \times \Omega} |u_n(x)-u_{\infty}(y)|^q \, d\pi_n(x,y) \rightarrow 0 \qquad \text{as } n\rightarrow \infty.$$

By H\"older's inequality we get
\begin{align*}
&\int_{\Omega \times \Omega} |u_n(x)-u_{\infty}(y)|^p \, d\pi_n(x,y)
= \int_{\Omega \times \Omega} 1 \cdot |u_n(x)-u_{\infty}(y)|^p \, d\pi_n(x,y) \\
&\leq \left( \int_{\Omega \times \Omega} 1^{\frac{q}{q-p}} d\pi_n \right)^{\frac{q-p}{q}} \left( \int_{\Omega \times \Omega} |u_n(x)-u_{\infty}(y)|^q  d\pi_n \right)^{\frac{p}{q}} \\
&= \left( \int_{\Omega \times \Omega} |u_n(x)-u_{\infty}(y)|^q  d\pi_n \right)^{\frac{p}{q}}\rightarrow 0 \qquad \text{as } n\rightarrow \infty.
\end{align*}
A similar argument shows that $$ \int_{\Omega \times \Omega} |x-y|^p \, d\pi_n(x,y) \leq \left( \int_{\Omega \times \Omega} |x-y|^q  d\pi_n \right)^{\frac{p}{q}}\rightarrow 0 \qquad \text{as } n\rightarrow \infty,$$ where the convergence follows from Definition~\ref{def:pstagnating}. 

These two estimates together imply the desired convergence.
\end{proof}

Next we introduce a useful lemma regarding $TL^p$ convergence.
\begin{definition}
Let $\Omega$ be a metric space and let $\Sigma$ be the Borel $\sigma$-algebra generated by the open subsets of $\Omega$. Given a sequence of probability measures $(\mu_n)_{n\in \mathbb{N}_{\infty}}$ on $(\Omega,\Sigma)$ we say $(\mu_n)$ converges weakly to $\mu_{\infty}$ as $n\rightarrow \infty$ if for all continuous bounded functions $f:\Omega \rightarrow \mathbb{R}$ we have $\int fd\mu_n \rightarrow \int fd\mu_{\infty}$ as $n\rightarrow \infty$. 
\end{definition}

\begin{lemma}
\label{tlplem3}
Let $\Omega \subset \mathbb{R}^N$ be open and let $(\mu_n)_{n\in \mathbb{N}_\infty}$ be a sequence of Borel probability measures on $\Omega$. Let $1\leq p <+\infty$ and suppose $(u_n)_{n\in \mathbb{N}_\infty} \subset_n L^p(\Omega;\mu_n)$ is a sequence such that $((u_n,\mu_n))$ converges to $(u_{\infty},\mu_{\infty})$ in $TL^p(\Omega)$ as $n\rightarrow \infty$. Then $(u_n \# \mu_n)$ converges weakly to $u_{\infty} \# \mu_{\infty}$ as $n\rightarrow \infty$.  
\end{lemma}

\begin{proof}
By Lemma~\ref{tlplem2} we can assume without loss of generality that $p=1$. By the Portmanteau theorem (see Remark~\ref{Portmanteau}) it suffices to show that, for all bounded and Lipschitz-continuous functions $f:\mathbb{R} \rightarrow \mathbb{R}$,
$$   \int_{\mathbb{R}} f(t) d(u_n \# \mu_n)(t) \rightarrow  \int_{\mathbb{R}} f(t) d(u_{\infty} \# \mu_{\infty})(t) \qquad \text{as } n\to \infty .$$
Fix $f$ as above and assume it has a Lipschitz constant $L>0$. Also choose a $p$-stagnating sequence $(\pi_n)_{n\in \mathbb{N}}$ of transport plans $\pi_n \in \Gamma(\mu_n,\mu_{\infty})$. Then
\begin{align*}
    \left| \int_{\mathbb{R}} f(t) d(u_n \# \mu_n)(t) -  \int_{\mathbb{R}} f(t) d(u_{\infty} \# \mu_{\infty})(t) \right| 
    &=  \left| \int_{\Omega} f(u_n(x)) d\mu_n(x) -  \int_{\Omega} f(u_{\infty}(y)) d\mu_{\infty}(y) \right| \\
    &= \left| \int_{\Omega} f(u_n(x)) - f(u_{\infty}(y)) \; \; d\pi_n (x,y) \right| \\
    &\leq \int_{\Omega} |f(u_n(x)) - f(u_{\infty}(y))| \; \; d\pi_n (x,y) \\
    &\leq L \int_{\Omega} |u_n(x) - u_{\infty}(y)| \; \; d\pi_n (x,y)  
    \rightarrow 0,
\end{align*}
as $n\to \infty$. The convergence follows from Proposition~\ref{tlpBS}.
\end{proof}

\begin{remark}
\label{Portmanteau}
The Portmanteau theorem refers to a collection of results that give equivalent definitions for weak convergence of measures. The precise statement we need in in the proof of Lemma~\ref{tlplem3} is as follows. 

Let $(\mu_n)_{n\in \mathbb{N}_{\infty}}$ be a sequence of Borel probability measures on $\mathbb{R}$ such that for all bounded Lipschitz-continuous functions $f:\mathbb{R} \rightarrow \mathbb{R}$ we have $\int_\mathbb{R} f d\mu_n \rightarrow \int_\mathbb{R} f d\mu_{\infty}$ as $n\rightarrow \infty$. Then $(\mu_n)$ converges weakly to $\mu_{\infty}$. 

We refer the reader to \cite[Theorem 2.1]{Billingsley99} for details. We note here that that theorem does not quite state exactly what we claimed above; our claim, however, is implied by the proof in \cite{Billingsley99}. In particular, that same proof still holds if `bounded, uniformly continuous function' in \cite[Theorem~2.1 condition~(ii)]{Billingsley99} (and its proof) is replaced by `bounded, Lipschitz-continuous function', because the function $f$ in \cite[formula (1.1)]{Billingsley99} which is used in the proof is not only bounded and uniformly continuous, but also Lipschitz continuous.
\end{remark}

We conclude this section with a remark regarding possible generalizations of the $TL^p$ metric.

\begin{remark}
For a function $f:\Omega\rightarrow \mathbb{R}$, we define its graph $Gf:\Omega \rightarrow \Omega\times\mathbb{R}$ by $Gf(x):=(x,f(x))$. Moreover for each $\mu\in\mathcal{P}_p(\Omega)$, define the map $G_{\mu}:L^p(\Omega;\mu)\rightarrow \mathcal{P}_p(\Omega\times\mathbb{R})$ by $G_{\mu}(u):=Gu\#\mu$. Let $u\in L^p(\Omega;\mu)$ and $v\in L^p(\Omega;\nu)$. From the definition of the $TL^p$ metric we obtain that
$$ d_{TL^p}((u,\mu),(v,\nu))= d_p(G_\mu (u),G_{\nu}(v)),$$
where $d_p$ is the $p$-Wasserstein metric on $\mathcal{P}_p(\Omega\times\mathbb{R})$ (see Definition~\ref{def:Wasserstein}). 

Thus we can think of the metric $d_{TL^p}$ as the $p$-Wasserstein metric between the probability measures induced on the graph of the functions. So rather than use to the Banach stacking $\big((L^p(\Omega,\mu), \xi_{\mu})_{\mu \in \mathcal{P}_p(\Omega)}, TL^p(\Omega)\big)$ we could equivalently use the Banach stacking $\big((L^p(\Omega,\mu), G_{\mu})_{\mu \in \mathcal{P}_p(\Omega)}, \mathcal{P}_p(\Omega\times\mathbb{R}) \big)$. We will not do this to keep in line with the common notation in the literature.

Nonetheless, this alternative perspective does motivate possible generalizations. In particular, we could swap the $p$-Wasserstein metric on $\mathcal{P}_p(\Omega\times\mathbb{R})$ for a different metric between probability measures.
\end{remark}

\section{Extension of a result by Brezis and Pazy to Banach stackings}\label{sec:BrazisPazyextension}

In this section we shall provide a generalisation of the Brezis--Pazy theorem from \cite[Theorem 3.1]{brezis1972convergence}. The original theorem concerns $\omega$-accretive operators on a Banach space and the semigroup they generate. Instead, in Theorem~\ref{theorem1} we consider $\omega$-accretive operators defined on a sequence of Banach spaces that together form a Banach stacking. 

First we state a lemma contained in \cite[Proof of Theorem I]{CrandallLiggett71}; see also \cite[Lemma 2.4]{brezis1972convergence}. This lemma is used in  the proof of the classical Brezis--Pazy theorem and will be important for our generalisation as well.

Given a non-empty set $B\subset X$, we define
\begin{equation}\label{eq:infnorm}
\inf \Vert B \Vert := \inf \{\Vert z \Vert \}_{z\in B}.
\end{equation}
We also remind the reader that for $\omega$-accretive operators that satisfy the range condition~\eqref{rangeCon} for some $\delta>0$, we can interpret $R_{\lambda}(A)$ as a Lipschitz continuous function from $X\rightarrow X$ when $\lambda\in \mathfrak{I}_{w}\cap (0,\delta)$. This is discussed further in Remark~\ref{accremark}.

\begin{lemma}\label{lem:lemmafromCL}
Let $A$ be an $\omega$-accretive operator on a Banach space $X$ satisfying the range condition~\eqref{rangeCon} and suppose $x\in D(A)$. Then we have, for all $t, \tau \in [0,+\infty)$,
\begin{equation}
\label{resolve1}
|S_A(t)x-R_{t/n}^n(A)x| \leq \frac{2t}{\sqrt{n}} \inf \Vert A(x) \Vert  e^{4\omega t}
\end{equation}
and
\begin{equation}
\label{continuity1}
|S_A(t)x-S_A(\tau)x| \leq 2|t-\tau| \inf \Vert A(x) \Vert [e^{4\omega t}+e^{2\omega(t+\tau)}].
\end{equation}
\end{lemma}
\begin{proof}
This is proven as part of the proof of Theorem I in \cite{CrandallLiggett71}; see in particular \cite[formulas (1.10) and (1.11)]{CrandallLiggett71}. To avoid confusion, we note that $J_{t/n}^n$ in \cite{CrandallLiggett71} corresponds to our $R_{t/n}^n$.
\end{proof}

We note that in Lemma~\ref{lem:lemmafromCL} $x$ is contained in the domain of the semigroup as a consequence of Theorem~\ref{CranLig}. This lemma gives us a quantitative estimate on how the semigroup of an $\omega$-accretive operator is approximated by the resolvent maps. Moreover we have an estimate on the modulus of continuity of the semigroup.

\begin{theorem}[Brezis--Pazy for Banach stackings] \label{theorem1}
Let $((X_n,\xi_n)_{n \in \mathbb{N}_{\infty}},\mathcal{M})$ be a Banach stacking. Let $(\omega_n)_{n \in \mathbb{N}_{\infty}}$ be a sequence of real numbers bounded above by $\omega\in\mathbb{R}$. Let $(A_n)_{n\in\mathbb{N}_\infty}$ be a sequence of $\omega_n$-accretive operators $A_n$ on $X_n$ that each satisfy the range condition~\eqref{rangeCon}. Assume that there exists a $\delta\in \mathfrak{I}_{\omega}$ such that, for all $\lambda \in (0,\delta)$ and for all sequences $(z_n)_{n\in\mathbb{N}_\infty}\subset \overline{D(A_n)}$ such that $z_n\underset{\mathcal{M}}\longtwoheadrightarrow z_\infty$ as $n\rightarrow \infty$, we have $R_{\lambda}(A_n)z_n \underset{\mathcal{M}}\longtwoheadrightarrow R_{\lambda}(A_{\infty})z_\infty$ as $n\rightarrow \infty$. 
Then, for all sequences $(x_n)_{n\in\mathbb{N}_\infty} \subset_n \overline{D(A_n)}$ that satisfy $x_n \underset{\mathcal{M}}\longtwoheadrightarrow x_{\infty}$ and for all $t\geq 0$, we have that $S_{A_n}(t)x_n \underset{\mathcal{M}}\longtwoheadrightarrow S_{A_{\infty}}(t)x_{\infty}$ as $n\rightarrow \infty$.
Moreover, for all $T\in [0,+\infty)$ the convergence is uniform in $t$ on the interval $[0,T]$.
\end{theorem}

For the proof of Theorem~\ref{theorem1} we need the following two lemmas. 

If $X$ is a set and $f:X\to X$ a function, then, for all $K\in\mathbb{N}$, we denote by $f^{(\circ K)}$ be the $K$-fold composition of $f$ with itself, i.e. $f^{(\circ K)} := \underbrace{f \circ \ldots \circ f}_{K \text {times}}$.

\begin{lemma} \label{thm1lem1}
Let $((X_n,\xi_n)_{n \in \mathbb{N}_{\infty}},\mathcal{M})$ be a Banach stacking. Suppose $(f_n)_{n\in \mathbb{N}_\infty}$ is a sequence of maps $f_n: X_n \rightarrow X_n$ such that, for all sequences $(x_n)_{n\in \mathbb{N}_\infty} \subset_n X_n$, we have that $x_n \underset{\mathcal{M}}\longtwoheadrightarrow x_{\infty}$ as $n\rightarrow \infty$ implies $f_n(x_n) \underset{\mathcal{M}}\longtwoheadrightarrow f_{\infty}(x_{\infty})$ as $n\rightarrow \infty$. Then, for all $K\in \mathbb{N}$ and all sequences $(x_n)_{n\in \mathbb{N}_\infty} \subset_n X_n$, we have that $x_n \underset{\mathcal{M}}\longtwoheadrightarrow x_{\infty}$ as $n\rightarrow \infty$ implies $f_n^{(\circ K)}(x_n) \underset{\mathcal{M}}\longtwoheadrightarrow f_{\infty}^{(\circ K)}(x_{\infty})$ as $n\rightarrow \infty$.
\end{lemma}
\begin{proof}
This follows by induction on $K$.
\end{proof}

\begin{lemma} \label{thm1lem2}
Let $(h_n)_{n\in \mathbb{N}}$ be a sequence of Lipschitz-continuous functions $h_n: [0,T] \rightarrow \mathbb{R}$ with the same Lipschitz constant, i.e., there exists a $C>0$ such that, for all $t,\tau \in [0,T]$,
$|h_n(t)-h_n(\tau)| \leq C|t-\tau|$. If $(h_n(t))$ converges pointwise to $0$, it also converges uniformly to $0$ for $t\in [0,T]$.
\end{lemma}
\begin{proof}
Suppose this fails. Then there exists an $\varepsilon>0$ and a sequence $(t_n)_{n\in\mathbb{N}} \subset [0,T]$ such that, for all $n\in \mathbb{N}$,
$|h_n(t_n)| \geq \varepsilon$. 
By compactness of $[0,T]$ there exists a $t^*\in [0,T]$ such that, up to taking a subsequence, $t_n \rightarrow t_*\in [0,T]$ as $n\to \infty$. From this we get that 
\[
    |h_n(t_*)| \geq |h_n(t_n)| - |h_n(t_*)-h_n(t_n)| \geq |h_n(t_n)| - C|t_n - t_*| 
    \geq \varepsilon - C|t_n - t_*| 
\]
and thus $ 0 = \limsup_{n\to\infty} |h_n(t_*)| \geq \varepsilon.$ This gives us a contradiction.
\end{proof}
\begin{proof}[Proof of Theorem~\ref{theorem1}]
Fix $T<+\infty$. Let $(x_n)_{n\in \mathbb{N}_{\infty}} \subset_n \overline{D(A_n)}$ be a sequence such that $x_n \underset{\mathcal{M}}\longtwoheadrightarrow x_{\infty}$ as $n\rightarrow \infty$. We first prove the result under the assumption that $x_\infty \in D(A_\infty)$. After that, we will extend the proof to all $x_\infty \in \overline{D(A_\infty)}$.

Fix $\lambda \in (0,\delta)$ and define, for all $n\in \mathbb{N}_\infty$, $z_n:=R_{\lambda}(A_n)x_n$. We will show that, on $[0,T]$, the sequence $(S_{A_n}(t)z_n)$ converges to $S_{A_{\infty}}(t)z_{\infty}$ uniformly in $t$. 

To show pointwise convergence first, fix $t>0$ and $K\in \mathbb{N}$ such that $K>t \omega$. Then
\begin{align*}
d^{n,\infty}( S_{A_n}(t)z_n , S_{A_{\infty}}(t)z_{\infty} )
&\leq d^{n,n}(  S_{A_n}(t)z_n ,  R_{t/K}^K(A_n)z_n  ) + d^{n,\infty}( R_{t/K}^K(A_n)z_n  ,  R_{t/K}^K(A_{\infty}) z_{\infty} ) \\
&\hspace{0.4cm}+ d^{\infty,\infty}( R_{t/K}^K(A_{\infty}) z_{\infty} , S_{A_{\infty}}(t)z_{\infty} )\\
&\leq \Vert S_{A_n}(t)z_n  - R_{t/K}^K(A_n)z_n \Vert_n + \Vert R_{t/K}^K(A_{\infty}) z_{\infty} - S_{A_{\infty}}(t)z_{\infty} \Vert_{\infty} \\
&\hspace{0.4cm}+ d^{n,\infty}(  R_{t/K}^K(A_n)z_n  ,  R_{t/K}^K(A_{\infty}) z_{\infty} )\\
&\leq \frac{2t}{\sqrt{K}}\big( \inf \Vert A_n(z_n) \Vert_n+ \inf \Vert A_{\infty}(z_{\infty}) \Vert_{\infty} \big)e^{4\omega t} \\
&\hspace{0.4cm}+ d^{n,\infty}( R_{t/K}^K(A_n)z_n  , R_{t/K}^K(A_{\infty}) z_{\infty}).
\end{align*}

The first inequality follows from the triangle inequality for $d$ and the second from condition~(i) in Definition~\ref{BanStack} (see also Remark~\ref{rem:triangle}). The third inequality follows from \eqref{resolve1}. By applying Lemma~\ref{thm1lem1} we find that the sequence $(R_{t/K}^K(A_n)z_n)$ converges to $R_{t/K}^K(A_{\infty}) z_{\infty}$ as $n\rightarrow \infty$. Taking the limit superior as $n\rightarrow \infty$ in the above inequality we obtain
$$
\limsup_{n \rightarrow \infty} d^{n,\infty}( S_{A_n}(t)z_n , S_{A_{\infty}}(t)z_{\infty} )
\leq \limsup_{n \rightarrow \infty}  \frac{2t}{\sqrt{K}}\big( \inf \Vert A_n(z_n) \Vert_n+ \inf \Vert A_{\infty}(z_{\infty}) \Vert_{\infty} \big) e^{4\omega t}.
$$

As a consequence of the definition of resolvent in Definition~\ref{def:operatorsoperations}, for all $n\in \mathbb{N}_{\infty}$, $\frac{1}{\lambda}(I_{X_n}-R_{\lambda}(A_n))x_n \in A_n(z_n)$. It follows that, for all $n\in \mathbb{N}_{\infty}$,
\begin{equation}
\label{operatorBound}
\inf \Vert A_n(z_n) \Vert \leq \frac{\Vert x_n \Vert_n + \Vert R_{\lambda}(A_n)x_n \Vert_n }{\lambda} .
\end{equation}
Thus
\begin{align*}
\limsup_{n \rightarrow \infty} d^{n,\infty}( S_{A_n}(t)z_n , S_{A_{\infty}}(t)z_{\infty} )
&\leq \limsup_{n \rightarrow \infty}  \frac{2te^{4\omega t}}{\lambda\sqrt{K}}\big( \Vert x_n \Vert_n + \Vert R_{\lambda}(A_n)x_n \Vert_n \\ 
&\hspace{2.7cm}+  \Vert x_{\infty} \Vert_{\infty} + \Vert R_{\lambda}(A_{\infty})x_{\infty} \Vert_{\infty} \big) \\
&= \frac{4te^{4\omega t}}{\lambda\sqrt{K}} \big(  \Vert x_{\infty} \Vert_{\infty} + \Vert R_{\lambda}(A_{\infty})x_{\infty} \Vert_{\infty} \big).
\end{align*}
The final equality above follows from condition~(iv) in Definition~\ref{BanStack}. We note that to use this condition, we require that $(R_{\lambda}(A_n))$ converges pointwise to $R_\lambda(A_\infty)$, as assumed in the theorem. The integer $K$ was chosen arbitrarily, so taking $K \rightarrow \infty$, we deduce that $S_{A_n}(t)z_n \underset{\mathcal{M}}\longtwoheadrightarrow S_{A_{\infty}}(t)z_{\infty}$ as $n\rightarrow \infty$. 

To get uniform convergence, we aim to make use of Lemma~\ref{thm1lem2}. For each $n\in \mathbb{N}_{\infty}$ we define
\begin{equation*}
    h_n(t) := d^{n,\infty}( S_{A_n}(t)z_n , S_{A_{\infty}}(t)z_{\infty} ).
\end{equation*} 
Then $(h_n)$ converges to $0$ pointwise on $[0,T]$ as $n \rightarrow \infty$. Let $t,\tau \in [0,T]$. We have that
\begin{align*}
|h_n(t)-h_n(\tau)| &= |d^{n,\infty}( S_{A_n}(t)z_n , S_{A_{\infty}}(t)z_{\infty} ) - d^{n,\infty}( S_{A_n}(\tau)z_n , S_{A_{\infty}}(t)z_{\infty} ) \\
& \hspace{0.6cm} + d^{n,\infty}( S_{A_n}(\tau)z_n , S_{A_{\infty}}(t)z_{\infty} ) - d^{n,\infty}( S_{A_n}(\tau)z_n , S_{A_{\infty}}(\tau)z_{\infty} )|\\
&\leq d^{n,n}( S_{A_n}(t)z_n , S_{A_n}(\tau)z_n )
+ d^{\infty,\infty}( S_{A_{\infty}}(t)z_{\infty} , S_{A_{\infty}}(\tau)z_{\infty} )\\
&\leq 2|t-\tau| \big( \inf \Vert A_n(z_n) \Vert_n + \inf \Vert A_{\infty}(z_{\infty}) \Vert_{\infty} \big)[e^{4\omega t}+e^{2\omega(t+\tau)}]\\
&\leq \frac{2|t-\tau|}{\lambda} \big( \Vert x_n \Vert_n + \Vert R_{\lambda}(A_n)x_n \Vert_n  \\
& \hspace{1.8cm} +\Vert x_{\infty} \Vert_{\infty} + \Vert R_{\lambda}(A_{\infty})x_{\infty} \Vert_{\infty} \big)[e^{4\omega t}+e^{2\omega(t+\tau)}].
\end{align*}
The first inequality holds by the (reverse) triangle inequality (Remark~\ref{rem:triangle}), the second by \eqref{continuity1}, and the third follows from \eqref{operatorBound}. By condition~(iv) in Definition~\ref{BanStack} (which we can use again by the assumption of pointwise convergence of $(R_{\lambda}(A_n))$ to $R_{\lambda}(A_\infty)$), we have that $(\Vert x_n \Vert_n+\Vert R_{\lambda}(A_n)x_n \Vert_n)$ converges to $\Vert x_\infty \Vert_\infty+\Vert R_{\lambda}(A_\infty)x_\infty \Vert_\infty$ and thus $L := \sup_{n\in \mathbb{N}_{\infty}} \Vert x_n \Vert_n +\Vert R_{\lambda}(A_n)x_n \Vert_n <+\infty$. Hence, for each $n\in  \mathbb{N}_{\infty}$, we have the bound
$$ |h_n(t) - h_n(\tau)| \leq \frac{8Le^{4\omega T}|t-\tau|}{\lambda}.$$

With the above control we apply Lemma~\ref{thm1lem2} to deduce that $(h_n(t))$ converges to $0$ uniformly for $t\in[0,T]$. Hence $(S_{A_n}(t)z_n)$ converges to $S_{A_{\infty}}(t)z_{\infty}$ uniformly in $t$ on $[0,T]$.

The next step is to obtain a uniform estimate over $[0,T]$ for the distance between $S_{A_n}(t)x_n$ and $S_{A_n}(t)z_n$. Indeed, for all $n\in \mathbb{N}_\infty$,
\begin{align}\label{eq:dnn_estimate}
d^{n,n}( S_{A_n}(t)x_n , S_{A_n}(t)z_n ) &\leq \Vert S_{A_n}(t)x_n - S_{A_n}(t)z_n\Vert_n
\leq e^{\omega T}\Vert x_n - z_n \Vert_n \notag \\
&= e^{\omega T}\Vert (I-R_{\lambda}(A_n)) x_n \Vert_n .
\end{align}
The first inequality follows from condition~(i) in Definition~\ref{BanStack} (see also Remark~\ref{rem:triangle}), the second from Corollary~\ref{SemigroupCon}, the third from our definition of $z_n$. We apply conditions~(iii) and~(iv) from Definition~\ref{BanStack} (again in combination with the assumption of pointwise convergence of $(R_\lambda(A_n))$) to find that 
$$ \Vert (I-R_{\lambda}(A_n))x_n \Vert_n \rightarrow  \Vert (I-R_{\lambda}(A_{\infty}))x_{\infty} \Vert_{\infty} \qquad \text{as } n\rightarrow \infty.$$
As previously observed, $\frac{1}{\lambda}(I-R_{\lambda}(A_{\infty}))x_{\infty} \in A_{\infty}(x_{\infty})$. It follows that 
\begin{equation}\label{eq:I-R_estimate}
\Vert (I-R_{\lambda}(A_{\infty}))x_{\infty} \Vert_{\infty} \leq \lambda \inf \Vert A_{\infty}(x_{\infty}) \Vert_{\infty}.
\end{equation}
Since $x_\infty \in D(A_\infty)$, $\inf \Vert A_{\infty}(x_{\infty}) \Vert_{\infty}< +\infty$ and thus there exists a constant $C\in (0, +\infty)$ independent of $\lambda$ such that
$$\limsup_{n \rightarrow \infty} \Vert (I-R_{\lambda}(A_n))x_n \Vert_n \leq \lambda C. $$
Hence
$$ \limsup_{n \rightarrow \infty} \sup_{t\in [0,T]}  d^{n,n}( S_{A_n}(t)x_n , S_{A_n}(t)z_n ) \leq e^{\omega T}\lambda C.$$
By the triangle inequality,
\begin{align*}
&d^{n,\infty}( S_{A_n}(t)x_n , S_{A_{\infty}}(t)x_{\infty} ) \\
&\leq d^{n,n}( S_{A_n}(t)x_n , S_{A_n}(t)z_n ) + d^{n,\infty}( S_{A_n}(t)z_n , S_{A_{\infty}}(t)z_{\infty} ) + d^{\infty,\infty}( S_{A_{\infty}}(t)z_{\infty} , S_{A_{\infty}}(t)x_{\infty} ).
\end{align*}
Using \eqref{eq:dnn_estimate} for $n=\infty$ together with \eqref{eq:I-R_estimate} and since $(d^{n,\infty}( S_{A_n}(t)z_n , S_{A_{\infty}}(t)z_{\infty}))$ converges to $0$ uniformly on $[0,T]$, we get
$$
\limsup_{n \rightarrow \infty} \sup_{t\in [0,T]} d^{n,\infty}( S_{A_n}(t)x_n , S_{A_{\infty}}(t)x_{\infty} ) \leq 2e^{\omega T}\lambda C.
$$
Since this holds for all $\lambda\in (0,\delta)$, we can take the limit $\lambda \downarrow 0$ in order to conclude uniform (in $t$) convergence of $S_{A_n}(t)x_n$ to $S_{A_\infty}(t)x_\infty$ on $[0,T]$ as required. This proves the required result under the condition that $x_\infty\in D(A_\infty)$.

Now we assume that $x_\infty \in \overline{D(A_\infty)}$. Given $\varepsilon>0$, let $y_{\infty}\in D(A_{\infty})$ be such that $\Vert x_{\infty}-y_{\infty}\Vert_{\infty} \leq \varepsilon$. Let $(y_n)_{n\in \mathbb{N}} \subset_n X_n$ be a sequence such that $y_n \underset{\mathcal{M}}\longtwoheadrightarrow y_{\infty}$; such a sequence exists by condition~(ii) in Definition~\ref{BanStack}. Since $y_\infty\in D(A_\infty)$, by the first part of this proof it follows that $S_{A_n}(t)y_n \underset{\mathcal{M}}\longtwoheadrightarrow S_{A_{\infty}}(t)y_{\infty} $ as $n\rightarrow \infty$ pointwise for all $t>0$ and uniformly in $t$ on $[0,T]$. By condition~(iii) in Definition~\ref{BanStack} we know that $x_n-y_n \underset{\mathcal{M}}\longtwoheadrightarrow x_{\infty}-y_{\infty}$, thus condition~(iv) implies
$\Vert x_n - y_n \Vert_n \rightarrow \Vert x_{\infty} - y_{\infty} \Vert_{\infty} $ as $n\rightarrow \infty$. Hence for $n$ sufficiently large we have $\Vert x_n - y_n \Vert_n  \leq 2\varepsilon$.

For all $t>0$,
\begin{align*}
d^{n,\infty}(  S_{A_n}(t)x_n ,  S_{A_{\infty}}(t)x_{\infty} )
&\leq d^{n,n}(  S_{A_n}(t)x_n  ,  S_{A_n}(t)y_n ) +
d^{n,\infty}(  S_{A_n}(t)y_n  , S_{A_{\infty}}(t)y_{\infty} ) \\
&\hspace{0.4cm} +d^{\infty,\infty}( S_{A_{\infty}}(t)y_{\infty} ,  S_{A_{\infty}}(t)x_{\infty} )\\
&\leq \Vert S_{A_n}(t)x_n - S_{A_n}(t)y_n \Vert_n + \Vert S_{A_{\infty}}(t)y_{\infty} - S_{A_{\infty}}(t)x_{\infty} \Vert_{\infty} \\ 
&\hspace{0.4cm} + d^{n,\infty}(  S_{A_n}(t)y_n  ,  S_{A_{\infty}}(t)y_{\infty} )\\
&\leq e^{\omega T}\Vert y_n - x_n \Vert_n + e^{\omega T}\Vert y_{\infty} - x_{\infty} \Vert_{\infty} + d^{n,\infty}(  S_{A_n}(t)y_n  ,  S_{A_{\infty}}(t)y_{\infty} )\\
&\leq 3 e^{\omega T} \varepsilon + d^{n,\infty}( S_{A_n}(t)y_n  ,  S_{A_{\infty}}(t)y_{\infty} ).
\end{align*}
The first inequality is the triangle inequality for $d$, the second follows from condition~(i) in Definition~\ref{BanStack}, and the third from Corollary~\ref{SemigroupCon}. It follows that
$$ \limsup_{n \rightarrow \infty} \sup_{t\in [0,T]} d^{n,\infty}( S_{A_n}(t)x_n  ,  S_{A_{\infty}}(t)x_{\infty} ) \leq 3 e^{\omega T} \varepsilon. $$
Taking the limit $\varepsilon\downarrow 0$ completes the proof.
\end{proof}
 
In the following remark we will discuss linear operators and $C_0$-semigroups on a Banach space $X$. A linear operator is a linear map $A:\operatorname{dom}(A)\rightarrow X$ where the domain $\operatorname{dom}(A)$ is a linear subspace of $X$. Given a linear operator $A$, $\hat{A}:=\{(x,y) \; | \; x\in \operatorname{dom}(A), \; y=A(x) \} \subset X \times X$ is an operator in the sense of Definition~\ref{def:operator}. The definitions of the resolvent operator of $A$ (Definition~\ref{def:operatorsoperations}) and semigroup generated by $A$ (Definition~\ref{def:semigroup}) translate similarly. Following \cite[Section 1.3]{Pazy1983}, we also define the resolvent set of $A$, denoted by $\rho(A)$, to be the set of all $\lambda \in \mathbb{R}$ such that $\lambda I- A:\operatorname{dom}(A)\rightarrow X$ is invertible, where $I$ is the identity map. Because $R_{\lambda^{-1}}(-\hat{A})=(\lambda I_X-\hat{A})^{-1}=\lambda^{-1}(I_X-\lambda^{-1}\hat{A})^{-1}$, for all $\lambda\in \rho(A)$, the resolvent operator $R_{\lambda^{-1}}(-\hat{A})$ of the linear operator $-\hat{A}$ is single-valued on its domain $X$ and linear.

A $C_0$-semigroup is a one-parameter collection $\{T(t)\}_{t\in[0,\infty)}$ of functions $T(t)$, such that, for all $t\in [0,\infty)$, $T(t)$ is a linear operator on $X$ with domain equal to $X$, the semigroup property is satisfied (i.e., for all $s,t \geq 0$, $T(t+s)=T(t)T(s)$) and, for all $x\in X$, $T(t)x\rightarrow x$ in the strong topology of $X$ as $t\downarrow 0$.

We say that a linear operator $A$ on a Banach space $X$ is within the set $G(M,\omega,X)$, for constants $M\geq 1$ and $\omega\in\mathbb{R}$, if the semigroup generated by $S_{-A}$ (as defined in Definition~\ref{def:semigroup}) has domain $X$, is a $C_0$-semigroup, and satisfies, for all $t\geq 0$, $\Vert S_{-A}(t)\Vert \leq Me^{\omega t}$. We point out the minus sign in the generator of the semigroup, which is due to a difference in convention between the current paper and \cite{Pazy1983}, whose results we wish to use. In particular, the underlying Cauchy equation for $S_{-A}$ is given by \cite[Equation 1.1]{Pazy1983} which has the opposite sign for the operator $A$ to our Cauchy equation in \eqref{AbCauchy}. After scaling the norm by $M^{-1}$, we can apply \cite[Corollary 3.8]{Pazy1983} to deduce that, if $A \in G(M,\omega,X)$, then the operator $-\hat{A}$ is $\omega$-accretive and satisfies the range condition \eqref{rangeCon}. Indeed, $\omega$-accretivity follows by a similar argument to that in Proposition~\ref{prop:accretivecontraction}, using conclusion (ii) from \cite[Corollary 3.8]{Pazy1983}; moreover, if $\lambda^{-1}$ is in the resolvent set $\rho(A)$, then $I-\lambda^{-1} A$ is surjective on $X$ and thus, again by conclusion (ii) from \cite[Corollary 3.8]{Pazy1983}, the range condition~\eqref{rangeCon} is satisfied for $-\hat{A}$, specifically for $\lambda \in (0,\omega^{-1})$.

\begin{remark}
\label{TrotterKato}
In this remark, whenever $B:X\rightarrow X$ is a linear map on a Banach space $X$ with norm $\|\cdot\|$, we write $\Vert B \Vert:= \sup_{x\neq 0} \Vert B(x)\Vert/\Vert x\Vert$ for the standard operator norm of $B$. 

There are a number of similarities between Theorem~\ref{theorem1} and the Trotter--Kato theorem regarding approximations of $C_0$-semigroups from \cite[Theorem 2.1]{Ito1998} (see also \cite{Kato95,Trotter58}) which we will now detail. 

To explain \cite[Theorem 2.1]{Ito1998} we require a sequence $(X_n)_{n\in \mathbb{N}}$ of Banach spaces and a Banach space $Z$ together with linear maps $E_n: X_n \rightarrow Z$ and $P_n:Z\rightarrow X_n$ for each $n\in \mathbb{N}$, that satisfy the following properties:
\begin{enumerate}[(a)]
    \item there exist $M_1>0$ and $M_2>0$ such that, for all $n\in \mathbb{N}$, $\Vert P_n \Vert \leq M_1$ and $\Vert E_n \Vert \leq M_2$; 
    \item for all $x\in Z$, $\Vert E_nP_n x - x\Vert \rightarrow 0$ as $n\rightarrow \infty$; 
    \item for all $n\in \mathbb{N}$, $P_nE_n=I_{X_n}$.
\end{enumerate}
Additionally we assume $X_\infty$ is a closed linear subspace of $Z$, $M\geq 1$ and $\omega\in \mathbb{R}$ are constants, and, for all $n\in\mathbb{N}_\infty$, $A_n \in G(M,\omega,X_n)$ is a linear operator.

Then \cite[Theorem 2.1]{Ito1998} states that for all $x\in X_\infty$, $(E_nS_{A_n}(t)P_n x)_{n\in\mathbb{N}}$ converges uniformly (in $t$) on bounded intervals to $S_{A_\infty}(t)x$ if and only if there exists a $\lambda \in \bigcap_{n\in\mathbb{N}_\infty} \rho(A_n) \subset \mathbb{R}$, such that, for all $x\in X_\infty$, $(E_nR_{\lambda}(A_n)P_nx)_{n\in\mathbb{N}}$ converges to $R_{\lambda}(A_\infty)x$. 

To draw a connection between \cite[Theorem 2.1]{Ito1998} and our result we define, for all $n\in \mathbb{N}$, $Z_n:=E_nX_n$ and $Z_\infty:=X_\infty$. Let $\iota_n: Z_n\rightarrow Z$ be the inclusion map. Then $((Z_n,\iota_n)_{n \in \mathbb{N}_{\infty}},Z)$ forms a Banach stacking (as in part~1 of Examples~\ref{BSexmpl}); property (b) is necessary for the required density condition in part (ii) of Definition~\ref{BanStack}  to hold. We observe that, for all $n\in \mathbb{N}$ and all $x\in Z_n$, $E_n P_n x = x$; this holds since, if $x\in Z_n$, then there exists a $y\in X_n$ such that $x=E_ny$ and thus $E_n P_n x = E_n P_n E_n y = E_n y = x$. Thus, restricted to $Z_n$, the operators $E_n$ and $P_n$ are each other's inverses.

Next we define, for all $n\in\mathbb{N}$,  $\widetilde{A_n}:=E_nA_nP_n$, which is a linear operator on $Z_n$. As a consequence of property~(c), for all $\lambda\in\mathbb{R}$ and all $x\in Z_n$, $R_{\lambda}(\widetilde{A_n})=E_nR_\lambda(A_n)P_n$. Thus, by the Crandall--Liggett formula in Theorem~\ref{CranLig} and the condition in property~(a) on the operators $E_n$, $S_{\widetilde{A_n}}(\cdot)=E_nS_{A_n}(\cdot)P_n$. Hence \cite[Theorem 2.1]{Ito1998} equivalently states that, for all $x\in X_\infty$, $S_{\widetilde{A_n}}(t)x$ converges uniformly (in $t$) on bounded intervals to $S_{A_\infty}(t)x$ if and only if there exists a $\lambda \in \rho(A_\infty) \bigcap_{n\in\mathbb{N}} \rho(\widetilde{A_n}) \subset \mathbb{R}$ such that, for all $x\in Z_\infty$, $(R_{\lambda}(\widetilde{A_n})x)_{n\in\mathbb{N}}$ converges to $R_{\lambda}(A_\infty)x$. This closely resembles the statement of Theorem~\ref{theorem1}.

Thus the setting of \cite[Theorem 2.1]{Ito1998} is a particular of example of that of Theorem~\ref{theorem1}. The latter covers possibly nonlinear and multivalued operators, whereas \cite[Theorem 2.1]{Ito1998} applies only to linear, and thus single-valued, operators. For Theorem~\ref{theorem1} we require the operators $A_n$ to be $\omega$-accretive. If $\hat A$, corresponding to a linear operator $A$, is $\omega$-accretive, satisfies the range condition in \eqref{rangeCon}, and has a domain that is dense in $X$, then $-A$ is in fact in $G(1,\omega,X)$ (see Corollary~\ref{SemigroupCon}). As noted just before this remark, if $-\hat{A}$ is in $G(1,\omega,X)$, then $\hat{A}$ is $\omega$-accretive and satisfies the range condition \eqref{rangeCon}. Thus, for linear operators, $\omega$-accretivity and the range condition together are equivalent to membership of $G(1,\omega,X)$, up to a sign change in the operator; since these concepts are also applicable to nonlinear operators, we can view them as the nonlinear analogues to membership of $G(1,\omega,X)$.

The restriction to linear operators in \cite[Theorem 2.1]{Ito1998} allows for a stronger conclusion than in Theorem~\ref{theorem1}. In particular, convergence of the resolvents for one admissible choice of $\lambda$ suffices to conclude uniform convergence of the semigroups. Furthermore, \cite[Theorem 2.1]{Ito1998} also contains the converse implication that concludes pointwise convergence of the resolvents from uniform convergence of the semigroups. In the current work, we do not explore this implication for semigroups generated by nonlinear operators and we are unsure if it holds in our setting. 

\end{remark}

\begin{remark} Another possible avenue for further research is to extend Theorem~\ref{theorem1} to cover convergence of mild solutions of a Cauchy problem with forcing,
$$\begin{cases}
u'+Au \ni f \\
u(0)=x. \\
\end{cases}  $$
To guarantee the existence of mild solutions we have to trade the assumption of the range condition~\eqref{rangeCon} for $\omega$-m-accretivity (see \cite[Theorem A.29]{andreu2010nonlocal}). For nontrivial $f$ we lose the semigroup structure of the solutions.

In the setting of a Banach stacking we would have a converging sequence of forcing functions $f_n\in L^1(0,T;X_n)$. Which notion of convergence is suitable to conclude convergence of the mild solutions is a question for future research.  
\end{remark}

\section{Convergence of semigroups and gradient flows}\label{sec:convresolvents}

In this section we prove convergence results for gradient flows and nonlinear semigroups in the setting of a Banach stacking, using the machinery we have built so far. Each of the main results uses our generalization of the Brezis--Pazy theorem given in theorem~\ref{theorem1}. In particular, this theorem obtains uniform convergence of nonlinear semigroups from pointwise convergence of the resolvents of the semigroup-generating operators. The three subsections of Section~\ref{sec:convresolvents} illustrate three different situations in which we can find pointwise convergence of resolvents, and thus convergence of the corresponding semigroups:
\begin{enumerate}[(1)]
    \item pointwise convergence of resolvents over a Banach stacking via direct inspection of the operators;
    \item pointwise convergence of resolvents of subdifferentials of a sequence of $\Gamma$-converging functions on a Banach stacking of Hilbert spaces;
    \item pointwise convergence of resolvents of subdifferentials of a sequence of $\Gamma$-converging functionals on the $TL^p(\Omega)$ Banach stacking of Proposition~\ref{tlpBanach}.
\end{enumerate}
In the second and third cases above, the corresponding semigroups are in fact gradient flows. In the first case we consider more general semigroups, but the conditions on the generating operators in that case are not always straightforward to check in specific situations.

\subsection{Direct convergence of accretive operators}

In order to deduce pointwise convergence of resolvents it stands to reason that we need only pointwise convergence of the operators. However, one difficulty is that an operator on $X$ can be any subset of $X\times X$ and thus need not correspond to a single-valued function. We clarify what we mean by pointwise convergence in this case below.

\begin{proposition}
\label{OpConvergence}
Let $\big((X_n, \xi_n)_{n\in \mathbb{N}_\infty}, \mathcal{M}\big)$ be a Banach stacking. Suppose $(A_n)_{n\in \mathbb{N}_\infty}$ is a sequence of $\omega_n$ -accretive operators $A_n$ on $X_n$ that each satisfy the range condition~\eqref{rangeCon}, with the the sequence $(\omega_n)_{n\in \mathbb{N}_{\infty}}$ bounded above by some constant $\omega\in \mathbb{R}$.

Suppose that, for all $(x_\infty,y_\infty)\in A_\infty$, there exists a sequence $\big((x_n,y_n)\big)_{n\in \mathbb{N}} \subset_n A_n$ such that $x_n\underset{\mathcal{M}}\longtwoheadrightarrow x_{\infty}$ and $y_n\underset{\mathcal{M}}\longtwoheadrightarrow y_{\infty}$ as $n\to\infty$. Then, for all $\lambda \in \mathfrak{I}_{\omega}$ (as defined in \eqref{OmegaInterval}) and all sequences $(z_n)_{n\in\mathbb{N}_\infty} \subset_n \overline{D(A_n)}$ for which $z_n \underset{\mathcal{M}}\longtwoheadrightarrow z_{\infty}$ as $n\to\infty$, we have $R_\lambda(A_n)z_n\underset{\mathcal{M}}\longtwoheadrightarrow R_\lambda(A_\infty)z_{\infty}$ as $n\to\infty$. Moreover, for all $t\geq 0$, $S_{A_n}(t)z_n \underset{\mathcal{M}}\longtwoheadrightarrow S_{A_{\infty}}(t)z_{\infty}$ as $n\rightarrow \infty$ and, for all $T\in [0,+\infty)$, the convergence is uniform in $t$ on the interval $[0,T]$.
\end{proposition}
\begin{proof}
Consider a convergent sequence $(z_n)_{n\in \mathbb{N}_\infty}\subset_n X_n$ in the Banach stacking sense, which exists by part (ii) of Definition~\ref{BanStack}, and let $\lambda \in \mathfrak{I}_{\omega}$. Define $x_{\infty} := R_{\lambda}(A_{\infty})z_{\infty}$ and let $y_{\infty} \in A_{\infty}(x_{\infty})$ be such that $x_{\infty}+ \lambda y_{\infty} = z_{\infty}$. Then, by the premise of the proposition, there exist sequences $(x_n)_{n\in \mathbb{N}} \subset_n X_n$ and $(y_n)_{n\in \mathbb{N}} \subset_n A_n(x_n)$ such that 
$x_n \underset{\mathcal{M}}\longtwoheadrightarrow x_{\infty}$ and $y_n \underset{\mathcal{M}}\longtwoheadrightarrow y_{\infty}$ as $n \to \infty$. 

Define $\hat{z}_n := x_n + \lambda y_n$. By condition~(iii) in Definition~\ref{BanStack} we have that $\hat{z}_n \underset{\mathcal{M}}\longtwoheadrightarrow z_{\infty}$ as $n\rightarrow \infty$. Moreover, we have
$$ R_{\lambda}(A_n)\hat{z}_n = x_n \underset{\mathcal{M}}\longtwoheadrightarrow x_{\infty} = R_{\lambda}(A_{\infty}) z_{\infty} \quad \text{as } n \to \infty$$
and
\begin{align*}
    d^{n,\infty}(  R_{\lambda}(A_n)z_n , R_{\lambda}(A_{\infty})z_{\infty} ) &\leq d^{n,n}( R_{\lambda}(A_n)z_n  , R_{\lambda}(A_n)\hat{z}_n ) 
    + d^{n,\infty}(  R_{\lambda}(A_n)\hat{z}_n  , R_{\lambda}(A_{\infty})z_{\infty} ) \\
    &\leq \Vert R_{\lambda}(A_n)z_n - R_{\lambda}(A_n)\hat{z}_n \Vert_n
    + d^{n,\infty}( R_{\lambda}(A_n)\hat{z}_n ,  R_{\lambda}(A_{\infty})z_{\infty} ) \\
    &\leq \frac{1}{1-\lambda \omega}\Vert z_n - \hat{z}_n \Vert_n
    + d^{n,\infty}( R_{\lambda}(A_n)\hat{z}_n  , R_{\lambda}(A_{\infty})z_{\infty} ).
\end{align*}
The first inequality holds by the triangle inequality, the second by condition~(i) of Definition~\ref{BanStack}, and the third by Proposition~\ref{Omega1} (see also Remark~\ref{accremark}) in combination with the upper bound on $(\omega_n)_{n\in \mathbb{N}_\infty}$.

By conditions~(iii) and~(iv) of Definition~\ref{BanStack} we have that $\Vert z_n - \hat{z}_n \Vert_n \rightarrow 0$ as $n\rightarrow \infty$ and thus
$$ \lim_{n\rightarrow \infty} d^{n,\infty}(R_{\lambda}(A_n)z_n ,  R_{\lambda}(A_{\infty})z_{\infty} )=0. $$
Convergence of the semigroups, pointwise for all $t\geq 0$ and uniformly on $[0,T]$, then follows directly from Theorem~\ref{theorem1}.
\end{proof}

The approach detailed in this subsection is the most straightforward one we present. The later results in Theorems~\ref{Hilthm} and~\ref{thm:P0theorem1} use $\Gamma$-convergence and a compactness assumption to deduce convergence of the resolvents. In Proposition~\ref{OpConvergence} no kind of compactness is assumed; instead we assume that for all $(x_\infty,y_\infty)$ a sequence converging to it exists, which is sometimes easier to check. In many applications, such as when the semigroups describe gradient flows, the operators generating the semigroups may be difficult to analyse directly, that is why the variational approach in the rest of Section~\ref{sec:convresolvents} is also of interest.

\subsection{Banach stacking of Hilbert spaces}\label{sec:BanachStackHilbert}

In this section $(H_n)_{n\in\mathbb{N}_\infty}$ is a sequence of Hilbert spaces and $\big((H_n, \xi_n)_{n\in \mathbb{N}_\infty}, \mathcal{M}\big)$ is a Banach stacking. We consider a sequence $(\Phi_n)_{n\in\mathbb{N}_\infty}$ of functions $\Phi_n: H_n \rightarrow (-\infty,+\infty]$ that satisfy the following assumptions. We define $\Gamma$-convergence and equicoercivity of a sequence of functions over a Banach stacking in Definition~\ref{BSdefs}.

\begin{assumptions}
\label{assumHilbert}
\begin{enumerate}
\item The sequence $(\Phi_n)_{n\in \mathbb{N}}$ $\Gamma$-converges to $\Phi_{\infty}$ over the Banach stacking $\big((H_n, \xi_n)_{n\in \mathbb{N}_\infty}, \mathcal{M}\big)$, as $n \rightarrow \infty$.
\item For all sequences $(x_n)_{n\in \mathbb{N}} \subset_n H_n$ for which the sequences $\left(\Vert x_n \Vert_n\right)_{n\in \mathbb{N}}$ and  $\left(\Phi_n(x_n)\right)_{n\in \mathbb{N}}$ are bounded, there exists an $x_\infty \in H_\infty$ and a subsequence $(x_{n_k})_{k\in\mathbb{N}}$ of $(x_n)_{n\in \mathbb{N}}$ such that $x_{n_k}\underset{\mathcal{M}}\longtwoheadrightarrow x_\infty$ as $k\rightarrow \infty$.
\item There exists a $\lambda\in \mathbb{R}$ such that, for all $n\in \mathbb{N}_{\infty}$, $\Phi_n$ is proper, $\lambda$-convex, and lower semicontinuous.
\end{enumerate}
\end{assumptions}

Under these assumptions, we prove the following theorem.
\begin{theorem}
\label{Hilthm}
Let $\big((H_n, \xi_n)_{n\in \mathbb{N}_\infty}, \mathcal{M}\big)$ be a Banach stacking as above and $(\Phi_n)_{n\in\mathbb{N}_\infty}$ a sequence of functions $\Phi_n: H_n \rightarrow (-\infty,+\infty]$ that satisfies parts~1--3 of Assumptions~\ref{assumHilbert}. Let $(x_n)_{n\in \mathbb{N}_\infty} \subset_n \overline{\operatorname{dom}(\Phi_n)}$ be a convergent sequence. For all $n\in \mathbb{N}_{\infty}$, let $u_n:[0,+\infty)\rightarrow H_n$ be the gradient flow of $\Phi_n$ starting at $x_n$ (as in Definition \ref{gradflowdef}). Then, for all $T \in [0,+\infty)$, $(u_n(t))_{n\in \mathbb{N}}$ converges uniformly to $u_{\infty}(t)$ for $t\in[0,T]$ over the Banach stacking $\big((H_n, \xi_n)_{n\in \mathbb{N}_\infty}, \mathcal{M}\big)$. Additionally, for all $t>0$,
$$ \Phi_n(u_n(t))\rightarrow \Phi_{\infty}(u_{\infty}(t)) \qquad \text{as }  n\rightarrow \infty.$$
\end{theorem}

\begin{proof}
Let $\lambda\in \mathbb{R}$ be such that, for all $n\in \mathbb{N}_\infty$, $\Phi_n$ is $\lambda$-convex. Such a $\lambda$ exists by part~3 of Assumptions~\ref{assumHilbert}.

By Lemma~\ref{lem:gradflowsemigroup} the gradient flow $u_n$ coincides with the semigroup $S_{\partial^{\lambda}\Phi_n}$ acting on $x_n$. This allows us to prove uniform convergence of the gradient flows by making use of Theorem~\ref{theorem1}. To do so, we will show that, for $\gamma\in \mathfrak{I}_{-\lambda}$, the resolvents converge pointwise over $\big((H_n, \xi_n)_{n\in \mathbb{N}_\infty}, \mathcal{M}\big)$. (We recall that $\mathfrak{I}_{-\lambda}$ is defined in \eqref{OmegaInterval}.) To prove convergence of $(\Phi_n(u_n(t)))$ we will use properties of $\Gamma$-convergence and the bound in terms of a Moreau envelope in Theorem~\ref{GradFlowTheorem}.

Let $(h_n)_{n\in \mathbb{N}_\infty} \subset_n H_n$ be convergent. Let $\gamma \in \mathfrak{I}_{-\lambda}$ and define for each $n\in \mathbb{N}_{\infty}$ the function $\mathfrak{H}_n[\gamma,h_n]:H_n \rightarrow [0,+\infty]$ by
 $$ \mathfrak{H}_n[\gamma,h_n](v) := \frac{1}{2\gamma}\Vert v-h_n \Vert_n^2 + \Phi_n(v).$$
By Lemma~\ref{ProxResolve} (see also Definition~\ref{MoreauProx}), for each $n\in \mathbb{N}_{\infty}$,
 \begin{equation}\label{eq:Risargmin}
 R_{\gamma}(\partial^{\lambda} \Phi_n)h_n = \underset{v\in H_n}{\argmin} \, \mathfrak{H}_n[\gamma,h_n](v).
 \end{equation}
 In particular we note that the minimizer on the right-hand side is unique. 
 Moreover by the definition of the Moreau envelope in Definition~\ref{MoreauProx} we have
 \begin{equation}\label{eq:envelopeisinf}
 [\Phi_n]^{\gamma}(h_n) = \inf_{v\in H_n} \mathfrak{H}_n[\gamma,h_n](v).
 \end{equation}

We will show that the sequence $\left(\mathfrak{H}_n[\gamma,h_n]\right)_{n\in \mathbb{N}}$ $\Gamma$-converges to $\mathfrak{H}_{\infty}[\gamma,h_{\infty}]$ over the Banach stacking $\big((H_n,\xi_n)_{n\in\mathbb{N}_\infty,\mathcal{M}}\big)$, in the sense of Definition~\ref{BSdefs}. Indeed let $(v_n)_{n\in \mathbb{N}_{\infty}} \subset_n H_n$ be convergent. By properties~(iii) and~(iv) in Definition~\ref{BanStack} we have that $\Vert v_n - h_n \Vert_n^2 \rightarrow  \Vert v_{\infty} - h_{\infty} \Vert_{\infty}^2$ as $n\rightarrow \infty$. Thus
\begin{align*}
\liminf_{n \rightarrow \infty} \mathfrak{H}_n[\gamma,h_n](v_n) &= \liminf_{n \rightarrow \infty} \Phi_n(v_n) + \liminf_{n \rightarrow \infty} \frac{1}{2\gamma} \Vert v_n - h_n \Vert_n^2 \\
&= \liminf_{n \rightarrow \infty} \Phi_n(v_n) +  \frac{1}{2\gamma} \Vert v_{\infty} - h_{\infty} \Vert_{\infty}^2 \\
&\geq \Phi_{\infty}(v_{\infty})  +  \frac{1}{2\gamma} \Vert v_{\infty} - h_{\infty} \Vert_{\infty}^2 \\
& = \mathfrak{H}_{\infty}[\gamma,h_{\infty}](v_{\infty}).
\end{align*}
The first and second equalities hold by the convergence of $(\|v_n-h_n\|_n^2)$ and the inequality holds by part~1 of Assumptions~\ref{assumHilbert}. To obtain a recovery sequence, let $(v_n)_{n\in \mathbb{N}_{\infty}} \subset_n H_n$ be a convergent sequence such that $\limsup_{n \rightarrow \infty}\Phi_n(v_n) \leq \Phi_{\infty}(v_{\infty})$; such a sequence exists by part~1 of Assumptions~\ref{assumHilbert}. Similar to above, we find that 
$$\limsup_{n\rightarrow \infty} \mathfrak{H}_n[\gamma,h_n](v_n) = \limsup_{n \rightarrow \infty} \Phi_n(v_n) +  \frac{1}{2\gamma} \Vert v_{\infty} - h_{\infty} \Vert_{\infty}^2 \leq  \mathfrak{H}_{\infty}[\gamma,h_{\infty}](v_{\infty}).$$
As required, we conclude that $\left(\mathfrak{H}_n[\gamma,h_n]\right)_{n\in \mathbb{N}}$ $\Gamma$-converges to $\mathfrak{H}_{\infty}[\gamma,h_{\infty}]$ as $n\rightarrow \infty$. 

Next we wish to show that $\left(\mathfrak{H}_n[\gamma,h_n]\right)_{n\in \mathbb{N}_{\infty}}$ is equicoercive over the Banach stacking $\big((H_n,\xi_n)_{n\in\mathbb{N}_\infty,\mathcal{M}}\big)$, in the sense of Definition~\ref{BSdefs}. Suppose we have a sequence $(v_n)_{n\in \mathbb{N}_{\infty}} \subset_n H_n$ and a constant $C>0$ such that, for all $n\in \mathbb{N}_{\infty}$, $\mathfrak{H}_n[\gamma,h_n](v_n)\leq C$. We observe that, for all $n\in \mathbb{N}_{\infty}$, $\Phi_n(v_n)\leq C $. 

Choose $\delta>0$ such that $\delta+\gamma \in \mathfrak{I}_{-\lambda}$, which is possible by openness of $\mathfrak{I}_{-\lambda}$. We recall that, by assumption, for all $n\in \mathbb{N}_{\infty}$, $\Phi_n$ is $\lambda$-convex and thus $\mathfrak{H}_n[\gamma+\delta,h_n]$ is $(\frac{1}{\gamma+\delta} + \lambda)$-convex, with $\frac{1}{\gamma+\delta} + \lambda>0$. Since also, as proved above, $(\mathfrak{H}_n[\gamma+\delta,h_n])_{n\in \mathbb{N}}$ $\Gamma$-converges to $\mathfrak{H}_{\infty}[\gamma+\delta,h_{\infty}]$, we can apply Lemma~\ref{Lem:lambdagammabound} to deduce that there exists an $L>0$ such that, for all $n\in \mathbb{N}_{\infty}$ and for all $w\in H_n$, $$\mathfrak{H}_n[\gamma+\delta,h_n](w) \geq -L.$$ 
It follows that 
\begin{align*}
\mathfrak{H}_n[\gamma,h_n](v_n) &= \left( \frac{\delta}{2\gamma(\gamma+\delta)} +\frac{1}{2(\gamma+\delta)} \right) \Vert v_n-h_n \Vert_n^2 + \Phi_n(v_n) \\
&= \frac{\delta}{2\gamma(\gamma+\delta)}  \Vert v_n-h_n \Vert_n^2 + \mathfrak{H}_n[\gamma,h_n](v_n) 
\geq \frac{\delta}{2\gamma(\gamma+\delta)}  \Vert v_n-h_n \Vert_n^2 - L
\end{align*}
and thus 
$$ \Vert v_n-h_n \Vert_n^2 \leq \frac{2\gamma(\gamma+\delta)}{\delta}(L+C).$$
Given that $(h_n)_{n\in \mathbb{N}_\infty}$ is convergent, we note that there exists a constant $K>0$ such that, for all $n\in \mathbb{N}_\infty$,
\begin{align*}
    \Vert v_n \Vert_n &\leq \Vert h_n \Vert_n + \Vert h_n - v \Vert_n
    \leq K + \sqrt{\frac{2\gamma}{\delta}(\gamma+\delta)(L+C)} 
    <+\infty.
\end{align*}
Thus $(\Vert v_n \Vert_n)_{n\in \mathbb{N}}$ and $(\Phi_n(v_n))_{n\in \mathbb{N}}$ are bounded. Part~2 of Assumptions~\ref{assumHilbert} then implies that there exists an $x_\infty\in X_\infty$ and a subsequence $(x_{n_k})_{k\in\mathbb{N}}$ such that $x_{n_k}\underset{\mathcal{M}}\longtwoheadrightarrow x_{\infty}$ as $k \rightarrow \infty$. Hence $\left(\mathfrak{H}_n[\gamma,h_n]\right)_{n\in \mathbb{N}}$ is equicoercive. Therefore, using Proposition~\ref{prop:gammaconBS}, we find from \eqref{eq:Risargmin} that $R_{\gamma}(\partial^{\lambda} \Phi_n)h_n \underset{\mathcal{M}}\longtwoheadrightarrow R_{\gamma}(\partial^{\lambda} \Phi_{\infty})h_{\infty} $ as $n\rightarrow \infty$. For later use we also note that  Proposition~\ref{prop:gammaconBS} and \eqref{eq:envelopeisinf} imply that 
\begin{equation}\label{eq:convergentenvelopes}
[\Phi_n]^{\gamma}(h_n)\rightarrow [\Phi_{\infty}]^{\gamma}(h_{\infty}) \qquad \text{as } n\rightarrow \infty.
\end{equation}

By property~(iii) in Assumptions~\ref{assumHilbert} and Theorem~\ref{thm:partiallambdamaccretive} we have that, for all $n\in \mathbb{N}_\infty$, $\partial^{\lambda} \Phi_n$ is ($-\lambda$)-m-accretive and thus, by Remark~\ref{rem:rangecondition}, satisfies the range condition~\eqref{rangeCon}. We now apply Theorem~\ref{theorem1} and Lemma~\ref{lem:gradflowsemigroup} to deduce for all $t\geq 0$, as $n\rightarrow \infty$,
 $$ u_n(t) = S_{\partial^{\lambda} \Phi_n}(t)x_n \underset{\mathcal{M}}\longtwoheadrightarrow S_{\partial^{\lambda} \Phi_{\infty}}(t)x_{\infty} =u_{\infty}(t).$$
 Moreover, given $T\in [0,+\infty)$, this convergence is uniform on $[0,T]$.

 For a proof of the final statement of the theorem, define, for all $n\in \mathbb{N}_\infty$, the function $f_n:[0,+\infty) \rightarrow (-\infty,+\infty]$ by
 $f_n(t) := \Phi_n\big( u_n(t) \big).$
 By Remark~\ref{rem:gradientflow} (see also Remark~\ref{rem:maxslope}), for all $n\in \mathbb{N}_{\infty}$, $f_n$ is non-increasing. Given the previously established convergence of $(u_n(t))_{n\in \mathbb{N}}$ plus the $\Gamma$-convergence of $(\Phi_n)_{n\in \mathbb{N}}$ we have, for all $t\geq 0$, 
 \begin{equation}\label{eq:liminffn}
 \liminf_{n\rightarrow \infty} f_n(t) \geq f_{\infty}(t).
 \end{equation}
Then, since $\gamma\in \mathfrak{I}_{-\lambda}$ and $t\geq 0$ we have
\begin{align*}
f_n(t+\gamma) &= \Phi_n \big(u_n(t+ \gamma) \big)
\leq [\Phi_n]^{\kappa(\gamma,\lambda)}\big( u_n(t) \big) 
\rightarrow [\Phi_{\infty}]^{\kappa(\gamma,\lambda)}\big( u_{\infty}(t) \big) 
\leq f_{\infty}(t).
\end{align*}
The first inequality holds by Theorem~\ref{GradFlowTheorem} (which also contains the definition of $\kappa(\gamma,\lambda)$), where we used the uniqueness and semigroup property of the gradient flow (see Remark~\ref{rem:gradientflow}) to interpret $u_n(t+\gamma)$ as the value at time $\gamma$ of the gradient flow with initial value $u_n(t)$. The convergence of the Moreau envelopes follows from \eqref{eq:convergentenvelopes} (which holds for all $\gamma \in \mathfrak{I}_{-\lambda}$), since $\kappa(\gamma,\lambda)\in \mathfrak{I}_{-\lambda}$ by Remark~\ref{rem:kappa}. The second inequality holds by Remark~\ref{rem:envelopesupport}.
  
 In particular we have, for all $\gamma\in \mathfrak{I}_{-\lambda}$ and for all $t\geq 0$, 
 \begin{equation}\label{eq:limsupfn}
 \limsup_{n \rightarrow \infty} f_n(t+\gamma) \leq f_{\infty}(t).
 \end{equation}
 Define $g:[0,+\infty)\rightarrow [0,+\infty]$ by
 $ g(t):= \limsup_{n\rightarrow \infty} f_n(t).$
 Since, for all $n\in \mathbb{N}$, $f_n$ is non-increasing, also $g$ is non-increasing and hence $g$ is continuous a.e. on $[0,+\infty)$. Assume $s$ is a continuity point of $g$ (i.e., $g$ is continuous at $s$), then
 $$ g(s) = \lim_{\gamma \downarrow 0} g(s+\gamma) \leq f_{\infty}(s).$$ 
 We are allowed to take the limit $\gamma\downarrow 0$, because \eqref{eq:limsupfn} holds for all $\gamma\in \mathfrak{I}_{-\lambda}$. It follows that 
 $$ f_{\infty}(s) \leq \liminf_{n\rightarrow \infty} f_n(s) \leq \limsup_{n\rightarrow \infty} f_n(s) \leq f_{\infty}(s)$$
and so $f_n(s) \rightarrow f_\infty(s)$ as $n\rightarrow \infty$. In particular $g(s)=f_{\infty}(s)$.

Now let $t>0$, let $(a_k)_{k\in \mathbb{N}}$ and $(b_k)_{k\in\mathbb{N}}$ be two sequences with $0\leq a_k < t < b_k$ such that, for all $k\in \mathbb{N}$, $a_k$ and $b_k$ are continuity points of $g$, $(a_k)$ and $(b_k)$ both converge to $t$, $(a_k)$ is strictly increasing, and $(b_k)$ is strictly decreasing. We note that such sequences exist, since the set of continuity points of $g$, being the complement of a null set in $[0,+\infty)$, is dense in $[0,+\infty)$. As explained in Remark~\ref{rem:maxslope}, the function $f_{\infty}$ is continuous on $(0,+\infty)$ and hence
\[
f_{\infty}(t) = \lim_{k \rightarrow \infty} f_\infty(a_k) = \lim_{k \rightarrow \infty} g(a_k) 
\geq g(t).
\]
Similarly it follows that
\[
f_{\infty}(t) = \lim_{k \rightarrow \infty} f_\infty(b_k) = \lim_{k \rightarrow \infty} g(b_k) 
\leq g(t).
\]

Thus, for all $t>0$, $f_{\infty}(t)=g(t)=\limsup_{n\rightarrow \infty}f_n(t)$. Together with the inequality for the limit inferior in \eqref{eq:liminffn} this implies that, for all $t>0$, $f_n(t)\rightarrow f_{\infty}(t)$ as $n\rightarrow \infty$. This concludes our proof. 
\end{proof}

We expect $TL^2(\Omega)$ (see part~6 of Examples~\ref{BSexmpl}) to be a useful example of a Banach stacking of Hilbert spaces in which Theorem~\ref{Hilthm} can be applied.

\subsection{Banach stacking of \texorpdfstring{$L^p$}{Lp} spaces: \texorpdfstring{$TL^p(\Omega)$}{TLp(Omega)}}\label{sec:TLp}

For this subsection (Section~\ref{sec:TLp}) let $\Omega \subset \mathbb{R}^d$ be an open subset and $(\mu_n)_{n\in \mathbb{N}_{\infty}}$ a sequence of Borel probability measures. We consider a sequence $(\Phi_n)_{n\in \mathbb{N}_{\infty}}$  of functionals $\Phi_n: \mathfrak{F}(\Omega;\mu_n) \rightarrow (-\infty,+\infty]$. We recall that the space of measurable functions $\mathfrak{F}(\Omega;\mu_n)$ is defined at the start of Section~\ref{sec:subdifferentials}. 

As in Section~\ref{sec:BanachStackHilbert}, again we are interested in the convergence of gradient flows over Hilbert spaces, in particular the convergence of the gradient flows of $(\Phi_n|_{L^2(\Omega;\mu_n)})_{n\in \mathbb{N}_{\infty}}$ over the Hilbert spaces $L^2(\Omega;\mu_n)$. We will be working, however, with different assumptions than in Section~\ref{sec:BanachStackHilbert}. In particular we will assume $\Gamma$-convergence and equicoercivity in the topology of $TL^p(\Omega)$, which, if $p\neq 2$, is not a Banach stacking of Hilbert spaces. 

Moreover, in the case when $p<2$ the assumption of equicoercivity which we will make in Assumptions~\ref{assump} is strictly weaker than the analogous assumption of relative compactness in Assumptions~\ref{assumHilbert}. The relevance of such an assumption is explained in Section~\ref{sec:convgradflowintro}. 

To deal with this technical difficulty we will assume (in Assumptions~\ref{assump}) that the functionals $\Phi_n$ are $P_0$-convex (see Definition~\ref{def:Pconvex}), rather than $\lambda$-convex. Henceforth for this subsection, fix $p\in [1,+\infty)$ and, for all $n\in \mathbb{N}_{\infty}$, define $X_n:=L^p(\Omega;\mu_n)$ with $\|\cdot\|_n$ denoting the corresponding $L^p(\Omega;\mu_n)$ norm. In this subsection, for any subset $A\subset X_n$, $\overline{A}$ always denotes the closure of $A$ with respect to the topology on $X_n$. We also assume that $(\mu_n)_{n\in \mathbb{N}}$ converges to $\mu_{\infty}$ as $n\rightarrow \infty$ in the $\max(2,p)$-Wasserstein metric (see Definition~\ref{def:Wasserstein}).

\begin{definition}
\label{movementeq}
Let $n\in \mathbb{N}_\infty$, $h\in X_n$, and $\lambda >0$. We define the functional $\mathfrak{M}[\lambda,h;\Phi_n]: \mathfrak{F}(\Omega;\mu_n) \rightarrow (-\infty,+\infty]$ by
$$ \mathfrak{M}[\lambda,h;\Phi_n](u) := \frac{1}{2}\int_{\Omega} |u-h|^2 \, d\mu_n + \lambda \Phi_n(u).$$
We also define the functional $\mathfrak{N}[\lambda;\Phi_n]:L^2(\Omega;\mu_n) \rightarrow (-\infty,+\infty]$ by
$$ \mathfrak{N}[\lambda;\Phi_n](u) := \frac{1}{2}\int_{\Omega} |u|^2 \, d\mu_n + \lambda \Phi_n(u).$$
\end{definition}

We note the similarity between the functionals $\mathfrak{M}[\lambda,h;\Phi_n]$ in Definition~\ref{movementeq} and the functionals $\mathfrak{H}_n[\gamma,h_n]$ in the proof of Theorem~\ref{Hilthm}. Once again the minimizers of these functionals will be important for the proof of convergence of the gradient flows. 

We note the following connection with the Moreau envelope from Definition~\ref{MoreauProx}:
$$ [\Phi_n]^{\gamma}(u) = \frac1\gamma \inf_{v\in L^2(\Omega;\mu_n)} \mathfrak{M}[\gamma,u;\Phi_n](v).$$
Moreover, the proximal operator $J_\gamma(\Phi_n)$, also from Definition~\ref{MoreauProx}, coincides with the minimizers of $\mathfrak{M}[\gamma,u;\Phi_n]$, in the sense that $(u,w)\in J_\gamma(\Phi_n)$ if and only if $w\in \argmin_{v\in L^2(\Omega;\mu_n)} \mathfrak{M}[\gamma,u;\Phi_n](v)$.

For the following assumptions, we recall Definition~\ref{def:Pconvex} for the definition of $P_0$-convexity and note that the definitions of $\Gamma$-convergence and equicoercivity over Banach stackings are given in Definition~\ref{BSdefs}.

\begin{assumptions}\label{assump}
\begin{enumerate}[(i)]
\item There exists a subset $\mathcal{A}\subset L^{\max(2,p)}(\Omega;\mu_{\infty})$ such that $\mathcal{A}$ is dense in $X_{\infty}$. Moreover, for all $\lambda>0$, for all $h_{\infty}\in \mathcal{A}$, and for all sequences $(h_n)_n \subset_n L^{\max(2,p)}(\Omega;\mu_n)$ such that $h_n\underset{TL^{\max(2,p)}(\Omega)}\longtwoheadrightarrow h_{\infty}$ as $n\rightarrow \infty$, the following hold:
\begin{itemize}
    \item the sequence $\left(\mathfrak{M}[\lambda,h_n;\Phi_n]\right)_{n\in\mathbb{N}}$ $\Gamma$-converges to $\mathfrak{M}[\lambda,h_{\infty};\Phi_{\infty}]$ as $n\rightarrow \infty$ over the Banach stacking $\big((X_n, \xi_n)_{n\in \mathbb{N}_\infty}, TL^p(\Omega)\big)$ and
    \item the sequence $\big( \mathfrak{M}[\lambda,h_n;\Phi_n] \big)_{n\in \mathbb{N}}$ is equicoercive over the Banach stacking $\big((X_n, \xi_n)_{n\in \mathbb{N}_\infty}, TL^p(\Omega) \big)$.
\end{itemize}
\item For all $n\in \mathbb{N}_{\infty}$, $\Phi_n: \mathfrak{F}(\Omega;\mu_n) \rightarrow (-\infty,+\infty]$ is non-negative, $P_0$-convex on $L^{\min (2,p)}(\Omega;\mu_n)$, lower semicontinuous with respect to the topology of $L^{\min (2,p)}(\Omega;\mu_n)$, and $\Phi_n(0)=0$.
\end{enumerate}
\end{assumptions}

\begin{remark}\label{rem:00subdiff}
By assumption~(ii) in Assumptions~\ref{assump}, for all $n\in \mathbb{N}_\infty$, $\Phi_n$ is non-negative and $\Phi_n(0)=0$. As a consequence, for all $n\in \mathbb{N}_\infty$ and for all $j\in [1,+\infty)$, $(0,0) \in \partial_{L^j} \Phi_n$, as can be checked directly from Definition~\ref{Subdiffdef}.
Moreover, $\Phi_n(0)=0\neq +\infty$ also implies that $\Phi_n$ is proper. 
\end{remark}

\begin{remark}
\label{rem:lowsemcont}
Let $j\geq \min(2,p)$. Since $\mu_n$ is a probability measure (and thus a finite measure), by H\"older's inequality we have that $L^j(\Omega;\mu_n)\subset L^{\min (2,p)}(\Omega;\mu_n)$ and convergence in $L^j(\Omega;\mu_n)$ implies convergence in $L^{\min (2,p)}(\Omega;\mu_n)$. It follows that lower semicontinuity with respect to $L^{\min (2,p)}(\Omega;\mu_n)$ implies lower semicontinuity with respect to $L^j(\Omega;\mu_n)$. Moreover, recalling Definition~\ref{def:sumtopology}, we note that convergence in the topology of $L^2(\Omega;\mu_n)+L^p(\Omega;\mu_n)$ implies convergence in the topology of $L^{\min (2,p)}(\Omega;\mu_n)$. Hence it also follows that lower semicontinuity with respect to $L^{\min (2,p)}(\Omega;\mu_n)$ implies lower semicontinuity with respect to $L^2(\Omega;\mu_n)+L^p(\Omega;\mu_n)$. 
\end{remark}

\begin{remark}\label{rem:restrictedPhin}
As a consequence of assumption~(ii) in Assumptions~\ref{assump} and Remark~\ref{rem:lowsemcont}, for all $n\in \mathbb{N}_\infty$, the restriction $\Phi_n|_{L^2(\Omega;\mu_n)}$ of $\Phi_n$ to $L^2(\Omega;\mu_n)$ is a lower semicontinuous function with respect to the topology of $L^2(\Omega;\mu_n)$. Furthermore, by Remark~\ref{rem:P0subset}, $\Phi_n|_{L^2(\Omega;\mu_n)}$ is $P_0$-convex on $L^2(\Omega;\mu_n)$, and, since $\Phi_n(0)=0$ and $0\in L^2(\Omega;\mu_n)$, $\Phi_n|_{L^2(\Omega;\mu_n)}$ is also proper. It follows from Proposition~\ref{Pprop} that $\Phi_n|_{L^2(\Omega;\mu_n)}$ is a convex function. So we may consider its gradient flow. 
\end{remark}

\begin{lemma}
\label{lem:Resolve}
Let $\big((X_n, \xi_n)_{n\in \mathbb{N}_\infty}, TL^p(\Omega)\big)$ be the Banach stacking of part~6 of Examples~\ref{BSexmpl} and let $(\Phi_n)_{n\in \mathbb{N}_\infty}$ be a sequence of functionals $\Phi_n: \mathfrak{F}(\Omega;\mu_n) \rightarrow (-\infty,+\infty]$. Assume Assumptions~\ref{assump} are satisfied. Suppose $(z_n)_{n\in \mathbb{N}_{\infty}} \subset_n X_n$ is a sequence such that $z_n\underset{TL^p(\Omega)}\longtwoheadrightarrow z_{\infty} $ as $n\rightarrow \infty$. Then, for all $\lambda>0$,
$$ R_{\lambda}(\overline{\partial_{L^p} \Phi_n})z_n \underset{TL^p(\Omega)}\longtwoheadrightarrow  R_{\lambda}(\overline{\partial_{L^p} \Phi_{\infty}})z_{\infty} \qquad \text{as } n\to \infty.$$
\end{lemma}
\begin{proof}
In this proof, given an operator $A$ defined on some subspace of $X_n$, the notation $\overline{A}$ always refers to the pointwise closure with respect to the topology of $X_n$ (see Definition~\ref{def:pointwiseclosure}). 

First we prove some properties of the operators that are involved in this proof. Let $n\in \mathbb{N}_\infty$. We apply parts~1, 2, and~4 of Proposition~\ref{SubdiffProp} to deduce that
\begin{equation}
\label{resolveEq2}
I_{L^2(\Omega; \mu_n)}+\lambda \partial_{L^2} \Phi_n \subset \partial_{L^2} \mathfrak{N}[\lambda;\Phi_n].
\end{equation}
Taking an intersection with $X_n$ we get that
$$ I_{X_n}+\lambda \partial_{L^2} \Phi_n \cap \big(X_n \times X_n \big) \subset  \big(I_{L^2(\Omega; \mu_n)}+\lambda \partial_{L^2} \Phi_n \big)\cap \big(X_n \times X_n \big) \subset \partial_{L^2} \mathfrak{N}[\lambda;\Phi_n] \cap \big(X_n \times X_n \big).$$
Subsequently taking the pointwise closure on both sides of the inclusion above, we obtain 
$$ \overline{I_{X_n}+\lambda \partial_{L^2} \Phi_n \cap \big(X_n \times X_n \big)} \subset \overline{\partial_{L^2}\mathfrak{N}[\lambda;\Phi_n] \cap \big(X_n \times X_n \big)}.$$
Applying Proposition~\ref{Closureprop} we obtain that
$$ I_{X_n}+ \lambda\overline{\partial_{L^2} \Phi_n \cap \big(X_n \times X_n \big)} \subset \overline{\partial_{L^2}\mathfrak{N}[\lambda;\Phi_n] \cap \big(X_n \times X_n \big)}.$$

We can apply Theorem~\ref{BenCran1991} to $\Phi_n$ for each $n\in \mathbb{N}_{\infty}$. Indeed, for all $n\in\mathbb{N}_\infty$, $\Phi_n$ is $P_0$-convex by assumption~(ii) in Assumptions~\ref{assump}, $(0,0)\in \partial_{L^p} \Phi_n$ by Remark~\ref{rem:00subdiff}, and by Remark~\ref{rem:lowsemcont} $\Phi_n$ is lower semicontinuous with respect to the topology of $L^2(\Omega;\mu_n)+L^p(\Omega;\mu_n)$. Thus
$$ \overline{\partial_{L^2}\Phi \cap (X_n \times X_n)} = \overline{\partial_{L^p} \Phi}$$
and it follows that 
\begin{equation}
\label{resolveEq1}
 I_{X_n} + \lambda\overline{\partial_{L^p}\Phi_n} \subset \overline{\partial_{L^2}\mathfrak{N}[\lambda;\Phi_n] \cap \big(X_n \times X_n \big)}.
\end{equation}
We observe that $I_{L^2(\Omega;\mu_n)}+\lambda \partial_{L^2} \Phi_n$ has range equal to $L^2(\Omega;\mu_n)$, since by Theorem~\ref{thm:partiallambdamaccretive} $\partial_{L^2} \Phi_n|_{L^2(\Omega;\mu_n)}$ is m-accretive (see Definition~\ref{def:m-accretive}) and thus, by Remark~\ref{rem:adjustedsubdiff}, so is $\partial_{L^2} \Phi_n$. Theorem~\ref{thm:partiallambdamaccretive} can be used here since, by Remark~\ref{rem:restrictedPhin}, $\Phi_n|_{L^2(\Omega;\mu_n)}$ is proper, convex (thus $0$-convex), and lower semicontinuous. 
Since, by Definition~\ref{Subdiffdef}, the range of $\partial_{L^2} \mathfrak{N}[\lambda;\Phi_n]$ is included in $L^2(\Omega;\mu_n)$, from the inclusion (\ref{resolveEq2}) it follows that its range is in fact equal to $L^2(\Omega;\mu_n)$. Hence, by Remark~\ref{rem:domainrange}, $\big(\partial_{L^2} \mathfrak{N}[\lambda;\Phi_n]\big)^{-1}$ has domain equal to $L^2(\Omega;\mu_n)$.

By Lemmas~\ref{basicP0} and~\ref{quadP0} $\mathfrak{N}[\lambda;\Phi_n]$ is $P_0$-convex on $\mathfrak{F}(\Omega;\mu_n)$. We apply Lemma~\ref{lemAcrete} to find that $\partial_{L^2} \mathfrak{N}[\lambda;\Phi_n]$ is accretive with respect to the norm on $X_n$ and thus, by Remark~\ref{rem:intersectionclosure}, so is $\partial_{L^2} \mathfrak{N}[\lambda;\Phi_n] \cap \big( X_n \times X_n \big)$. By the same remark, $\overline{\partial_{L^2} \mathfrak{N}[\lambda;\Phi_n] \cap \big(X_n \times X_n \big) }$ is also accretive with respect to the norm on $X_n$. 

It follows from Theorem~\ref{BenCran1991} that $\overline{ \partial_{L^p} \Phi_n }$ is m-accretive on $X_n$ and thus so is $I_{X_n}+\lambda \overline{ \partial_{L^p} \Phi_n }$ by Proposition~\ref{accprop}. By Proposition~\ref{propMaccret}, Definition~\ref{def:maximalmonotone}, and equation \eqref{resolveEq1}, we have $$I_{X_n}+\lambda \overline{ \partial_{L^p} \Phi_n } = \overline{ \partial_{L^2} \mathfrak{N}[\lambda;\Phi_n]  \cap \big( X_n \times X_n\big)}$$
and thus, taking inverses, we find that
\begin{equation*}
\big( \overline{ \partial_{L^2} \mathfrak{N}[\lambda;\Phi_n] \cap \big(X_n \times X_n\big)} \big)^{-1} = R_{\lambda} \big( \overline{ \partial_{L^p} \Phi_n} \big) .
\end{equation*}
Define $A_n:= \big(\partial_{L^2} \mathfrak{N}[\lambda;\Phi_n] \cap \big(X_n \times X_n\big)\big)^{-1}$, which is an operator on $X_n$. In fact, it is an operator on $X_n \cap L^2(\Omega;\mu_n)$, since it is the inverse of the restriction of the operator $\partial_{L^2} \mathfrak{N}[\lambda;\Phi_n]$ (which itself is an operator on $L^2(\Omega;\mu_n)$) to $X_n \cap L^2(\Omega;\mu_n)$. Since $\mu_n$ is a finite measure, $X_n \cap L^2(\Omega;\mu_n) = L^{\max(2,p)}(\Omega;\mu_n)$ (as can be shown using H\"older's inequality) and thus $D(A_n) \subset L^{\max(2,p)}(\Omega;\mu_n)$. By Lemma~\ref{lem:inverseclosure} we know that $\overline{A_n} = \big( \overline{ \partial_{L^2} \mathfrak{N}[\lambda;\Phi_n] \cap \big(X_n \times X_n\big)} \big)^{-1}$. Since $\overline{ \partial_{L^p} \Phi_n}$ and thus also $R_{\lambda} \big( \overline{ \partial_{L^p} \Phi_n} \big)$ are operators on $X_n$, so is $\overline{A_n}$. By Proposition~\ref{prop:accretivecontraction}, since $\overline{ \partial_{L^p} \Phi_n }$ is m-accretive (and thus accretive) on $X_n$, $R_{\lambda} \big( \overline{ \partial_{L^p} \Phi_n} \big)$ is a contraction with respect to the norm on $X_n$ and hence so is $\overline{A_n}$. Given that $A_n \subset \overline{A_n}$ we get that $A_n$ is also a contraction with respect to the norm on $X_n$.

As mentioned in Remark~\ref{rem:contraction}, we can identify a contraction on $X_n$ with a $1$-Lipschitz-continuous map on a subset of $X_n$. Thus we can view $R_{\lambda}(\overline{\partial_{L^p} \Phi_n})$ as the continuous extension of $A_n$ onto all of $X_n$, given that $\overline{\partial_{L^p} \Phi_n}$ is m-accretive on $X_n$ and so $R_{\lambda}(\overline{\partial_{L^p} \Phi_n})$ has domain equal to $X_n$. Moreover, since $\overline{A_n}=R_{\lambda}(\overline{\partial_{L^p} \Phi_n})$ and $D(\overline{A_n})=X_n$, we have by Remark~\ref{rem:domainsofclosures} that $\overline{D(A_n)}\supset D(\overline{A_n})=X_n$, in particular the domain of $A_n$ is dense in $X_n$. 

We now have collected the prerequisites to prove the theorem. We wish to do so via the density argument in Lemma~\ref{thResLem} applied to the Banach stacking $\big((L^p(\Omega;\mu), \xi_{\mu})_{\mu \in \mathcal{P}_p(\Omega)}, TL^p(\Omega)\big)$ (part~6 of Examples~\ref{BSexmpl}), where for $\mathcal{A}$ in the lemma we will take the dense subset $\mathcal{A} \subset L^{\max(2,p)}(\Omega;\mu_\infty)$ whose existence is guaranteed by part~(i) of Assumptions~\ref{assump} and for $G_n$ we will use $R_\lambda(\overline{\partial_{L^p}\Phi_n})$. We note that $L^{\max(2,p)}(\Omega;\mu_\infty) \subset X_\infty$ and thus $\mathcal{A}\subset X_\infty$.

A key ingredient in the remainder of the proof is that, for all $h_\infty\in \mathcal{A}$, there exists a sequence $(h_n)_{n\in\mathbb{N}} \subset_n D(A_n)$ such that $h_n \underset{TL^{\max(2,p)}(\Omega)}\longtwoheadrightarrow h_{\infty}$ as $n\to \infty$, as we will prove shortly. We recall that $D(A_n) \subset L^{\max(2,p)}(\Omega;\mu_n)$, and thus such a sequence $(h_n)$ satisfies the conditions from part~(i) of Assumptions~\ref{assump}. Moreover, since $\mu_n$ is a finite measure, by H\"older's inequality we have $L^{\max(2,p)}(\Omega;\mu_n) \subset X_n$; furthermore, by Lemma~\ref{tlplem2}, $\max(2,p)\geq p$ implies that $h_n \underset{TL^p(\Omega)}\longtwoheadrightarrow h_{\infty}$ as $n\to \infty$. Thus, if it also holds that $R_\lambda(\overline{\partial_{L^p}\Phi_n}) h_n \underset{TL^p(\Omega)}\longtwoheadrightarrow R_\lambda(\overline{\partial_{L^p}\Phi_n}) h_\infty$ as $n\rightarrow \infty$, then all the requirements are fulfilled to make use of Lemma~\ref{thResLem} in the way explained in the previous paragraph. The fact that $(h_n)_{n\in\mathbb{N}} \subset_n D(A_n)$ and the availability of part~(i) of Assumptions~\ref{assump} will allow us to prove this required convergence condition on the resolvents.

We prove the existence such a sequence $(h_n)$ separately in two complementary cases: if $p\leq 2$ and if $p>2$. Let $h_\infty \in \mathcal{A}$.
\begin{itemize}
    \item[($p\leq2$)] In this case $\max(2,p)=2$, thus we need to prove the existence of a sequence $(h_n)_{n\in \mathbb{N}} \subset_n D(A_n)$ such that $h_n \underset{TL^2(\Omega)}\longtwoheadrightarrow h_{\infty}$ as $n\to \infty$. According to Proposition~\ref{tlpBanach}, $\big((L^2(\Omega;\mu), \xi_{\mu})_{\mu \in \mathcal{P}_2(\Omega)}, TL^2(\Omega)\big)$ is a Banach stacking, thus per property~(ii) of Banach stackings in Definition~\ref{BanStack}, there exists a sequence $(h_n)_{n\in \mathbb{N}} \subset_n L^2(\Omega;\mu_n)$ with the required convergence property. We claim that, for all $n\in \mathbb{N}$, $D(A_n)=L^2(\Omega;\mu_n)$, which completes the existence proof in the case that $p\leq 2$.
    
    To prove this claim, we recall from earlier in this proof that, for all $n\in \mathbb{N}_\infty$, $\operatorname{Range}(\partial_{L^2} \mathfrak{N}[\lambda;\Phi_n])=L^2(\Omega;\mu_n)$. Since $\mu_n$ is a finite measure, by H\"older's inequality we have $L^2(\Omega;\mu_n)\subset X_n$ and thus $\partial_{L^2} \mathfrak{N}[\lambda;\Phi_n] \cap \big(X_n \times X_n\big) = \partial_{L^2} \mathfrak{N}[\lambda;\Phi_n]$. Therefore $\operatorname{Range}\big(\partial_{L^2} \mathfrak{N}[\lambda;\Phi_n] \cap \big(X_n \times X_n\big)\big) = L^2(\Omega;\mu_n)$ and hence, by Remark~\ref{rem:domainrange}, for all $n\in \mathbb{N}_\infty \supset \mathbb{N}$, $D(A_n)=L^2(\Omega;\mu_n)$. 
    
    \item[($p>2$)] In this case $X_n = L^p(\Omega;\mu_n) = L^{\max(2,p)}(\Omega;\mu_n)$. Similarly to the previous case, there exists a sequence $(\Tilde{h}_n)_{n\in\mathbb{N}}\subset_n X_n$ such that $\Tilde{h}_n \underset{TL^p(\Omega)}\longtwoheadrightarrow h_{\infty}$ as $n\rightarrow \infty$, because, by Proposition~\ref{tlpBanach}, $\big((L^p(\Omega;\mu), \xi_{\mu})_{\mu \in \mathcal{P}_p(\Omega)}, TL^p(\Omega)\big)$ is a Banach stacking. We recall that earlier in this proof we established that $D(A_n)$ is dense in $X_n$, thus there exists a sequence $(h_n)_n \subset_n D(A_n)$ such that, for all $n\in \mathbb{N}$, $\Vert \Tilde{h}_n -h_n \Vert_{L^p(\Omega;\mu_n)}\leq \frac{1}{n}$. By the triangle inequality from Remark~\ref{rem:triangle} and property~(i) in Definition~\ref{BanStack}, it follows that, for all $n\in \mathbb{N}$,
    \begin{align*}
        d_{TL^p}^{n,\infty}(h_n,h_{\infty}) &\leq d_{TL^p}^{n,n}(h_n,\Tilde{h}_n) + d_{TL^p}^{n,\infty}(\Tilde{h}_n,h_{\infty})
        \leq \Vert h_n - \Tilde{h}_n \Vert_{L^p}  + d_{TL^p}^{n,\infty}(\Tilde{h}_n,h_{\infty}) \\
        &\leq \frac{1}{n} + d_{TL^p}^{n,\infty}(\Tilde{h}_n,h_{\infty}) \rightarrow 0
    \end{align*}
    as $n\rightarrow \infty$. Hence $h_n \underset{TL^p(\Omega)}\longtwoheadrightarrow h_{\infty}$ as $n\rightarrow \infty$ as required.
\end{itemize}

We return to the general setting with $p\in [1,+\infty)$. Let $h_\infty \in \mathcal{A}$ and let $(h_n)_{n\in \mathbb{N}}\subset_n D(A_n)$ be the corresponding sequence specified above. Since $D(A_n) \subset L^{\max(2,p)}(\Omega;\mu_n) \subset X_n$ and $\mathcal{A} \subset L^{\max(2,p)}(\Omega;\mu_n) \subset X_n$, we have, for all $n\in \mathbb{N}_\infty$, $h_n \in X_n$, and thus, by Remark~\ref{rem:inverseofrestriction}, $A_n(h_n) \in \big(\partial_{L^2} \mathfrak{N}[\lambda;\Phi_n] \big)^{-1}(h_n)$. We established earlier that $A_n$ is a contraction on $X_n$, which justifies interpreting it as a function as in Remark~\ref{rem:contraction}.

By Remark~\ref{rem:restrictedPhin}, for all $n\in \mathbb{N}_\infty$, the restriction of $\Phi_n$ to $L^2(\Omega;\mu_n)$ is lower semicontinuous with respect to the topology of $L^2(\Omega;\mu_n)$ and convex. Since $u\mapsto \frac12 \int_\Omega |u|^2 \, d\mu_n$ is continuous with respect to the same topology and convex, we deduce that $\mathfrak{N}[\lambda; \Phi_n]$ is lower semicontinuous with respect to the topology of $L^2(\Omega;\mu_n)$ and convex. Hence, by Lemma~\ref{subdiffLem}, we obtain, for all $n\in \mathbb{N}_{\infty}$,
\begin{align}\label{eq:Anhnargmin}
 A_n(h_n)&\in  \big(\partial_{L^2}  \mathfrak{N}[\lambda;\Phi_n] \big)^{-1}(h_n) 
 = \argmin_{w\in L^2(\Omega; \mu_n)} \left(  \mathfrak{N}[\lambda;\Phi_n](w) - \int_\Omega h_nw \, d\mu_n \right) \notag \\ 
&= \argmin_{w\in L^2(\Omega; \mu_n)} \left( \frac{1}{2} \int_\Omega |h_n -w|^2 \, d\mu_n + \lambda \Phi_n(w) - \frac{1}{2}\int_\Omega |h_n|^2 \, d\mu_n \right) \notag \\
&= \argmin_{w\in L^2(\Omega; \mu_n)} \mathfrak{M}[\lambda,h_n;\Phi_n](w).
\end{align}
We recall our observation, following Definition~\ref{movementeq}, that $J_\gamma(\Phi_n)$ coincides with the minimizers of $\mathfrak{M}[\lambda,h_n;\Phi_n]$. For reasons similar to the ones we gave above for $\mathfrak{N}[\lambda; \Phi_n]$, also $\mathfrak{M}[\lambda,h_n;\Phi_n]$ is lower semicontinuous with respect to the topology of $L^2(\Omega;\mu_n)$ and convex. Moreover, since $h_n \in L^{\max(2,p)}(\Omega) \subset L^2(\Omega)$ and $\Phi_n(0)=0$, $\mathfrak{M}[\lambda,h_n;\Phi_n](0)<+\infty$ and thus $\mathfrak{M}[\lambda,h_n;\Phi_n]$ is proper. Therefore, by part~1 of Proposition~\ref{envelopeProp}, $\mathfrak{M}[\lambda,h_n;\Phi_n]$ has a unique minimizer over $L^2(\Omega;\mu_n)$. By \eqref{eq:Anhnargmin}, this minimizer is necessarily given by $A_n(h_n)$. We recall that $A_n(h_n) \in X_n \cap L^2(\Omega;\mu_n)$, so, for all $n\in \mathbb{N}_\infty$,
\begin{align*}
     A_n(h_n) &= \argmin_{w\in L^2(\Omega; \mu_n)} \mathfrak{M}[\lambda,h_n;\Phi_n](w)
     = \argmin_{w\in X_n \cap L^2(\Omega; \mu_n)} \mathfrak{M}[\lambda,h_n;\Phi_n](w)\\ 
     &= \argmin_{w\in X_n} \mathfrak{M}[\lambda,h_n;\Phi_n](w).
\end{align*}
The first equality holds by uniqueness of the minimizer, the second because the minimizer lies in the subset $X_n \cap L^2(\Omega;\mu_n)$. The third equality follows because, for all ${w\in X_n\setminus L^2(\Omega;\mu_n)}$, $\mathfrak{M}[\lambda,h_n;\Phi_n](w)=+\infty$. We then apply assumption (i) in Assumptions~\ref{assump} alongside Proposition~\ref{prop:gammaconBS} to deduce
\begin{align*}
 A_n(h_n) &= \argmin_{w\in X_n} \mathfrak{M}[\lambda,h_n;\Phi_n](w)
&\underset{TL^p(\Omega)}\longtwoheadrightarrow \argmin_{w\in X_{\infty}} \mathfrak{M}[\lambda,h_{\infty};\Phi_{\infty}](w)
&=  A_{\infty}(h_{\infty})
\end{align*}
as $n\rightarrow \infty$. 

As established earlier in this proof, $ R_{\lambda} \big( \overline{ \partial_{L^p} \Phi_n} \big)$ is the continuous extension of $A_n$ onto all of $X_n$, so
$$  R_{\lambda} \big( \overline{ \partial_{L^p} \Phi_n} \big)h_n \underset{TL^p(\Omega)}\longtwoheadrightarrow R_{\lambda} \big( \overline{ \partial_{L^p} \Phi_{\infty}} \big)h_{\infty} \qquad \text{as } n\to\infty.$$
Together with the properties $(h_n)_{n\in\mathbb{N}} \subset_n X_n$ and $h_n \underset{TL^{\max(2,p)}(\Omega)}\longtwoheadrightarrow h_{\infty}$ as $n\to \infty$, that were proven earlier, this allows us to use Lemma~\ref{thResLem} to conclude that, for all sequences $(z_n)_{n\in \mathbb{N}} \subset_n X_n$ such that $z_n \underset{TL^p(\Omega)}\longtwoheadrightarrow z_{\infty} $ as $n\rightarrow \infty$,
$$  R_{\lambda} \big( \overline{ \partial_{L^p} \Phi_n} \big)z_n \underset{TL^p(\Omega)}\longtwoheadrightarrow R_{\lambda} \big( \overline{ \partial_{L^p} \Phi_{\infty}} \big)z_{\infty}.$$
\end{proof}

In the next steps of our analysis we will make use of the Moreau envelope from Definition~\ref{MoreauProx}. Given $\gamma>0$ and a function $\Psi:\mathfrak{F}(\Omega;\mu)\rightarrow (-\infty,+\infty]$, we will use the shorthand 
\begin{equation}\label{eq:Moreau2}
[\Psi]_2^{\gamma} := \left[\Psi|_{L^2(\Omega;\mu)}\right]^{\gamma}.
\end{equation}

\begin{lemma}
\label{lem:envelopes}
Suppose $(\Phi_n)_{n\in\mathbb{N}_{\infty}}$ satisfies Assumptions~\ref{assump}. Then, for all $\gamma>0$ and for all sequences ${(u_n)_{n\in\mathbb{N}_{\infty}}\subset_n X_n \cap L^2(\Omega;\mu_n)}$ such that $u_n \underset{TL^2(\Omega)}\longtwoheadrightarrow u_{\infty}$,
$$ [\Phi_n]_2^{\gamma}(u_n) \rightarrow  [\Phi_\infty]_2^{\gamma}(u_\infty) \qquad \text{as } n\to\infty.$$
\end{lemma}

\begin{proof}
By Remark~\ref{rem:restrictedPhin}, for all $n\in \mathbb{N}_{\infty}$, $\Phi_n|_{L^2(\Omega;\mu_n)}$ is convex, proper and lower semicontinuous on $L^2(\Omega;\mu_n)$; thus the same properties hold for $\mathfrak{M}[\gamma,u_n;\Phi_n]$. Hence, by Lemma~\ref{ProxResolve} the unique minimizer of $\mathfrak{M}[\gamma,u_n;\Phi_n|_{L^2(\Omega;\mu_n)}]$ over $L^2(\Omega;\mu_n)$, and thus also the unique minimizer of $\mathfrak{M}[\gamma,u_n;\Phi_n]$ over $L^2(\Omega;\mu_n)$, is given by $R_{\gamma}(\partial_{L^2}\Phi_n)u_n$, where we can interpret $R_\gamma(\partial_{L^2}\Phi_n)$ as a $1$-Lipschitz-continuous function on $L^2(\Omega;\mu_n)$. Moreover, by Lemma~\ref{lemAcrete}, $R_{\gamma}(\partial_{L^2}\Phi_n)$ is accretive with respect to the $L^p(\Omega;\mu_n)$ norm and thus, per Proposition~\ref{prop:accretivecontraction}, it is also a contraction with respect to this norm.

Since $R_{\gamma}(\partial_{L^2}\Phi_n)0=0$ as a consequence of part~(ii) of Assumptions~\ref{assump} (in particular the non-negativity of $\Phi_n$ and the condition $\Phi_n(0)=0$), we deduce that
$$ \Vert R_{\gamma}(\partial_{L^2}\Phi_n)u_n \Vert_{L^p}\leq \Vert u_n \Vert_{L^p}<+\infty. $$
It follows that the minimizer of $\mathfrak{M}[\gamma,u_n;\Phi_n]$ over $L^2(\Omega;\mu_n)$ is also in $X_n$. Hence
$$ \inf_{v\in L^2(\Omega;\mu_n)}  \mathfrak{M}[\gamma,u_n;\Phi_n](v) = \inf_{v\in L^2(\Omega;\mu_n)\cap X_n}  \mathfrak{M}[\gamma,u_n;\Phi_n](v) = \inf_{v\in X_n} \mathfrak{M}[\gamma,h_n;\Phi_n](v).$$
The second equality holds because $\mathfrak{M}[\gamma,h_n;\Phi_n](v)=+\infty$ if $v\notin L^2(\Omega;\mu_n)$, as follows from Definition~\ref{movementeq}, since $h_n\in L^2(\Omega;\mu_n)$.

Finally, by part~(i) from Assumptions~\ref{assump} we can apply Proposition~\ref{prop:gammaconBS} to the sequence $(\mathfrak{M}[\gamma,u_n;\Phi_n])_{n\in\mathbb{N}_\infty}$ to conclude that
\begin{align*}
[\Phi_n]_2^{(\gamma)}(u_n) &= \frac{1}{\gamma}\inf_{v\in L^2(\Omega;\mu_n)}  \mathfrak{M}[\gamma,u_n;\Phi_n](v)
=\frac{1}{\gamma}\inf_{v\in X_n}  \mathfrak{M}[\gamma,u_n;\Phi_n](v)\\
&\rightarrow \frac{1}{\gamma}\inf_{v\in X_\infty} \mathfrak{M}[\gamma,u_\infty;\Phi_\infty](v) 
= [\Phi_{\infty}]_2^{(\gamma)}(u_{\infty}) \qquad \text{as } n\to\infty.
\end{align*} 
\end{proof}

\begin{lemma}
\label{mnthmlem}
Suppose $(\Phi_n)_{n\in\mathbb{N}_{\infty}}$ satisfies Assumptions~\ref{assump}.
For all $t\in [0,+\infty)$ and all sequences $(x_n)_{n\in \mathbb{N}_\infty} \subset_n \overline{D(\partial_{L^p}\Phi_n)}$ such that $x_n \underset{TL^p(\Omega)}\longtwoheadrightarrow x_{\infty}$ as $n\rightarrow \infty$,
$$ S_{\overline{\partial_{L^p} \Phi_n}}(t)x_n \underset{TL^p(\Omega)}\longtwoheadrightarrow S_{\overline{\partial_{L^p} \Phi_{\infty}}}(t)x_{\infty} \quad \text{as } n\to\infty.$$
Moreover, given $T \in [0,+\infty)$, the above convergence is uniform in $t$ on the interval $[0,T]$. 
\end{lemma}

\begin{proof}
By Lemma~\ref{lem:Resolve}, we have for all $\lambda>0$ and for all sequences $(z_n)_{n\in \mathbb{N}_\infty} \subset_n X_n$ such that $z_n\underset{TL^p(\Omega)}\longtwoheadrightarrow z_{\infty}$ as $n\rightarrow \infty$, that
$$ R_{\lambda}(\overline{\partial_{L^p} \Phi_n})z_n \underset{TL^p(\Omega)}\longtwoheadrightarrow  R_{\lambda}(\overline{\partial_{L^p} \Phi_{\infty}})z_{\infty} \qquad \text{as } n \rightarrow \infty.$$
By part~(ii) of Assumptions~\ref{assump} and Theorem~\ref{BenCran1991}, for all $n\in \mathbb{N}_\infty$, $\overline{\partial_{L^p} \Phi_n}$ is m-accretive on $X_n$. (For details we refer to the proof of Lemma~\ref{lem:Resolve}, where the same property was derived from part~(ii) of Assumptions~\ref{assump}, using Remarks~\ref{rem:00subdiff} and~\ref{rem:lowsemcont}.) In particular, by Proposition~\ref{prop:maccretive}, $\overline{\partial_{L^p} \Phi_n}$ is accretive on $X_n$. Furthermore, by Remark~\ref{rem:rangecondition}, $\overline{\partial_{L^p} \Phi_n}$ satisfies the range condition in \eqref{rangeCon}. Moreover, by Remark~\ref{rem:domainsofclosures}, $\overline{D(\overline{\partial_{L^p} \Phi_n})}= \overline{D(\partial_{L^p} \Phi_n)} $, so, for all $n\in\mathbb{N}_\infty$, $x_n \in\overline{D(\overline{\partial_{L^p} \Phi_n})} $. 

Thus all conditions of Theorem~\ref{theorem1} are satisfied (with, for all $n\in \mathbb{N}_\infty$, $\omega_n=0$). This theorem implies pointwise and uniform convergence of the sequence $\left(S_{\overline{\partial_{L^p} \Phi_n}}(\cdot)x_n\right)$, as required.
\end{proof}

\begin{lemma}
\label{lem:liminf}    
Suppose $(\Phi_n)_{n\in\mathbb{N}_{\infty}}$ satisfies Assumptions~\ref{assump}. Then, for all sequences $(x_n)_{n\in \mathbb{N}_\infty} \subset_n L^{\max(2,p)}(\Omega;\mu_n)$ for which $x_n \underset{TL^{\max(2,p)}(\Omega)}\longtwoheadrightarrow x_{\infty}$ as $n\rightarrow \infty$,
$$ \Phi_\infty(x_\infty)\leq\liminf_{n\rightarrow \infty} \Phi_n(x_n).$$
\end{lemma}

\begin{proof}
Let $\mathcal{A}$ be as in part~(i) of Assumptions~\ref{assump} and $h_\infty \in \mathcal{A}$. By part~(ii) of Definition~\ref{BanStack} and Proposition~\ref{tlpBanach} there exists a sequence $(h_n)_n \subset_n L^{\max(2,p)}(\Omega;\mu_n)$ such that $h_n\underset{TL^{\max(2,p)}(\Omega)}\longtwoheadrightarrow h_{\infty}$ as $n\rightarrow \infty$. By part~(i) of Assumptions~\ref{assump} the sequence $\left(\mathfrak{M}[\lambda,h_n;\Phi_n]\right)_{n\in\mathbb{N}}$ $\Gamma$-converges to $\mathfrak{M}[\lambda,h_{\infty};\Phi_{\infty}]$ as $n\rightarrow \infty$ over the Banach stacking $\big((X_n, \xi_n)_{n\in \mathbb{N}_\infty}, TL^p(\Omega)\big)$. Let $(x_n)_{n\in \mathbb{N}_\infty} \subset_n L^{\max(2,p)}(\Omega;\mu_n)$ be a sequence for which $x_n \underset{TL^{\max(2,p)}(\Omega)}\longtwoheadrightarrow x_{\infty}$ as $n\rightarrow \infty$. Then, using $\Gamma$-convergence,
\begin{align*}
    \Phi_\infty(x_\infty) &= \mathfrak{M}[1,h_{\infty};\Phi_{\infty}](x_\infty) - \frac{1}{2}\int_\Omega |h_\infty - x_\infty|^2 \, d\mu_\infty \\
    &\leq \liminf_{n\rightarrow \infty}  \mathfrak{M}[1,h_n;\Phi_n](x_n) - \frac{1}{2}\int_\Omega |h_\infty - x_\infty|^2 \, d\mu_\infty \\
    &= \liminf_{n\rightarrow \infty} \left\{ \mathfrak{M}[1,h_n;\Phi_n](x_n) - \frac{1}{2}\int_\Omega |h_n - x_n|^2 \, d\mu_n \right\} 
    = \liminf_{n\rightarrow \infty} \Phi_n(x_n).
\end{align*}
The second equality holds because, by Proposition~\ref{tlpBanach}, $TL^2(\Omega)$ is a Banach stacking and thus in particular, by condition~(iv) of Definition~\ref{BanStack} we have continuity of the norms.
\end{proof}

We can now prove the main theorem of this subsection.

\begin{theorem}
\label{thm:P0theorem1}
Suppose $(\Phi_n)_{n\in\mathbb{N}_{\infty}}$ satisfies Assumptions~\ref{assump}. Let $(x_n)_{n\in \mathbb{N}_\infty} \subset_n L^{\max(2,p)}(\Omega;\mu_n)$ be a sequence that satisfies the following three conditions:
\begin{enumerate}[(a)]
    \item there exist a $q>\max(2,p)$ and a $C>0$ such that, for all $n\in \mathbb{N}_{\infty}$, we have ${\Vert x_n \Vert_{L^q(\Omega; \mu_n)} \leq C}$;
    \item for all  $n\in \mathbb{N}_{\infty}$, $x_n \in \overline{\operatorname{dom}(\Phi_n|_{L^2(\Omega;\mu_n)})}$ (where the closure is taken with respect to the topology of $L^2(\Omega;\mu_n)$); and
    \item $x_n \underset{TL^p(\Omega)}\longtwoheadrightarrow x_{\infty}$ as $n\rightarrow \infty$.    
\end{enumerate}
For all $n\in\mathbb{N}_{\infty}$, let $u_n:[0,+\infty)\rightarrow L^2(\Omega;\mu_n)$ be the gradient flow of $\Phi_n|_{L^2(\Omega;\mu_n)}$ starting from $x_n$ as specified in Definition~\ref{gradflowdef}. Then, for all $t\in [0,+\infty)$,
$$u_n(t) \underset{TL^{\max(2,p)}(\Omega)}\longtwoheadrightarrow u_{\infty}(t) \qquad \text{as } n\to\infty.$$
Given $T \in [0,+\infty)$, the above convergence is uniform in $t$ on the interval $[0,T]$. Moreover for any $t>0$ we have,
$$ \Phi_n(u_n(t))\rightarrow \Phi_\infty(u_\infty(t)) \qquad \text{as } n\to\infty.$$
\end{theorem}
\begin{proof}
By Remarks~\ref{rem:adjustedsubdiff} and~\ref{rem:restrictedPhin} and by Lemma~\ref{lem:gradflowsemigroup}, we have, for all $n\in\mathbb{N}_{\infty}$, $$u_n(t)=S_{\partial \Phi_n|_{L^2(\Omega;\mu_n)}}(t)x_n=S_{\partial_{L^2}\Phi_n}(t)x_n.$$ Additionally, for all $n\in \mathbb{N}_\infty$, $\Phi_n$ is convex and lower semicontinuous with respect to the topology of $L^2(\Omega;\mu_n)$, so by Theorem~\ref{thm:partiallambdamaccretive} $\partial_{L^2}\Phi_n$ is m-accretive on $L^2(\Omega;\mu_n)$. As a consequence, for all $\lambda>0$, the resolvent $R_\lambda(\partial_{L^2}\Phi_n)$ can be interpreted as a $1$-Lipschitz continuous map from $L^2(\Omega;\mu_n)$ to $L^2(\Omega;\mu_n)$, as per Remark~\ref{accremark}. In particular, $D(\partial_{L^2}\Phi_n) = L^2(\Omega;\mu_n)$. Moreover, by Remark~\ref{rem:rangecondition}, $\partial_{L^2}\Phi_n$ satisfies the range condition \eqref{rangeCon} on $L^2(\Omega;\mu_n)$. 

Define $A_n:=\partial_{L^2}\Phi_n\cap(L^p(\Omega;\mu_n)\times L^p(\Omega;\mu_n))$. By part~(ii) of Assumptions~\ref{assump} and Remark~\ref{rem:lowsemcont}, Theorem~\ref{BenCran1991} is applicable, hence $\overline{A_n}=\overline{\partial_{L^p}\Phi_n}$. We claim that, for all $n\in\mathbb{N}_{\infty}$,
\begin{equation}\label{eq:claiminproof}
x_n \in \overline{D(\partial_{L^p}\Phi_n)} \quad \text{and} \quad u_n(t)=S_{\overline{\partial_{L^p}\Phi_n}}(t)x_n.
\end{equation}
We will now prove this.

Let $T>0$ and $n\in\mathbb{N}_\infty$. For all $k\in\mathbb{N}$, let $[k]:=\mathbb{N}\cap[1,k]$, $[k]_0:=\{0\}\cup[k]$, and let $\mathcal{P}_k:=\left(0,\frac{T}{k},\frac{2T}{k},\ldots,\frac{(k-1)T}{k},T\right)$ be a $\frac{1}{k}$-discretisation of the interval $[0,T]$ as in Definition~\ref{def:epsdiscr}. For all $i \in [k]_0$, define $v^{(k)}_i:= R^i_{T/k}(\partial_{L^2}\Phi_n)x_n$; we recall the notation for repeated applications of the resolvent operator from Remark~\ref{rem:whyrangecon} and interpret $R^0_{T/k}(\partial_{L^2}\Phi_n)$ as the identity operator on $L^2(\Omega;\mu_n)$. Additionally define $\overline{v}^{(k)}:[0,T]\rightarrow L^2(\Omega;\mu_n)$ to be the piecewise-constant function such that $\overline{v}^{(k)}(0):=x_n$ and, for all $i\in[k-1]_0$ and all $t\in \left(\frac{iT}{k},\frac{(i+1)T}{k}\right]$, $\overline{v}^{(k)}(t):=v_{i+1}^{(k+1)}$. The triplet $(\mathcal{P}_{k},v^{(k)},0)$ is a $\frac{1}{k}$-approximate solution to the Cauchy equation \eqref{AbCauchy} (with $A=\partial_{L^2}\Phi_n$, $X=L^2(\Omega;\mu_n)$, and $x=x_n$), as in Definition~\ref{def:discsolution}. By \cite[Corollary 4.3]{barbu2010nonlinear} there exists a $\tilde C>0$ independent of $k$ (but not of $T$) such that, for all $y\in D(\partial_{L^2}\Phi_n)$,
\begin{equation}
\label{eq:mildbound}
\sup_{t\in[0,T]} \Vert u_n(t) - \overline{v}^{(k)}(t)\Vert_{L^2(\Omega;\mu_n)} \leq \tilde C\left( \Vert x_n -y\Vert_{L^2}+ \inf\Vert \partial_{L^2}\Phi_n (y)\Vert/\sqrt{k} \right),
\end{equation}
where we recall the definition in \eqref{eq:infnorm}. In particular, for all $t\in[0,T]$, $\overline{v}^{(k)}(t)\rightarrow u_n(t)$ as $k\rightarrow \infty$ in $L^2(\Omega;\mu_n)$. 

As a consequence of Lemma~\ref{lemAcrete}, for all $r\geq 1$, $\partial_{L^2}\Phi_n$ is accretive with respect to the $L^r(\Omega;\mu_n)$ norm and thus, by Proposition~\ref{prop:accretivecontraction}, for all $\lambda>0$, $R_\lambda(\partial_{L^2}\Phi_n)$ is a contraction with respect to the same norm. Taking $r=q$, with $q$ as in condition~(a) of the theorem with corresponding $C>0$, we observe that, for all $k\in \mathbb{N}$ and all $i\in[k]_0$, $\Vert v_i^{(k)}\Vert_{L^q(\Omega;\mu_n)}\leq \Vert x_n\Vert_{L^q(\Omega;\mu_n)} \leq C$. We then deduce that, for all $t\in[0,T]$, 
\begin{equation}\label{eq:overlinevbound}
\Vert \overline{v}^{(k)}(t)\Vert_{L^q(\Omega;\mu_n)}\leq \Vert x_n\Vert_{L^q(\Omega;\mu_n)}.
\end{equation}
By Lemma~\ref{lem:lscab} the $L^q$ norm is lower semicontinuous with respect to the topology of $L^2(\Omega;\mu_n)$, thus, taking the limit $k\rightarrow\infty$, we obtain 
\begin{equation}\label{eq:unbound}
\Vert u_n(t)\Vert_{L^q(\Omega;\mu_n)}\leq \Vert x_n \Vert_{L^q(\Omega;\mu_n)}.
\end{equation}

Next we make three observations. Firstly $v_i^{(k)}\in L^q(\Omega;\mu_n) \subset L^p(\Omega;\mu_n)$, secondly $ -\frac{k}{T}(v_{i}^{(k)}-v_{i-1}^{(k)})\in L^p(\Omega;\mu_n)$, and thirdly, by the definition of $v_i^{(k)}$, for all $i\in [k]$,
$$ -\frac{k}{T}(v_{i}^{(k)}-v_{i-1}^{(k)})\in\partial_{L^2}\Phi_n(v_i^{(k)}) .$$
Combining these we deduce that, for all $i\in [k]$, $ -\frac{k}{T}(v_i^{(k)}-v_{i-1}^{(k)})\in A_n(v_i^{(k)})$ and so $(\mathcal{P}_{k},v^{(k)},0)$ is a $\frac{1}{k}$-approximate solution to the Cauchy equation in \eqref{AbCauchy} with $X=L^2(\Omega;\mu_n)$, $A=A_n$, and $x=x_n$.

We want to prove that $u_n$ is a mild solution of the Cauchy equation \eqref{AbCauchy} with $X=L^p(\Omega;\mu_n)$, $A=A_n$, and $x=x_n$. To do so, it suffices to establish that $\overline{v}^{(k)}(t)$ converges uniformly on $t\in[0,T]$ to $u_n(t)$ in $L^p(\Omega;\mu_n)$. If $p\leq 2$ this is indeed true, since $\mu_n$ is a finite measure and \eqref{eq:mildbound} holds. Now assume that $p>2$. Let $\theta\in (0,1)$ be such that $\frac{1}{p}=\frac{\theta}{2}+\frac{1-\theta}{q}$. Then, by H\"older's inequality alongside the bounds in \eqref{eq:mildbound}, \eqref{eq:overlinevbound}, and \eqref{eq:unbound}, we obtain that, for all $y\in D(\partial_{L^2}\Phi_n)$,
\begin{align*}
    \sup_{t\in[0,T]}\Vert u_n(t) - \overline{v}^{(k)}(t)\Vert_{L^p(\Omega;\mu_n)} &\leq  \sup_{t\in[0,T]} \Vert u_n(t) - \overline{v}^{(k)}(t)\Vert_{L^2(\Omega;\mu_n)}^{\theta}\Vert u_n(t) - \overline{v}^{(k)}(t)\Vert_{L^q(\Omega;\mu_n)}^{1-\theta} \\
    &\leq C\left( \Vert x_n -y\Vert_{L^2}+ \inf\Vert \partial_{L^2}\Phi_n (y)\Vert/\sqrt{k} \right)^{\theta}\left( 2\Vert x_n \Vert_{L^q(\Omega;\mu_n)}\right)^{1-\theta}.
\end{align*}
By condition~(b) in the theorem, $x_n$ is in the closure of $ D(\partial_{L^2}\Phi_n)$ with respect to the topology of $L^2(\Omega;\mu_n)$ and thus, by applying the inequality just obtained to a sequence $(y_\ell)_{\ell\in\mathbb{N}} \subset D(\partial_{L^2}\Phi_n)$ that converges to $x_n$, a diagonal argument (with $\ell \to \infty$ and $k\to\infty$) shows that $\overline{v}^{(k)}(t)$ converges to $u_n(t)$ in $L^p(\Omega;\mu_n)$ uniformly for $t\in[0,T]$. Hence $u_n$ is a mild solution to \eqref{AbCauchy} on $[0,T]$ with $A=A_n$, $x=x_n$, and $X=L^p(\Omega;\mu_n)$, as desired. Since $T$ is arbitrary, $u_n$ is in fact a mild solution on $[0,+\infty)$. By part~(a) of Proposition~\ref{mildsolutionsProp}, for all $t\in[0,T]$, $u_n(t)\in \overline{D(A_n)}$, where the closure is now taken in $L^p(\Omega;\mu_n)$; thus in particular $x_n\in \overline{D(A_n)}$. To conclude the proof of our claim in \eqref{eq:claiminproof}, we recall that $\overline{A_n}= \overline{\partial_{L^p}\Phi_n}$; thus $x_n \in \overline{D(\partial_{L^p}\Phi_n)}$ and $u_n$ is a mild solution to \eqref{AbCauchy} with $A=\overline{\partial_{L^p}\Phi_n}$, $x=x_n$, and $X=L^p(\Omega;\mu_n)$. By Theorem~\ref{BenCran1991}, it holds that $\overline{\partial_{L^p}\Phi_n}$ is m-accretive on $L^p(\Omega;\mu_n)$ and thus, by Theorem~\ref{CranLig} it $u_n$ is the unique mild solution on $[0,+\infty)$. By Definition~\ref{def:semigroup} it follows that $u_n(t)=S_{\overline{\partial_{L^p}\Phi_n}}(t)x_n$, completing the proof of the claim in \eqref{eq:claiminproof}.

We now apply Lemma~\ref{mnthmlem} to deduce that, for all $t\in [0,+\infty)$, 
$$ u_n(t)= S_{\overline{\partial_{L^p} \Phi_n}}(t)x_n \underset{TL^p(\Omega)}\longtwoheadrightarrow S_{\overline{\partial_{L^p} \Phi_{\infty}}}(t)x_{\infty}=u_{\infty}(t) \qquad \text{as } n\to\infty.$$

Moreover, given $T \in [0,+\infty)$, the above convergence is uniform in $t$ on the interval $[0,T]$. Since the sequence $(u_n(t))_{n\in \mathbb{N}_\infty}$ is uniformly bounded in $L^q(\Omega; \mu_n)$ and thus also in $L^q(\Omega; \mu_n)$ (since $q>\max(2,p)$) and since, by assumption (at the start of Section~\ref{sec:TLp}) $(\mu_n)_{n\in\mathbb{N}}$ converges to $\mu_\infty$ as $n\to\infty$ in the $\max(2,p)$-Wasserstein metric, we can apply Lemma~\ref{tlplem} with $r:=\max(2,p)$ to deduce that
\begin{equation}
\label{thmEQ2}   
u_n(t) \underset{TL^{\max(2,p)}(\Omega)}\longtwoheadrightarrow u_{\infty}(t) \qquad \text{as } n\to\infty.
\end{equation}
Moreover, since the bound in Lemma~\ref{tlplem} is uniform in $t$, for every $T\in [0,+\infty)$ this convergence is uniform in $t$ on the interval $[0,T]$.

Next we apply Lemma~\ref{lem:liminf} to deduce, for all $t\in [0,+\infty)$, 
\begin{equation}
\label{thmInf}
\liminf_{n \rightarrow \infty} \Phi_n\big( u_n(t) \big) \geq \Phi_{\infty}\big( u_{\infty}(t) \big).    
\end{equation}

For notational convenience, define, for all $n\in \mathbb{N}_{\infty}$,
$f_n(t) := \Phi_n\big( u_n(t) \big)$. 
The inequality in \eqref{thmInf} becomes $$\liminf_{n \rightarrow \infty} f_n(t) \geq f_{\infty}(t).$$
Since, for all $n\in \mathbb{N}_\infty$ and all $t>0$, $u_n(t)\in L^2(\Omega;\mu_n)$, by the semigroup property of gradient flows and Theorem~\ref{GradFlowTheorem}, we have that, for all $\gamma>0$,
\begin{align*}
f_n(t+\gamma) &= \Phi_n\big( u_n(t+\gamma) \big) = \Phi_n|_{L^2(\Omega;\mu_n)}\big( u_n(t+\gamma) \big)
\leq [\Phi_n]_2^{\gamma} \big( u_n(t) \big).
\end{align*}
(We recall the notation from \eqref{eq:Moreau2}.) 
The convergence in \eqref{thmEQ2} and Lemma~\ref{tlplem2} allow us to apply Lemma~\ref{lem:envelopes} to deduce that
\begin{align*}
\limsup_{n \rightarrow \infty} f_n(t+\gamma) &\leq[\Phi_{\infty}]_2^{\gamma}\big( u_{\infty}(t) \big) 
\leq  \Phi_\infty|_{L^2(\Omega;\mu_\infty)}\big( u_\infty(t) \big)
=\Phi_{\infty}\big( u_{\infty}(t) \big)
=f_{\infty}(t),
\end{align*} 
where the second inequality follows from the bound in Remark~\ref{rem:envelopesupport}.

By the same method as that following \eqref{eq:limsupfn} at the end of the proof of Theorem~\ref{Hilthm} (where we can use $(0,+\infty)$ instead of $\mathfrak{I}_{-\lambda}$) we conclude that, for all $t>0$,
$f_n(t) \rightarrow f_{\infty}(t)$ as $n\rightarrow \infty$. We can use the same method, because, for all $n\in \mathbb{N}$, $f_n$ is non-increasing (see Remarks~\ref{rem:gradientflow} and~\ref{rem:maxslope}) and $f_\infty$ is continuous on $(0,+\infty)$ (see Remark~\ref{rem:maxslope}).
\end{proof}

Next we find conditions that imply part~(i) of Assumptions~\ref{assump} and that may be easier to check in practice. They include conditions reminiscent of $\Gamma$-convergence for the sequence $(\Phi_n)_{n\in\mathbb{N}_\infty}$, instead of the conditions on $(\mathfrak{M}[\lambda,\cdot,\Phi_n])_{n\in\mathbb{N}_\infty}$ in Assumptions~\ref{assump}. 

\begin{assumptions}\label{assumpN}
\begin{enumerate}[(i)] 
\item For all sequences $(u_n)_{n\in \mathbb{N}_\infty} \subset_n X_n$ such that $u_n \underset{TL^p(\Omega)}\longtwoheadrightarrow u_{\infty}$ as $n\rightarrow \infty$, 
$$ \liminf_{n\rightarrow \infty} \Phi_n(u_n) \geq \Phi_{\infty}(u_{\infty}).$$
\item For all $u_{\infty} \in L^{\max (2,p)}(\Omega;\mu_{\infty})$ there exists a sequence $(u_n)_{n\in \mathbb{N}} \subset_n L^{\max (2,p)}(\Omega;\mu_n)$ such that $u_n \underset{TL^{\max(2,p)}(\Omega)}\longtwoheadrightarrow u_{\infty}$ as $n\rightarrow \infty$ and
$$ \limsup_{n\rightarrow \infty} \Phi_n(u_n) \leq \Phi_{\infty}(u_{\infty}).$$
\item For all sequences $(x_n)_{n\in \mathbb{N}} \subset_n X_n$ for which the sequences $(\Vert x_n \Vert_n)_{n\in \mathbb{N}}$ and $(\Phi_n(x_n))_{n\in \mathbb{N}}$ are bounded, there exists an $x_\infty \in L^p(\Omega;\mu_\infty)$ and a subsequence $\big((x_{n_k},\mu_{n_k})\big)_{k\in\mathbb{N}}$ of $\big((x_n,\mu_n)\big)_{n\in\mathbb{N}}$ that converges to $(x_\infty,\mu_\infty)$ in $TL^p(\Omega)$.
\item For all $n\in \mathbb{N}_\infty$, $\Phi_n$ is non-negative. 
\item For all $\lambda>0$ there exists a strictly increasing function $f:[0,+\infty) \rightarrow [0,+\infty)$ such that $f(0)=0$ and $f(r)\rightarrow+\infty$ as $r\rightarrow \infty$, moreover for all $n \in \mathbb{N}$ and all $u\in X_n$,
$$ \mathfrak{N}[\lambda;\Phi_n](u) \geq f(\Vert u \Vert_n ).$$ 
\end{enumerate}
\end{assumptions}

\begin{remark}
\label{rem:assumpiv}
If part~(iv) of Assumptions~\ref{assumpN} is satisfied and $p\leq 2$, then also part~(v) of Assumptions~\ref{assumpN} holds. Indeed, for all $n\in \mathbb{N}_\infty$, $\mu_n$ is a probability measure and thus
$$ \frac{1}{2}\Vert u \Vert_n^2 \leq \frac{1}{2} \int |u|^2 \, d\mu_n \leq  \mathfrak{N}[\lambda;\Phi_n](u).  $$
The first inequality follows from H\"older's inequality, the second from the definition of $\mathfrak{N}[\lambda;\Phi_n]$ in Definition~\ref{movementeq}.
\end{remark}

For the next proposition, we recall that definitions of $\Gamma$-convergence and equicoercivity over a Banach stacking are given in Appendix~\ref{app:Gammaconvergence}.
\begin{proposition}
\label{prop:assumpreduce}
Suppose $\big(\Phi_n\big)_{n\in \mathbb{N}_{\infty}}$ satisfies Assumptions~\ref{assumpN}. Then for all sequences $(h_n)_{n\in \mathbb{N}_\infty} \subset_n L^2(\Omega;\mu_n)$ such that $h_n \underset{TL^2(\Omega)}\longtwoheadrightarrow h_{\infty}$ as $n\rightarrow \infty$ we have that, for all $\lambda>0$,
\begin{enumerate}[(a)]
    \item the sequence $(\mathfrak{M}[\lambda,h_n;\Phi_n])_{n\in \mathbb{N}_\infty}$ $\Gamma$-converges to $\mathfrak{M}[\lambda,h_{\infty};\Phi_{\infty}]$ as $n\rightarrow \infty$ over the Banach stacking $\big((X_n, \xi_n)_{n\in \mathbb{N}_\infty}, TL^p(\Omega)\big)$ and
    \item the sequence $\big( \mathfrak{M}[\lambda,h_n;\Phi_n] \big)_{n\in \mathbb{N}}$ is equicoercive over the Banach stacking $\big((X_n, \xi_n)_{n\in \mathbb{N}_\infty}, TL^p(\Omega) \big)$. 
\end{enumerate} 
In particular, part~1 from Assumptions~\ref{assump} holds with $\mathcal{A}:=L^{\max(2,p)}(\Omega;\mu_{\infty})$.
\end{proposition}
\begin{proof}
Let $(h_n)_{n\in\mathbb{N}_\infty} \subset_n L^2(\Omega;\mu_n)$ be a sequence such that $h_n \underset{TL^2(\Omega)}\longtwoheadrightarrow h_{\infty}$ as $n\rightarrow \infty$. 

First we prove (a). To establish the required inequality for the limit inferior, let $(u_n)_{n\in \mathbb{N}_\infty} \subset_n X_n$ be a sequence such that $u_n \underset{TL^p(\Omega)}\longtwoheadrightarrow u_{\infty}$ as $n\rightarrow \infty$. By Lemma~\ref{tlplem3} it follows that $((h_n-u_n) \# \mu_n)$ converges weakly to $(h_{\infty}-u_{\infty}) \# \mu_{\infty}$ as $n\rightarrow \infty$. Next we apply Skorokhod's representation theorem \cite[Theorem 6.7]{Billingsley99} to establish the existence of a probability space $(\Pi,\Sigma,\nu)$ and a sequence $(\mathcal{U}_n)_{n\in \mathbb{N}_\infty}$ of random variables $\mathcal{U}_n:\Pi \rightarrow \mathbb{R}$ such that the law of $\mathcal{U}_n$ is given by $(h_n-u_n)\# \mu_n$ and $(\mathcal{U}_n)_{n\in \mathbb{N}}$ converges pointwise to $\mathcal{U}_{\infty}$, $\nu$-a.e. on $\Pi$. Then
\begin{align*}
\liminf_{n\rightarrow \infty} \mathfrak{M}[\lambda,h_n;\Phi_n](u_n) &= \liminf_{n\rightarrow \infty} \left(  \frac{1}{2}\int_{\Omega} |h_n-u_n|^2 \, d\mu_n +\lambda \Phi_n(u_n) \right)\\
 &= \liminf_{n\rightarrow \infty} \left( \frac{1}{2}\mathbb{E}_{\nu}[|\mathcal{U}_n|^2] +\lambda \Phi_n(u_n) \right)\\
 &\geq  \frac{1}{2}\mathbb{E}_{\nu}[|\mathcal{U}_{\infty}|^2] +\lambda \Phi_{\infty}(u_{\infty})
 = \mathfrak{M}[\lambda,h_{\infty};\Phi_{\infty}](u_{\infty}).
\end{align*}
The second and third equalities hold by the law of the unconscious statistician combined with the definition of the pushforward measure and the inequality follows by Fatou's lemma (see \cite[Lemma 123B]{Fremlin1}) and part~(i) in Assumptions~\ref{assumpN}. This establishes the required inequality for the limit inferior.

To prove the inequality for the limit superior that is needed for $\Gamma$-convergence, let $u_{\infty}\in X_{\infty}=L^p(\Omega; \mu_\infty)$. First assume that $u_{\infty} \notin L^2(\Omega;\mu_{\infty})$. Let $(u_n)_{n\in \mathbb{N}} \subset_n X_n$ be a sequence such that $u_n \underset{TL^p(\Omega)}\longtwoheadrightarrow u_{\infty}$ as $n\rightarrow \infty$. Then we have
\[
 \limsup_{n\to\infty} \mathfrak{M}[\lambda,h_n;\Phi_n](u_n) \leq \mathfrak{M}[\lambda,h_{\infty};\Phi_{\infty}](u_{\infty})=+\infty.
\]
The right-hand side above is $+\infty$ because, by part~(iv) of Assumptions~\ref{assumpN}, each $\Phi_n$ is non-negative, and moreover $h_{\infty} \in  L^2(\Omega;\mu_{\infty})$ and $u_{\infty} \notin L^2(\Omega;\mu_{\infty})$, thus ${h_\infty-u_\infty \notin L^2(\Omega;\mu_{\infty})}$ and $ \int |h_{\infty} -u_{\infty}|^2 d\mu_{\infty} = +\infty$.

Now we assume that $u_\infty \in X_{\infty} \cap L^2(\Omega;\mu_{\infty})$. By part~(ii) in Assumptions~\ref{assumpN}, there exists a sequence $(u_n)_{n\in \mathbb{N}} \subset_n X_n$ such that $u_n \underset{TL^{\max(2,p)}(\Omega)}\longtwoheadrightarrow u_{\infty}$ as $n\rightarrow \infty$ and
$$ \limsup_{n\rightarrow \infty} \Phi_n(u_n) \leq \Phi_{\infty}(u_{\infty}).$$
By property~(iii) in Definition~\ref{BanStack} and Lemma~\ref{tlplem2} we have $h_n-u_n \underset{TL^2(\Omega)}\longtwoheadrightarrow h_{\infty}- u_{\infty}$ as $n\rightarrow \infty$. 
Then, by property~(iv) in Definition~\ref{BanStack}, 
$$ \int_{\Omega} |h_n -u_n|^2 \, d\mu_n \rightarrow \int_{\Omega} |h_{\infty} -u_{\infty}|^2 \, d\mu_{\infty} \qquad \text{as } n\to\infty.$$
Combining this with the earlier inequality for the limit superior of the $\Phi_n$, we deduce that
$n\rightarrow \infty$,
$$ \limsup_{n\rightarrow \infty} \mathfrak{M}[\lambda,h_n;\Phi_n](u_n) \leq  \mathfrak{M}[\lambda,h_{\infty};\Phi_{\infty}](u_{\infty}).$$
This proves $\Gamma$-convergence of $(\mathfrak{M}[\lambda,h_n;\Phi_n](u_n))$ to $\mathfrak{M}[\lambda,h_\infty;\Phi_\infty](u_\infty)$ and thus part~(a) of the Lemma.

Next we prove part~(b) of the lemma. Indeed we will show that the sequence $\big( \mathfrak{M}[\lambda,h_n;\Phi_n] \big)_{n\in \mathbb{N}}$ is equicoercive as a consequence of parts~(iii), (iv), and~(v) of Assumptions~\ref{assumpN}. Indeed, let $(u_n)_{n\in \mathbb{N}}\subset_n X_n$ be a sequence such that $(\mathfrak{M}[\lambda,h_n;\Phi_n](u_n))_{n\in \mathbb{N}}$ is bounded, say by a constant $L_1>0$. In particular, since, for all $n\in \mathbb{N}$, $\Phi_n$ is non-negative, the sequence $(\Phi_n(u_n))_{n\in\mathbb{N}}$ is bounded.

Part~(v) of Assumptions~\ref{assumpN} implies that, for all $n\in \mathbb{N}$, 
\begin{align*}
\Vert u_n \Vert_n &\leq f^{-1}( \mathfrak{N}[\lambda;\Phi_n](u_n) )
= f^{-1} \left(\frac{1}{2} \int |u_n -h_n +h_n|^2 \, d\mu_n + \Phi_n(u_n) \right) \\
&\leq f^{-1} \left( \int \big( |u_n -h_n|^2 + |h_n|^2 \big) \, d\mu_n + \Phi_n(u_n) \right)\\
&\leq f^{-1} \left( 2\mathfrak{M}[\lambda,h_n;\Phi_n](u_n) + \int |h_n|^2 \, d\mu_n \right).
\end{align*}
For all three inequalities we have used that $f^{-1}$ is non-decreasing because $f$ is. Moreover, for the third inequality, we use the non-negativity of $\Phi_n$ from part~(iv) of Assumptions~\ref{assumpN}.

Since $h_n \underset{TL^2(\Omega)}\longtwoheadrightarrow h_{\infty}$ as $n\rightarrow \infty$, we have that $\int_{\Omega} |h_n|^2 \, d\mu_n \rightarrow \int_{\Omega} |h_{\infty}|^2 \, d\mu_{\infty}$ as $n\rightarrow \infty$, by property~(iv) in Definition~\ref{BanStack}. So there exists a constant $L_2>0$ such that, for all $n\in \mathbb{N}_{\infty}$, $\int_{\Omega} |h_n|^2 \, d\mu_n \leq L_2$. Hence, for all $n\in \mathbb{N}$, 
$$ \Vert u_n \Vert_n \leq  f^{-1} \big( 2L_1 +L_2 \big).$$

Thus the sequences $(\Vert u_n \Vert_n)_{n\in \mathbb{N}}$ and $(\Phi_n(u_n))_{n\in \mathbb{N}}$ are bounded, so by part~(iii) of Assumptions~\ref{assumpN}, there exists an $x_\infty \in L^p(\Omega;\mu_\infty)$ and a subsequence $\big((x_{n_k},\mu_{n_k})\big)_{k\in\mathbb{N}}$ of $\big((x_n,\mu_n)\big)_{n\in\mathbb{N}}$ that converges to $(x_\infty,\mu_\infty)$ in $TL^p(\Omega)$. Hence the required equicoercivity holds.

Finally, let $\mathcal{A}:=L^{\max(2,p)}(\Omega;\mu_{\infty})$. Then $\mathcal{A}$ is a dense subset of $X_\infty$, because $\mathcal{A}\subset X_\infty$ (as before, since $\mu_\infty$ is a finite measure) and the set of simple functions is dense in both these spaces (see \cite[Proposition 6.7]{Folland99}). Moreover, if $h_\infty\in \mathcal{A}$ and $(h_n)_{n\in\mathbb{N}}\subset_n L^{\max(2,p)}(\Omega;\mu_n)$ is such that $h_n\underset{TL^{\max(2,p)}(\Omega)}\longtwoheadrightarrow h_{\infty}$ as $n\rightarrow \infty$, then $(h_n)_{n\in \mathbb{N}_\infty} \subset_n L^2(\Omega;\mu_n)$ and, by Lemma~\ref{tlplem2}, $h_n\underset{TL^2(\Omega)}\longtwoheadrightarrow h_{\infty}$ as $n\rightarrow \infty$. Thus, by what we have proven above, the required $\Gamma$-convergence and equicoercivity properties hold for $(\mathfrak{M}[\lambda,h_n;\Phi_n])_{n\in \mathbb{N}_\infty}$, to conclude that part~(i) of Assumptions~\ref{assump} holds. 
\end{proof}

\begin{remark}
By Lemma~\ref{tlplem2}, for all $p>1$, the topology of $TL^p(\Omega)$ is stronger than that of $TL^1(\Omega)$. So if we wish to establish the inequality for the limit inferior in part~(i) of Assumptions~\ref{assumpN}, then it suffices to do so for all sequences $(u_n)_{n\in \mathbb{N}_\infty}$ that are convergent in $TL^1(\Omega)$. Furthermore, since, for all $n\in \mathbb{N}_\infty$, $\mu_n$ is a probability measure and so $L^{\min (2,p)}(\Omega;\mu_n)\subset L^1(\Omega;\mu_n)$, it suffices to require that the functions $\Phi_n$ are $P_0$-convex on $L^1(\Omega;\mu_n)$, in place of the $P_0$-convexity condition on $L^{\min(2,p)}(\Omega;\mu_n)$ in part~(ii) of Assumptions~\ref{assump}.
\end{remark}

\subsection{Contributions of of Theorems~\ref{Hilthm} and~\ref{thm:P0theorem1}}

We conclude Section~\ref{sec:convresolvents} by summarising the main contributions of Theorem~\ref{Hilthm} and Theorem~\ref{thm:P0theorem1}. 
\begin{itemize}
    \item We use the bound on $\Phi(u(t))$ by a Moreau envelope from Theorem~\ref{GradFlowTheorem}, combined with $\Gamma$-convergence, in the proofs of Theorems \ref{Hilthm} and \ref{thm:P0theorem1} as a novel proof strategy to deduce convergence of the values of the gradient flows. This allows us to avoid the requirement of a well-preparedness assumption on the initial conditions such as can be found in \cite{EvoEq} and \cite{SerfatyGam}. The only assumption we require in these theorems for the initial conditions is that they lie in the closure of the effective domains of the functionals and possibly a uniform bound on their $L^q$-norms for some sufficiently large $q$. 
    \item We use a generalized Brezis--Pazy theorem (Theorem~\ref{theorem1}) to conclude uniform convergence of the gradient flows. In Section~\ref{sec:literature} we have discussed the results that other works obtained. To the best of our knowledge, no other work uses a Brezis--Pazy-type theorem to establish such results.
    \item Both theorems are formulated in the newly introduced setting of a Banach stacking. Theorem~\ref{Hilthm} is applicable to Banach stackings of Hilbert spaces, while Theorem~\ref{thm:P0theorem1} deals with the specific Banach stacking $TL^p(\Omega)$ that was introduced in \cite{garcia2016continuum}. 
    \item In Theorem~\ref{thm:P0theorem1} we show one strategy to overcome the mismatch (as illustrated by the total-variation example in Section~\ref{sec:convgradflowintro}) between the topology in which compactness holds ($TL^p(\Omega)$ with $p<2$) and the topology with respect to which the gradient flow is taken ($TL^2(\Omega)$) by strengthening our convexity assumption; namely by using the notion of convexity introduced in \cite{benilan1991comp} (which went unnamed in \cite{benilan1991comp}, but we refer to as $P_0$-convexity). 
\end{itemize}

\section*{}

\paragraph{Acknowledgements}

The authors acknowledge receiving funding from the Dutch Research Council (NWO) via its Open Competition Domain Science (ENW) - M.

\appendix

\section{Some basic properties of operators}\label{app:basicproperties}

\begin{remark}\label{rem:restricteddomain}
    The restricted operator $A|_Y$ in Definition~\ref{def:restriction} is an operator on a set $X$ and an operator on $Y$. Moreover, 
    \begin{align*}
D(A|_Y) &= \{x\in X \,| \, \text{there exists a } y\in X \text{ such that } (x,y)\in A|_Y\}\\
&= \{x\in Y \,| \, \text{there exists a } y\in Y \text{ such that } (x,y)\in A\}.
    \end{align*}
    Thus $D(A|_Y) \subset D(A) \cap Y$. The converse inclusion does not hold in general. For example, if $X:=\{0,1\}$, $Y:=\{0\}$, and $A:=\{(0,1), (1,1)\}$, then $D(A) = X$, hence $D(A)\cap Y = Y$, yet $A|_Y = \emptyset$, so $D(A|Y)=\emptyset \not\supset D(A)\cap Y$.

    If, however, $x\in D(A)\cap Y$ is such that $A(x)\cap Y \neq\emptyset$, then there exists a $y\in Y$ such that $(x,y)\in A$ and thus $x\in D(A|_Y)$.

Similarly, we have
\begin{align*}
    \operatorname{Range}(A) &= \{y\in X \,| \, \text{there exists an } x\in X \text{ such that } (x,y)\in A\}\\
    \operatorname{Range}(A|_Y) &= \{y\in Y \,| \, \text{there exists an } x\in Y \text{ such that } (x,y)\in A\}.
\end{align*}
Thus $\operatorname{Range}(A|_Y)\subset \operatorname{Range}(A)\cap Y$. The converse does not hold in general. As a counterexample, we can take the same sets $X$ and $Y$ as above, with $A:=\{(1,0),(1,1)\}$. Then $A|_Y=\emptyset$, thus $\operatorname{Range}(A) = X$ and $\operatorname{Range}(A)\cap Y=Y$, yet ${\operatorname{Range}(A|_Y)=\emptyset \not\supset \operatorname{Range}(A)\cap Y}$.

If, however, $y\in \operatorname{Range}(A)\cap Y$ is such that $A^{-1}\cap Y \neq\emptyset$, then $y\in \operatorname{Range}(A|_Y)$.
    
\end{remark}

\begin{remark}\label{rem:domainrange}
If $A$ is an operator on a set $X$, then ${D(A^{-1})=\operatorname{Range}(A) \subset X}$. From Definition~\ref{def:inverse} it follows that $(A^{-1})^{-1}=A$ and thus also ${\operatorname{Range}(A^{-1}) = D(A) \subset X}$.

Moreover, if $B$ is also an operator on $X$ with $A\subset B$, then $D(A)\subset D(B)$ and $\operatorname{Range}(A) \subset \operatorname{Range}(B)$.
\end{remark}

\begin{lemma}\label{lem:restrictionofinverse}
    Let $A$ be an operator on a set $X$ and $Y\subset X$. Then $(A|_Y)^{-1} = (A^{-1})|_Y$.
\end{lemma}
\begin{proof}
    From the definitions of the restricted operator and the inverse it follows that
    \begin{align*}
        (A|_Y)^{-1} &= \{(x,y)\in X\times X| (y,x) \in A \cap (Y\times Y)\}\\
        &= \{(x,y)\in X \times X| (y,x)\in A\} \cap (Y\times Y) = (A^{-1})|_Y.
    \end{align*}
\end{proof}

\begin{remark}\label{rem:inverseofrestriction}
Lemma~\ref{lem:restrictionofinverse} implies that, if $(x,y)\in Y\times Y$, then $(x,y)\in (A|_Y)^{-1}$ if and only if $(x,y)\in A^{-1}$. In particular, since $D((A|_Y)^{-1})\subset Y$ and $\operatorname{Range}((A|_Y)^{-1})\subset Y$, if $x\in D((A|_Y)^{-1})$, then, for all $y\in (A|_Y)^{-1}(x)$, $(x,y) \in A^{-1}$.
\end{remark}

In the remainder of Appendix~\ref{app:basicproperties}, $X$ will denote a topological vector space.

\begin{remark}\label{rem:closureinclusions}
From Definition~\ref{def:pointwiseclosure} it is immediate that if $A$ and $B$ are operators on $X$ with $A \subset B$, then $\overline{A}\subset \overline{B}$. Moreover, for all $x\in X$, $A(x) \subset \overline{A}(x)$; hence also $D(A)\subset D(\overline{A})$. 

Furthermore, for all $x\in X$, $\overline{A(x)} \subset \overline{A}(x)$. Indeed, if $y\in \overline{A(x)}$, then there exists a sequence $(y_n)_{n\in \mathbb{N}} \subset A(x)$ such that $y_n \to y$, as $n\to \infty$. If we define, for all $n\in \mathbb{N}$, $x_n:=x$, then, for all $n\in \mathbb{N}$, $y_n\in A(x_n)$, and $(x_n,y_n)\to (x,y)$ as $n\to \infty$; hence, $y\in \overline{A}(x)$ by Definition~\ref{def:pointwiseclosure}.

The opposite inclusion does not always hold, as the following counterexample shows. Let $X=\mathbb{R}$ endowed with the Euclidean topology. Define $A\subset \mathbb{R} \times \mathbb{R}$ by, for all $x\in X$,
\[
A(x) := \begin{cases}
    \{x\}, &\text{if } x\neq 0,\\
    \{1\}, &\text{if } x=0,
\end{cases}
\]
and, for all $n\in\mathbb{N}$, $x_n:=y_n:=\frac1n$. Then, for all $n\in \mathbb{N}$, $y_n\in A(x_n)$, and $(x_n,y_n)\to (0,0)$ as $n\to \infty$. Hence $0\in \overline{A}(0)$. However, $\overline{A(0)}=\{1\}$ and thus $0 \not\in \overline{A(0)}$.
\end{remark}

\begin{remark}\label{rem:domainsofclosures}
The set $D(\overline{A})$ can be characterised as the set of all $x\in X$ such that there exist a $y\in X$ and sequences $(x_n)_{n\in\mathbb{N}}$ and $(y_n)_{n\in\mathbb{N}}$ in $X$ that converge to $x$ and $y$, respectively, such that, for all $n\in \mathbb{N}$, $(x_n,y_n)\in A$. Moreover, $\overline{D(A)}$ is the set of all $x\in X$ such that there exists a sequence $(x_n)_{n\in\mathbb{N}}$ in $X$ that converges to $x$ and a sequence $(y_n)_{n\in\mathbb{N}}$ in $X$, such that, for all $n\in \mathbb{N}$, $(x_n,y_n)\in A$. It follows that $D(\overline{A}) \subset \overline{D(A)}$. 

That the converse inclusion does not not hold, is shown by the following counterexample. Define $A:= \{(1/n, n) \in \mathbb{R}\times\mathbb{R}: n \in \mathbb{N}\}$. Then $\overline{A}=A$ and thus $D(\overline{A})=D(A) = \{1/n\in \mathbb{R}: n\in \mathbb{N}\}$. However, $\overline{D(A)} = D(A) \cup \{0\} \not\subset D(\overline{A})$.

From $D(\overline{A}) \subset \overline{D(A)}$ it follows that $\overline{D(\overline{A})} \subset \overline{D(A)}$. Moreover, it is true that $\overline{D(A)} \subset \overline{D(\overline{A})}$. Indeed, $\overline{D(\overline{A})}$ is the set of all $x\in X$ such that there exist sequences $(x_n)_{n\in\mathbb{N}}$ and $(y_n)_{n\in\mathbb{N}}$ in $X$ and, for all $n\in \mathbb{N}$, a sequence $(z_{n_k})_{k\in \mathbb{N}}$ in $A(x_n)$, such that $x_n \to x$ as $n\to\infty$ and $z_{n_k} \to y_n$ as $k\to \infty$. As we saw above, if $x\in \overline{D(A)}$, then there exist sequences $(x_n)_{n\in\mathbb{N}}$ and $(y_n)_{n\in\mathbb{N}}$ in $X$ such that $x_n\to x$ as $n\to\infty$ and, for all $n\in \mathbb{N}$, $(x_n,y_n)\in A$. By defining, for all $n,k\in \mathbb{N}$, $z_{n_k}:=y_n$, we obtain the sequence $(z_{n_k})_{k\in\mathbb{N}}$ necessary to conclude that $x\in \overline{D(\overline{A})}$.
\end{remark}

\begin{lemma}\label{lem:inverseclosure}
    If $A$ is an operator on $X$, then $\overline{A^{-1}} = (\overline{A})^{-1}$.
\end{lemma}
\begin{proof}
    It follows from Definitions~\ref{def:inverse} and~\ref{def:pointwiseclosure} that $(x,y) \in \overline{A^{-1}}$ if and only if there exist sequences $(x_n)_{n\in\mathbb{N}}$ and $(y_n)_{n\in\mathbb{N}}$ in $X$ that converge to $x$ and $y$, respectively, and that satisfy, for all $n\in \mathbb{N}$, $(y_n,x_n)\in A$. A similar unfolding of the definitions shows that the same condition is equivalent to $(x,y)\in (\overline{A})^{-1}$.
\end{proof}

\begin{remark}
    Because $X$ is a vector space, the space of all operators on $X$ with summation and scalar multiplication as in Definition~\ref{def:operatorsoperations} is a vector space. Moreover, if $A$, $B$, and $C$ are operators on $X$, $x\in X$, $\lambda \in \mathbb{R}$, and $C\subset B$, then $(A+C)(x)\subset (A+B)(x)$ and $\lambda C(x) \subset \lambda B(x)$.
\end{remark}

\begin{proposition}
\label{Closureprop}
Let $A$ be an operator on $X$, then
$$ \overline{I_X+A} = I_X+\overline{A}.$$
Moreover, if $\lambda\in \mathbb{R}$, then $\lambda \overline{A} \subset \overline{\lambda A}$. If $\lambda\neq 0$, then also $\overline{\lambda A} \subset \lambda \overline{A}$.
\end{proposition}

\begin{proof}
Let $y\in \overline{I_X+A}(x)$, i.e., there exist sequences $(x_n)_{n\in\mathbb{N}}, (y_n)_{n\in\mathbb{N}}\subset X$ such that $x_n \rightarrow x$ and $y_n \rightarrow y$ as $n\to \infty$ and, for all $n\in \mathbb{N}$, $y_n \in (I_X+A)(x_n)$. So, for every $n\in \mathbb{N}$ there exists an $a_n \in A(x_n)$ such that $y_n = x_n +a_n$. By this equality, $(a_n)_{n\in\mathbb{N}}$ converges to $y-x$ and thus $y-x\in \overline{A}(x)$. Therefore $y\in (I_X+\overline{A})(x)$ and 
$$ \overline{I_X+A} \subset I_X+\overline{A}.$$

To prove the converse inclusion now let $y\in (I_X+\overline{A})(x)$. So there exists an $a\in \overline{A}(x)$ such that $y=a+x$. Then there exist sequences $(x_n)_{n\in\mathbb{N}}, (a_n)_{n\in\mathbb{N}}\subset X$ such that $(x_n)$ converges to $x$, $(a_n)$ converges to $a$, and, for all $n\in \mathbb{N}$, $a_n \in A(x_n)$. It follows that $x_n+a_n \in (I_X+A)(x_n)$ and, since $(x_n+a_n)$ converges to $x+a$, we conclude that $y=x+a \in \overline{I_X +A}(x)$, as required.

To prove the second claim in the proposition, let $\lambda \in \mathbb{R}$. Then $(x,y)\in \lambda \overline{A}$ if and only if there exists a $z\in X$ such that $y=\lambda z$ and $(x,z)\in \overline{A}$. This in turn is equivalent to the existence of a $z\in X$ and of sequences $(x_n)_{n\in\mathbb{N}}$ and $(z_n)_{n\in\mathbb{N}}$ in $X$ such that $y=\lambda z$, $x_n\to x$ and $z_n \to z$ as $n\to \infty$, and, for all $n\in \mathbb{N}$, $(x_n,z_n)\in A$. We call this latter statement condition I.

On the other hand, $(x,y)\in \overline{\lambda A}$ if and only if there exist sequences $(\tilde x_n)_{n\in\mathbb{N}}$ and $(\tilde y_n)_{n\in\mathbb{N}}$ in $X$ such that $\tilde x_n\to x$ and $\tilde y_n\to y$ as $n\to \infty$ and, for all $n\in \mathbb{N}$, $(\tilde x_n,\tilde y_n)\in \lambda A$. This is equivalent to the existence of sequences $(\tilde x_n)_{n\in\mathbb{N}}$, $(\tilde y_n)_{n\in\mathbb{N}}$, and $(\tilde z_n)_{n\in\mathbb{N}}$ in $X$ such that $\tilde x_n\to x$ and $\tilde y_n\to y$ as $n\to \infty$ and, for all $n\in \mathbb{N}$, $(\tilde x_n,\tilde z_n)\in A$ and $\tilde y_n=\lambda \tilde z_n$. We name this latter statement condition II.

Assume that $(x,y)\in \lambda \overline{A}$. Let $z$ be the element of $X$ and $(x_n)_{n\in\mathbb{N}}$ and $(z_n)_{n\in\mathbb{N}}$ the sequences in $X$ whose existence is guaranteed by condition I. By defining, for all $n\in \mathbb{N}$, $\tilde x_n := x_n$, $\tilde y_n := \lambda z_n$ and $\tilde z_n := z_n$, we obtain sequences $(\tilde x_n)$, $(\tilde y_n)$, and $(\tilde z_n)$ that satisfy condition II, hence $(x,y)\in \overline{\lambda A}$.

Now instead assume that $\lambda \neq 0$ and $(x,y) \in \overline{\lambda A}$. Let $(\tilde x_n)_{n\in\mathbb{N}}$, $(\tilde y_n)_{n\in\mathbb{N}}$, and $(\tilde z_n)_{n\in\mathbb{N}}$ be the sequences in $X$ whose existence is guaranteed by condition II. Define $z:= \lambda^{-1} y$ and, for all $n\in \mathbb{N}$, $x_n:=\tilde x_n$ and $z_n:= \tilde z_n$. Then $x_n \to x$ and $z_n = \lambda^{-1} \tilde y_n \to \lambda^{-1} y = z$ as $n\to \infty$. Moreover, for all $n\in \mathbb{N}$, $(x_n,z_n)\in A$. Thus condition I is satisfied and therefore $(x,y) \in \lambda \overline{A}$.
\end{proof}

\begin{remark}\label{rem:countertomultiplication}
    Proposition~\ref{Closureprop} shows that, if $\lambda \neq 0$, then $\overline{\lambda A} \subset \lambda \overline{A}$. The same counterexample as in Remark~\ref{rem:domainsofclosures} shows that this is not necessarily the case if $\lambda=0$.

    For the operator $A$ defined in that remark we know that $\overline{A}=A$ and thus $0 \overline{A} = 0 A = {\{(1/n,0) \in \mathbb{R}\times\mathbb{R}: n \in \mathbb{N}\}}$. On the other hand $\overline{0 A} = 0 A \cup \{0,0\} \not\subset 0 \overline{A}$.    
\end{remark}

\section{Properties of \texorpdfstring{$\lambda$}{lambda}-convex functions and subdifferentials}\label{sec:C}

This appendix is dedicated to the statements and proofs of some known properties of $\lambda$-convex functions (Definition~\ref{def:lambdaconvex}) and subdifferentials (Definitions~\ref{def:subdifferential} and~\ref{def:subdiflambda}) that we need in this paper. Throughout we assume $H$ to be a Hilbert space with inner product $\langle \cdot , \cdot \rangle$ and corresponding norm $\| \cdot \|$.

\begin{lemma}\label{lem:lambdaconvexequivalent}
Let $\Phi:H\rightarrow (-\infty,+\infty]$ and $\lambda \in \mathbb{R}$. The function $\Phi$ is $\lambda$-convex if and only if the function $\Phi-\frac{\lambda}{2}\Vert \cdot \Vert^2$ is convex.
\end{lemma}
\begin{proof}
Let $x, y\in H$ and $t\in [0,1]$. As a preliminary observation, we note that
    \begin{align*}
\|tx + (1-t)y\|^2 - t\|x\|^2 - (1-t) \|y\|^2 &= (t^2-t) \|x\|^2 + \big((1-t)^2-(1-t)\big) \|y\|^2\\
&\hspace{0.4cm} + 2 t (1-t) \langle x, y\rangle\\
&= -t(1-t) \left(\|x\|^2 + \|y\|^2- 2 \langle x, y\rangle \right)\\
&= -t(1-t) \|x-y\|^2.
    \end{align*}
Thus the definitional inequality for $\lambda$-convexity of $\Phi$ in Definition~\ref{def:lambdaconvex} is equivalent to
\[
\Phi( tx + (1-t)y) \leq t\Phi(x) +(1-t) \Phi(y) + \frac{1}{2} \lambda \left(\|tx + (1-t)y\|^2 - t\|x\|^2 - (1-t) \|y\|^2\right),
\]
which, through reordering of terms, is itself equivalent to the definitional inequality for convexity of $\Phi-\frac{\lambda}2\|\cdot\|$:
\[
\Phi( tx + (1-t)y) - \frac{\lambda}2\|tx + (1-t)y\|^2 \leq t \left(\Phi(x)-\frac{\lambda}2\|x\|^2\right) +(1-t) \left(\Phi(y)-\frac{\lambda}2 \|y\|^2\right).
\]
\end{proof}

\begin{lemma}\label{lem:lambdaconvexanstrictlyconvex}
Let $\gamma,\lambda\in \mathbb{R}$ and let $\Phi:H\rightarrow (-\infty,+\infty]$ be $\lambda$-convex. If $\gamma\leq \lambda$, then $\Phi$ is also $\gamma$-convex. In particular, if $\lambda \geq 0$, then $\Phi$ is $0$-convex, i.e. convex. Moreover, if instead $\lambda>0$, then $\Phi$ is strictly convex.
\end{lemma}
\begin{proof}
This is a direct consequence of the definitional inequality for $\lambda$-convexity of $\Phi$ in Definition~\ref{def:lambdaconvex}.   
\end{proof}

\begin{lemma}
\label{subdiffLem}
Let $\Phi:H \rightarrow (-\infty,+\infty]$ be convex and lower semicontinuous. Then $(x,y)\in \partial \Phi$ if and only if
$$ x\in \underset{w\in H}{\argmin} \big( \Phi(w) - \langle y,w \rangle \big) .$$
\end{lemma}
\begin{proof}
We note that $(x,y)\in \partial\Phi$ if and only if, for all $w\in H$,
$\Phi(w) \geq \Phi(x) + \langle w-x,y \rangle$, which is equivalent to, for all $w\in H$,
$\Phi(w) - \langle w,y \rangle \geq \Phi(x) - \langle x,y \rangle$. 
This proves the proposition.
\end{proof}

\begin{proposition}
\label{subdiffProp}
Let $\Phi:H \rightarrow (-\infty,+\infty]$ be proper, convex, and lower semicontinuous. Then the following hold.
\begin{enumerate}
    \item The function $\Phi$ is continuous on the interior of $\operatorname{dom}(\Phi)$.
    \item The function $\Phi$ is lower semicontinuous with respect to the weak topology on $H$.
    \item For all $x$ in the interior of $\operatorname{dom}(\Phi)$, $\partial \Phi(x)$ is non-empty.
    \item The set $D(\partial \Phi)$ is a subset of $\operatorname{dom}(\Phi)$ and dense in $\operatorname{dom}(\Phi)$.
\end{enumerate}
\end{proposition}

\begin{proof}
Statements~1 and~3 follow from \cite[Proposition 16.27]{Bauschke}. Statement~4 follows from  \cite[Corollary 16.39]{Bauschke} and statement~2 is established by \cite[Theorem 9.1]{Bauschke}. 
\end{proof}

\begin{lemma}
\label{ConvexMin}
Let $\lambda>0$ and let $\Phi: H \rightarrow (-\infty,+\infty]$ be $\lambda$-convex and lower semicontinuous. Then $\inf_{x\in H}\Phi(x) > -\infty$ and $\Phi$ attains its minimum at a unique $x_*\in H$.
\end{lemma}
\begin{proof}
By the assumptions and Lemma~\ref{lem:lambdaconvexequivalent}, $\Phi - \frac{\lambda}{2}\Vert\cdot \Vert^2$ is convex and lower semicontinuous. Thus, by statement 3 in Proposition~\ref{subdiffProp}, $D(\partial(\Phi - \frac{\lambda}{2}\Vert\cdot \Vert^2))$ is not empty. If we fix $(u,v)\in \partial(\Phi - \frac{\lambda}{2}\Vert\cdot \Vert^2)$, then, for all $x\in H$, 
\begin{align*}
\Phi(x) &= \Phi(x) - \frac{\lambda}{2}\Vert x \Vert^2 + \frac{\lambda}{2}\Vert x \Vert^2
\geq \Phi(u) - \frac{\lambda}{2}\Vert u\Vert^2 +\langle v,x-u \rangle + \frac{\lambda}{2}\Vert x \Vert^2\\
&= \Phi(u) - \frac{\lambda}{2}\Vert u\Vert^2 -\langle v,u \rangle +\frac{\lambda}{2} \Vert x + \frac{1}{\lambda}v \Vert^2 - \frac{1}{2\lambda} \Vert v \Vert^2\\
&\geq \Phi(u) - \frac{\lambda}{2}\Vert u\Vert^2 - \frac{1}{2\lambda} \Vert v \Vert^2 -\langle v,u \rangle.
\end{align*}
Therefore
$$ \inf_{x\in H} \Phi(x) \geq \Phi(u) - \frac{\lambda}{2}\Vert u\Vert^2 - \frac{1}{2\lambda} \Vert v \Vert^2 -\langle v,u \rangle > -\infty .$$
Moreover, our calculation shows that $\Phi$ is coercive, i.e. $\Phi(x) \to \infty$ as $\|x\|\to \infty$, since it is bounded below by the coercive function $x\mapsto \frac{\lambda}{2} \Vert x + \frac{1}{\lambda}v \Vert^2 + \Phi(u) - \frac{\lambda}{2}\Vert u\Vert^2 -\langle v,u \rangle - \frac{1}{2\lambda} \Vert v \Vert^2$. 

Let $(x_n)_{n\in \mathbb{N}} \subset H$ be a sequence such that $\Phi(x_n)\rightarrow \inf_{x\in H} \Phi(x)$ as $n \rightarrow \infty$. By coercivity of $\Phi$, $(x_n)$ is bounded in $H$ and so, by the Banach--Alaoglu--Bourbaki theorem \cite[Theorem 3.16]{BrezisBook}, there exist an $x_* \in H$ and a subsequence $(x_{n_k})_{k\in \mathbb{N}}$ of $(x_n)$ such that $(x_{n_k})$ converges weakly in $H$ to $x_*$. By part 2 of Proposition~\ref{subdiffProp}, $\Phi$ is weakly lower semicontinuous and thus
$$ \Phi(x_*) \leq \underset{k \rightarrow \infty}{\liminf} \Phi(x_{n_k}) = \inf_{x\in H} \Phi(x). $$
Hence the infimum is attained. Uniqueness of the minimum follows by strict convexity of $\Phi$. Indeed, by Lemma~\ref{lem:lambdaconvexanstrictlyconvex}, $\Phi$ is strictly convex and thus, if the infimum is attained at both $x^1\in H$ and $x^2\in H$ with $x^1\neq x^2$ and $t\in (0,1)$, then
\[
\Phi(tx^1+(1-t)x^2) < t \Phi(x^1) + (1-t) \Phi(x^2) = \inf_{x\in H} \Phi(x).
\]
Since $tx^1+(1-t)x^2\in H$, this is a contradiction and hence $x^1=x^2$.
\end{proof}

From this result we deduce a useful lower bound for $\lambda$-convex functions.

\begin{corollary}
\label{cor:lambdaconvexbound}
Let $\lambda>0$ and let $\Phi: H \rightarrow (-\infty,+\infty]$ be $\lambda$-convex, proper, and lower semicontinuous. Then there exists an $x^* \in H$ such that, for all $x\in H$,
$$ \Phi(x) \geq \frac{\lambda}{2} \Vert x -x^* \Vert^2 + \Phi(x^*).$$ 
Moreover, this $x^*$ is the unique minimizer of $\Phi$ over $H$.
\end{corollary}

\begin{proof}
By Lemma~\ref{ConvexMin}, $\Phi$ is bounded below and has a unique minimizer at $x^*\in H$. Let $\delta \in (0, \frac{\lambda}{2})$. We note that $\Phi(\cdot) - \delta\Vert\cdot -x^*\Vert^2$ is $(\lambda-2\delta)$-convex. Thus, by Lemma~\ref{lem:lambdaconvexanstrictlyconvex}, $\Phi(\cdot) - \delta\Vert\cdot -x^*\Vert^2$ is strictly convex and hence, by Lemma~\ref{ConvexMin}, it has a unique minimizer $z$.
Assume for a contradiction that $z\neq x^*$. By uniqueness of the minimizer of  $\Phi(\cdot) - \delta\Vert\cdot -x^*\Vert^2$ we have 
$$ \Phi(z) - \delta \Vert z- x^*\Vert^2 < \Phi(x^*) -\delta\Vert x^* - x^*\Vert^2= \Phi(x^*).$$
If $t:= \frac{\delta}{\lambda/2} \in (0,1)$, then
\begin{align*}
\Phi(tx^*+(1-t)z) &\leq t\Phi(x^*) +(1-t)\Phi(z) -\frac{\lambda}{2} t(1-t)\Vert z-x^*\Vert^2 \\
&= t \Phi(x^*) +(1-t)[\Phi(z)-\delta\Vert z-x^* \Vert^2] \\
&< t\Phi(x^*) + (1-t)\Phi(x^*)
= \Phi(x^*).
\end{align*}
The first inequality above follows by $\lambda$-convexity of $\Phi$ and the equality $$t\|x^*\|^2+(1-t)\|z\|^2-\|tx^*+(1-t)z\|^2 = t(1-t) \|z-x^*\|^2.$$

Since $x^*$ is the minimizer of $\Phi$, this gives us a contradiction and we deduce that $z=x^*$. Hence the minimizer of $\Phi(\cdot) - \delta\Vert\cdot -x^*\Vert^2$ is $x^*$. It follows that, for all $x\in H$, 
$$\Phi(x) - \delta\Vert x -x^*\Vert^2 \geq \Phi(x^*) - \delta\Vert x^* -x^*\Vert^2 = \Phi(x^*) .$$
The above result holds for all $\delta \in (0,\frac{\lambda}{2})$ so we can take the limit as $\delta \uparrow \frac{\lambda}{2}$ to conclude our result.
\end{proof}

\begin{lemma}\label{lem:subdiffinclusion}
    Let $\lambda_1, \lambda_2 \in \mathbb{R}$ with $\lambda_2 \geq \lambda_1$ and $\Phi: H \rightarrow (-\infty,+\infty]$.
    \begin{enumerate}
        \item Then $\partial^{\lambda_2} \Phi \subset \partial^{\lambda_1} \Phi$.
        \item If $\Phi$ is $\lambda_2$-convex, then $\partial^{\lambda_2} \Phi = \partial^{\lambda_1} \Phi$.
        \end{enumerate}
\end{lemma}
\begin{proof}
\begin{enumerate}
\item  Let $(x,y)\in \partial^{\lambda_2} \Phi$. By Definition~\ref{def:subdiflambda}, there exists a $y_2\in H$ such that $y=y_2 + \lambda_2 x$ and, for all $h\in H$, 
    \[
\Phi(x+h) -\frac{\lambda_2}2 \|x+h\|^2 \geq \Phi(x) -\frac{\lambda_2}2 \|x\|^2 + \langle y_2, h\rangle.
    \]
Define $y_1 := y-\lambda_1 x = y_2 + (\lambda_2-\lambda_1) x$, such that $y=y_1+\lambda_1 x$. Let $h\in H$, then
\begin{align*}
-\frac{\lambda_2}2 \|x\|^2 + \langle y_2, h\rangle + \frac{\lambda_2-\lambda_1}2 \|x+h\|^2 &= -\frac{\lambda_1}2 \|x\|^2 + \frac{\lambda_2-\lambda_1}2 \|h\|^2 + \langle y_2 + (\lambda_2-\lambda_1) x, h \rangle\\
&\geq -\frac{\lambda_1}2 \|x\|^2 + \langle y_1, h \rangle,
\end{align*}
where the inequality holds by the definition of $y_1$ and $\frac{\lambda_2-\lambda_1}2\geq 0$. Thus, for all $h\in H$,
\begin{align*}
\Phi(x+h) - \frac{\lambda_1}2 \|x+h\|^2 &= \Phi(x+h) - \frac{\lambda_2}2 \|x+h\|^2 + \frac{\lambda_2-\lambda_1}2 \|x+h\|^2\\
&\geq \Phi(x) - \frac{\lambda_2}2 \|x\|^2 + \langle y_2, h\rangle + \frac{\lambda_2-\lambda_1}2 \|x+h\|^2\\
&\geq \Phi(x) - \frac{\lambda_1}2 \|x\|^2 + \langle
y_1, h \rangle.
\end{align*}
Thus $(x,y)\in \partial^{\lambda_1} \Phi$.

\item Let $(x,y)\in \partial^{\lambda_1} \Phi$. Then, for all $h\in H$,
\begin{align*}
\Phi(x+h) -\frac{\lambda_2}2 \|x+h\|^2 &= \Phi(x+h) - \frac{\lambda_1}2 \|x+h\|^2 - \frac{\lambda_2-\lambda_1}2 \|x+h\|^2\\
& \geq \Phi(x) - \frac{\lambda_1}2 \|x\|^2 + \langle y-\lambda_1 x, h\rangle - \frac{\lambda_2-\lambda_1}2 \|x+h\|^2\\
&= \Phi(x) - \frac{\lambda_2}2 \|x\|^2 + \langle y-\lambda_2 x, h\rangle - \frac{\lambda_2-\lambda_1}2 \|h\|^2,
\end{align*}
where the inequality follows from the definition of $\partial^{\lambda_1} \Phi$ and the last equality holds since
\begin{align*}
&\hspace{0.4cm} -\frac{\lambda_1}2\|x\|^2 + \langle y-\lambda_1 x, h\rangle - \frac{\lambda_2-\lambda_1}2 \|x+h\|^2\\
&= \langle y-\lambda_1 x, h\rangle - \frac{\lambda_2}2 \left(\|x\|^2+\|h\|^2\right) - \lambda_2 \langle x, h\rangle + \frac{\lambda_1}2 \|h\|^2 + \lambda_1 \langle x, h\rangle\\
&= -\frac{\lambda_2}2 \|x\|^2 + \langle y-\lambda_2 x, h\rangle 
- \frac{\lambda_2-\lambda_1}2 \|h\|^2.
\end{align*}
For notational convenience write $f(x):=\Phi(x) - \frac{\lambda_2}2\|x\|^2$, then we have, for all $h\in H$,
\[
f(x+h) \geq f(x) + \langle y-\lambda_2 x, h\rangle - \frac{\lambda_2-\lambda_1}2 \|h\|^2.
\]
Fix a $h \in H$, then, for all $t\in \mathbb{R}$, $th\in H$ and thus
\[
f(x+th) \geq f(x) + t \langle y-\lambda_2 x, h\rangle - \frac{\lambda_2-\lambda_1}2 t^2 \|h\|^2.
\]
Since $\Phi$ is $\lambda_2$-convex, $f$ is convex, and thus, for all $t\in [0,1]$,
\[
f(x+th) = f\big((1-t)x+t(h+x)\big) \leq (1-t) f(x) + t f(x+h).
\]
Combining the preceding two inequalities, we obtain, for all $t\in [0,1]$,
\[
(1-t) f(x) + t f(x+h) \geq f(x) + t \langle y-\lambda_2 x, h\rangle - \frac{\lambda_2-\lambda_1}2 t^2 \|h\|^2,
\]
hence
\[
t f(x+h)-tf(x) \geq t \langle y-\lambda_2 x, h\rangle - \frac{\lambda_2-\lambda_1}2 t^2 \|h\|^2.
\]
Dividing by $t>0$ and taking the limit $t\downarrow 0$ gives
\[
f(x+h)-tf(x) \geq \langle y-\lambda_2 x, h\rangle,
\]
which implies that $(x,y)\in \partial^{\lambda_2} \Phi$ and thus, in combination with part 1 of this lemma, $\partial^{\lambda_2} \Phi = \partial^{\lambda_1} \Phi$.
\end{enumerate}
\end{proof}

\begin{remark}
   Without the additional assumption of $\lambda_2$-convexity of $\Phi$, part 2 of Lemma~\ref{lem:subdiffinclusion} is not true.  For example, if $\lambda_1=0<\lambda_2$ and $\Phi(x)=0$ (which is easily seen not to be $\lambda_2$-convex), then $(0,0)\in \partial^{\lambda_1} \Phi = \partial \Phi$, but $(0,0)\not\in \partial^{\lambda_2} \Phi$, because for, $(x,y)=(0,0)$, $\Phi(x+h)-\frac{\lambda_2}2 \|x+h\|^2 = -\frac{\lambda_2}2 \|h\|^2$ which for all nonzero $h\in H$ is strictly smaller than $\Phi(x)-\frac{\lambda_2}2\|x\|^2 + \langle y-\lambda_2 x, h \rangle = 0$.
\end{remark}

\section{Lower bound on the metric derivatives}\label{sec:B}

This section is dedicated to proving Proposition~\ref{prop:derivativelowbound}, i.e. that the lower bound on the metric derivatives, a natural assumption introduced in \cite[Theorem 2]{SerfatyGam}, holds in the setting of a Banach stacking.

\begin{definition}
\label{def:ACcurves}
Let $T>0$ and let $(M,d)$ be a complete metric space. We say that a curve $u:(0,T)\rightarrow M$ is absolutely continuous, denoted as $u\in AC(0,T;M)$, if there exists a non-negative function $g\in L^1(0,T)$ such that, for all $s,t \in (0,T)$ with $s\leq t$,
\begin{equation}
\label{eq:ACadmissable}
 d(u(s),u(t))\leq \int_s^t g(r)\, dr .   
\end{equation}
\end{definition}

For absolutely continuous curves we are able to define pointwise (almost everywhere) a derivative analogous to the magnitude of a tangent to a curve. We do so as part of the next theorem.

\begin{theorem}
\label{thm:metderivative}
Let $T>0$ and let $(M,d)$ be a complete metric space. Assume ${u\in AC(0,T;M)}$. For a.e. $t\in(0,T)$ the limit
$$ |u'|(t) := \lim_{s\rightarrow t} \frac{d(u(s),u(t))}{|s-t|}$$
exists.
Moreover, the function $t\mapsto |u'|(t)$ belongs to $L^1(0,T)$. We call $|u'|$ the metric derivative. The inequality in \eqref{eq:ACadmissable} is satisfied with $g=|u'|$. Additionally, $|u'|$ is minimal among all admissible functions $g$ in~\eqref{eq:ACadmissable}; i.e. if $g\in L^1(0,T)$ is non-negative and satisfies~\eqref{eq:ACadmissable}, then, for a.e. $t\in(0,T)$, $|u'|(t)\leq g(t)$.  
\end{theorem}
\begin{proof}
See \cite{greenbook}[Theorem 1.1.2].
\end{proof}

We now prove Proposition~\ref{prop:derivativelowbound}.

\begin{proof}[Proof of Proposition~\ref{prop:derivativelowbound}]
We assume for a proof by contradiction that the proposition fails. So there exists an $S\in (0,T)$ such that
$\displaystyle
\liminf_{n\to\infty} \int_0^S |u_n'|^2(s) \, ds < \int_0^S |u'|^2(s) \, ds. 
$
Let $\tilde L\geq 0$ be such that $\displaystyle \liminf_{n\to\infty} \int_0^S |u_n'|^2(s) \, ds < \tilde L < \int_0^S |u'|^2(s) \, ds$. Take $\varepsilon:= \frac12\left(\int_0^S |u'|^2(s) \, ds - \tilde L \right)$ and $L:=\tilde L+\varepsilon$. Then, by definition of the limit inferior, there exists a strictly increasing sequence $(n_k)_{k\in\mathbb{N}}\subset \mathbb{N}$ such that, for all $k\in\mathbb{N}$, 
$$\int_0^S |u_{n_k}'|^2(s) \,ds<\widetilde{L}+\varepsilon=L<\int_0^S|u'|^2(s) \,ds.$$
Since $(|u_{n_k}'|)_{k\in\mathbb{N}}$ is bounded in $L^2(0,S)$ we can assume, after potentially passing to another subsequence, that $(|u_{n_k}'|)$ converges weakly in $L^2(0,S)$ to a function $g\in L^2(0,S)$. Then, for all $r,t\in (0,S)$ with $r\leq t$,
\begin{align*}
\Vert u(t) - u(r) \Vert_{\infty} &= \lim_{k\rightarrow \infty} \Vert u_{n_k}(t)-u_{n_k}(r) \Vert_{n_k} 
\leq \lim_{k\rightarrow \infty} \int_r^t |u_{n_k}'|(s)\,ds 
= \int_r^t g(s)\,ds. 
\end{align*}
The first equality holds by property~(iv) in Definition~\ref{BanStack} and the second holds by weak convergence of $(|u_{n_k}'|)$ to $g$, since the indicator function of the interval $(r,t)$ is in $L^2(0,S)$. The inequality is a consequence of \eqref{eq:ACadmissable} with $g=|u_{n_k}'|$, as per Theorem~\ref{thm:metderivative}. To justify the use of that theorem, we recall that, for all $k\in \mathbb{N}$, $|u_{n_k}'|$ denotes the metric derivative of $u_{n_k}$ with respect to the distance on $X_{n_k}$.

Since $L^2(0,S)$ is continuously embedded in $L^1(0,S)$ by H\"older's inequality and since $g$, being the weak-$L^2(0,S)$ limit of non-negative functions, is non-negative, $g$ is admissible in \eqref{eq:ACadmissable}. So, by Theorem~\ref{thm:metderivative}, $g\geq |u'|$ a.e. on $(0,S)$. We then deduce
\begin{align*}
   L\geq \liminf_{k\rightarrow \infty} \int_0^S |u_{n_k}'|^2(s)\,ds 
   \geq \int_0^S g^2(s)\,ds \geq \int_0^S |u'|^2(s)\,ds,
\end{align*}
where the second inequality holds since the $L^2(0,S)$ norm is lower semicontinuous with respect to weak-$L^2(0,S)$ convergence and because $\displaystyle \liminf_{k\to\infty} a_k^2 = \left(\liminf_{k\to\infty} a_k\right)^2$ for sequences $(a_k)_{k\in \mathbb{N}}$ of non-negative numbers $a_k$. This is a contradiction, hence we conclude that the result holds.
\end{proof}

\section{Examples of Banach stackings}\label{sec:BanStackEx}

In this appendix we list a number of examples of the Banach stacking structure introduced in Definitions~\ref{BanStack} and~\ref{def:BanStackI}. 

\begin{example}\label{BSexmpl}
\begin{enumerate}

    \item Let $Z$ be a Banach space and let $(Z_n)_{n\in\mathbb{N}_\infty}$ be a sequence of closed linear subspaces of $Z$. For all $n\in \mathbb{N}_\infty$, let $\iota_n$ denote the inclusion map of $Z_n$ into $Z$. Assume that for all $x\in Z_\infty$ there exists a sequence $(x_n)_{n\in \mathbb{N}} \subset_n Z_n$ such that $x_n\rightarrow x$ in $Z$ as $n\rightarrow \infty$. Then $\big((Z_n,\iota_n)_{n\in\mathbb{N_\infty}},Z\big)$ is a Banach stacking. 

    \item Let $I$ be a topological space, $X$ a Banach space with norm $\Vert \cdot \Vert$, and $d$ a metric on $I\times X$ such that
    \begin{itemize}
        \item the induced metric topology of $(I\times X,d)$ is equivalent to the product topology on $I\times X$; and,
        \item for all $i\in I$ and all $x,y\in X$, $d\big((i,x),(i,y)\big)\leq \Vert x-y\Vert$.
    \end{itemize}
    Define, for all $i\in I$, the map $\xi_i: X \rightarrow I\times X$ by $\xi_i(x):=(i,x)$. Then ${\big( (X,\xi_i)_{i\in I},I\times X\big)}$ is a Banach stacking.
    
    \item Let $(\pi,E,M)$ be a smooth vector bundle of rank $k$ over $M$ (see \cite[Chapter 10]{Lee2013}), with $M$ (and therefore also $E$) connected. In particular, $M$ is a connected $C^\infty$ manifold called the base space, $E$ is an $C^\infty$-manifold called the total space, $\pi:E\rightarrow M$ is a $C^\infty$ and surjective map called the projection map, and, for all $p\in M$, $\pi^{-1}(p)$ is a $k$-dimensional vector space. 
    
    Let $U$ be an open subset of $M$ in the usual manifold topology. We will construct a Banach stacking in which the fibres $(\pi^{-1}(p))_{p\in U}$ are the Banach spaces. To this end we need to equip each fibre with a norm to each fibre and the total space $E$ with a metric. To do so we require a choice of local frame, which is equivalent to a choice of local trivialisation (\cite[Chapter 10]{Lee2013}).

 Let $g$ be a Riemannian metric \cite[Chapter 13]{Lee2013} on $M$ (which takes finite values by connectedness of $M$) and let $\Phi:\pi^{-1}(U)\rightarrow U \times \mathbb{R}^k$ be a local trivialisation; thus $\pi \circ \Phi^{-1}$ is the projection from $U\times \mathbb{R}^k$ onto $U$ and, for all $q\in U$, the restriction $\Phi|_{\pi^{-1}(q)}$ is an isomorphism between the vector spaces $\pi^{-1}(q)$ and $\{q\}\times \mathbb{R}^k$. Next construct a $C^\infty$ local frame $\sigma_i(p):=\Phi^{-1}(p,e_i)$, for $1\leq i \leq k$ and $p\in U$. Here $e_1, \ldots ,e_k$ are the standard basis vectors of $\mathbb{R}^k$. This enables us to define, for all $p\in U$, a norm on $\pi^{-1}(p)$ as follows. Let $v\in \pi^{-1}(p)$, then there exist unique real coefficients $(v_i)_{1\leq i\leq k}$ such that $v=\sum_{i=1}^k v_i\sigma_i(p)$. Define $ \Vert v \Vert_p := \left( \sum_{i=1}^k |v_i|^2 \right)^{1/2}$.
    
    Denote by $h$ the standard Euclidean Riemannian metric on $\mathbb{R}^k$, then $\Phi^*(g|_U\oplus h)$ is a Riemannian metric on $\pi^{-1}(U)$. Here $g|_U\oplus h$ is the product metric on $U\times \mathbb{R}^k$ as in \cite[Example 13.2]{Lee2013}, $\Phi^*$ denotes the pullback by $\Phi$ \cite[Chapter 12, Proposition 13.9]{Lee2013}, and $g|_U$ is the restriction of $g$ to the submanifold $U$.
    
    Let $d$ denote the distance function on $\pi^{-1}(U)$ that is induced by $\Phi^*(g|_U\oplus h)$. This makes $\pi^{-1}(U)$ into a metric space. Let, for all $p\in U$, $\iota_p:\pi^{-1}(p)\rightarrow \pi^{-1}(U)$ be the inclusion map. One can then verify that $\big( (\pi^{-1}(p),\iota_p)_{p\in U} , \pi^{-1}(U)\big)$ is a Banach stacking. This can be seen by utilizing the local trivialisation $\Phi$ to reduce us to example 2.

    \item Let $\mathcal{I}$ denote the space of $d\times d$ symmetric positive-definite real matrices and equip $\mathcal{I}$ with the standard operator norm $\Vert x \Vert_{op} := \sup_{\substack{v\in \mathbb{R}^d\\x\neq 0}} \frac{\|Ax\|}{\|x\|}$, where $\|\cdot\|$ is the Euclidean norm on $\mathbb{R}^d$.  Given $A\in \mathcal{I}$ we denote $H_A$ to be the Hilbert space $\mathbb{R}^d$ equipped with the inner product $\langle x,y \rangle_A := x^TAy$, where $x$ and $y$ are interpreted as column vectors and superscript $T$ denotes transposition.
    
    Take $E$ to be $\mathbb{R}^d$ equipped with the standard Euclidean inner product. For each $A\in \mathcal{I}$ define the map $\xi_A: H_A \rightarrow E$ by
    $ \xi_A(x) := A^{1/2}x.$

    Then $\big( (H_A,\xi_A)_{A\in \mathcal{I}} , E)$ forms a Banach stacking, where we equip the index set $\mathcal{I}$ with the operator norm topology.

\item For an overview of Sobolev spaces and their properties from the perspective of Fourier transforms, useful for this example, we refer the reader to \cite[Chapter 9.3]{Folland99}.

We denote the Fourier transform (\cite[Chapter 8]{Folland99}) of a function $f$ by $\widehat{f}$ and the complex conjugate of $z\in\mathbb{C}$ by $\overline{z}$.   
Let $d\in\mathbb{N}$, $s\in \mathbb{R}$, and denote by $H^s(\mathbb{R}^d)$ the Sobolev space of square-Lebesgue-integrable functions on $\mathbb{R}^d$ for which the norm induced by the inner product
$$  \langle u,v\rangle_s:= \int_{\mathbb{R}^d} \widehat{u}(\xi) \, \overline{\widehat{v}(\xi)} \, (1+|\xi|^2)^{s} \, d\xi$$
is finite.

For a function $f$ we denote the inverse Fourier transform by $f^{\vee}$.  Define the operator $\Lambda_{s}:H^{s}(\mathbb{R}^d)\rightarrow L^2(\mathbb{R}^d)$ by
$$ \Lambda_{s}u := \left((1+|\xi|^2)^{s/2} \widehat{u}\right)^{\vee}.$$
Then $\big((H^s(\mathbb{R}^d), \Lambda_s)_{s\in \mathbb{R}}, L^2(\mathbb{R}^d)\big)$ is a Banach stacking, where we equip the index set $\mathbb{R}$ with the standard topology. From the basic properties of the Fourier transform most axioms of a Banach stacking are easy to verify; we leave this to the reader. The one condition that is more difficult to check is property~(ii) from Definition~\ref{BanStack}. To see that it too is satisfied, let $s_\infty\in \mathbb{R}$, $u_\infty\in H^s(\mathbb{R}^d)$, and let $(s_n)_{n\in\mathbb{R}}$ be a sequence of real numbers such that $s_n\rightarrow s_\infty$ as $n\to\infty$.   
Define the kernel,
$$ S(t,x):= t^{n/2}e^{-\pi |x|^2t}$$
and, for all $n\in \mathbb{N}$, $\sigma_n:= \frac{2\pi}{1/n +|s_n-s_\infty|}$. Property~(ii) in Definition~\ref{BanStack} requires the existence of an approximating sequence $(u_n)_{n\in\mathbb{N}} \subset_n H^{s_n}(\mathbb{R}^d)$ such that $(\Lambda_{s_n} u_n)_{n\in\mathbb{N}}$ converges to $\Lambda_s u_\infty$ in $L^2(\mathbb{R}^d)$. We will show that such a sequence is given by $u_n:= S(\sigma_n,\cdot)*u$. 

First we observe that indeed $u_n\in H^{s_n}(\mathbb{R}^d)$. To prove the required convergence, it suffices, by Parseval's equality, to show that $\left((1+|\xi|^2)^{s_n/2} \widehat{u_n}\right)_{n\in\mathbb{N}}$ converges to $(1+|\xi|^2)^{s/2} \widehat{u}$ in $L^2(\mathbb{R})^d$. By \cite[Theorem 8.22(c) and Proposition 8.24]{Folland99}, we have
$$ \widehat{u_n}(\xi)= \widehat{S}(\sigma_n,\xi)\widehat{u}(\xi) = e^{-\pi|\xi|^2/\sigma_n}\widehat{u}(\xi)$$
and thus
\begin{align*}
I_n(\xi) &:= \left|(1+|\xi|^2)^{s/2} \widehat{u}-(1+|\xi|^2)^{s_n/2} \widehat{u_n}\right|\\
&= \left|(1+|\xi|^2)^{s/2} \widehat{u}\right| \, \, \left|1-(1+|\xi|^2)^{(s_n-s)/2}e^{-\pi|\xi|^2/\sigma_n}\right|.
\end{align*}
It remains to show that $\int_{\mathbb{R}^d} I_n^2(\xi)d\xi\rightarrow 0$ as $n\rightarrow \infty$, which we will prove by the dominated convergence theorem. For all $\xi \in \mathbb{R}^d$, $I_n(\xi) \to 0$ as $n\to\infty$, because $\sigma_n \to \infty$ as $n\to\infty$. Thus to use the dominated convergence theorem we need to bound $\left((1+|\xi|^2)^{(s_n-s)/2}e^{-\pi|\xi|^2/\sigma_n}\right)_{n\in\mathbb{N}}$ uniformly. Define, for all $r\geq 0$ and all $n\in\mathbb{N}$,
$$ f_{\lambda,\theta}(r):= (1+r^2)^{\lambda/2}e^{-\pi r^2/\theta}.$$
If $\lambda\leq 0$, then $f_{\lambda,\theta}$ is maximised at $r=0$. If $\lambda>0$ and $\theta \leq \frac{2\pi}{\lambda}$, it follows that $f_{\lambda,\theta}$ is again maximised at $r=0$. As a consequence we have $f_{s_n-s,\sigma_n}(r)$, for $r\geq 0$, is maximised at $r=0$.  
Hence
$$
(1+|\xi|^2)^{(s_n-s)/2}e^{-\pi|\xi|^2/\sigma_n} = f_{s_n-s,\sigma_n}(|\xi|)\leq f_{s_n-s,\sigma_n}(0)=1.     
$$
Thus, for all $n\in\mathbb{N}$, $I_n^2(\xi)\leq 4(1+|\xi|^2)^{s} |\widehat{u}(\xi)|^2$. The bound is integrable, thus, by the dominated convergence theorem, $\int_{\mathbb{R}^d} I_n^2(\xi)d\xi\rightarrow 0$ as $n\rightarrow \infty$.

\item This example is based on ideas from optimal transport and is important for this paper, in particular in Section~\ref{sec:TLp}. Fix an open subset $\Omega \subset \mathbb{R}^d$ and $1\leq p <+\infty$. For all $\mu \in \mathcal{P}_p(\Omega)$, define the map $\xi_{\mu}: L^p(\Omega;\mu) \rightarrow TL^p(\Omega)$ by
$\xi_{\mu}(u) := (u,\mu).$ 

Then $\big((L^p(\Omega;\mu), \xi_{\mu})_{\mu \in \mathcal{P}_p(\Omega)}, TL^p(\Omega)\big)$ defines a Banach stacking in the sense of Definition~\ref{def:BanStackI}where we equip the index set $\mathcal{P}_p(\Omega)$ with the topology induced by the $p$-Wasserstein metric.
\end{enumerate}

\medskip

We leave it to the reader to verify that the first five examples above indeed describe Banach stackings. The sixth example is the most relevant one for this paper and a proof of its correctness is given in Proposition~\ref{tlpBanach}. 
\end{example}

\section{\texorpdfstring{$\Gamma$}{Gamma}-convergence over a Banach stacking}\label{app:Gammaconvergence}

We briefly provide an overview of the main definitions and results that we need to make use of $\Gamma$-convergence. Definitions~\ref{def:Gammaconvergence} and~\ref{def:equicoercivity} and Theorem~\ref{fundGamma} are from \cite{DalMaso93}. In the remainder of Appendix~\ref{app:Gammaconvergence} we adapt them to the context of a Banach stacking.

\begin{definition}\label{def:Gammaconvergence} 
Let $(\mathcal{M},d)$ be a metric space and $(\Psi_n)_{n\in \mathbb{N}_{\infty}}$ a sequence of functions $\Psi_n:\mathcal{M} \rightarrow (-\infty,+\infty]$. We say that the sequence $(\Psi_n)_{n\in \mathbb{N}}$ $\Gamma$-converges to $\Psi_{\infty}$ as $n\rightarrow \infty$ if both of the following conditions hold.
\begin{itemize}
    \item For all sequences $(x_n)_{n\in \mathbb{N}_{\infty}} \subset \mathcal{M}$ such that $x_n \rightarrow x_{\infty}$ as $n\rightarrow \infty$, we have
    $$ \liminf_{n \rightarrow \infty} \Psi_n(x_n) \geq \Psi_{\infty}(x_\infty) .$$
    \item For all $x_{\infty}\in \mathcal{M}$ there exists a sequence $(x_n)_{n\in \mathbb{N}} \subset \mathcal{M}$ such that $x_n \rightarrow x_{\infty}$ as $n \rightarrow \infty$ and
    $$ \limsup_{n \rightarrow \infty} \Psi_n(x_n) \leq \Psi_{\infty}(x_{\infty}) .$$
    In this case we call $(x_n)_{n\in \mathbb{N}}$ a recovery sequence.
\end{itemize}
\end{definition}

\begin{remark}\label{rem:recovery}
    If $(\Psi)_{n\in\mathbb{N}}$ $\Gamma$-converges to $\Psi_\infty$ and $(x_n)_{n\in \mathbb{N}}$ is a recovery sequence in the sense of Definition~\ref{def:Gammaconvergence} converging to $x_\infty$, then in fact
    \[
\Psi_{\infty}(x_\infty) \leq \liminf_{n \rightarrow \infty} \Psi_n(x_n) \leq \limsup_{n \rightarrow \infty} \Psi_n(x_n) \leq \Psi_{\infty}(x_{\infty})
    \]
    and thus $\Psi_n(x_n) \to \Psi_\infty(x_\infty)$ as $n\to \infty$.
\end{remark}

\begin{definition}\label{def:equicoercivity}
Let $(\mathcal{M},d)$ be a metric space and $(\Psi_n)_{n\in \mathbb{N}}$ a sequence of functions $\Psi_n:\mathcal{M} \rightarrow (-\infty,+\infty]$. We say that the sequence $(\Psi_n)_{n\in \mathbb{N}}$ is equicoercive in $\mathcal{M}$ if for all $a\in \mathbb{R}$, the set $\cup_{n\in \mathbb{N}} \{x\in M\,|\,\Psi_n(x) \leq a\}$ is relatively compact in $\mathcal{M}$. 
\end{definition}

The following classical result illustrates the usefulness of $\Gamma$-convergence in minimization problems.

\begin{theorem} \label{fundGamma}

Let $(\mathcal{M},d)$ be a metric space and $(\Psi_n)_{n\in \mathbb{N}_{\infty}}$ a sequence of functions $\Psi_n:\mathcal{M} \rightarrow (-\infty,+\infty]$.
Suppose that the sequence $(\Psi_n)_{n\in \mathbb{N}}$ is equicoercive and $(\Psi_n)_{n\in \mathbb{N}}$ $\Gamma$-converges to $\Psi_\infty$. Additionally assume that, for all $n \in \mathbb{N}_{\infty}$, $\Psi_n$ has a unique minimizer $x^*_n$ in $\mathcal{M}$. Then $x^*_n \rightarrow x^*_{\infty}$ and $\Psi_n(x^*_n) \rightarrow \Psi(x^*_{\infty})$ as $n \rightarrow \infty$.
\end{theorem}
\begin{proof}
    See Corollary 7.24 in \cite{DalMaso93}.
\end{proof}

Next we adapt the definitions and result above to the context of a Banach stacking.

\begin{definition}
\label{BSdefs}
Suppose $\big((X_n, \xi_n)_{n\in \mathbb{N}_\infty}, \mathcal{M}\big)$ is a Banach stacking and $(\Phi_n)_{n\in \mathbb{N}_\infty}$ a sequence of functions $\Phi_n:X_n \rightarrow (-\infty,+\infty]$.
We say that $(\Phi_n)_{n\in\mathbb{N}}$ $\Gamma$-converges to $\Phi_\infty$ over the Banach stacking $\big((X_n, \xi_n)_{n\in \mathbb{N}_\infty}, \mathcal{M}\big)$ if both of the following conditions hold. 
\begin{itemize}
\item For all sequences $(x_n)_{n\in \mathbb{N}_{\infty}} \subset_n X_n$ such that $x_n \underset{\mathcal{M}}\longtwoheadrightarrow x_{\infty}$ as $n\rightarrow \infty$, we have
    $$ \liminf_{n \rightarrow \infty} \Phi_n(x_n) \geq \Phi_{\infty}(x_\infty) .$$
    \item For all $x_{\infty}\in X_\infty$, there exists a sequence $(x_n)_{n\in \mathbb{N}} \subset_n X_n$ such that $x_n \underset{\mathcal{M}}\longtwoheadrightarrow x_{\infty}$ as $n \rightarrow \infty$ and
    $$ \limsup_{n \rightarrow \infty} \Phi_n(x_n) \leq \Phi_{\infty}(x_{\infty}) .$$
    In this case we call $(x_n)_{n\in \mathbb{N}}$ a recovery sequence.    
\end{itemize}

Additionally we say that the sequence $(\Phi_n)_{n\in\mathbb{N}}$ is equicoercive over the Banach stacking $\big((X_n, \xi_n)_{n\in \mathbb{N}_\infty}, \mathcal{M}\big)$ if, for all $c\in \mathbb{R}$ and all sequences $(x_n)_{n\in\mathbb{N}}\subset_n X_n$ with $\Phi_n(x_n)\leq c$, there exists an $x_\infty\in X_\infty$ and a subsequence $(x_{n_k})_{k\in\mathbb{N}}$ of $(x_n)_{n\in\mathbb{N}}$ such that $x_{n_k}\underset{\mathcal{M}}\longtwoheadrightarrow x_{\infty}$ as $k \rightarrow \infty$.
\end{definition}

Our definition of equicoercivity over a Banach stacking in Definition~\ref{BSdefs} differs slightly from the definition over a metric space in Definition~\ref{def:equicoercivity}. We explain the reason for this difference in Remark~\ref{rem:whydifferentequicoercivity}. First we prove the analogous result to Theorem~\ref{fundGamma} in the Banach-stacking setting. We note that the observation about recovery sequences in Remark~\ref{rem:recovery} still holds in this setting.

\begin{proposition}
\label{prop:gammaconBS}
Suppose $\big((X_n, \xi_n)_{n\in \mathbb{N}_\infty}, \mathcal{M}\big)$ is a Banach stacking and $(\Phi_n)_{n\in \mathbb{N}_\infty}$ a sequence of functions $\Phi_n:X_n \rightarrow (-\infty,+\infty]$. Assume that $(\Phi_n)_{n\in \mathbb{N}}$ is equicoercive and $(\Phi_n)_{n\in \mathbb{N}}$ $\Gamma$-converges to $\Phi_\infty$ over the Banach stacking $\big((X_n, \xi_n)_{n\in \mathbb{N}_\infty}, \mathcal{M}\big)$. Additionally suppose that, for all $n \in \mathbb{N}_{\infty}$, $\Phi_n$ has a unique minimizer $x^*_n$ in $X_n$. Then $x^*_n \underset{\mathcal{M}}\longtwoheadrightarrow x^*_{\infty}$ and $\Phi_n(x^*_n) \rightarrow \Phi_{\infty}(x^*_{\infty})$ as $n\rightarrow \infty$.
\end{proposition}
\begin{proof}
Let $(z_n)_{n\in\mathbb{N}}\subset X_n$ be a recovery sequence such that $\Phi_n(z_n)\rightarrow \Phi_\infty(x_\infty^*)$ as $n\rightarrow \infty$. Since $\Phi_n(x^*_n)\leq \Phi_n(z_n)$ and $\Phi_\infty$ is proper,  $\limsup_{n\rightarrow \infty} \Phi_n(x_n^*) \leq \Phi_\infty(x_\infty^*)<+\infty$. Suppose for a proof by contradiction that $(x_n^*)_{n\in\mathbb{N}}$ does not converge to $x_\infty^*$ over $\big((X_n, \xi_n)_{n\in \mathbb{N}_\infty}, \mathcal{M}\big)$. In particular, there exist an $\varepsilon>0$ and a subsequence $(x_{n_k}^*)_{k\in\mathbb{N}}$ of $(x_n^*)_{n\in\mathbb{N}}$ such that, for all $k\in \mathbb{N}$, $d(\xi_{n_k}(x^*_{n_k}),\xi_\infty(x^*_\infty))\geq \varepsilon$. By equicoercivity of $(\Phi_n)_{n\in\mathbb{N}}$, there exists a $z_\infty \in X_\infty$ such that, up to passing to a further subsequence, $x^*_{n_k} \underset{\mathcal{M}}\longtwoheadrightarrow z_{\infty}$ as $k\rightarrow \infty$. By the inequality for the limit inferior, $\Phi_\infty(z_\infty)\leq \liminf_{k\rightarrow \infty} \Phi_{n_k}(x^*_{n_k})\leq \Phi_\infty(x_\infty^*)$. The minimizer of $\Phi_\infty$ is unique and so $z_\infty=x_\infty^*$, which gives a contradiction. Thus $x^*_n \underset{\mathcal{M}}\longtwoheadrightarrow x^*_{\infty}$ as $n\rightarrow \infty$. By the inequality for the limit inferior, $\liminf_{n\rightarrow \infty} \Phi_n(x_n^*)\geq \Phi_\infty(x_\infty^*)$; combined with the inequality for the limit superior this implies that $\Phi_n(x_n^*)\rightarrow \Phi_\infty(x_\infty^*)$ as $n\rightarrow\infty$.     
\end{proof}

\begin{remark}\label{rem:whydifferentequicoercivity}
Suppose $\big((X_n, \xi_n)_{n\in \mathbb{N}_\infty}, \mathcal{M}\big)$ is a Banach stacking and $(\Phi_n)_{n\in \mathbb{N}_\infty}$ a sequence of functions $\Phi_n:X_n \rightarrow (-\infty,+\infty]$. Mirroring Definition~\ref{def:equicoercivity}, we might reasonably think to define $(\Phi_n)_{n\in \mathbb{N}}$ to be equicoercive over the Banach stacking $\big((X_n, \xi_n)_{n\in \mathbb{N}_\infty}, \mathcal{M}\big)$ if, for any $c\in \mathbb{R}$, $\cup_{n\in\mathbb{N}} \xi_n(\Phi_n^{-1}(-\infty,c])$ is relatively compact in $\mathcal{M}$. However this turns out to be insufficient for Proposition~\ref{prop:gammaconBS} to hold, which is why we chose differently in Definition~\ref{BSdefs}. In fact, if we were to reformulate the definition of equicoercivity from Definition~\ref{BSdefs} for the metric-space setting of Definition~\ref{def:equicoercivity}, we would obtain an equivalent definition to that in Definition~\ref{def:equicoercivity}, but in the Banach-stacking setting the two do not coincide, as we will now illustrate.

First we define a map $\theta:\mathbb{R}\rightarrow [0,2\pi]$ by
$$ \theta(x):= \pi+ \frac{\pi x}{\sqrt{1+x^2}}.$$
Then
$$\theta'(x)= \frac{\pi\sqrt{1+x^2}-\frac{\pi x^2}{\sqrt{1+x^2}}}{1+x^2} = \frac{\pi(1+x^2)-\pi x^2}{(1+x^2)^{\frac32}}=\frac{\pi}{(1+x^2)^{\frac32}}\in [0,\pi].$$
Next we define the embedding $\xi:\mathbb{R}\rightarrow S^1$ via $\xi(x)=e^{i\theta(x)}$. Here $S^1 \subset \mathbb{C}$ is the unit circle in the complex plane, equipped with the angular distance metric divided by $\pi$, which we denote by $d$. Since we divided by $\pi$ and by the bounds on $\theta'$, we find that $\xi$ is a $1$-Lipschitz map. For all $n\in\mathbb{N}_\infty$, let $X_n:=\mathbb{R}$ and $\xi_n:=\xi$, and set $\mathcal{M}:=S^1$. Then $\big((X_n, \xi_n)_{n\in \mathbb{N}_\infty}, \mathcal{M}\big)$ is a Banach stacking; we leave it to the reader to verify this. 

Now we define a sequence $(\Phi_n)_{n\in\mathbb{N}}$ of functions $\Phi_n:\mathbb{R}\rightarrow [0,+\infty)$ by $\Phi_n(n):=0$ and, for all $x\neq n$, $\Phi_n(x):=|x|+1/n$. Additionally, we define $\Phi_\infty:\mathbb{R}\rightarrow [0,+\infty)$ by $\Phi_\infty(x):=|x|$. One can then verify that $(\Phi_n)_{n\in\mathbb{N}}$ $\Gamma$-converges to $\Phi_\infty$ over the Banach stacking. Moreover, $S^1$ is compact in $\mathbb{C}$ and thus every subset is relatively compact; hence $(\Phi_n)_{n\in\mathbb{N}}$ is ``equicoercive'' according to the ``mirrored'' definition at the start of this remark. However, for all $n\in \mathbb{N}$, the minimizer of $\Phi_n$ is $n$, and $(n)_{n\in\mathbb{N}}$ does not converge to the minimizer of $\Phi_\infty$, which is $0$.  
\end{remark}

We close this appendix with a lemma regarding a uniform lower bound for a sequence of $\lambda$-convex functions on Hilbert spaces which $\Gamma$-converges.

\begin{lemma}
\label{Lem:lambdagammabound}
Suppose $\big((H_n, \xi_n)_{n\in \mathbb{N}_\infty}, \mathcal{M}\big)$ is a Banach stacking of Hilbert spaces $H_n$ and $(\Phi_n)_{n\in \mathbb{N}_\infty}$ a sequence of functions $\Phi_n:H_n \rightarrow (-\infty,+\infty]$. Assume that there exists a $\lambda>0$ such that, for all $n\in \mathbb{N}$, $\Phi_n$ is $\lambda$-convex, proper, and lower semicontinuous, and, moreover, $\Phi_\infty$ is $\lambda$-convex and proper. Furthermore, assume that $(\Phi_n)_{n\in \mathbb{N}}$ $\Gamma$-converges to $\Phi_\infty$ over the Banach stacking. Then there exists an $L\in \mathbb{R}$ such that, for all $n\in \mathbb{N}_\infty$ and all $x\in H_n$, $\Phi_n(x)\geq L$.
\end{lemma}

\begin{proof}
The function $\Phi_\infty$, being a $\Gamma$-limit, is lower semicontinuous with respect to the topology on $H_\infty$ (see \cite[Definition 4.1 and Proposition 6.8]{DalMaso93}). By assumption it is also $\lambda$-convex and thus, by Proposition~\ref{ConvexMin}, $\Phi_{\infty}$ is bounded below and has a unique minimizer. So without loss of generality we can assume that $\Phi_{\infty}(0)=0$ and $\Phi_{\infty}\geq0$; indeed, if this is not the case, then we can replace each $\Phi_n$ by $\Phi_n(\cdot - x_\infty)-\alpha_\infty$ with $x_\infty := \argmin_{x\in H_\infty} \Phi_{\infty}(x)$ and $\alpha_\infty := \Phi_{\infty}(x_\infty)$. It can be checked directly from the definitions that this preserves the required $\lambda$-convexity, properness, and $\Gamma$-convergence.

For all $n\in \mathbb{N}$, let $x_n^*:=\argmin_{x\in H_n} \Phi_n(x)$ (which exists and is indeed unique according to Lemma~\ref{ConvexMin}) and $\alpha_n:= \Phi_n(x_n^*)$. By Corollary~\ref{cor:lambdaconvexbound} we have, for all $n\in \mathbb{N}$ and for all $x\in H_n$,
\begin{equation}
\label{coercivebound}
\Phi_n(x) \geq \frac{\lambda}{2}\Vert x-x_n^* \Vert_n^2 + \alpha_n .
\end{equation}
By our assumption of $\Gamma$-convergence, there exists a recovery sequence $(z_n)_{n\in \mathbb{N}}\subset_n H_n$ such that $z_n \rightarrow 0$ as $n\rightarrow \infty$ and $\Phi_n(z_n)\rightarrow \Phi_\infty(0)=0$ as $n\rightarrow 0$ (see Remark~\ref{rem:recovery}). Using the bound in~\eqref{coercivebound} we obtain, for all $n\in \mathbb{N}$,
\[
    \Vert x_n^* \Vert_n \leq \Vert x_n^* -z_n \Vert + \Vert z_n \Vert_n
    \leq \sqrt{\frac{\Phi_n(z_n)-\alpha_n}{\lambda/2}} + \Vert z_n \Vert_n .
\]
Now assume, for a proof by contradiction, that $(\Phi_n)_{n\mathbb{N}_{\infty}}$ is not uniformly bounded below. By passing to a subsequence we can assume WLOG $\alpha_n \rightarrow -\infty$ as $n\rightarrow \infty$. 
Define $t_n:= \frac{1}{1+|\alpha_n|^{2/3}}\in (0,1]$. Then
$$ \Vert t_n x_n^*\Vert_n \leq \frac{1}{1+|\alpha_n|^{2/3}} \left(\sqrt{\frac{\Phi_n(z_n)+|\alpha_n|}{\lambda/2}} + \Vert z_n \Vert_n  \right) \rightarrow 0.$$
Thus $t_n x_n^*+(1-t_n)z_n \rightarrow 0$ as $n\rightarrow \infty$. By the inequality for the limit inferior in our assumption of $\Gamma$-convergence it follows that
\begin{align*}
    0 &= \Phi_{\infty}(0)
    \leq \liminf_{n\rightarrow \infty} \Phi_n(t_nx_n^*+(1-t_n)z_n)
    \leq \liminf_{n\rightarrow \infty} t_n \Phi_n(x_n^*)+(1-t_n) \Phi_n(z_n) \\
    &= \liminf_{n\rightarrow \infty} \frac{\alpha_n}{1+|\alpha_n|^{2/3}}
    = -\infty.
\end{align*}
The second inequality follows from Lemma~\ref{lem:lambdaconvexanstrictlyconvex} and the second equality by definition of $\alpha_n$ and convergence of $(\Phi_n(z_n))$ to $0$.
This is a contradiction and so $(\Phi_n)_{n\mathbb{N}_{\infty}}$ is uniformly bounded below. \end{proof}

\bibliographystyle{plain} 
\bibliography{mybib}

\end{document}